\documentclass[hidelinks,12pt,reqno]{amsart}
\usepackage[dvipsnames]{xcolor}
\usepackage{xcolor,amssymb,enumerate,tikz-cd,hyperref, amsmath,mathrsfs,mathtools,enumitem,changepage,
mathdots,eucal,xspace,amsthm,parskip,cancel,
amsfonts,fge,xpatch,array,pgfplots,mathabx,stackrel,stmaryrd, 
blkarray, bigstrut, booktabs}
\usepackage{xurl}
\hypersetup{breaklinks=true}
\usetikzlibrary{3d,knots,patterns}
\usepackage[T1]{fontenc}
\usepackage{crimson}

\usepackage[margin=1in]{geometry}

\usetikzlibrary{decorations.markings}
\newcommand{\myarrowlength}{10pt}
\tikzset{mytip/.tip={Butt Cap[black, length=\myarrowlength, sep=-1.6pt]>[black]},
    std/.style={white, text=black, #1, decoration={transform={xshift=.5*\myarrowlength}, markings, mark=at position .5 with {\arrow{mytip}}}, postaction=decorate},
    myarrow/.default={}}

\newcommand{\rightloop}[1][]{%
  \begin{tikzpicture}[baseline={(0,-0.1)}]
    \draw[->, in=30, out=-30, looseness=6] (0,-0.1) to (0,0.1);
  \end{tikzpicture}%
}

\newcommand{\leftloop}[1][]{%
  \begin{tikzpicture}[baseline={(0,-0.1)}]
    \draw[->, in=-30+180, out=30+180, looseness=6] (0,-0.1) to (0,0.1);
  \end{tikzpicture}%
}

\definecolor{pastelpink}{rgb}{1.0, 0.82, 0.86}
\definecolor{pastelorange}{rgb}{1.0, 0.7, 0.28}
\definecolor{pastelblue}{rgb}{0.68, 0.78, 0.81}
\definecolor{pastelgreen}{rgb}{0.47, 0.87, 0.47}
\definecolor{pastelred}{rgb}{1.0, 0.41, 0.38}
\definecolor{pastelgray}{rgb}{0.81, 0.81, 0.77}
\definecolor{lightpastelpurple}{rgb}{0.69, 0.61, 0.85}
\definecolor{pastelbrown}{rgb}{0.51, 0.41, 0.33}

\ExplSyntaxOff

\newtheorem{thmX}{Theorem}

\newtheorem{corX}[thmX]{Corollary}

\theoremstyle{definition}
\newtheorem{defn}[subsubsection]{Definition}
\newtheorem*{defn*}{Definition}

\newtheorem{innercustomass}{Assumption}
\newenvironment{customass}[1]
  {\renewcommand\theinnercustomass{#1}\innercustomass}
  {\endinnercustomass}

\theoremstyle{theorem}
\newtheorem{prop}[subsubsection]{Proposition}
\newtheorem*{prop*}{Proposition}
\newtheorem{theorem}[subsubsection]{Theorem}
\newtheorem*{theorem*}{Theorem}
\newtheorem{lem}[subsubsection]{Lemma}
\newtheorem*{lem*}{Lemma}
\newtheorem{cor}[subsubsection]{Corollary}
\newtheorem*{cor*}{Corollary}
\newtheorem{conj}[subsubsection]{Conjecture}
\newtheorem*{conj*}{Conjecture}

\newtheorem*{ques*}{Question}

\newcommand{\Gm}{\textup{\textbf{G}}_m\xspace}
\newcommand{\Ga}{\textup{\textbf{G}}_a\xspace}
\newcommand{\BGm}{\textup{B}\textup{\textbf{G}}_m\xspace}

\newcommand{\Md}{\text{-}\textup{Mod}\xspace}

\newcommand{\rank}{\textup{rk}\xspace}

\newcommand{\Ag}{\text{-}\textup{Alg}\xspace}

\newcommand{\FactAg}{\text{-}\textup{FactAlg}\xspace}

\newcommand{\smprod}{\textstyle{\prod}}

\newcommand{\hatotimes}{\, \hat{\otimes}\, }

\newcommand{\defeq}{\vcentcolon=}

\usepackage[numbers]{natbib}

\newcommand{\protosbt}{\,\begin{picture}(-0.5,1)(-0.5,-2.2)\circle*{2.5}\end{picture}\ }
\newcommand{\sbt}{{\protosbt}}

\newcommand{\overbar}[1]{\mkern 1.5mu\overline{\mkern-1.5mu#1\mkern-1.5mu}\mkern 1.5mu}

\usepackage{xparse,subcaption,hyperref,cleveref}

\ExplSyntaxOn

\NewDocumentCommand{\massdefine}{m}
 {
  \clist_map_inline:nn { #1 }
   {
    \cs_new_protected:cpn { ##1 } { \operatorname{##1}\xspace }
   }
 }

\ExplSyntaxOff

\massdefine{Higgs,Bun,Set,Fun,Maps,loc,Hilb,QCoh,Coh,triv,
BM,Obj,Ext,Hom,Mor,Aut,pt,gr,Vect,Frac,im,Rep,Gr,corank,
coker,ad,cl,Res,act,Ind,End,ch,ev,tr,Perf,Pic,Ran,Sh,CoSh,id,Zhu,Mod,coMod,Alg,Art,Supp,colim,Stk,Fuk,
Cob,LocSys,Stab,Perv,IC,IH,cSh,fib,cofib,oblv,cone,cocone,Top,Loc,
Ann,Sym,Gal,IndCoh,LGr,HF,PreStk,fin,Whit,Dix,Vir,AbCat,Cat,Div,ComCoAlg,
fact,Lie,CoLie,CoAlg,HH,HN,Jac,CoHA,Hitch,BD,BV,vol,crit,Cone,Hol,Isom,Sch,reg,
PSL,PGL,PSp,PSO,Proj,ShvCat,Groth,dgCat,Conf,Fin,SES,FinSet,Prim,dR,Spf,
LieAlg,LieCoAlg,LieBiAlg,Aff,Tr,Disk,Disks,tors,Wh,CommAlg,CommCoAlg,PoisAlg,
PoisCoAlg,PoisBiAlg,Rib,BPS,PervCoh,corr,Tot,CE,codim,BiAlg,Comm,Pois,CoCommCoAlg,
BU,BSU,BConf,Ar,Diff,Crit,Fact,Harr,PSet,Inst,Tay,Spec,Op,Th,FactCat,FactAlg,FactCoAlg,CoFactCat,CoMod,FactBiAlg,Cf,surj}

\usepackage{xparse,subcaption,hyperref,cleveref}

\ExplSyntaxOn

\NewDocumentCommand{\massdefinetextup}{m}
 {
  \clist_map_inline:nn { #1 }
   {
    \cs_new_protected:cpn { ##1 } { \textup{##1}\xspace }
   }
 }

\ExplSyntaxOff

\massdefinetextup{BGL,BSp,BO,BSL,GL,SL,BG,BT,BH,ord,obj,Sp,BP,BSO,BOSp}

\ExplSyntaxOn

\NewDocumentCommand{\massdefinemathcal}{m}
 {
  \clist_map_inline:nn { #1 }
   {
    \cs_new_protected:cpx { ##1 } { \exp_not:N \mathcal { \tl_range:nnn { ##1 } { 1 } { -2 } } }
   }
 }

\NewDocumentCommand{\massdefinemathfrak}{m}
 {
  \clist_map_inline:nn { #1 }
   {
    \cs_new_protected:cpx { ##1 } { \exp_not:N \mathfrak { \tl_range:nnn { ##1 } { 1 } { -2 } } }
   }
 }
 
 \NewDocumentCommand{\massdefinetext}{m}
 {
  \clist_map_inline:nn { #1 }
   {
    \cs_new_protected:cpx { ##1 } { \exp_not:N \textup { \tl_range:nnn { ##1 } { 1 } { -2 } } }
   }
 }

\NewDocumentCommand{\textbftextup}{m}{%
  \textbf{\textup{#1}}%
}

 \NewDocumentCommand{\massdefinetextbf}{m}
 {
  \clist_map_inline:nn { #1 }
   {
    \cs_new_protected:cpx { ##1 } { \exp_not:N \textbftextup { \tl_range:nnn { ##1 } { 1 } { -2 } } }
   }
 }
 
\ExplSyntaxOff

\massdefinemathcal{Al,Bl,Cl,Dl,El,Fl,Gl,Hl,Il,Jl,Kl,Ll,Ml,Nl,Ol,Pl,Ql,Rl,Sl,Tl,Ul,Vl,Wl,Xl,Yl,Zl}
\massdefinemathfrak{Ak,Bk,Ck,Dk,Ek,Fk,Gk,Hk,Ik,Jk,Kk,Lk,Mk,Nk,Ok,Pk,Qk,Rk,Sk,Tk,Uk,Vk,Wk,Xk,Yk,Zk,
ak,bk,ck,dk,ek,fk,gk,hk,ik,jk,kk,lk,mk,nk,ok,pk,qk,rk,sk,tk,uk,vk,wk,xk,yk,zk,
slk,glk,sok,spk,ospk}
\massdefinetext{At,Bt,Ct,Dt,Et,Ft,Gt,Ht,It,Jt,Kt,Lt,Mt,Nt,Ot,Pt,Qt,Rt,St,Tt,Ut,Vt,Wt,Xt,Yt,Zt,
at,bt,ct,dt,et,ft,gt,jt,kt,lt,mt,nt,ot,qt,rt,st,ut,vt,wt,xt,yt,zt,sst,SSt,SUt,SLt,GLt,HHt,SOt,Spt,CEt,SpOt,OSpt,HCt,CSt,MCt}
\massdefinetextbf{Ab,Bb,Cb,Db,Eb,Fb,Gb,Hb,Ib,Jb,Kb,Lb,Mb,Nb,Ob,Pb,Qb,Rb,Sb,Tb,Ub,Vb,Wb,Xb,Yb,Zb,
ab,bb,cb,db,eb,fb,gb,hb,ib,jb,kb,lb,mb,nb,ob,pb,qb,rb,tb,ub,vb,wb,xb,yb,zb,Grb}

\makeatletter   
\xpatchcmd{\@tocline}
{\hfil\hbox to\@pnumwidth{\@tocpagenum{#7}}\par}
{\ifnum#1<0\hfill\else\dotfill\fi\hbox to\@pnumwidth{\@tocpagenum{#7}}\par}
{}{}
\makeatother    

\makeatletter
\def\l@subsection{\@tocline{2}{-9pt}{4pc}{6pc}{}}
\def\l@subsubsection{\@tocline{3}{-9pt}{8pc}{8pc}{}}
 \makeatother

\usetikzlibrary{decorations.pathmorphing,shapes}
\usetikzlibrary{arrows.meta}
\tikzcdset{arrow style=tikz,
    squigarrow/.style={
        decoration={
        snake, 
        amplitude=.4mm,
        segment length=2mm
        }, 
        rounded corners=.2pt,
        decorate
        }
    }

\begin{document}

\author{Samuel DeHority and Alexei Latyntsev}
\title{Orthosymplectic modules of cohomological Hall algebras}
\maketitle
\begin{adjustwidth}{20pt}{20pt}
\small{\textsc{Abstract}: We study modules and comodules for cohomological Hall algebras equipped with their vertex coproducts arising as objects with classical type stabiliser groups. Specifically we consider how classical type parabolic induction gives rise to actions of CoHAs of quivers with potential, of preprojective algebras, and of dimension zero sheaves on a smooth proper surface. In all cases the CoHA action is compatible with a localised (and vertex) coaction making the module a twisted Yetter-Drinfeld module over the CoHA with its localised braided bialgebra structure. In the case of dimension zero sheaves on a surface the action is related to an approach to the AGT conjecture in classical type using moduli stacks of orthosymplectic perverse coherent sheaves, a compactification of the stack of classical type bundles on a surface. }
\end{adjustwidth}

\tableofcontents

\setcounter{section}{-1}
\section{Introduction}

\noindent
\textit{Parabolic induction} in type $A$, described by the diagram
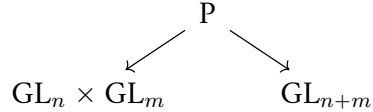
\begin{center}
\begin{tikzcd}[row sep = {30pt,between origins}, column sep = {45pt, between origins}]
 & \Pt \ar[rd] \ar[ld]  & \\ 
 \GL_n \times \GL_m& & \GL_{n+m}
\end{tikzcd}
\end{center}
underlies the multiplication of multiplication in various flavours of \textit{Hall algebras}, and their interpretation as \textit{Eisenstein series}, \textit{shuffle products}, \textit{free field realisations} of vertex algebras, etc. 
 
In this paper, we study the Hall algebra implications of parabolic induction in classical types. In various examples like quivers, preprojective algebras, and coherent sheaves on surfaces, the parabolic induction diagram is replaced by the diagram
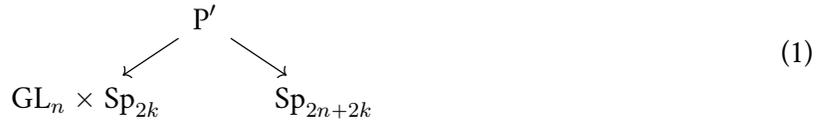
\begin{equation}
  \label{fig:SpParabolicInduction} 
  \begin{tikzcd}[row sep = {30pt,between origins}, column sep = {45pt, between origins}]
   & \Pt' \ar[rd] \ar[ld]  & \\ 
   \GL_n \times \Sp_{2k}& & \Sp_{2n+2k}
  \end{tikzcd}
\end{equation}
We define \textit{modules} over the Hall algebra via such parabolic induction diagrams.

The main idea is to interpret \eqref{fig:SpParabolicInduction} as \textit{fixed points} of an involution acting on a three-step parabolic induction diagram. 

Cohomological Hall algebras are known to have compatible \textit{vertex} or \textit{localised} coalgebraic structures, where the singularities of the coproduct are related to the diagonals in type $A$ configuration spaces $\Conf \Ab^1=\sqcup \Ab^n//\Sk_n=\Spec \Ht^\sbt(\BGL)$
\begin{center}
  \begin{tikzcd}[row sep = {30pt,between origins}, column sep = {65pt, between origins}]
   &(\Conf \Ab^1 \times \Conf \Ab^1)_\circ \ar[ld] \ar[rd]   & \\ 
   \Conf \Ab^1 \times\Conf \Ab^1& & \Conf \Ab^1
  \end{tikzcd}
  \end{center}
  We find an analogous vertex comodule structure living over the classical type configuration spaces, for instance $\Conf_{\Sp} \Ab^1=\Spec \Ht^\sbt(\BSp)$, where the singularities of the coaction are related to the root hyperplanes.

 Of independent interest, our categorical fixed point contruction of moduli stacks gives a compatification $\Ml^{\OSpt}$ of classical-type bundles on a surface, and we obtain a representation of the cohomological Hall algebra of dimension zero sheaves on the surface on $\Ht^{\BM}_\sbt(\Ml^{\OSpt})$. $\Ml^{\OSpt}$ consists of perfect complexes $\El$ with an isomorphism $\eta:\El \simeq \Dt\El$ to its derived dual satisfying a cocycle and stability condition. When $\El$ is a vector bundle, $\eta$ induces a symmetric form; an antisymmetric form arises from a similar construction.

\subsection{Orthosymplectic vertex algebras and cohomological Hall modules}

\subsubsection{} Let $\Ml$ be the moduli stack of objects in an \textit{orthosymplectic} abelian category: an abelian category $\Al$ equipped with an antiinvolution $\Al\simeq \Al^{\textup{op}}$. It carries a canonical involution 
$$\tau \ : \ \Ml \ \stackrel{\sim}{\to} \ \Ml$$
and the results of this paper concern how the structures on $\Ml$ interact with those on its stacky fixed locus $\Ml^\tau$.

\begin{thmX} \label{thmX:CoHAVA} \emph{(Theorems \ref{thm:coham}, \ref{thm:OrthosymplecticJoyceLiuLocalised}, \ref{cor:OrthosymplecticJoyceLiuVertex})}
  For any invariant function $W$ on a stack $\Ml$ satisfying assumptions \ref{ass:StkCoHAM} or \ref{ass:StkVA} respectively, there is a module structure 
  $$m_{\GL \textup{-}\OSpt} \ : \  \Ht^\sbt(\Ml, \varphi) \otimes \Ht^\sbt(\Ml^\tau, \varphi^\tau) \ \to \ \Ht^\sbt(\Ml^\tau, \varphi^\tau)$$
  and an orthosymplectic vertex comodule structure 
  \begin{align*}
    \Delta_{\GL \textup{-}\OSpt} \ : \  \Ht^\sbt(\Ml^\tau, \varphi^\tau) & \ \stackrel{}{\to} \ \Ht^\sbt(\Ml^\tau, \varphi^\tau) \, \hatotimes\, \Ht^\sbt(\Ml, \varphi)(((z\pm w)^{-1}))
  \end{align*}
where $\varphi^{(\tau)}$ is the vanishing cycle sheaf of $W\vert_{(\Ml^\tau)}$.
\end{thmX}
Setting $W=0$ gives Theorem \ref{thmX:CoHAVA} for the ordinary cohomologies $\Ht^\sbt(\Ml),\Ht^\sbt(\Ml^\tau)$. 

The vertex coproduct on $\Ht^\sbt(\Ml, \varphi)$ was originally defined in \cite{Jo,Li,JKL} and the cohomological Hall product in \cite{KS}. In special cases the vertex coaction was introduced by Bu in the dual setting of homology \cite{Bu1,Bu2}. The main examples of interest are where the critical locus $\Crit(W)$ is the moduli stack of objects for a Calabi-Yau-three category; such a moduli stack is always locally of this form. Our next result shows the above structures are compatible:

\begin{thmX} \emph{(Theorem \ref{thm:YD})} \label{thmX:YD}
  $\Ht^\sbt(\Ml^\tau,\varphi^\tau)$ forms a twisted Yetter Drinfeld vertex module for $\Ht^\sbt(\Ml,\varphi)$:
\begin{center}
  \begin{tikzpicture}[scale = 0.9]
  \newcommand{\mydrawingsmall}{
  \draw[black, ultra thick] (-1.5-0.2,2) .. controls +(0,-0.75) and +(0,0.75) .. (0,0.2);

  \draw[black, line width=5pt] (0,-2) -- (0,2);
  }
  \mydrawingsmall

  \begin{scope}[xscale=-1]
      \mydrawingsmall
      \node[white] at (1.5+0.7,0) {a};
  \end{scope}

  \begin{scope}[yscale=-1]
      \mydrawingsmall
  \end{scope}

  \begin{scope}[xscale=-1, yscale=-1]
      \mydrawingsmall
  \end{scope}

  \fill[white,opacity=1] (0,-2) rectangle (1.5+0.8,2);
  \draw[black, line width=5pt] (0,-2) -- (0,2);

  \node[] at (3.5-1,0) {$\stackrel{\textup{\eqref{fig:YDConditionVertex}}}{=}$};


  \begin{scope} 
   [xshift=7cm] 

   \newcommand{\mydrawing}{%
   \draw[black, ultra thick] (-1.5-0.2,-2) -- (-1.5-0.2,-1);
   \draw[black, ultra thick] (-1.5-0.2,-2) .. controls +(0,0.25) and +(0,-0.25) .. (-1.5-0.7,-1.5);
   \draw[black, ultra thick] (-1.5-0.2,-2) .. controls +(0,0.25) and +(0,-0.25) .. (-1.5+0.3,-1.5);
   
   \draw[black, ultra thick] (-1.5-0.2,-1.5) -- (-1.5-0.2,-1);
   \draw[black, ultra thick] (-1.5-0.2,-1) -- (-1.5-0.2,-0.5);
   \draw[black, ultra thick] (-1.5+0.3,-1.5) -- (-1.5+0.3,-1);
   \draw[black, ultra thick] (-1.5-0.7,-1.5) -- (-1.5-0.7,0);
 
   \draw[white,line width=5pt] (-1.5+0.3,-1) .. controls +(0,1.5) and +(0,-1.5) .. (1.5-0.3,3-2);
   \draw[black, ultra thick] (-1.5+0.3,-1) .. controls +(0,1.5) and +(0,-1.5) .. (1.5-0.3,3-2);
   \draw[white,line width=5pt] (-1.5-0.2,-0.5) .. controls +(0,1) and +(0,-1) .. (0,2);
   \draw[black, ultra thick] (-1.5-0.2,-0.5) .. controls +(0,1) and +(0,-1) .. (0,2);
 
   \draw[white, line width=7pt] (0,-0.1) -- (0,0.1);
   }
   \mydrawing

   \begin{scope}[xscale=-1]
    \mydrawing
\end{scope}

   \begin{scope}[xscale=-1, yscale=-1]
       \mydrawing
   \end{scope}
   
   \begin{scope}[yscale=-1]
    \mydrawing
    
    \draw[white,line width=5pt] (-1.5-0.2,-0.5) .. controls +(0,1) and +(0,-1) .. (0,2);
    \draw[black, ultra thick] (-1.5-0.2,-0.5) .. controls +(0,1) and +(0,-1) .. (0,2);
    \end{scope}

   \fill[white,opacity=1] (0,-2) rectangle (1.5+0.8,2);
   \draw[black, line width=5pt] (0,-2) -- (0,2);
   \end{scope}
  \end{tikzpicture}
  \end{center}
  which expresses $\Delta_{\GL \textup{-}\OSpt}\cdot m_{\GL \textup{-}\OSpt}$ (on the left) as a different expression (on the right) where the coaction precedes the action.
\end{thmX}

The above diagram takes place in a \textit{braided module category} $(\Bl, \otimes_{\Al\Bl},\kappa)$: a category $\Bl$ with a monoidal action $\otimes_{\Al \Bl}  :  \Al \times \Bl  \to  \Bl$ of a braided monoidal category $\Al=(\Al, \otimes, \beta)$, and a natural transformation $\kappa: \otimes_{\Al \Bl}\to \otimes_{\Al \Bl}$ which together with $\beta$ satisfies the type B/C braid relations and act on tensor products
$$A_1 \otimes \cdots  \otimes A_k \otimes_{\Al \Bl}B \ \in \ \Bl$$
which we represent as $k+1$ strands, and the actions of $\beta,\kappa$ as braidings of the strands. See Appendix \ref{sec:Cherednik} for more details. A \textit{twisted Yetter Drinfeld module} over a bialgebra $A\in \Al$ is a $B\in \Bl$ together with an associative action $A \otimes_{\Al\Bl} B \to B$ and coassociative coaction $ B \to  A \otimes_{\Al\Bl}B$ satisfying the conditions of Theorem \ref{thmX:YD}, where we represent as a thin braid merging or splitting off from the heavy central braid.

The braided module category we use is modules over the cohomology with its cup product
$$\Al \ = \ \Ht^\sbt(\Ml)\Md, \hspace{15mm} \Bl \ = \ \Ht^\sbt(\Ml^\tau)\Md$$
where $\kappa=(\tau \times \id)$; geometrically this corresponds to the fact that the $A=\Ht^\sbt(\Ml, \varphi)$ and $B=\Ht^\sbt(\Ml^\tau, \varphi^\tau)$ factorise over the (orthosymplectic) configuration space. For example, in the quiver with potential case we may instead take the categories of modules over (localisations of) the tautological rings
$$\Lambda \ = \ k[\{c_{k}(\El_i)\}], \hspace{15mm} \Lambda^{\OSpt} \ = \ k[\{c_k(\El_{\overbar{\iota}}^{\OSpt})\}]$$
generated by the chern classes of the tautological vector bundles on the moduli stack and its fixed locus, labelled by $i\in Q_0$ and $\overbar{\iota}\in Q_0/\tau$. The pullback $\El_{\overbar{\iota}}^{\OSpt} \mapsto \El_{i} \oplus\El_{\overbar{\iota}}^{\OSpt}\oplus\El_i^\vee$ by the action of $\Ml$ on $\Ml^\tau$ defines a coaction of $\Lambda$ on $\Lambda^{\OSpt}$ hence a module category structure. Finally, we require localisation because the braidings $\beta$ and $\kappa$ are expressed as rational functions in chern classes; see section \ref{sec:Configuration} for more details.

\subsubsection{Construction sketch} The construction \cite{KS,Jo,Li,JKL} of the cohomological Hall product and vertex coproduct on $\Ht^\sbt(\Ml, \varphi)$ proceeds by considering the maps 
\begin{center}
  \centering
  \begin{minipage}{.5\textwidth} \centering
   \begin{tikzcd}[row sep = {15pt,between origins}, column sep = {45pt, between origins}]
     &[-3pt] &\SES\ar[rdd] \ar[ldd]  &&[-28pt] \\ 
     & & & & \\
    & \Ml  \times \Ml& & \Ml&
     \end{tikzcd}
  \end{minipage}%
  \begin{minipage}{.5\textwidth} \centering
   \begin{tikzcd}[row sep = {15pt,between origins}, column sep = {45pt, between origins}]
     &[-15pt]&[20pt] &[-35pt]\\ 
   &\Ml \times \Ml \ar[r,"\oplus"] &\Ml&\\
   && & 
    \end{tikzcd}
  \end{minipage}
  \end{center}
induced by the moduli stack $\SES$ of short exact sequences, we get a quasitrianguar vertex bialgebra structure on $\Ht^\sbt(\Ml, \varphi)$ because these maps form a \textit{lax bialgebra} in an appropriate category. See \cite{JKL,La1,Li} or section \ref{sssec:FinalRemarksBialg}. 

The orthosymplectic variants are constructuted by taking \textit{invariants} of the threefold correspondences
\begin{center}
  \centering
  \begin{minipage}{.5\textwidth} \centering
   \begin{tikzcd}[row sep = {15pt,between origins}, column sep = {45pt, between origins}]
     &[-3pt] &\SES_3\ar[rdd] \ar[ldd]  &&[-28pt] \\ 
     \Bigg(& & & & \Bigg)^\tau \\
    & \Ml \times \Ml \times \Ml& & \Ml&
     \end{tikzcd}
  \end{minipage}%
  \begin{minipage}{.5\textwidth} \centering
   \begin{tikzcd}[row sep = {15pt,between origins}, column sep = {45pt, between origins}]
     &[-7pt]&[20pt] &[-30pt]\\ 
   \Big(&\Ml \times \Ml \times \Ml \ar[r,"\oplus"] &\Ml& \Big)^\tau\\
   && & 
    \end{tikzcd}
  \end{minipage}
  \end{center}
where on the right $\SES_3$ parametrises three-step short exact sequences in $\Al$, and is acted upon by
$$a \ \subseteq \ b \ \subseteq \ c \hspace{10mm} \ \mapsto \ \hspace{10mm} c^* \ \subseteq \ b^* \ \subseteq \ a^*.$$
The Yetter-Drinfeld condition of Theorem \ref{thmX:YD} then follows from the fact that these maps endow $\Ml^\tau$ with a lax Yetter-Drinfeld module structure over $\Ml$.

These structures should be viewed as being attached to the following $\Zb/2$-equivariant cobordisms,\footnote{In the former, which is not commutative, we need to consider correspondences between labelled one-manifolds.} and considering other equivariant cobordisms gives the entire orthosymplectic vertex structure:
\begin{center}
  \centering
  \begin{minipage}{.5\textwidth} \centering
   
   \begin{tikzpicture}[scale=0.7]
     \begin{scope}
       [xshift=1.5cm]
       \draw (0,4) ellipse (0.5 and 0.25);
       \draw (-0.5,3.5) -- (-0.5,4);
       \draw (0.5,3.2) -- (0.5,4);
     \end{scope}
     \begin{scope}
       [xshift=-1.5cm]
       \draw (0,4) ellipse (0.5 and 0.25);
       \draw (-0.5,3.2) -- (-0.5,4);
       \draw (0.5,3.5) -- (0.5,4);
     \end{scope}
 \fill[pattern=north east lines, pattern color=black!30]
 (1.5-0.5,3.5) to[out=-90,in=90,looseness=0.8] (-0.5,0.5) to[out=90,in=-90,looseness=0.8]  (1.5-0.5,3.5) -- (1.5-0.5,4) arc (180:360:0.5 and 0.25)  (1.5+0.5,4) --(1.5+0.5,3.2) to[out=-90,in=90,looseness=0.8] (0.5,0.2) -- (0.5,0) -- (-0.5,0)  -- (-0.5,0.5) to[out=90,in=-90,looseness=0.8]  (1.5-0.5,3.5) -- cycle ;
   
 \fill[pattern=north east lines, pattern color=black!30]
 (-1.5-0.5,3.2) to[out=-90,in=90,looseness=0.8] (-0.5,0.2) to[out=90,in=-90,looseness=0.8]  (-1.5-0.5,3.2) -- (-1.5-0.5,4) arc (180:360:0.5 and 0.25)  (-1.5+0.5,4) --(-1.5+0.5,3.5) to[out=-90,in=90,looseness=0.8] (0.5,0.5) -- (0.5,0) -- (-0.5,0)  -- (-0.5,0.2) to[out=90,in=-90,looseness=0.8]  (-1.5-0.5,3.2) -- cycle ;
 
     \draw (-1.5-0.5,3.2) to[out=-90,in=90,looseness=0.8] (-0.5,0.2);
     \draw (-1.5+0.5,3.5) to[out=-90,in=90,looseness=0.8] (0.5,0.5);
     \draw (1.5-0.5,3.5) to[out=-90,in=90,looseness=0.8] (-0.5,0.5);
     \draw (1.5+0.5,3.2) to[out=-90,in=90,looseness=0.8] (0.5,0.2);
 
     \filldraw[white,opacity=1] (-0.49,0) rectangle (0.49,2.4);
     \draw[dashed] (0.5,0) arc (0:180:0.5 and 0.25);
     \filldraw[white,opacity=0.5] (-0.49,0) rectangle (0.49,1.9);
 
   \fill[pattern=north east lines, pattern color=black!30] (-0.5,4) -- (-0.5,0) arc (180:360:0.5 and 0.25) -- (0.5,4) arc (360:180:0.5 and 0.25);
 
     \draw (0,4) ellipse (0.5 and 0.25);
     \draw (-0.5,0) arc (180:360:0.5 and 0.25);
     \draw (-0.5,2.35) -- (-0.5,4);
     \draw (0.5,2.35) -- (0.5,4);
     \draw (-0.5,0) -- (-0.5,0.2);
     \draw (0.5,0) -- (0.5,0.2);
 
     \draw[dashed] (-0.1,0-0.25+0.14) -- (0-0.1,4-0.25-0.14) -- (0+0.1,4+0.25+0.14) ;
     \draw[<->] (0.7,4.5) -- (-0.7,4.5);
     \node[above] at (0,4.5) {$\tau$};
 

     \node[] at (3,2) {$\rightsquigarrow$};
 
 \begin{scope} 
   [xshift = 5cm] 
  
   \begin{scope}[rotate=-90]
   \begin{scope}
    [rotate = 90]
    \begin{scope}
      [xshift=1.5cm]
      \draw (0,4) ellipse (0.5 and 0.25);
      \draw (-0.5,3.5) -- (-0.5,4);
      \draw (0.5,3.2) -- (0.5,4);
    \end{scope}

    \draw (1.5-0.5,3.5) to[out=-90,in=90,looseness=0.8] (-0.5,0.5);
    \draw (1.5+0.5,3.2) to[out=-90,in=90,looseness=0.8] (0.5,0.2);
  
    \filldraw[white,opacity=1] (-0.5,0) rectangle (0.5,4);
  
  \fill[pattern=north east lines, pattern color=black!30]
  (1.5-0.5,3.5) to[out=-90,in=90,looseness=0.8] (-0.5,0.5) to[out=90,in=-90,looseness=0.8]  (1.5-0.5,3.5) -- (1.5-0.5,4) arc (180:360:0.5 and 0.25)  (1.5+0.5,4) --(1.5+0.5,3.2) to[out=-90,in=90,looseness=0.8] (0.5,0.2) -- (0.5,0) -- (-0.5,0)  -- (-0.5,0.5) to[out=90,in=-90,looseness=0.8]  (1.5-0.5,3.5) -- cycle ;

    \filldraw[white,opacity=0.8] (-0.5,0) rectangle (0.5,4);
    
    \draw[dashed,line width = 2pt] (0.5,0) arc (0:180:0.5 and 0.25);
  
  \fill[pattern=north east lines, pattern color=black!30] (-0.5,4) -- (-0.5,0) arc (180:360:0.5 and 0.25) -- (0.5,4) arc (360:180:0.5 and 0.25);
  
    \draw[line width = 2pt] (0,4) ellipse (0.5 and 0.25);
    \draw[line width = 2pt] (-0.5,0) arc (180:360:0.5 and 0.25);
    \draw (-0.5,0) -- (-0.5,4);
    \draw (0.5,2.3) -- (0.5,4);
    \draw[] (0.5,0) -- (0.5,0.5);

    \node[above] at (0,4.3) {$\Ml^\tau$};
    \node[above] at (1.5,4.3) {$\Ml$};
    \node[below] at (0,-0.4) {$\Ml^\tau$};
  
  \end{scope}
  \end{scope}
  \end{scope}

 \end{tikzpicture} 
  \end{minipage}%
  \begin{minipage}{.5\textwidth} \centering
   \begin{tikzpicture}[scale=0.7,rotate=180]
     \begin{scope}
       [xshift=1.5cm]
       \draw[dashed] (0.5,4) arc (0:180:0.5 and -0.25);
       \draw[] (0.5,4) arc (0:180:0.5 and 0.25);
       \draw (-0.5,3.5) -- (-0.5,4);
       \draw (0.5,3.2) -- (0.5,4);
     \end{scope}
     \begin{scope}
       [xshift=-1.5cm]
       \draw[dashed] (0.5,4) arc (0:180:0.5 and -0.25);
       \draw[] (0.5,4) arc (0:180:0.5 and 0.25);
       \draw (-0.5,3.2) -- (-0.5,4);
       \draw (0.5,3.5) -- (0.5,4);
     \end{scope}
     
     \draw[dashed] (0.5,4) arc (0:180:0.5 and -0.25);
     \draw[] (0.5,4) arc (0:180:0.5 and 0.25);
 
 \fill[pattern=north east lines, pattern color=black!30]
 (1.5-0.5,3.5) to[out=-90,in=90,looseness=0.8] (-0.5,0.5) to[out=90,in=-90,looseness=0.8]  (1.5-0.5,3.5) -- (1.5-0.5,4) arc (180:360:0.5 and -0.25)  (1.5+0.5,4) --(1.5+0.5,3.2) to[out=-90,in=90,looseness=0.8] (0.5,0.2) -- (0.5,0) -- (-0.5,0)  -- (-0.5,0.5) to[out=90,in=-90,looseness=0.8]  (1.5-0.5,3.5) -- cycle ;
   
 \fill[pattern=north east lines, pattern color=black!30]
 (-1.5-0.5,3.2) to[out=-90,in=90,looseness=0.8] (-0.5,0.2) to[out=90,in=-90,looseness=0.8]  (-1.5-0.5,3.2) -- (-1.5-0.5,4) arc (180:360:0.5 and -0.25)  (-1.5+0.5,4) --(-1.5+0.5,3.5) to[out=-90,in=90,looseness=0.8] (0.5,0.5) -- (0.5,0) -- (-0.5,0)  -- (-0.5,0.2) to[out=90,in=-90,looseness=0.8]  (-1.5-0.5,3.2) -- cycle ;
 
     \draw (-1.5-0.5,3.2) to[out=-90,in=90,looseness=0.8] (-0.5,0.2);
     \draw (-1.5+0.5,3.5) to[out=-90,in=90,looseness=0.8] (0.5,0.5);
     \draw (1.5-0.5,3.5) to[out=-90,in=90,looseness=0.8] (-0.5,0.5);
     \draw (1.5+0.5,3.2) to[out=-90,in=90,looseness=0.8] (0.5,0.2);
 
     \filldraw[white,opacity=1] (-0.49,0) rectangle (0.49,2.4);
     \draw[dashed] (0.5,0) arc (0:180:0.5 and 0.25);
     \filldraw[white,opacity=0.5] (-0.49,0) rectangle (0.49,1.9);
 
   \fill[pattern=north east lines, pattern color=black!30] (-0.5,4) -- (-0.5,0) arc (180:360:0.5 and 0.25) -- (0.5,4) arc (360:180:0.5 and -0.25);
 
     \filldraw[draw=black,fill=white] (0,0) ellipse (0.5 and 0.25);
     \draw (-0.5,2.35) -- (-0.5,4);
     \draw (0.5,2.35) -- (0.5,4);
     \draw (-0.5,0) -- (-0.5,0.2);
     \draw (0.5,0) -- (0.5,0.2);
 
     \draw[dashed] (-0.1,-0.1-0.25+0.14) -- (0+0.1,-0.1+0.25+0.14) -- (0+0.1,4+0.14);
     \draw[<->] (0.7,-0.5) -- (-0.7,-0.5);
     \node[above] at (0,-0.5) {$\tau$};

     \node[] at (-3,2) {$\rightsquigarrow$};
 
 \begin{scope} 
   [xshift = -5cm,yscale=-1] 
  
   \begin{scope}[rotate=90]
   \begin{scope}
    [rotate = 90]

    \begin{scope}
     [xshift=1.5cm]
     \draw[dashed] (0.5,4) arc (0:180:0.5 and -0.25);
     \draw[] (0.5,4) arc (0:180:0.5 and 0.25);
     \draw (-0.5,3.5) -- (-0.5,4);
     \draw (0.5,3.2) -- (0.5,4);
   \end{scope}
   
   \draw[dashed] (0.5,4) arc (0:180:0.5 and -0.25);
   \draw[] (0.5,4) arc (0:180:0.5 and 0.25);
 
 \fill[pattern=north east lines, pattern color=black!30]
 (1.5-0.5,3.5) to[out=-90,in=90,looseness=0.8] (-0.5,0.5) to[out=90,in=-90,looseness=0.8]  (1.5-0.5,3.5) -- (1.5-0.5,4) arc (180:360:0.5 and -0.25)  (1.5+0.5,4) --(1.5+0.5,3.2) to[out=-90,in=90,looseness=0.8] (0.5,0.2) -- (0.5,0) -- (-0.5,0)  -- (-0.5,0.5) to[out=90,in=-90,looseness=0.8]  (1.5-0.5,3.5) -- cycle ;

   \draw (1.5-0.5,3.5) to[out=-90,in=90,looseness=0.8] (-0.5,0.5);
   \draw (1.5+0.5,3.2) to[out=-90,in=90,looseness=0.8] (0.5,0.2);
 
   \filldraw[white,opacity=1] (-0.49,0) rectangle (0.49,2.4);
   \draw[dashed] (0.5,0) arc (0:180:0.5 and 0.25);
   \filldraw[white,opacity=0.5] (-0.49,0) rectangle (0.49,1.9);
 
 \fill[pattern=north east lines, pattern color=black!30] (-0.5,4) -- (-0.5,0) arc (180:360:0.5 and 0.25) -- (0.5,4) arc (360:180:0.5 and -0.25);
 
   \filldraw[draw=black,fill=white] (0,0) ellipse (0.5 and 0.25);
   \draw (-0.5,0) -- (-0.5,4);
   \draw (0.5,2.35) -- (0.5,4);
   \draw (-0.5,0) -- (-0.5,0.2);
   \draw (0.5,0) -- (0.5,0.2);

   \draw[dashed,line width = 2pt] (0.5,4) arc (0:180:0.5 and -0.25);
   \draw[line width = 2pt] (0,0) ellipse (0.5 and 0.25);
   \draw[line width = 2pt] (-0.5,4) arc (180:360:0.5 and -0.25);

    \node[below] at (0,4.3) {$\Ml^\tau$};
    \node[below] at (1.5,4.3) {$\Ml$};
    \node[above] at (0,-0.4) {$\Ml^\tau$};
  \end{scope}
  \end{scope}
  \end{scope}
 
 \end{tikzpicture} 
  \end{minipage}
  \end{center}

\subsection{Axiomatics} \label{ssec:Axiomatic}

\subsubsection{} It has been known for some time that the cohomology of moduli stacks carry vertex structures \cite{Jo}. In this paper we will deduce this statement and its orthosymplectic analogues from a richer structure--that the (critical) cohomology of moduli stacks $\Ht^\sbt(\Ml, \varphi)$ lifts to a quasicoherent factorisation coalgebra over the \textit{$\Ml$-configuration space}:
$$\Conf_\Ml \Ab^1 \ = \ \Spec \Ht^\sbt(\Ml).$$
This should be thought of as parametrising all possible coalgebraic structures $\Ht^\sbt(\Ml, \varphi)$ has, and is compatible with the CoHA. See \cite{JKL}, and 

\begin{thmX} \label{thmX:Vertex}
  \emph{(Theorem \ref{thm:OrthosymplecticJoyceLiuLocalised})} Viewed as quasicoherent sheaves, $\Ht^\sbt(\Ml,\varphi)$ and $\Ht^\sbt(\Ml^\tau, \varphi^\tau)$ form a braided cocommutative factorisation coalgebra and comodule over $\Conf_{\Ml^{(\tau)}} \Ab^1$.
\end{thmX}

To relate this configuration space data to vertex coalgebras,
$$\otimes \ : \ \Vect \times \Cl \ \to \ \Cl$$
we consider the action on our abelian category of (symplectic) vector spaces, 
$$\GL\times \Ml \ \to \ \Ml,\hspace{15mm} \OSpt \times \Ml^{\tau} \ \to \ \Ml^{\tau}$$ 
and pull back our quasicoherent sheaves along the induced map
$$\Gb_{a,z} \times \Conf_{\Ml^{(\tau)}} \Ab^1 \ \to \ \Conf_{\Ml^{(\tau)}} \Ab^1.$$
Loosely speaking, our functor $\Phi$ is constructed by taking this translation map on the configuration space, then in the limit of translating by $z\to \infty$ all the hyperplanes in the configuration space are removed, i.e. the resulting structure on the punctured formal neighbourhood of $z=\infty$ gives a vertex coproduct with singularities in $z^{-1}$ rather than the hyperplanes. 
\begin{center}
\begin{tikzpicture}

  \begin{scope} 
   [scale=0.3] 
    \filldraw[black, fill=black, fill opacity = 0.2, pattern=north west lines] (-6,0) -- (-1,2) --  (6,0) -- (1,-2) -- cycle;
     \draw[black, line width = 1.5pt] (-1,2) -- (1,-2);
     \draw[black, line width = 1.5pt] (-6,0) -- (6,0);

     \draw[->] (-6,0) -- (1+1.4,-2-0.4);
     \draw[->] (-6,0) -- (-1+1.0,2+0.4);
     \draw[->] (-6,0) -- (-6,11);

     \node[left] at (-6,11) {$z$};
      \node[below] at (1+1.4,-2-0.4) {};
      \node[above] at (-1+1.0,2+0.4) {};
      \node[] at (9.5,0) {$\Conf_{\Gb} \Ab^1$};
      \node[left] at (-6.5,5) {$\Ga$};

     \foreach \x/\y in {-6/0, -1/2, 6/0, 1/-2}
      {\draw[] (\x,\y) -- (\x,\y+10);};

    \filldraw[black, fill=black, fill opacity = 0.2, pattern=north west lines] (-6,0) -- (-1,2) --  (6,0) -- (6,10) -- (-1,2+10) -- (-6,10) -- cycle;

   \begin{scope} 
    [yshift=10cm] 
    \filldraw[black, fill=black, fill opacity = 0.2, pattern=north west lines] (-6,0) -- (-1,2) --  (6,0) -- (1,-2) -- cycle;
    \clip (-6,0) -- (-1,2) --  (6,0) -- (1,-2) -- cycle;
    \begin{scope} 
     [xshift=1.5cm,yshift=-0.5cm] 
     \draw[black, line width = 1.5pt] (-2,4) -- (2,-4);
     \draw[black, line width = 1.5pt] (-6,0) -- (6,0);
     \end{scope}
    \end{scope}

   \end{scope}

 \end{tikzpicture}
 \end{center}

\begin{thmX}\label{thmX:LocalisedToVertex} \emph{(Theorem \ref{thm:LocalisedCoproducts})}
  There is a functor $\Phi$ (and $\varphi_{\OSpt}$) from equivariant factorisation coalgebras on $\Conf_{\Ml^{(\tau)}}\Ab^1$ to (orthosymplectic) vertex coalgebras. 
\end{thmX}

Additionally, we prove in Theorem \ref{thm:LocalisedBiAlgebras} that these functors preserve the localised coaction and CoHA action; Theorem \ref{thmX:YD} may thus be proven over the configuration space, without working with vertex algebras directly.  Combining Theorems \ref{thmX:Vertex} and \ref{thmX:LocalisedToVertex} gives the following.
\begin{corX} \emph{(Corollary \ref{cor:OrthosymplecticJoyceLiuVertex})} \label{corX:OSpGLVertex}
  There are maps
  $$\Delta_{\GL \textup{-}\OSpt}(z,w) \ : \ \Ht^\sbt(\Ml^\tau, \varphi^\tau) \ \to \ \Ht^\sbt(\Ml, \varphi)\, \hat{\otimes}\, \Ht^\sbt(\Ml^\tau, \varphi^\tau)(((z\pm w)^{-1}))$$
  making $\Ht^\sbt(\Ml^\tau, \varphi^\tau)$ into an orthosymplectic vertex comodule for the vertex coalgebra $\Ht^\sbt(\Ml, \varphi)$.
\end{corX}

Vertex algebras have an equivalent description \cite{FBZ,FG} as chiral algebras over the Ran space of $\Ab^1$, and in Theorem \ref{thm:OrthosymplecticFactorisationAlgebras} we show that $\Gt$-vertex algebras (with singularities along relatively general collections of hyperplanes) are equivalent to factorisation algebras over the \textit{$\Gt$-Ran space},
\begin{center}
  \begin{tikzcd}[row sep = 20pt, column sep = 20pt]
    \Ran_{\Gb} \Ab^1\ =\ \colim_{I\in \FinSet_{\Gb}^{\surj,\textup{op}}}X^I \ = \  \colim\bigg( &[-20pt] \tk_{\gk_1}\ar[r,std]\arrow[loop, distance=2em, in=50, out=130, looseness=5,"W_{\Gb_1}"] & \tk_{\gk_2}\ar[r,{yshift=3pt},std]\ar[r,{yshift=-3pt},std] \arrow[loop, distance=2em, in=50, out=130, looseness=5,"W_{\Gb_2}"] & \tk_{\gk_3} \arrow[loop, distance=2em, in=50, out=130, looseness=5,"W_{\Gb_3}"] \ar[r,{yshift=6pt},std] \ar[r,{yshift=0pt},std]\ar[r,{yshift=-6pt},std] & \cdots\ \ \bigg)
  \end{tikzcd}
  \end{center}
and Corollary \ref{corX:OSpGLVertex} is the case where $\Gb$ a linear-orthosymplectic parabolic. For example, the coaction $\Delta_{\GL \textup{-}\OSpt}$, related to braided module categories is uniquely calibrated to the type B/C configuration spaces and braid groups, and $\Zb/2$ is singled out as the automorphism group of the $A_n$ Dynkin diagram.  This should be compared to notions \cite{Li} of multiplicative vertex algebras, whose singularities lie along shifts of the diagonals by roots of unity. 

 We expect that many of the techniques developed in this paper naturally generalise away from the classical type case, which will be subject of future work. It would be interesting to find analogues of vertex algebras and the twisted Yetter-Drinfeld condition for type D and exceptional types.

\subsection{Orthosymplectic instantons on surfaces} 

\subsubsection{} If $S$ is a smooth proper algebraic surface or $\Ab^2$ and $\Gt$ is a symplectic or orthogonal group, one can consider the Uhlenbeck compactification $M_{U, G}$ of the moduli space of $\Gt$-bundles on the surface. When $S = \Ab^2$ it is expected \cite{BFN} that, as a conjectural generalization of the AGT correspondence to general groups, there is a representation 
\begin{equation}\label{eq:AGT} W(\hat{\gk}^L) \ \to \  \End(\mathrm{IH}^\bullet_{G\times \Cb^\times\times \Cb^\times}(M_{U,G}))\end{equation}
of a $W$-algebra associated to the Langlands dual affine algebra $\hat{\gk}^L$ on the equivariant intersection homology of $M_{U,G}$. 

In our setting we obtain in section \ref{sec:PervCoh} an action of the cohomological Hall algebra $\Al_{ 0}(S)$ of (cohomological shifts of) dimension zero sheaves on $S$ on the Borel Moore homology of $\Ml^{\OSpt,\textup{ss}}$, the semistable orthosymplectic perverse coherent sheaves.

\begin{thmX} \emph{(Theorem \ref{thm:Surface})}
  There is a representation 
  \begin{equation}
    \Al_{0}(S) \ \to \  \End(\Ht_\bullet^{\BM}(\Ml^{\OSpt,\textup{ss}}))
  \end{equation}
  making $\Ht_\bullet^{\BM}(\Ml^{\OSpt})$ a $\tau$-twisted localised Yetter-Drinfeld module for the zero-dimensional CoHA $\Al_0(S)$. 
\end{thmX}

For a smooth proper surface $S$ the CoHA of dimension zero sheaves has been described in \cite{MMSV} as the positive modes  $\Al_{0}(S)\simeq W^+(S)$ of a defomed W algebra $W(S)$ defined by explicit generators and relations. 
The operation of derived duality $\Dt:\Coh_0(S)[1]\stackrel{\sim}{\to}\Coh_0(S)[1]^{op}$ preserves the Serre subcategory of shifts of dimension zero sheaves and acts as an anti-involution $\tau$ on $W^+(S)$.  
\begin{corX} 
  There is a representation 
  \[ W^+(S) \ \to \  \operatorname{End}(\Ht_\bullet^{\BM}(\Ml^{\OSpt,\textup{ss}}))\]
  making $\Ht_\bullet^{\BM}(\Ml^{\OSpt})$ a $\tau$-twisted localised Yetter-Drinfeld module for $W^+(S)$. 
\end{corX}

\subsubsection{Conjectural relation to AGT} 
To relate this construction to the AGT conjectures, there is an analogue of the module construction for framed sheaves on $\mathbf{P}^2$, thought of as perverse coherent sheaves on $\Ab^2$ we expect in the $S = \Ab^2$ case that there is an analogue of the Kirwan map 
\[ \kappa: \Ht_\bullet^{\BM}(\Ml^{\OSpt,\textup{fr}}) \ \to \  \mathrm{IH}^\bullet_{G\times \Cb^\times\times \Cb^\times}(M_{U,G})\] 
such that the positive modes of $W((\widehat{\gk})^\vee)$ act via equivalence classes of elements of $\Al_{0}(S)$. The $\tau$-twisted Yetter-Drinfeld condition should provide the appropriate commutation relations between the positive and negative modes.

\subsection{Orthosymplectic  quivers with potential} \label{ssec:QuiverIntro} 

\subsubsection{} Attached to any quiver $Q$ with an orientation-reversing involution  $ \theta : Q \stackrel{\sim}{\to} Q^{\textup{op}}$ 
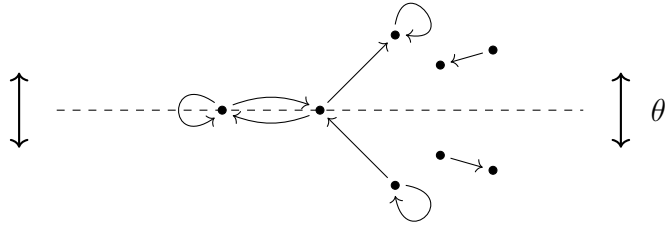
\begin{center}
\begin{tikzpicture}
\begin{scope} 
 [rotate = -90] 
 \draw[dashed] (0,-3.5) -- (0,3.5);
 \begin{scope} 
 \clip (0,0) circle (2.2);
 \filldraw[] (0,0) circle (1.5pt);
 \filldraw[] (1,1) circle (1.5pt);
 \filldraw[] (-1,1) circle (1.5pt);
 \filldraw[] (0,-1.3) circle (1.5pt);
 \draw[<-] (0.1,0.1) -- (0.9,0.9);
 \draw[->] (-0.1,0.1) -- (-0.9,0.9);
 \draw[->] (0.07,-0.14) to[out=-90+20,in=90-20] (0.07,-1.16);
 \draw[<-] (-0.07,-0.14) to[out=-90-20,in=90+20] (-0.07,-1.16);
 \draw[->,looseness=10] (1,1+0.14) to[out=80,in=10] (1+0.14,1);
 \draw[<-,looseness=10] (-1,1+0.14) to[out=80,in=180-10] (-1-0.14,1);
 \draw[->,looseness=10]  (-0.1,-1.3-0.1) to[out=-135+10,in=-45-10]  (0.1,-1.3-0.1);
  \end{scope}

  \filldraw[] (0.6,1.6) circle (1.5pt);
  \filldraw[] (-0.6,1.6) circle (1.5pt);
  \filldraw[] (0.8,2.3) circle (1.5pt);
  \filldraw[] (-0.8,2.3) circle (1.5pt);
  \draw[->, shorten <=4pt, shorten >=4pt] (0.6,1.6) -- (0.8,2.3);
  \draw[<-, shorten <=4pt, shorten >=4pt] (-0.6,1.6) -- (-0.8,2.3);
  \draw[<->,thick] (0.5,4) -- (-0.5,4);
  \draw[<->,thick] (0.5,-4) -- (-0.5,-4);

  \node[] at (0,4.5) {$\theta$};
  \node[white] at (0,-4.5) {$\theta$};
 \end{scope}
 \end{tikzpicture}
 \end{center}
 we get an involution on the category of $Q$-representations 
 $$\tau \ = \ (-)^\vee \cdot \theta \ : \  \Rep Q \ \stackrel{\sim}{\to} \ \Rep Q$$
 by applying $\theta$ then taking duals of all vector spaces and edge maps, and therefore an involution $\tau: \Ml \stackrel{\sim}{\to}\Ml$ on the moduli stack of $Q$-representations. More generally one can consider self-dual quiver representations \cite{DW} which modify the above construction by signs on the edges. 
 
 \begin{corX} \label{corX:W0}
  $\Ht^\sbt(\Ml^\tau) $ forms a $\tau$-equivariant module and vertex comodule for $\Ht^\sbt(\Ml)$. It forms a $\tau$-twisted vertex Yetter-Drinfeld module.
\end{corX}

The same is true for $\Ht^\sbt(\Ml^\tau,\varphi^\tau)$ and $\Ht^\sbt(\Ml,\varphi)$ whenever we take the vanishing cycles sheaf $\varphi=\varphi_{\tr W}$ of a function given by a $\tau$-invariant potential $W$. When the potential is zero, the CoHAM action recovers \cite{Yo}. We give shuffle formulas for these structures:

 \begin{thmX} \emph{(Theorems \ref{thm:ShuffleNoPotential}, Proposition \ref{prop:Intertwine})}
  The CoHA action of $\Ht^\sbt(\Ml,\varphi_W)$ on $\Ht^\sbt(\Ml^\tau,\varphi^\tau_W)$ is intertwined by map to the action of the shuffle algebra on an orthsymplectic shuffle module. 
 \end{thmX}
 
In the special case where $Q = Q^{(3)}$ is the tripled quiver of a quiver $Q$ and the involution $\tau$ preserves the symplecitc form on the space of representations of the doubled quiver $Q^{(2)}$ we obtain representations of the cohomological Hall algebra of the preprojective algebra via dimensional reduction. 

With view to applications to the AGT correspondence we also give a shuffle algebra formula for representations of the cohomological Hall algebra of the preprojective algebra on framed shuffle modules $\Ht^{\BM}_\sbt(\Ml^{\tau}_{\Pi_{Q_\wb}^{\textup{fr}}})$.

\begin{corX} \emph{(Corollary \ref{cor:preproj_w_intertwiner} )}
There is an action of $\Ht^{\BM}_{\bullet}(\Ml_{\Pi_{Q}})$ on $\Ht^{\BM}_{\bullet}(\Ml^{\tau}_{\Pi_{Q_\wb}^{\textup{fr}}})$ and compatible coaction making the latter a twisted Yetter-Drinfeld module. This action is intertwined by a map to a framed orthosymplectic module of the shuffle algebra. 
\end{corX} 
 
 Finally, at the end of  section \ref{sec:TwistedYangians} this case we then extend the CoHA action on $\Ht^{\BM}_{\bullet}(\Ml_{\Pi_{Q}})$ to an action of an extended CoHA which is roughly a braided module version in $\Lambda^{\OSpt}\Md$ of the Radford-Majid bosonisation construction \cite{Ma}.

\subsection{Relations with twisted Yangians via physics heuristics} \label{ssec:Heuristics}

\subsubsection{Chern-Simons theory} 
In \cite{BS}, Bittleson and Skinner consider orbifolded $4d$ Chern-Simons theory on 
$$\Rb\times \frac{\Rb\times\Cb}{\Zb/2},$$
whose category of line operators for $\Rb \times t$ is the category of modules for the Yangian $Y_\hbar(\gk)$ if $t\ne0$, and the \textit{twisted} Yangian $Y_\hbar(\gk,\sigma)$ if $t=0$. Thus we expect a factorisable sheaf of categories over $(\Rb\times \Cb)/\pm$.
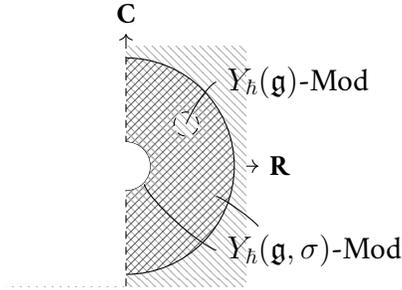
\begin{center}
  \begin{tikzpicture}[scale=0.8]
    \node[left,white] at (2-1.0-0.5,2-0.7-0.7) {$Y_\hbar(\gk,\sigma)\Md$};

    \draw[->] (4,2) -- (4.2,2);
    \draw[->] (2,4) -- (2,4.2);
    \node[right] at (4.2,2) {\small{$\Rb$}};
    \node[above] at (2,4.2) {\small{$\Cb$}};
    \fill[pattern=north west lines, pattern color=black!30] (0,0) rectangle (4,4);

    \fill[pattern=north east lines, pattern color=black!60] (2,2) circle [radius=1.8];
    \draw[] (2,2) circle [radius=1.8]; 

    \fill[white] (2,2) circle [radius=0.4];
    \fill[pattern=north west lines, pattern color=black!30] (2,2) circle [radius=0.6];
    \draw[] (2,2) circle [radius=0.4]; 

    \fill[white] (2+1.0,2+0.7) circle [radius=0.2];
    \fill[pattern=north west lines, pattern color=black!30] (2+1.0,2+0.7) circle [radius=0.4];
    \draw[dashed] (2+1.0,2+0.7) circle [radius=0.2];

    \fill[white] (2-1.0,2-0.7) circle [radius=0.2];
    \fill[pattern=north west lines, pattern color=black!30] (2-1.0,2-0.7) circle [radius=0.4];
    \draw[dashed] (2-1.0,2-0.7) circle [radius=0.2];

  \fill[white,opacity=0.8] (2,0) rectangle (0,4);

  \draw[dashed] (2+0.0,0) -- (2-0.0,4);
  \fill[white,opacity=0.5] (2,2) circle [radius=0.4];

  \node[right] at (2+1.0+0.5,2+0.7+0.7) {$Y_\hbar(\gk)\Md$};
  \node[right] at (2+1.0+0.5,2-0.7-0.7) {$Y_\hbar(\gk,\sigma)\Md$};
  \node[left] at (2-0.9,2-0.3) {};
  \draw[] (2+1.0+0.5,2+0.7+0.7) to[in=70,out=180+20,looseness=0.5]  (2+1.0,2+0.7);
  \draw[] (2+1.0+0.5,2-0.7-0.7) to[in=-70,out=180-20,looseness=0.5]  (2+0.3,2-0.3);
  \draw[] (4.2,1) to[in=-20,out=130,looseness=0.5]  (3.5,1.5);
  \end{tikzpicture} 
\end{center}
This predicts that $Y_\hbar(\gk)$ carries an ordinary and meromorphic coproduct each with an equivariance property and related by meromorphic $R$-matrices \cite{GTW}, and $Y_\hbar(\gk,\sigma)$ is a comodule for both of these.

\subsubsection{Cohomological Hall algebras and BPS states} Cohomological Hall algebras arise through a different physical construction to the above, as an algebra of BPS states \cite{HM, KS}. As a consequence, we expect an algebra $A$ with a module $M$ arising from scattering via a $4d$ $\Zb/2$-equivariant cobordism 
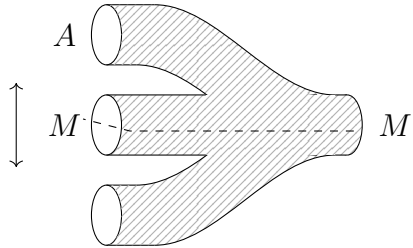
\begin{center}
  
  \begin{tikzpicture}[rotate = 90,scale=0.8]
    \begin{scope}
      [xshift=1.5cm]
      \draw (0,4) ellipse (0.5 and 0.25);
      \draw (-0.5,3.5) -- (-0.5,4);
      \draw (0.5,3.2) -- (0.5,4);
    \end{scope}
    \begin{scope}
      [xshift=-1.5cm]
      \draw (0,4) ellipse (0.5 and 0.25);
      \draw (-0.5,3.2) -- (-0.5,4);
      \draw (0.5,3.5) -- (0.5,4);
    \end{scope}
  \fill[pattern=north east lines, pattern color=black!30]
  (1.5-0.5,3.5) to[out=-90,in=90,looseness=0.8] (-0.5,0.5) to[out=90,in=-90,looseness=0.8]  (1.5-0.5,3.5) -- (1.5-0.5,4) arc (180:360:0.5 and 0.25)  (1.5+0.5,4) --(1.5+0.5,3.2) to[out=-90,in=90,looseness=0.8] (0.5,0.2) -- (0.5,0) -- (-0.5,0)  -- (-0.5,0.5) to[out=90,in=-90,looseness=0.8]  (1.5-0.5,3.5) -- cycle ;
  
  \fill[pattern=north east lines, pattern color=black!30]
  (-1.5-0.5,3.2) to[out=-90,in=90,looseness=0.8] (-0.5,0.2) to[out=90,in=-90,looseness=0.8]  (-1.5-0.5,3.2) -- (-1.5-0.5,4) arc (180:360:0.5 and 0.25)  (-1.5+0.5,4) --(-1.5+0.5,3.5) to[out=-90,in=90,looseness=0.8] (0.5,0.5) -- (0.5,0) -- (-0.5,0)  -- (-0.5,0.2) to[out=90,in=-90,looseness=0.8]  (-1.5-0.5,3.2) -- cycle ;
  
    \draw (-1.5-0.5,3.2) to[out=-90,in=90,looseness=0.8] (-0.5,0.2);
    \draw (-1.5+0.5,3.5) to[out=-90,in=90,looseness=0.8] (0.5,0.5);
    \draw (1.5-0.5,3.5) to[out=-90,in=90,looseness=0.8] (-0.5,0.5);
    \draw (1.5+0.5,3.2) to[out=-90,in=90,looseness=0.8] (0.5,0.2);
  
    \filldraw[white,opacity=1] (-0.49,0) rectangle (0.49,2.4);
    \draw[dashed] (0.5,0) arc (0:180:0.5 and 0.25);
    \filldraw[white,opacity=0.5] (-0.49,0) rectangle (0.49,1.9);
  
  \fill[pattern=north east lines, pattern color=black!30] (-0.5,4) -- (-0.5,0) arc (180:360:0.5 and 0.25) -- (0.5,4) arc (360:180:0.5 and 0.25);
  
    \draw (0,4) ellipse (0.5 and 0.25);
    \draw (-0.5,0) arc (180:360:0.5 and 0.25);
    \draw (-0.5,2.35) -- (-0.5,4);
    \draw (0.5,2.35) -- (0.5,4);
    \draw (-0.5,0) -- (-0.5,0.2);
    \draw (0.5,0) -- (0.5,0.2);
  
    \draw[dashed] (-0.1,0-0.25+0.14) -- (0-0.1,4-0.25-0.14) -- (0+0.1,4+0.25+0.14) ;
    \draw[<->] (0.7,5.5) -- (-0.7,5.5);

     \node[] at (0,4.7) {$M$};
     \node[] at (1.5,4.7) {$A$};
     \node[] at (0,-0.8) {$M$};
  \end{tikzpicture} 
\end{center}
along with a coaction $M \to A\otimes M$ arising from the same cobordism read the other direction. These are the structures we consider in this paper; for instance the CoHA of dimension zero sheaves on a surface $S$ is expected to be related to the algebra of BPS states in a string theory compactificaiton on $\operatorname{tot}(K_S)$, and $M$ we expect to be related to an involution acting on $K_S$ linear in the fibres. 

CoHAs are known to be related to Yangians by taking Drinfeld doubles \cite{YZ}, but the twisted Yangian is not a module over the Yangian, so we expect some relation between $M$ and $Y_\hbar(\gk,\sigma)^+$, but not an equality.

\subsubsection{Partially dualised Hopf algebras} Algebraically, the representations of CoHAs given by taking the image of $A \to \End(M)$ are not faithful in general; the quotients we obtain appear to be generalisations of the construction of \cite{BLS} of \emph{partially dualized Hopf algebras} which are quotients of a Hopf algebra compatible with bialgebra pairings.

\subsubsection{Langlands-type dual of coideal subalgebras} 

Coideal subalgebras of quantum groups and affine quantum groups, such as the twisted Yangian, have attracted a great deal of interest in geometric representation theory in recent years \cite{LW}. 

While the previous section explains via physics heuristics that the algebraic structures arising in this paper are different from those of coideal subalgebras, it is natural to expect there to be some relationship between the quotients of CoHAs arising in orthosymplectic representations and the better-studied twisted Yangians. This should also be related to the Langlands duality of affine Lie algebras appearing in \ref{eq:AGT} for the AGT correspondence for non-simply laced Dynkin types.


In the classical case of finite dimensional Lie algebras, fixed point subalgebras under diagram involutions (which preserve the triangular decomposition and so are most relevant to cohomological Hall algebras) are described by the folding of Dynkin diagrams 
\[ A_{2n-1} \leadsto C_n, \qquad D_{n+1} \leadsto B_n, D_4 \leadsto G_2, \ldots  \] 
whereas the Langlands dual procedure arises via the orbit Lie algebra. While the schematic description of the dual folding is given by 
\[ A_{2n-1} \leadsto B_n, \qquad D_{n+1} \leadsto C_n, D_4 \leadsto G_2, \ldots  \] 
there is not a surjection $\gk \to (\gk^\sigma)^\vee$ because $\gk$ is simple. 

On the other hand, the nilpotent Lie algebras $\nk$ are not semisimple and we propose to construct Langlands-type duals of coideal subalgebras of quantum groups by taking quotients $U_q(\nk) \to R$ (whose kernel is the right radical of a coideal subalgebra with respect to a bialgebra pairing).

Theorem \ref{thmX:YD}, one of the main results of this paper, gives the commutation relations required to normal order the positive and negative modes of the double of the algebra acting on orthosymplectic modules, and is therefore relevant to the construction of dual-folded affine quantum groups rather than just their positive halves. A step in the Majid-type construction of $U_q(\gk)$ from $U_q(\nk)$ involves the construction of $U_q(\bk)$ from $U_q(\nk)$ via the procedure called bosonisation. At the level of torus localised CoHAs and othosympleictc modules we construct a variant of the bosonisaiton of the shuffle algbera which also acts on the orthosymplectic modules. Upon adjoining the (commutative) cartan elements in our variant of bosonisation, our extended algebra becomes neither naturally a subalgebra nor a quotient of the usual extended CoHA.

\subsection{Acknowledgements}

\subsubsection{} We thank Chenjing Bu, Hiraku Nakajima and Tianqing Zhu for useful discussions. A.L. would like to thank Danish National Research Foundation and Villum Fonden.

\newpage
\section{Orthosymplectic moduli stacks and CoHAs} \label{sec:OrthosymplecticStacks}

\noindent An \textit{orthogonal form} on a vector space $V$ is a bilinear map 
$$\kappa \ : \ V \otimes V \ \to \ k$$
which is symmetric and nondegenerate. It is \textit{symplectic} if it is skew-symmetric and nondegenerate. For any rigid braided monoidal category $\Cl$, an \textbf{orthosymplectic form} on an object $c$ is a map 
$$\kappa \ : \ c \otimes_\Cl c \ \to \ 1_\Cl$$
which commutes with the braiding $\kappa \cdot \beta \simeq \kappa$ and induces an isomorphism $c \simeq c^\vee$. Applying this to $\Vect$ with its swap braiding $\sigma$ and its negation $-\sigma$ gives back the notions of orthogonal and symplectic forms. 

In this section, we show that the moduli stack $\Ml_{\Cl^{\OSpt}}$ of orthosymplectic objects is the fixed point stack for an involution acting on $\Ml_{\Cl}$. We then show for suitable $\Cl$ that its critical cohomology gives a module for the cohomological Hall algebra respecting the involution. This generalises \cite{Yo} in case of the quiver representations with zero potential.


\subsection{Orthosymplectic categories and objects} 

\subsubsection{} Let $\Cl$ be a category with an isomorphism 
$$\theta \ : \  \Cl  \ \stackrel{\sim}{\to}\ \Cl^{op}$$
satisfying a 2-categorical version of $\theta^2 = \id$: there is a natural isomorphism
\begin{equation}
  \begin{tikzcd}[row sep={30pt,between origins}, column sep={45pt,between origins}]
    &\Cl^{op}\ar[rd,"\theta^{\textup{op}}"]  & \\[-15pt]
  \Cl \ar[ru,"\theta"] \ar[rr,"\id"'{name=U,inner sep=1pt}, bend left = -30]  &  & \Cl 
  \arrow[Rightarrow, to=U, from=1-2, "\alpha","\wr"', shorten <=0pt, shorten >=2pt] 
  \end{tikzcd}
\end{equation}
We call such $\theta$ an \textit{antiinvolution} and $(\Cl,\theta)$ an \textit{orthosymplectic} or \textit{self-dual category}.

\subsubsection{Remark} \label{sssec:RemarkZ2Cat} There is an action of $\Zb/2$ on $\Cat$, sending a category to its opposite. An orthosymplectic category $(\Cl,\theta)$ is precisely an object of the fixed category $\Cat^{\OSpt}=\Cat^{\Zb/2}$; see \cite{Ro} for how to take fixed points in a 2-category.
 
\subsubsection{Example} If $\Cl$ is rigid monoidal (every object has a dual) then the map 
$$(-)^\vee \ : \ \Cl \ \stackrel{\sim}{\to} \ \Cl^{\textup{op}}, \hspace{15mm}  c \ \mapsto \ c^\vee$$
is an antiinvolution. For instance, $\Cl=\Vect^{\textup{f.d.}}$ with involution dualising vector spaces.

\subsubsection{} An \textbf{orthosymplectic object} of orthosymplectic category $\Cl$ is an object $c$ together with an isomorphism $\varphi: c \stackrel{\sim}{\to} \theta(c)$ satisfying the cocycle condition: that
$$c \ \stackrel[\raisebox{2pt}{$\sim$}]{\varphi}{\to} \ \theta(c) \ \stackrel[\raisebox{2pt}{$\sim$}]{\theta(\varphi)}{\leftarrow} \ \theta^2(c) \ \simeq \ c$$
is the identity, where the final isomorphism was induced by $\alpha_c$.

This same notion is called different things in different sources. For example \cite{Bu, Yo} refer to such objects as self-dual. We choose the name orthosymplectic to relate to the notion of orthosymplecitc quiver varieites.

\begin{prop}
  Let $\Cl$ be an abelian or dg category. There is an involution on its moduli stack \emph{\cite{TV}} of objects $\tau  :  \Ml_{\Cl}  \stackrel{\sim}{\to}  \Ml_{\Cl}$  whose fixed points stack 
  $$\Ml^{\OSpt}_{\Cl}\ \defeq\ \Ml_\Cl^\tau$$ 
  has as points the orthosymplectic objects of $\Cl$.
\end{prop}
\begin{proof}
  By functoriality of $\Ml_{(-)}$, the antiinvolution induces a map 
  $$\theta \ : \ \Ml_\Cl \ \stackrel{\sim}{\to} \ \Ml_{\Cl^{\textup{op}}}.$$
  We obtain $\tau$ by composing this with the isomorphism $(-)^{-1} : \Ml_{\Cl^\textup{op}}\stackrel{\sim}{\to} \Ml_{\Cl}$ given on its ($\infty$-)groupoid of $R$-points
  $$(-)^{-1} \ : \ \Maps_{\Cat}(R\Md,\Cl^{\textup{op}})_\circ \ \stackrel{\sim}{\to} \ \Maps_{\Cat}(R\Md, \Cl)_\circ$$
  fixing each object and sending each morphism to its inverse. It is easily checked that this gives an involution whose fixed points are the orthosymplectic objects.
\end{proof}

We call $\Ml_{\Cl}^{\OSpt}$ the \textbf{orthosymplectic moduli stack} of $(\Cl,\theta)$.

\subsubsection{Remark} Let $\Cat_{\obj}$ denote the category whose objects are categories equipped with an object. This has a weak involution sending $(\Cl, c)  \mapsto (\Cl^{\textup{op}},c)$. The category of orthosymplectic categories $(\Cl,\theta)$ with orthosymplectic object $(c,\varphi)$ is then equivalent to $\Cat_{\obj}^{\Zb/2}$, and we call
$$\Cl^{\OSpt}\ = \ \Cat_{\textup{obj}}^{\Zb/2}\times_{\Cat^{\Zb/2}}\{(\Cl,\theta)\}$$ 
the category of \textit{orthosymplectic objects} in $\Cl$.

\subsubsection{Simplest example: vector spaces} \label{sssec:VectEx} The abelian category $\Cl=\Vect^{\textup{f.d.}}$ of finite dimensional vector spaces has a involution 
$$\tau \ : \  V\ \mapsto\ V^\vee,$$
where for $\alpha$ we use the canonical evaluation map $\ev:V^{\vee \vee} \stackrel{\sim}{\to}V$ or $-\ev$. The moduli stack is $\Ml_\Cl$ is the union of $\BGL_n$, whose fixed locus in both cases respectively is
$$(\BGL_n)^{\tau}\ =\ \BO_{n}, \ \BSp_{n},$$
where the latter is of course is empty unless $n$ is even. The two isomorphisms
$$(-)^{-1}, \, \theta \ : \ \Ml_{\Cl} \ \stackrel{\sim}{\to} \ \Ml_{\Cl^{\textup{op}}}, \hspace{15mm} \BGL_n \ \stackrel{\sim}{\to} \ \Bt(\GL_n^{\textup{op}})$$
 is given on transition functions by $A \mapsto A^{-1}$, and the second either by $A \mapsto A^T$ or $A \mapsto \Omega A^T\Omega$, respectively.\footnote{Here $\Omega= \left( \begin{smallmatrix} 
  &1 \\ 
  -1& 
 \end{smallmatrix} \right)$.}

\subsubsection{Example: Jordan quiver} The next simplest example is the category $\Cl$ of finite dimensional vector spaces with endomorphism. This has an antiinvoltion 
$$(V, \rho) \ \mapsto \ (V^\vee,\rho^\vee)$$
and an orthosymplectic structure is an isomorphism $\varphi:(V,\rho)\simeq (V^\vee, \rho^\vee)$, i.e. a commuting diagram 
\begin{equation}\label{fig:DiagJordan}
\begin{tikzcd}[row sep = {30pt,between origins}, column sep = {20pt}]
 V \ar[r,"\varphi","\sim"'] \ar[d,"\rho"]  & V^\vee  \\ 
V \ar[r,"\varphi","\sim"'] & V^\vee \ar[u,"\rho^\vee"']  
\end{tikzcd}
\end{equation}
such that if we concatanate \eqref{fig:DiagJordan} with $\eqref{fig:DiagJordan}^\vee$ 
\begin{center}
\begin{tikzcd}[row sep = {30pt,between origins}, column sep = {20pt}]
 V^\vee \ar[r,<-,"\varphi^\vee","\sim"'] \ar[d,<-,"\rho"']  & V^{\vee\vee}  \\ 
V^\vee \ar[r,<-,"\varphi^\vee","\sim"'] & V^{\vee\vee} \ar[u,<-,"\rho^\vee"]  
\end{tikzcd}
\end{center}
we get the identity after identifying $\alpha_V:V^{\vee\vee}\simeq V$ by the evaluation map. The moduli stack is a disjoint union of $\glk_n/\GL_n$ and the fixed locus is $\ok_n/\Ot_n$ or $\spk_n/\Sp_n$ depending on the choice of $\alpha$.

\subsubsection{Symmetry via the cocyle condition} \label{ssec:SymmetryCocycle}

We now assume that $\Cl$ is a rigid monoidal category and has an internal Hom functor $\Hom(-,-)$ satisfying tensor-Hom adjunction 
$$\Hom_\Cl(a \otimes_\Cl b, c) \ \simeq \ \Hom_\Cl(a, \Hom(b, c)).$$
It follows from this that if $(-)^\vee=\Hom(-,1)$, we have 
$$\Hom_\Cl(a \otimes_\Cl b, 1) \ \simeq \ \Hom_\Cl(a, b^\vee).$$
Thus if $\Cl$ is additional braided monoidal with braiding $\beta$, we have as in \cite{Zub, Yo} the following. 

\begin{prop}\label{lem:fixed_is_osp}
 An orthosymplectic object $(c,\varphi)$ is equivalent to an object $c\in \Cl$  together with a non-degenerate braided symmetric bilinear form 
 $$\kappa \ : \ c \otimes_\Cl c \ \to \ 1_\Cl.$$
\end{prop}

\subsection{Example: quivers with potential} 
\label{ssec:W_sign_invrariance}

\subsubsection{} 
As in section \ref{ssec:QuiverIntro}, let $Q = (Q_0, Q_1)$ be a quiver with orientation-reversing involution, viewed as an isomorphism $\theta: Q \stackrel{\sim}{\to} Q^{\textup{op}}$. Compose this with 
$$ (-)^\vee \ : \ k Q \ \stackrel{\sim}{\to} \ k Q^{\textup{op}},$$
sending an arrow in the path algebra to its opposite, to get an involution 
$$\tau \ : \ \Rep Q \ \stackrel{\sim}{\to} \ \Rep Q.$$
Finally we ask that the potential $W$ is invariant under $\tau: (-)^\vee \cdot \theta$.

\subsubsection{} In what follows we will use a slight generalisation of the above. A \textit{duality structure} \cite{DW} on $Q$ is a pair of $\theta$-invariant maps
$$\textup{sgn}\ : \ Q_0,\, Q_1 \ \to \ \{ \pm 1\}$$
compatible via 
$$\textup{sgn}(e)\cdot \textup{sgn}(\theta(e)) \ = \ \textup{sgn}(i)\cdot \textup{sgn}(j)$$
for any edge $e:i\to j$. We now define the \textit{dualising functor}
$$ \Dt \ = \ \textup{sgn}\cdot (-)^\vee \ : \ k Q \ \stackrel{\sim}{\to} \ k Q^{\textup{op}},$$
which induced the antiinvolution on representation categories
$$(-)^\vee \ : \ \Rep Q \ \stackrel{\sim}{\to} \  \Rep Q^{\textup{op}}, \hspace{15mm} (V_i,\rho_e) \ \mapsto \ (V_i^\vee,\textup{sgn}(e) \rho_e^\vee).$$

\begin{lem}
  This defines an antiinvolution on $\Rep Q$.
\end{lem}
\begin{proof}
For this to be an involution, we need a natural transformation $\alpha: \Dt^2 \stackrel{\sim}{\to} \id$, in other words a collection of commuting diagrams 
\begin{center}
\begin{tikzcd}[row sep = {30pt,between origins}, column sep = {45pt, between origins}]
V_i \ar[r,"\rho_e^{\vee \vee}"] \ar[d,"\alpha_i","\wr"']  & V_j\ar[d,"\alpha_j","\wr"']  \\ 
V_i \ar[r,"\rho_e"]  & V_j 
\end{tikzcd}
\end{center}
for every arrow $e:i\to j$ in the quiver which is natural in $(V_i,\rho_e)$. Setting $\alpha_i=(-1)^{\textup{sgn}(i)}$ gives the required natural transformation.
\end{proof}

Thus we have an involution $\tau=\Dt\cdot \theta: \Rep Q \stackrel{\sim}{\to} \Rep Q$, generalising the above.

\begin{prop} 
  An orthosymplectic quiver representation
  $$(V_i,\rho_e) \ \in \ (\Rep Q)^{\tau}$$
  is a quiver representation together with a nondegenerate (skew) symmetric bilinear form on $V_i$ whenever $\theta(i)=i$, and isomorphisms $V_i \cong V_{\theta(i)}^\vee $ for other $i$, such that in either case $\rho_e$ are intertwined by the duality isomorphisms $\eta: V_i \stackrel{\sim}{\to} V_{\theta(i)}^\vee$.
\end{prop}

For instance, for every self-loop $e:i \to i$ for a fixed vertex $i$, the diagram
\begin{center}
\begin{tikzcd}[row sep = {30pt,between origins}, column sep = {20pt}]
  V_i & V_i \\
  V_i & V_i \\
  {V_i^\vee} & {V_i^\vee}
  \arrow["\rho_e", from=1-1, to=1-2]
  \arrow[from=1-1, to=2-2]
  \arrow[from=1-2, to=2-1]
  \arrow["\eta_i"', from=2-1, to=3-1]
  \arrow["\rho_e"', from=2-2, to=2-1]
  \arrow["\eta_i", from=2-2, to=3-2]
  \arrow["{\rho_e^\vee}"', from=3-1, to=3-2]
\end{tikzcd}
\end{center}
 commutes, which is equivalent to $\rho_e^\vee\eta_i \rho_e = \eta_i$. It follows from this that $\rho_e$ is an orthogonal or symplectic matrix if $\textup{sgn}(i)=1$ or $\textup{sgn}(i)=-1$, respectively. If $i$ is not fixed there are no extra conditions on $\rho_e$ except an identification with $\rho_{\theta(e)}^\vee$.


\subsection{Example: preprojective algebra} 

\subsubsection{} 
Next we consider the category 
$$\Al\ =\ T^*\Rep(Q)$$
given by taking the two-Calabi-Yau completion of the category of representations of a finite quiver $Q$. Equivalently, $\Al$ is the category of representations of the \textit{preprojective algebra} $\Pi_Q$. Via our description we have chosen an orientation $\omega$ on the doubled quiver $Q_{(2)}$ consisting of those arrows already in $Q_1$. This defines a symplectic form  $\omega$ on $\Rep(Q_{(2)})$. 

Pick an involution and duality structure on $Q_{(2)}$. In the following, we will assume that 
\[ \tau^* \omega = \omega . \] 
This will ensure that the homotopy fixed points inherit a symplectic structure as a symplectic reduction via a fixed quotient group. We give two examples when this occurs. 

\begin{enumerate} 
    \item First, when $Q$ is the Jordan quiver, and each arrow is sent to itself.   

\item Assume that $Q_0$ is bipartite $Q_0 = Q_\Ot \sqcup Q_{\Spt}$ and let $\textup{sgn}(i) = 1$ if $i \in Q_{\Ot}$ and $-1$ otherwise. If $a \in \omega$ then $\textup{sgn}(a) = 1$ otherwise $\textup{sgn}(a) = -1$. 
\end{enumerate} 
Define an antiinvolution
\begin{equation}\label{eq:quiv_dual} \Dt \ : \  \Al \ \stackrel{\sim}{\to} \ \Al, \hspace{15mm}  (V_i, \rho(a)) \ \mapsto \  (V_i^\vee,  \textup{sgn}(a)\rho(\overline{a})^\vee )  \end{equation}
with the action on morphism the inverse of the dual map $\Dt(\varphi: V\to W)$ given by $(\varphi^\vee)^{-1}$. This is equivalent to the duality map above where $\theta$ is an involution depending on the orientation $\omega$ of the doubled quiver, or more generally any orientation reversing isomorphism of the quiver. 

Now defining $Q^{(3)}$ to be the tripled quiver with a self-loop $\omega_i$ adjoined at each vertex $i$, we have that the potential 
\[ W = \sum_{i \in Q_0} \omega_i \sum_{a \in \omega} [a, \overline{a}] \] 
is invariant under $\tau$.

\subsubsection{Dimensional reduction and symplectic reduction} 

Consider a $\tau$-fixed representation $(V, \rho) \in (\Rep Q_{(3)})^\tau$ of the tripled quiver. Because the fixed locus condition forces $\rho(\omega_i)$  to lie in the subalgebra $\mathfrak{g}_i$ of $\mathfrak{gl}(V_i)$ preserving the form, we have that 
\[ \operatorname{Tr}(W)|_{\Rep_{Q_{(3)}}^\tau} \ = \  \langle \omega, \mu(a) \rangle \] 
where $\omega \in  \prod_{i \in Q_0/\theta} \mathfrak{g}_i$ and $\mu(a) = \sum_{a \in \omega}[a,\overline{a}]$. The zero locus of $\omega$ is exactly the preprojective condition. 

Its moduli stack of orthosymplectic objects has a description as a symplectic reduction:

\begin{prop}
  $\Ml^{\textup{\OSpt}}$ is given by the symplectic reduction appearing in \emph{\cite{Cho}}
    $$\Ml^{\OSpt}_{\Pi_Q, \db}\ =\ \mu^{-1}(0)\big/\smprod_{i \in Q_{\Ot}^\theta } \Ot_{d_i} \times \smprod_{j \in Q_{\Spt}^\theta} \Spt_{d_j} \times \smprod_{k \in Q_0/\theta \smallsetminus Q_0^\theta} \GL_{d_i} $$ 
    where  $\mu$ is the moment map for the action of $\smprod_{i \in Q_{\Ot}^\theta } \Ot_{d_i} \times \smprod_{j \in Q_{\Spt}^\theta} \Spt_{d_j} \times \smprod_{k \in Q_0/\theta \smallsetminus Q_0^\theta} \GL_{d_i}$ on $\tau$-fixed representations of $Q^{(2)}$. 
\end{prop}

\subsection{Example: perfect complexes} 

\subsubsection{} 

Let $X$ be a smooth proper algebraic variety and $\Ml[-k,k]$ the stack of perfect universally glueable complexes shifted by $[k]$, with Tor amplitude in $[-k,k]$. Then the derived duality functor 
\begin{equation} \label{eq:derived_dual}
  (\El^\bullet, d_{\bullet}) \mapsto \Dt(\El^\bullet) := (\El^{-\bullet, \vee}, -(-1)^{-\bullet} d_{-\bullet+1}^\vee)
\end{equation} 
gives rise to a $\Zb/2$-action on $\Ml[-k,k]$. 
As a natural transformation of groupoids, we have 
\[ \Dt(\varphi: \El \xrightarrow{\sim} \Fl) = (\varphi^\vee)^{-1}: \Dt(\El) \xrightarrow{\sim} \Dt(\Fl).\] The action is \emph{not} a strict $\Zb/2$ action, instead required by \eqref{eq:derived_dual} we have a nontrivial natural transformation
$\alpha_{-1, -1}:\Dt\cdot \Dt \implies \operatorname{Id}$ 
given on points $\El$ by the isomorphism $\Dt\cdot \Dt(\El) \xrightarrow{\alpha_{-1,-1}(\El)} \El $ which is of the form $\alpha_{-1,-1}(\El)_i = (-1)^i$ where  $(-1)^{i}: (\El_i^{\vee \vee}) \to \El_i$ is this multiple of the usual map on locally free sheaves. 

Its strictification is isomorphic to the stack $\Ml[-k,k]^{str}$ whose objects consist of pairs $(a, \El)$ for $a \in \Zb/2, \El \in \Ml[-k,k]$ morphisms from $(a, \El)$ to $(b, \Fl)$ are maps $\El \to a^{-1}b \cdot \Fl$ and $\varphi:(a, \El)\to (b, \Fl)$, $\psi: (b, \Fl)\to (c, \Gl)$ compose as 
\[ \El \xrightarrow{\varphi} a^{-1} b\cdot \Fl \xrightarrow{a^{-1} b(\psi)} (a^{-1}b)\cdot (b^{-1} c)\cdot \Gl \xrightarrow{ \alpha_{a^{-1}b, b^{-1} c}(\Gl)} a^{-1} c \cdot \Gl .\] 

The nontrivial data of its stack of fixed points $\Ml[-k,k]^{str, \Zb/2}\simeq \Ml[-k,k]^{\Zb/2}$ over a test scheme $T$ consists of \cite{Ro} perfect complexes $\El \in \operatorname{Perf}(T \times X)$ of Tor-amplitude $[-k,k]$ together with a qausiisomorphism 
\[ \eta: \El \xrightarrow{\sim} \Dt(\El)\] satisfying the cocycle condition 
\[ (\Dt(\eta)) \cdot \eta = \alpha_{-1,-1}^{-1}({\El})\]
which is equivalent to 
\[ \eta^\vee = \alpha_{-1,-1}(\eta).  \] 
This is further equivalent under the chain of isomorphisms
\begin{equation}\label{eq:hom_composition} \Hom(\El, \El^{\vee}) \to \Hom(\El \otimes \El^{\vee \vee}, \Ol) \xrightarrow{\sigma} \Hom(\El^{\vee  \vee}\otimes \El, \Ol) \to \Hom(\El^{\vee \vee}, \El^\vee)\to \Hom(\El, \El^{\vee})
\end{equation} 
to $\eta$ providing a symmetric pairing on $\El$ under the supersymmetric braiding $\El\otimes \Fl \xrightarrow{\sigma} \Fl \otimes \El$. The composition of the first four terms in in \eqref{eq:hom_composition} agrees with the map $\eta\mapsto \eta^\vee$ by taking zeroth hypercohomology of the sequence of equalities
\[ \Hl om(\El,\Fl)  \to \El^\vee\otimes \Fl \to \Fl\otimes \El^\vee \to \Hl om(\Fl^\vee, \El^\vee)\]  
for perfect complexes $\El, \Fl$.

\subsection{Orthosymplectic short exact sequences} \label{ssec:OSpCorrespondences}

\subsubsection{} If $\Ml$ is a moduli stack of objects in an abelian category $\Al$, we may form the correspondence formed by the moduli stack $\SES=\Ml_{\SES(\Al)}$ of short exact sequences in $\Al$:
\begin{equation}
   \label{fig:CoHACorrespondence}
  \begin{tikzcd}[row sep = {30pt,between origins}, column sep = {45pt,between origins}]
    & \textcolor{white}{a\to e\to b}& &[30pt]&\SES\ar[rd,"p"]\ar[ld,"q"'] & &[30pt]& a\to e\to b \ar[rd,|->]\ar[ld,|->]&  \\ 
    \textcolor{white}{a}& &\textcolor{white}{(a,b)} & \Ml\times\Ml& & \Ml & (a,b)& & e
   \end{tikzcd} 
\end{equation}
This satisfies an associativity condition: an isomorphism
$$\alpha \ : \ (\Ml\times \SES )\times_{\Ml \times\Ml}\SES\ \simeq \ (\SES\times \Ml )\times_{\Ml \times\Ml}\SES$$
of correspondences satisfying coherence conditions. This means $\Ml\in \Alg(\PreStk^{\corr})$, i.e. makes $\Ml$ into an associative monoid in the category of algebraic prestacks with morphisms the correspondences between them.

\begin{prop} \label{prop:MActionOnMt}
  There is a pointed left $\Ml$-module structure on $\Ml^\tau$ given by the correspondence
  \begin{equation}\label{fig:OSpCorr}
  \begin{tikzcd}[row sep = {30pt,between origins}, column sep = {45pt, between origins}]
   &\SES_3^\tau \ar[rd] \ar[ld] & \\ 
  \Ml \times\Ml^\tau&& \Ml^\tau
  \end{tikzcd}
  \end{equation}
   In other words, there is an associativity isomorphism satisfying coherence conditions and a unit condition for $0\in \Ml$ and $0 \in \Ml^\tau$.
\end{prop}
\begin{proof}
  The above correspondence \eqref{fig:CoHACorrespondence} is equivariant for any involution $\tau$, i.e. we have a map
  \begin{center}
  \begin{tikzcd}[row sep = {30pt,between origins}, column sep = {45pt, between origins}]
   & \textcolor{white}{a \to e \to b}& &   &\SES \ar[rd] \ar[ld]\ar[dd,dashed,"\wr","\eta"']   & & & a \to e \to b \ar[dd,|->,"\eta"']  & \\ 
    & & & \Ml \times\Ml \ar[dd,"\sigma\cdot (\tau \textup{,}\tau)"',"\wr"] & & \Ml \ar[dd,"\tau"] & & &  \\[-30pt]
   &\textcolor{white}{\tau(b)\to \tau(e)\to \tau(a)}&&&\SES \ar[rd] \ar[ld]  & & & \tau(b)\to \tau(e)\to \tau(a) & \\ 
   &&&\Ml \times\Ml& & \Ml& & & 
  \end{tikzcd}
  \end{center}
  defining a $2$-isomorphism. Likewise we get an involution acting on the threefold correspondence parametrising length three flags in $\Al$.
  \begin{equation}
    \label{fig:SES3} 
    \begin{tikzcd}[row sep = {30pt,between origins}, column sep = {45pt,between origins}]
      &\SES_3\ar[rd,"p_3"]\ar[ld,"q_3"'] & &  \\
      \Ml \times \Ml\times\Ml& & \Ml
     \end{tikzcd} 
  \end{equation}
  Note that the data of a length three flag $a \subseteq \ker \subseteq e$ is equivalent the following square of short exact sequences, on which the maps $p_3,q_3$ act as 
  \begin{equation*}\label{eqn:SES3Corr}
    \begin{tikzcd}[row sep = {30pt,between origins}, column sep = {10pt}]
    &&a \ar[d,equals] \ar[r]  & \ker \ar[r] \ar[d] & b\ar[d]&[10pt]&[10pt] &  \\ 
   (a,b,c)&\mapsfrom& a\ar[r]  &e\ar[r] \ar[d]  &\coker\ar[d]& \mapsto&e  \\
   && 0  & c \ar[r,equals] &c &&&
    \end{tikzcd}
  \end{equation*}
  We define the involution on $\SES_3$ by sending this to 
  \begin{equation*}
    \begin{tikzcd}[row sep = {30pt,between origins}, column sep = {10pt}]
    &&\tau(a) \ar[d,equals]  & \tau(\ker)  \ar[r,<-]\ar[l]   & \tau(b)\ar[d,<-]&[5pt]&[10pt] &  \\ 
   (\tau(c),\tau(b),\tau(a))&\mapsfrom& \tau(a)\ar[r,<-]  &\tau(e)\ar[r,<-]\ar[u]  &\tau(\coker)\ar[d,<-]& \mapsto&\tau(e)  \\
   &&0 & \tau(c) \ar[r,equals] \ar[u] &\tau(c) &&&
    \end{tikzcd}
  \end{equation*}
  making \eqref{fig:SES3} equivariant, and taking the fixed point stack gives \eqref{fig:OSpCorr}.
  
  To show that the action \eqref{fig:OSpCorr} is associative, we note that the associativity of $\SES$ gives a $\tau$-equivariant identification between the two correspondences 
  \begin{center}
  \begin{tikzcd}[row sep = {30pt,between origins}, column sep = {65pt, between origins}]
   & \SES_{3}\times_\Ml \SES_3 \ar[rd] \ar[ld]  & \\ 
  \Ml \times (\Ml \times \Ml \times \Ml) \times\Ml & & \Ml
  \end{tikzcd}
  \end{center} 
and 
\begin{center}
  \begin{tikzcd}[row sep = {30pt,between origins}, column sep = {65pt, between origins}]
   & \SES\times_\Ml \SES_3\times_{\Ml}\SES \ar[rd] \ar[ld]  & \\ 
 (\Ml\times \Ml)\times \Ml \times (\Ml\times \Ml) & & \Ml
  \end{tikzcd}
  \end{center} 
To be concrete, both parametrise length five flags in $\Al$. Taking fixed points under $\tau$ gives the associativity isomorphism
\begin{center}
  \begin{tikzcd}[row sep = {30pt,between origins}, column sep = {45pt, between origins}]
    &\SES\times_\Ml\SES_3^\tau \ \simeq \ \SES_3^\tau\times_{\Ml^\tau}\SES_3^\tau \ar[rd] \ar[ld] & \\ 
   \Ml \times\Ml \times\Ml^\tau&& \Ml^\tau
   \end{tikzcd} 
\end{center}
compatible with the associator for $\SES$, thus finishing the proof.
\end{proof}

A point of $\SES_3^\tau$ consists of a square of short exact sequences which is reflection-symmetric:
\begin{center}
\begin{tikzcd}[row sep = {30pt,between origins}, column sep = {10pt}]
 &[40pt]a \ar[d,equals] \ar[r] & a^\perp  \ar[r] \ar[d]  & a^\perp/a\ar[d] &[40pt]\\ 
 \textcolor{white}{e\simeq \tau(e), \hspace{10pt} b \simeq \tau(b)}& a\ar[r]  &e\ar[r]\ar[d]  &\tau(a^\perp/a)\ar[d] & e\simeq \tau(e), \hspace{10pt} a^\perp/a \simeq \tau(a^\perp/a)\\
 &0 & \tau(a) \ar[r,equals]  &\tau(a) &
 \arrow[from=1-4, to=3-2,-, dashed,opacity=0.3]
\end{tikzcd}
\end{center}
In the above $a$ is \textbf{isotropic}, i.e.
$$\kappa_e\vert_{a \otimes a} \ = \ 0$$
and $a^\perp  \subseteq e$ is the orthogonal to $a$; the quotient naturally inherits an orthosymplectic structure from $e$. We call the above an \textit{orthosymplectic short exact sequence}, and sometimes use the shorthand 
$$0 \ \to \ a \ \to \ e \ \to \ e/a \ \to \ \tau(a) \ \to \ 0$$
to denote it.

\subsubsection{Example} When $\Ml=\BGL$ is the moduli stack of finite dimensional vector spaces and $\Ml^\tau=\BSp$ is the moduli stack of symplectic vector spaces, we have $\SES=\BP_{\GL \textup{-}\GL}$ and $\SES_3^\tau=\Pt_{\GL \textup{-}\OSpt}$ where 
\begin{center}
\begin{tikzcd}[row sep = {30pt,between origins}, column sep = {45pt, between origins}]
 & \Pt_{\GL \textup{-}\GL} \ar[rd] \ar[ld] && & & \Pt_{\GL \textup{-}\Sp} \ar[rd] \ar[ld] & \\
 \GL \times \GL  & & \GL & & \GL \times \Sp & &\Sp
\end{tikzcd}
\end{center}
 are the lower-triangular parabolics. 

\subsubsection{Left vs. right} There are two different identifications 
$$\Ml \times \Ml^\tau \ \stackrel[\raisebox{2pt}{$\sim$}]{\Delta_{12}}{\to} \  (\Ml \times \Ml \times \Ml)^{\Zb/2} \ \stackrel[\raisebox{2pt}{$\sim$}]{\Delta_{23}}{\leftarrow} \ \Ml^\tau \times \Ml.$$
Thus, taking invariants of the correspondence \eqref{fig:SES3} also gives rise to a right $\Ml$-action 
\begin{center}
  \begin{tikzcd}[row sep = {30pt,between origins}, column sep = {45pt, between origins}]
   &\SES_3^\tau \ar[rd] \ar[ld] & \\ 
  \Ml^\tau\times \Ml&& \Ml^\tau
  \end{tikzcd}
  \end{center}
and the left and right action are intertwined by $\tau: \Ml \stackrel{\sim}{\to}\Ml$.

\subsection{Virtual fundamental classes under fixed points}

\subsubsection{} 
Let $\Xl, \Yl, \Zl$ be derived Artin stacks, and let $f: \Xl \to \Yl$ be a morphism. Let $\Gamma$ be a finite group.

We obtain a morphism $f^\Gamma: \Xl^\Gamma \to \Yl^\Gamma$. For $\Zl = \Xl$ or $\Yl$, let $\iota: \Zl^\Gamma \to \Zl$ denote the forgetful map. There is an equality of cotangent complexes
\[ \Tt^*_{\Zl^\Gamma} = (\iota^* \Tt^*_{\Zl})^\Gamma. \]
by Appendix A of \cite{AKLPR}. The relative cotangent sequence becomes
\begin{equation*}
f^{\Gamma *}(\iota_{\Yl}^* \Tt^*_{\Yl})^\Gamma \to (\iota_{\Xl}^* \Tt^*_{\Xl})^\Gamma \to \Tt^*_{\Xl^\Gamma/\Yl^\Gamma} \to 
\end{equation*}
so because $\iota_{\Yl} \cdot f^\Gamma = f \cdot \iota_{\Xl}$, we obtain a quasi-isomorphism
\begin{equation} \label{eqn:RelativeCotangentFixedPoints}
   \Tt^*_{\Xl^\Gamma/ \Yl^\Gamma} = (\iota_{\Xl}^* \Tt^*_{\Xl/\Yl})^\Gamma.
\end{equation}

\subsubsection{} We here note that many properties of $\Ml^\tau$ and related correspondences are inherited from $\Ml$: 

\begin{prop}\label{prop:InheritedProperties}
  We have: 
  \begin{enumerate}
   \item if $p$ is proper, then $p^\Gamma$ is proper;
   \item if $q$ is (quasi)smooth, then $q^\Gamma$ is (quasi)smooth; 
   \item if $q$ is a smooth affine fibration and $\tau$ acts linearly on the fibers of $\SES \times_{\Ml \times \Ml \times \Ml}(\Ml\times \Ml^\OSpt)$, then $q^\Gamma$ is a smooth affine fibration.
  \end{enumerate} 
\end{prop}
\begin{proof}
  Note that for any $\Gamma$-equivariant map $\Xl\to \Yl$, we have equality of cotangent complexes \eqref{eqn:RelativeCotangentFixedPoints} and so quasismoothnes, a conidition on the tor-amplitude of $\Lb$, is inherited upon taking fixed points.  Next, we have that $\Xl^\Gamma\to \Xl\times_\Yl \Yl^\Gamma$ is proper if $\Xl\to \Yl$ is proper, and so properness is also inherited.
\end{proof}

\subsection{Orthosymplectic CoHAs and CoHAMs}

\subsubsection{} 
In this section we define orthosymplectic cohomological Hall module actions on the critical cohomology of certain moduli stacks satisfying assumptions listed below. They are all examples of (or related to) certain three-Calabi-Yau categories endowed with an involution preserving the heart and the Calabi-Yau structure. Our main examples will be 
\begin{itemize}
  \item Dimension zero sheaves on smooth surfaces.
 \item Moduli of representations of quivers $Q$ with potential $W$. 
 \item Local curves and more generally \textit{deformed CY3 completions}, see \cite{JKL}
\end{itemize}
All examples are either quasismooth stacks, or critical loci critical loci
$$\textup{crit}(W) \ \subseteq  \ \Ml$$
of a function $W$ on a smooth stack $\Ml$ which we call a \textbf{smooth model}. In both cases, these satisfy the following assumptions

\begin{customass}{\textup{StkCoHAM}}\label{ass:StkCoHAM}
  $\Ml$ is a stack and $W:\Ml\to\Ab^1$ is a function, such that:
  \begin{enumerate}
   \item \label{item:OplusWCompatibility} There is a point $0\in \Ml$ and a correspondence 
  \begin{equation}
    \label{fig:CoHACorrespondence} 
    \begin{tikzcd}[row sep = {30pt,between origins}, column sep = {45pt, between origins}]
     &\SES \ar[rd,"p"]\ar[ld,"q"']  & \\
    \Ml \times\Ml & & \Ml 
    \end{tikzcd}
  \end{equation}
  which satisfy an associativity and unit condition, such that $W\vert_0=0$ and $q^*(W \boxplus W)=p^*W$. 
  \item $q$ is quasismooth and $p$ is proper.
  \item There is an involution $\tau :  \Ml \stackrel{\sim}{\to} \Ml$ such that $\tau^*W=W$, and an involution making the following commute:
  \begin{equation}
    \label{fig:CoHACorrespondenceOSp} 
    \begin{tikzcd}[row sep = {30pt,between origins}, column sep = {45pt, between origins}]
     &\SES \ar[rd,"p"]\ar[ld,"q"']\ar[dd,"\tau_{\SES}","\wr"']  & \\
    \Ml \times\Ml \ar[dd,"\sigma\cdot (\tau \times \tau)"',"\wr"]  & & \Ml \ar[dd,"\tau","\wr"'] \\[-30pt]
    &\SES \ar[rd,"p"]\ar[ld,"q"']  & \\
   \Ml \times\Ml & & \Ml 
    \end{tikzcd}
  \end{equation}
  such that the involution respects the associativity and unit conditions of the correspondence, i.e. $\tau(0)=0$ and the associativity condition diagrams are $\tau$-equivariant.
  \end{enumerate}
\end{customass}

Note that for the CoHA of zero-dimensional sheaves on a surface, the map $q$ is quasismooth and $\Ml$ itself is not smooth; we take $W=0$. In the other three cases, the stacks in \eqref{fig:CoHACorrespondence} are in addition smooth.

\begin{lem}
  There is a map of sheaves on $\Ml$: 
  \begin{equation}
    \label{eqn:JoyceVanishing} 
    q^*(\varphi\boxtimes\varphi)\ \stackrel{\Dl}{\to} \  p^!\varphi[2 d_q], \hspace{15mm} q^{\tau,*}_3(\varphi\boxtimes\varphi^\tau)\ \stackrel{\Dl^\tau}{\to} \  p^{\tau,!}_3\varphi[2 d_{q_3^\tau}]
  \end{equation}
  where $d_\bullet$ denotes the dimension of the map.  It satisfies an associaticity, module and unit property.
\end{lem}
\begin{proof}
  The map $\Dl$ is constructed from functoriality properties of vanishing cycles functor with respect to $*$- and $!$-pullbacks:
  $$ \varphi q^*\omega  \ \to \ \varphi q^*\omega \hspace{15mm}   \varphi p^! \omega \ \to \ p^!\varphi \omega,$$
  as well as the fundamental class map $q^*\omega \to q^!\omega [2 d_q]$. Using Proposition \ref{prop:InheritedProperties}, $q_3^\tau$ and $p_3^\tau$ are also quasismooth and proper, so we can define $\Dl^\tau$ in the same way.

\end{proof}
Now consider the \emph{critical cohomology} 
$$\Ht^\sbt(\Ml,\varphi)\ =\ \Ht^\sbt(\Ml,\varphi_W(\omega_\Ml)), \hspace{15mm}  \Ht^\sbt(\Ml^\tau,\varphi^\tau)\ =\ \Ht^\sbt(\Ml^\tau,\varphi_{W\vert_{\Ml^\tau}}(\omega_{\Ml^\tau})),$$
given by applying the vanishing cycle functor of $W$ to the dualising sheaf of $\Ml$ then taking cohomology, and likewise for the fixed locus stack $\Ml^\tau$.

\begin{theorem} \label{thm:coham}
 For any stack $\Ml$ as in \ref{ass:StkCoHAM}, there is an associative algebra and left and right module structure
$$\Ht^\sbt(\Ml,\varphi)\otimes \Ht^\sbt(\Ml,\varphi)\ \stackrel{m}{\to} \  \Ht^{\sbt}(\Ml,\varphi), \hspace{15mm}  \Ht^\sbt(\Ml,\varphi)\otimes \Ht^\sbt(\Ml^\tau,\varphi^\tau) \ \stackrel{a_{L,R}}{\to} \  \Ht^{\sbt}(\Ml^\tau,\varphi^\tau)$$
with a unit and augmentation. Second, we have an involutive antiautomorphism 
$$\tau \ : \  (\Ht^\sbt(\Ml,\varphi), m) \ \stackrel{\sim}{\to}\ (\Ht^\sbt(\Ml,\varphi),m^{op})$$
exchanging $a_L$ and $a_R$.
\end{theorem}
\begin{proof} Consider the pullback by $q$, the map \eqref{eqn:JoyceVanishing}, and proper pushforward by $p$: omitting the cohomological shifts,
  $$\varphi \boxtimes \varphi \ \to \ q_*q^*(\varphi \boxtimes \varphi), \hspace{10mm} q^*(\varphi \boxtimes \varphi)\ \stackrel{\Dl}{\to} \  p^! \varphi, \hspace{10mm} p_*p^!\varphi \ \to \ \varphi.$$ 
  After taking sheaf cohomology the domains and codomains of these maps match up, so we compose them to get the CoHA product $m$. Likewise, the unit $1$ comes from the point $0:\pt \to \Ml$. The map $\Dl^\tau$ and the augmentation $1^\tau\in \Ht^\sbt(\Ml^\tau,\varphi^\tau)$ is constructed in the same way. 
\end{proof}

\subsubsection{} We thus call the above the \textit{cohomological Hall algebra} (\textit{CoHA}) and \textit{cohomological Hall module} (\textit{COHAM}). The \textit{orthosymplectic CoHA} is by definition the image of the algebra map 
$$\Ht^\sbt(\Ml,\varphi) \ \to \ \End \left(\Ht^\sbt(\Ml^\tau,\varphi^\tau) \right),$$
in particular it is a quotient of the ordinary CoHA.

\subsection{Dimensional reduction}

\subsubsection{}  We now consider the case of a dimensional reduction isomorphism.

\begin{prop}  \label{prop:DimensionalReduction}
  Assume that 
  $$\pi\ : \ \Ml \ \to \ \Nl$$
  is a $\Zb/2$-equivariant map of spaces satisfying assumption \ref{ass:StkCoHAM}. In addition, assume that $\pi$ is a smooth affine fibration such that the function $W_\Ml$ is linear and the $\Zb/2$-action is linear in the fibres of $\pi$. Then we have an isomorphism 
  $$\Ht^\sbt(\Ml^{\OSpt},\varphi_{W_\Ml^{\OSpt}}) \ \stackrel{\sim}{\to} \ \Ht^{\BM}_\sbt(\Nl^{\OSpt}).$$
\end{prop}
\begin{proof}
   Note that $\pi^{\OSpt}:\Ml^{\OSpt}\to \Nl^{\OSpt}$ is also a smooth affine fibration and $\varphi_{W_\Ml}^{\OSpt}$ is linear in its fibres. Thus the result follows by Corollary A.9 of \cite{Da}. 
\end{proof}

When the $\Zb/2$ action is trivial, this recovers the ordinary dimensional reduction formula. The main example we will consider is 
$$\pi \ : \ \Ml_{\Rep \Jac(Q^{(3)},W)} \ \to \ \Ml_{\Rep \Pi_Q}$$
equates the critical cohomology of the moduli stack of tripled quivers with its canonical cubic potential with the preprojective algebra; see section \ref{sec:TwistedYangians}.

\newpage 
\section{Moduli stacks as hyperplane arrangements}
\label{sec:ModuliStacksRootData}

\noindent In this section, we write down a common generalisation of the classical theory of root systems for Lie algebras $\gk$ and the cohomology of moduli stacks $\Ml$. 
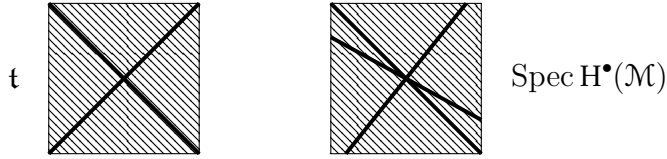
\begin{equation*}
  \label{fig:SymplecticFactorisationAlgebra} 
  \begin{tikzpicture}[scale=0.5]
    \begin{scope} 
     [xshift=-7.5cm]

   \draw[black, fill=black, fill opacity = 0.2, pattern=north west lines] (-2,-2) rectangle (2,2);

   \draw[black, line width = 1.5pt] (-2,-2) -- (2,2);
   \draw[black, line width = 1.5pt] (-2,2) -- (2,-2);
    \node[left,white] at (-2.5,0) {$\Spec \Ht^\sbt(\Ml)$};
    \node[left] at (-2.5,0) {$\tk$};
     \end{scope}

   \draw[black, fill=black, fill opacity = 0.2, pattern=north west lines] (-2,-2) rectangle (2,2);

    \draw[black, line width = 1.5pt] (-1.6,-2) -- (1.6,2);
    \draw[black, line width = 1.5pt] (-2,1.1) -- (2,-1.1);
    \draw[black, line width = 1.5pt] (-2,2) -- (2,-2);
    \node[right] at (2.5,0) {$\Spec \Ht^\sbt(\Ml)$};
   \end{tikzpicture}
\end{equation*}
This relates hyperplanes $\alpha=0$ in $\gk$ with the zeros and poles of the Euler class $e(\Tt_\Ml)$ of the tangent complex of $\Ml$, and relates parabolics $\pk$ with moduli stacks $\SES$ of short exact sequences.

\subsection{Moduli group stacks and moduli hyperplane arrangement}

\subsubsection{} Let $\Gb$ be an Artin stack with an associative monoid structure 
$$\oplus \ : \ \Gb \times \Gb \ \to \ \Gb, \hspace{15mm} 1_\Gb \ : \ \pt \ \to \ \Gb.$$
A \textbf{hyperplane arrangement} for $\Gb$ is a collection $\Phi=\{H_\alpha\}$ of hyperplanes $H_\alpha \subseteq \Spec \Ht^\sbt(\Gb)$, or equivalently their defining functions $f_\alpha$, such that 
$$\Phi \star \Phi  \ \subseteq \  \oplus^*\Phi, \hspace{15mm} \{0\} \ \subseteq \ 1_\Gb^*\Phi$$
where $\star$ is defined in \eqref{eqn:MonoidalStructureGb}. We call the pair $(\Gb, \Phi)$ a \textbf{moduli group stack}.

\subsubsection{} These form a category, where morphisms $(\Gb,\Phi_\Gb)\to (\Hb, \Phi_\Hb)$ are maps of monoid stacks
$$f \ : \ \Gb \ \to \ \Hb$$
such that $\Phi_\Gb \subseteq f^{-1}(\Phi_\Hb)$. It is called \textit{strict} if $\Phi_\Gb = f^{-1}(\Phi_\Hb)$.

\subsubsection{} 
We can form the category $\textup{StkHyp}$ of tuples $(\Xl,\Phi)$ of an Artin stack and set $\Phi$ of subschemes $H_\alpha  \subseteq \Spec \Ht^\sbt(\Xl)$. This is symmetric monoidal, for 
\begin{equation}
  \label{eqn:MonoidalStructureGb} 
  (\Xl, \Phi) \star (\Yl, \Psi) \ = \ (\Xl\times \Yl, \Phi \times \Spec \Ht^\sbt(\Yl) \, \sqcup \, \Spec \Ht^\sbt(\Xl) \times \Psi). 
\end{equation}

\begin{lem}
  The category of moduli group stacks is equivalent to $\Alg(\textup{StkHyp})$, the category of associative monoids in $\textup{StkHyp}$.
\end{lem}

This allows us to define $\Gb$-modules, as well as the conditions on $\Gb$ of being $\Eb_n$-commutative, etc. 

\subsubsection{Colours} A \textit{colour} is a connected component $c \in\pi_0(\Gb)$. A colour is \textit{simple} if it is primitive for the coproduct 
$$\oplus^{-1} \ : \ \pi_0(\Gb)\ \to \ \pi_0(\Gb) \times \pi_0(\Gb).$$
In other words, there are no pairs of nonzero colours $c_i$ with $c_1\oplus c_2=c$.

\subsection{Lie-theoretic and moduli stack examples}

\subsubsection{} Let $\GL$ be the infinite disjoint union of $\GL_n$ for $n \ge 0$. Then $\Gb=\BGL$ is a moduli group stack, where
$$\oplus \ : \ \BGL \times \BGL \ \to \ \BGL$$
is given by direct sum of vector bundles, and the hyperplane arrangement $\varphi_{\GL}$ is given by the diagonals 
$$\Delta_{ij} \ \subseteq \ \tk // W  \ = \ \Spec \Ht^\sbt(\BGL) $$
given by $x_i=x_j$, where $\tk$ and $W$ are the colimits of $\tk_{\GL_n}$ and $W_{\GL_n}$ respectively, and $x_i$ are coordinates on $\tk$.

\subsubsection{} \label{sssec:MaximalTorus} To get noncommutative examples, we pass to $\Gt=\Tt_{\GL}$ the disjoint union of the maximal tori in $\GL_n$. Then 
$$\oplus \ : \ \BT \times \BT \ \to \ \BT$$
is given by concatenation of ordered tuples of line bundles, and the hyperplane arrangement $\varphi_{\GL}^{\ord}$ is given by the diagonals $\Delta_{ij} \subseteq \tk = \Spec \Ht^\sbt(\BT)$. There is a map of moduli group stacks $\BT \to \BGL$ induced by inclusion of the maximal torus, or on functors of points by taking the direct sum of line bundles in the ordered tuple.

\subsubsection{} We may likewise consider the stack $\Gb=\BGL^{\times Q_0}$ parametrising $Q_0$-coloured vector bundles, with its product hyperplane arrangement.

\subsubsection{} If $\Gt\in \{\Ot,\Sp\}$ and $\Tt_{\Gt}$ is the maximal torus, then $\Gb=\Bt\Tt_{\Gt}$ is a moduli group stack with hyperplane arrangement
$$H_\alpha  \ \subseteq \ \tk \ = \ \Spec \Ht^\sbt(\Bt\Tt_{\Gt})$$
given by the hyperplanes of $\Gt$, likewise we may take $\Gb=\BG$ with hyperplane arrangement the images $ \pi(H_\alpha)  \subseteq \tk//W$ under the quotient map $\pi:\tk \to \tk//W$.

\subsubsection{} \label{sssec:Parabolic} Next, let $\Gt\in \{\Ot,\Sp\}$ and consider a parabolic $\Pt$ between it and the general linear group. Then $\Gb=\BP$ and $\Gb=\Bt\Tt_{\Pt}\simeq \BT_{\GL}\times \BT_{\Gt}$ are module group stacks with hyperplane arrangement given by the hyperplanes not lying in the Levi:
$$\varphi_{\Pt} \ = \ \{H_\alpha \ : \ \alpha \in \pk/\lk\}$$
where the hyperplanes are viewed as elements of $\tk_{\pk}^{(W_\Pt)}\simeq \tk_{\glk}^{(W)}\times \tk_{\gk}^{(W_\Gt)}$. The induction diagram
\begin{center}
  \begin{tikzcd}[row sep = {30pt,between origins}, column sep = {45pt, between origins}]
   &\Bt\Pt \ar[rd] \ar[ld]  & \\ 
  \BGL \times \BG & & \BG
  \end{tikzcd}
  \end{center}
conists of morphisms of moduli group stacks.

\subsubsection{Standard parabolics} If $\Gt$ has classical type, the \textit{standard} (or \textit{negative}) parabolics is the family 
\begin{center}
  \begin{tikzcd}[row sep = {30pt,between origins}, column sep = {45pt, between origins}]
   &\Pt \ar[rd] \ar[ld]  & \\ 
   \GL^{\times k} \times \OSpt & & \OSpt
  \end{tikzcd}
  \end{center}
Given by strictly lower-triangular matrices, i.e. such that all roots in $\pk/\lk$ are contained in $\gk^-$ for the standard triangular decomposition on $\gk$.

\subsubsection{Moduli stack examples} Let $\Cl$ be an abelian or dg category and $\Gb=\Ml_\Cl$ its moduli stack of objects. Then the direct sum functor 
$$\oplus \ : \ \Cl \times \Cl \ \to \ \Cl, \hspace{15mm} 0 \ : \ \triv \ \to \ \Cl$$ 
gives by functoriality of moduli stacks a commutative monoid structure on $\Gb$. Then given any multiplicatively closed subset of cohomology classes $S \subseteq \Ht^\sbt(\Ml)$ containing the unit and satisfying $S_1 \cdot S_2  \subseteq \oplus^*S$, we get a hyperplane arrangement $\varphi_{S}$ given by 
$$H_s \ = \ (s=0)  \ \subseteq \ \Spec \Ht^\sbt(\Ml)$$ 
making $(\Ml,\Phi_S)$ into a moduli group stack.


\subsubsection{} We write the analogue of the maximal torus example \ref{sssec:MaximalTorus}. Let $\Cl^s$ be a collection of subcategories of $\Cl$ whose objects are the points of locally closed substack $\Ml^s \subseteq \Ml$. Call these objects \textit{split}. Then the \textit{split moduli stack}
$$\Gb \ = \ \Ml^s \ \defeq \ \coprod_{n} (\Ml^s \times \cdots \times \Ml^s)$$
parametrising tuples of $n\ge 0$ split objects has an associative monoid structure 
$$\oplus \ : \ \Ml^s \times \Ml^s \ \to \ \Ml^s$$
given by concatenation of tuples. The map $i:\Ml^s \to \Ml$ taking a direct sum of tuple is a map of moduli group stacks, for the hyperplane arrangement given by $\varphi_{\Ml^s}= i^{-1}(\Phi_\Ml)$. 

Relevant examples for $\Cl=\Rep Q$ is the subcategory of objects $\Cl^s=\prod \Cl_{\delta_i}$ which have dimension one, and for $\Cl=\Coh X$ for variety $X$ is the subcategory $\Cl^s=\Cl^{\le 1}$ of sheaves with rank zero or one.

\subsection{Hyperplane arrangement from tangent complexes}

\subsubsection{} To begin, notice that for any smooth algebraic group $\Gt$ the tangent complex of $\BG$ is $\Tt_{\BG}=\gk[1]/\Gt$. If $\Gt$ is reductive or parabolic,
$$e(\Tt_{\BG})^{-1} \ = \ \prod_{\alpha \in \Phi\sqcup 0} x_\alpha \ \in \ \Ht^\sbt(\BG)$$
is a product over all hyperplanes $\alpha$ of coweights $x_\alpha=e_\Gt(\gk_\alpha)\in \Ht^\sbt(\BT)$, where $\Tt$ is the maximal torus of $\Gt$.
\begin{lem}
  Taking $\Phi$ to be the irreducible components of $e(\Tt_\BG[-1])=e_\Gt(\gk)$ gives a moduli hyperplane arrangement on $\BG$. 
\end{lem}

Thus we have obtained the root system from taking the tangent complex of $\Gb$, and will now generalise this to the moduli stack case. Note that of course the link between root data and tangent bundles is not a new observation, and goes back to at least Atiyah-Bott's work on torus localisation.

\subsubsection{} If $\Gb=\Ml$ is a moduli stack of objects in an abelian category $\Cl$, then its fibre at an object $c\in \Cl$ is the dg vector space
$$\Tt_{\Ml,c} \ = \ \Ext(c,c)[1].$$
Assuming that this is a global complex of vector bundles, or more generally that we are working in a setting where its (localised) Euler class is a well-defined element $e(\Tt_\Ml)_{\loc} \in \Ht^\sbt(\Ml)_{\loc}$, we get a rational function on $\Spec \Ht^\sbt(\Ml)$. Taking $\Phi$ to be the zeroes and poles of $e(\Tt_{\Ml}[-1])_{\loc}$ gives the structure of a moduli group stack on $\Ml$.

\subsubsection{}
\label{ssec:TangentBorcherds} Let $\Ml$ be an Artin stack with an associative unital correspondence 
\begin{center}
\begin{tikzcd}[row sep = {30pt,between origins}, column sep = {45pt, between origins}]
  &\SES \ar[rd,"p"] \ar[ld,"q"']  & \\
\Ml \times \Ml & & \Ml 
\end{tikzcd}
\end{center}
In this section we show that the tangent complex $\Tt_q$ formally satisfies the hexagon relations.

\begin{prop}
  Writing  
  \begin{center}
  \begin{tikzcd}[row sep = {30pt,between origins}, column sep = {45pt, between origins}]
    &\SES_3 \ar[rd,"p_3"] \ar[ld,"q_3"']  & \\
  \Ml \times \Ml \times \Ml & & \Ml 
  \end{tikzcd}
  \end{center}
  we have as classes in the Grothendieck group that $[\Tt_{q_3}]\simeq [\Tt_{q,12}\oplus\Tt_{q,13}\oplus\Tt_{q,23}]$, where $\Tt_{q,ij}$ denotes the pullback along $\pi_{ij}:\SES_3\to \SES$.
\end{prop}
\begin{proof}
  We will in fact show there is a distinguished triangle 
  $$\Tt_{q,13} \ \to \ \Tt_{q_3} \ \to \ \Tt_{q,12}\oplus \Tt_{q,23}$$
  This follows by applying functoriality of $\Tt_{(-)}$ for pullbacks and compositions, to 
  \begin{center}
  \begin{tikzcd}[row sep = {30pt,between origins}, column sep = {20pt}]
   \SES_{3,a,b,c} \ar[r] \ar[d] &\SES_{a,b}\times_{\Ml_b}\SES_{b,c}\ar[r] \ar[d]  &\Ml_{a}\times \Ml_b \times \Ml_c \\ 
   \SES_{a,c} \ar[r] &\Ml_a\times \Ml_c 
  \end{tikzcd}
  \end{center} 
  Thus it remains to show the square is a pullback, which is easily checked by associativity of $\SES$.
\end{proof}

Now assume that $q$ admits a section $s$, and write $\Nt_s=\Tt_s[1]$ for the normal complex. 

\begin{cor}
 We have $[\Nt_{s_3}]\simeq [\Nt_{s,12}\oplus\Nt_{s,13}\oplus\Nt_{s,23}]$.
\end{cor}
\begin{proof}
  We have $\Nt_s=s^*\Tt_q$ by applying the distinguished triangle for a tangent complex of a composition to the identity map $q\cdot s$. 
\end{proof}

Writing $\oplus = p\cdot s$ for the composition, this implies 

\begin{cor} \label{cor:EulerGLHexagon}
  As elements of the Grothendieck group,
  $[(\id\times \oplus)^*\Nt_s]=[\Nt_{s,12}\oplus\Nt_{s,13}]$ and $[(\oplus\times \id)^*\Nt_s]=[\Nt_{s,13}\oplus\Nt_{s,23}]$. 
\end{cor}

Now assume that $\Ml$ is endowed with an involution for which  $\SES$ is in addition equivariant, as in the proof of Proposition \ref{prop:MActionOnMt}. Arguing as above, we get

\begin{cor} \label{cor:EulerGLOSpModuleHexagon}
  As elements of the Grothendieck group,
  $[(\id\times \oplus_{\OSpt})^*\Nt_{s_3}^\tau]=[\Nt_{s,12} \oplus\Nt_{s_3,13}^\tau \oplus (\id\times \tau)^*\Nt_{s,12}]$ and $[(\oplus\times \id)^*\Nt_s]=[(\id\times\tau)^*\Nt_{s,12}\oplus\Nt_{s_3,13}^\tau\oplus\Nt_{s_3,23}^\tau]$. 
\end{cor}

\subsubsection{Remark} We explicitly compute the orthosymplectic normal complex. Let 
$$\Kt_s \ = \ (\Nt_s \vert_\Ml)^{\Zb/2}$$
 denote invariants applied to the pullback of the normal complex along the map $\Ml \simeq (\Ml \times \Ml)^{\Zb/2}$ where the involution sends $a \mapsto (a,\tau(a))$. For instance, $\Kt_{s,a} = \Ext(a,a^*)^{\Zb/2}\simeq(a^* \otimes a^*)^{\sigma}$ when we consider the moduli stack of finite dimensional vector spaces.

\begin{prop} \label{prop:OSpNormalComplex}
  As a perfect complex on $\Ml^\tau \times\Ml$,
  $$\Nt_{s^\tau_3}\ =\ {\Nt_s}\vert_{\Ml^\tau \times\Ml} \oplus \Kt_s $$ 
  where the last summand is the the normal complex of $s^\tau:\Ml\to \SES^\tau$.
\end{prop}
\begin{proof}
  Note that for a $\Gt$-equivariant map $f:\Xl\to \Yl$, 
  \begin{equation}
    \label{eqn:NormalComplexOfInvariants} 
    \Tt_{f^G} \ = \ \left({\Tt_f}\vert_{\Xl^G}\right)^G
  \end{equation}
  relates its tangent complex with to that of $f^G:\Xl^G\to \Yl^G$, even though in general $f^G$ is not ofbtained from $f$ by base change; see \cite{AKLPR}. Thus 
  \begin{align*}
    \Nt_{s_3^\tau} & \ = \ \left(\Nt_{s_3}\vert_{\Ml \times \Ml^\tau}\right)^{\tau}\\
    & \ = \ \left(\Nt_{s,12} \oplus \Nt_{s,13} \oplus \Nt_{s,23}\right)^{\tau}\\
    & \ = \ (\Nt_{s,12} \oplus \Nt_{s,13})^{\tau} \oplus \Nt_{s,23}^{\tau}\\
    & \ = \ \Nt_{s}\vert_{\Ml \times \Ml^\tau} \oplus \Nt_{s,2}^{\tau}
  \end{align*}
  and in the isomorphism exhibiting equivariance $\Nt_{s_3}\simeq \tau^*\Nt_{s_3}$ identifies the two outer summand and fixes the middle one. 
\end{proof}

\newpage 
 \section{Configuration spaces} \label{sec:Configuration}

\noindent In this section, we will study generalisations of the \textit{configuration space} of unordered tuples of points on the affine line
$$\Conf\Ab^1 \ = \ \sqcup_{n \ge 0}\Ab^n //\Sk_n.$$
For instance, we consider variants attached to any appropriate sequence $\Gt=\{\Gt_n\}$ of algebraic groups or moduli stacks $\Ml$:
$$\Conf_\Gt \Ab^1 \ = \ \sqcup_{n \ge 0}\gk_n, \hspace{10mm} \Conf_{\Ml}\Ab^1 \ = \ \Spec \Ht^\sbt(\Ml),$$
and most generally any moduli group stack $\Gb$.

Our main results of the section (Theorems \ref{thm:LocalisedCoproducts} and \ref{thm:LocalisedBiAlgebras}) characterise quasicoherent sheaves on these configuration spaces with \textit{factorisation} structure, as ($\varnothing$/co/bi)algebras with singularities over the hyperplanes.

Loosely speaking, a factorisation algebra on $\Conf_\Gb\Ab^1$ is $\Al \in \QCoh(\Conf_{\Gb}\Ab^1)$ together with a map 
$$ (\Al\boxtimes \Al)\vert_\circ \ \to \ (\cup^*\Al)\vert_\circ$$
defined on the restriction to the complement $\circ$ of ``off-diagonal'' hyperplanes:
\begin{equation*}
  \label{fig:SymplecticFactorisationAlgebra} 
  \begin{tikzpicture}[scale=0.7]
    \begin{scope} 
     [xshift=-7.5cm]

   \draw[black, fill=black, fill opacity = 0.2, pattern=north west lines] (-2,-2) rectangle (2,2);

   \draw[black, line width = 1.5pt] (-1.6,-2) -- (1.6,2);
   \draw[black, line width = 1.5pt] (-2,1.1) -- (2,-1.1);
   \draw[black, line width = 1.5pt] (-2,2) -- (2,-2);
    \node[right] at (2.5,0) {$(\Conf_\Gb \Ab^1\times \Conf_\Gb \Ab^1)_\circ$};
    \node[left,white] at (-2.5,0) {$(\Conf_\Gb \Ab^1\times \Conf_\Gb \Ab^1)_\circ$};
     \end{scope}

   \end{tikzpicture}
\end{equation*}
This generalises \cite{Da} and \cite{JKL}.

\subsection{Configuration spaces from $\BGL$} 

\label{ssec:OrthosymplecticConfigurationSpaces}

\subsubsection{} We begin by noticing that the configuration space of points on the affine line is equivalently
$$\Conf\Ab^1 \ = \ \Spec \Ht^\sbt(\BGL) $$
and the \textit{ordered configuration space} $\sqcup_{n \ge 0}\Ab^n$ is
$$\Conf^{\ord}\Ab^1 \ = \ \Spec \Ht^\sbt(\Bt\Tt) $$
where as in section \ref{sec:ModuliStacksRootData}, the moduli stacks $\BT,\BGL$ parametrise finite tuples of line bundles and of vector bundles of arbitrary finite rank.

\begin{lem}
 $\Conf^{\ord}\Ab^1$ and $\Conf \Ab^1$ are (commutative) ring objects, and there is a map $\Conf^{\ord}\Ab^1 \to \Conf \Ab^1$ of rings.
\end{lem}
\begin{proof}
Addition and tensor product of vector bundles gives
$$\oplus_{\GL} \ : \ \BGL \times \BGL \ \to \ \BGL, \hspace{15mm} \otimes_{\GL} \ : \ \BGL \times\BGL \ \to \ \BGL,$$
which also restricts to ring structure on $\Bt\Tt$, so $\BT \to \BGL$ is a map of ring objects. Note that the classifying stack $\Bt \Tt$ parametrises ordered tuples of line bundles, and  
$$\oplus_{\Tt} \ : \ \BT \times \BT \ \to \ \BT, \hspace{15mm} \otimes_{\GL} \ : \ \BT \times\BT \ \to \ \BT,$$
 concatenates one tuple to another and take the termwise tensor product of one tuple with the other, respectively. Since applying spectrum of cohomology is symmetric monoidal, this implies the result on configuration spaces. 
\end{proof}

To be explicit, we have two maps 
$$\cup \ :\ \Conf^n \Ab^1 \times\Conf^m \Ab^1 \ \to \ \Conf^{n+m}\Ab^1, \hspace{10mm}   + \ :\ \Conf^n \Ab^1 \times\Conf^m \Ab^1 \ \to \ \Conf^{nm}\Ab^1 $$
sending a pair of tuples $(S,T)$ to
$$S\cup T \ = \ (s_1,\ldots ,s_n,t_1,\ldots , t_m), \hspace{15mm}  (s_1+t_1,\ldots ,s_i+t_j,\ldots ,s_n+t_m).$$

\begin{lem}
   The correspondence 
\begin{center}
  \begin{tikzcd}[row sep = {30pt,between origins}, column sep = {65pt, between origins}]
   &(\Conf \Ab^1 \times \Conf \Ab^1)_\circ \ar[ld,"j"'] \ar[rd,"\cup j"]   & \\ 
   \Conf \Ab^1 \times\Conf \Ab^1& & \Conf \Ab^1
  \end{tikzcd}
  \end{center}
  induced by the locus $\circ$ of disjoint tuples satisfies a commutativity, associativity and unit condition. The analogous correspondence for $\Conf^{\ord}\Ab^1$ satisfies an associativity and unit condition.
\end{lem}

We call the above the \textit{chiral} decomposition structure on the configuration space.

\subsection{Configuration spaces attached to moduli hyperplane arrangement} \label{ssec:ConfigurationSpacesModuliRootData}

\subsubsection{} For $\Gb$ any Artin stack with associative monoid structure $\oplus$, then the \textit{$\Gb$-configuration space} of points on $\Ab^1$ to be
$$\Conf_{\Gb}\Ab^1 \ = \ \Spec \Ht^\sbt(\Gb)$$
has the structure of an associative monoid, denoted by $\cup$. If $\Phi$ is hyperplane arrangement making $\Gb$ into a moduli group stack, then 

\begin{prop} \label{prop:GChiralStructureConf}
  There is a correspondence
  \begin{equation}
\label{fig:GChiralStructure}
\begin{tikzcd}[row sep = {30pt,between origins}, column sep = {65pt, between origins}]
  &(\Conf_{\Gb}\Ab^1 \times \Conf_{\Gb}\Ab^1)_\circ \ar[ld,"j"'] \ar[rd,"\cup j"]   & \\ 
  \Conf_{\Gb}\Ab^1 \times\Conf_{\Gb}\Ab^1& & \Conf_{\Gb}\Ab^1
\end{tikzcd}
  \end{equation}
satisfying an associativity and unit condition.
\end{prop}
\begin{proof}
  We define  
  $$(\Conf_{\Gb}\Ab^1 \times \Conf_{\Gb}\Ab^1)_\circ \ = \ (\Conf_{\Gb}\Ab^1 \times \Conf_{\Gb}\Ab^1) \smallsetminus (\oplus^*\Phi \smallsetminus \Phi \star \Phi)$$
  by subtracting the subspace given by unions of off-diagonal hyperplanes. Associativity
  $$((\Conf_{\Gb}\Ab^1 \times \Conf_{\Gb}\Ab^1)_\circ \times \Conf_{\Gb}\Ab^1)_\circ \ \simeq \ (\Conf_{\Gb}\Ab^1 \times (\Conf_{\Gb}\Ab^1 \times \Conf_{\Gb}\Ab^1)_\circ)_\circ$$
  follows since both sides are the complement inside $(\Conf_\Gb\Ab^1)^{\times 3}$ of 
  $$(\oplus\times \id)^*\oplus^*\Phi \smallsetminus  \Phi \star \Phi\star \Phi \ = \ (\id\times \oplus)^* \Phi \smallsetminus  \Phi \star \Phi\star \Phi.$$
  The unit is given by the point $1_\Gb: \pt \to \Conf_\Gb\Ab^1$ coming from the unit in $\Gb$: 
  $$1_\Gb \ : \ \Conf_{\pt}\Ab^1 \ = \ \pt \ \stackrel{1_\Gb}{\to} \ \Conf_{\Gb}\Ab^1,$$
  the pullback along $1_\Gb\times \id$ or $\id\times 1_\Gb$ of $(\Conf_\Gb\Ab^1\times \Conf_\Gb\Ab^1)$ both being identified with $\Conf_\Gb\Ab^1$, which follows because 
  $$(1_\Gb\times \id)^*\oplus^*\Phi \ = \ \Phi \ = \ (1_\Gb\times \id)^*(\Phi \star \Phi).$$
\end{proof}

\subsubsection{Unipotent hyperplanes} \label{sssec:UnipotentRoots}
We call the above the \textit{$\Gb$-chiral} decomposition structure on the $\Gb$-configuration space, and $\Phi_2=\oplus^*\Phi \smallsetminus \Phi \star \Phi$ the set of \textit{unipotent} hyperplanes.

\subsubsection{} We now consider functoriality with respect to a map $f:(\Gb,\Phi_\Gb) \to (\Hb,\Phi_\Hb)$ of moduli group stacks.
\begin{lem} \label{lem:FunctorialityConfigurationSpaces}
  The induced map 
$$f \ : \ \Conf_{\Gb}\Ab^1 \ \to \ \Conf_{\Hb}\Ab^1$$
is a lax map of monoids with respect to the chiral decomposition structures $ch_\Gb, ch_\Hb$. If $f$ is strict, then $f$ is a strict map of monoids with respect to the chiral decomposition structures.
\end{lem}
\begin{proof}
  Note that $f$ laxly (resp. strictly) preserving the chiral decomposition structure means that there is a two-(iso)morphism in $\PreStk^{\corr}$
  \begin{center}
  \begin{tikzcd}[row sep = {30pt,between origins}, column sep = {20pt}]
    \Conf_{\Gb}\Ab^1 \times \Conf_{\Gb}\Ab^1 \ar[d,"f \times f"]
\ar[r,"\cup_\Gb"{name=U}] &  \Conf_{\Gb}\Ab^1 \ar[d,"f"]  \\ 
\Conf_{\Hb}\Ab^1 \times \Conf_\Gb \Ab^1\ar[r,"\cup_{\Hb}"'{name=D}] & \Conf_{\Hb}\Ab^1 
  \arrow[Rightarrow, from=U, to=D, "\eta", shorten <= 7pt, shorten >= 7pt] 
  \end{tikzcd}
  \end{center}
  or in other words that there is a map 
  $$ \left( \Conf_{\Gb}\Ab^1 \times \Conf_{\Gb}\Ab^1\right)\times_{\Conf_{\Hb}\Ab^1 \times \Conf_{\Hb}\Ab^1}(\Conf_{\Hb}\Ab^1 \times \Conf_{\Hb}\Ab^1)_{\Hb \circ} \ \to \  (\Conf_{\Gb}\Ab^1 \times \Conf_{\Gb}\Ab^1)_{\Gb \circ}$$
  of correspondences from $\Conf_{\Gb}\Ab^1 \times \Conf_{\Gb}\Ab^1$ to $\Conf_{\Hb}\Ab^1$. Such a map is induced from $f\times f$, so long as $\varphi_{2,\Gb} \subseteq (f\times f)^*\varphi_{2,\Hb}$, and is an isomorphism whenever this is an equality. This is true whenever $\Phi_\Gb \subseteq f^*\Phi_\Hb$ and $f$ is strict respectively. 
\end{proof}

\begin{cor} \label{cor:FunctorialityConfigurationSpacesCoAlgebra}
  There is a map on quasicoherent factorisation algebras 
  $$ f_* \ : \ \FactAlg(\Conf_\Gb \Ab^1) \ \to \ \FactAlg(\Conf_{\Hb}\Ab^1)$$
  and if $f$ is strict, this has a left adjoint 
  $$ \FactAlg(\Conf_\Gb \Ab^1) \ \leftarrow \ \FactAlg(\Conf_{\Hb}\Ab^1) \ : \ f^* $$
  The same is true for localised coalgebras. 
\end{cor}

\subsection{Orthosymplectic-linear configuration spaces}  \label{ssec:OSpConfSpaces}

\subsubsection{} The \textit{orthosymplectic configuration space} of points on $\Ab^1$ is
$$\Conf_{\Gt}\Ab^1 \ = \ \Spec \Ht^\sbt(\BG) \ = \ \sqcup_{n\ge 0}\tk_{\Gt_n}//W_n,$$
for any sequence of groups $\Gt\in \{\Ot,\Ot_{\textup{ev}},\Sp\}$ of classical type. The \textit{ordered orthosymplectic configuration space} is
$$\Conf_{\Gt}^{\ord}\Ab^1 \ = \ \Spec \Ht^\sbt(\Bt\Tt_{\Gt}) \ = \ \sqcup_{n \ge 0} \tk_{\Gt_n}$$
 where $\Tt_{\Gt}$ is the maximal torus in $\Gt$. Intermediate to this is the \textit{split orthosymplectic configuration space}
$$\Conf_{\Gt}^{\textup{s}}\Ab^1 \ = \ \Spec \Ht^\sbt(\BG^s)$$
 where $\BG^s$ classifies ordered tuples of rank two bundles with nondegenerate (anti)symmetric bilinear form;\footnote{Above each fibre $\Cb^2$ such a bundle will be $\hspace{5mm} \kappa \ = \ \left( \begin{smallmatrix} 
  &1 \\ 
 1 & 
 \end{smallmatrix} \right)$ or $\left( \begin{smallmatrix} 
  &1 \\ 
 -1 & 
 \end{smallmatrix} \right)$. } note that $\Tt_\Gt  \subseteq \Gt^s  \subseteq \Gt$.

\begin{prop} \label{prop:ConfActionConfSp} 
  There are maps 
  $$\cup_\Gt,\, +_\Gt \ : \ \Conf_{\Gt}^{(\ord)}\Ab^1 \times \Conf_{\Gt}^{(\ord)}\Ab^1 \ \to \  \Conf_{\Gt}^{(\ord)}\Ab^1$$
  forming a ring structure on the $\Gt$-configuration space.  There is a monoid map 
  $$\iota_{\Gt \textup{-}\GL} \ : \ \Conf^{(\ord)}\Ab^1 \ \to \ \Conf_{\Gt}^{(\ord)}\Ab^1$$
  making them into pointed modules for $(\Conf^{(\ord)}\Ab^1,\cup)$. 
 \end{prop}
 \begin{proof}
 There is a ring structure on $\BG$ induced on functors of points by taking direct sums and tensor produts of vector bundles with their bilinear forms. The maps 
 $$\BGL_n \ \to \ \BO_{2n}, \hspace{15mm} \BGL_n \ \to \ \BSp_{2n}$$
 defined by $V \mapsto (V \oplus V^*, \left( \begin{smallmatrix} 
  &\ev \\ 
 \pm \ev & 
 \end{smallmatrix} \right))$ induce the maps $\iota_{\Gt \textup{-}\GL}$ on configuration spaces, where $\ev:V \otimes V^*\to \Ol$ is the evaluation map. These maps intertwine the direct sum maps on both sides, making $\iota_{\Gt \textup{-}\GL}$ into a monoid map.
\end{proof}

\begin{prop}
   There is a (lax) action 
   \begin{equation}
    \label{fig:ConfActionConfSp} 
    \begin{tikzcd}[row sep = {30pt,between origins}, column sep = {65pt, between origins}]
     &(\Conf \Ab^1 \times \Conf_{\Gt} \Ab^1)_{\Gt\circ} \ar[rd] \ar[ld]  & \\ 
     \Conf \Ab^1 \times \Conf_{\Gt} \Ab^1 & & \Conf_{\Gt} \Ab^1 
    \end{tikzcd}
  \end{equation}
  of the decomposition space $(\Conf^{(\ord)}\Ab^1,ch)$ on $(\Conf_{\Gt}^{(\ord)}\Ab^1,ch_\Gt)$.
\end{prop}
\begin{proof}
  The action is defined by pulling back the $\GL$-$\Gt$ hyperplanes along the map $\iota_{\Gt \textup{-}\GL}$, 
  $$(\Conf \Ab^1 \times \Conf_{\Gt} \Ab^1)_{\Gt\circ}  \ = \ (\Conf \Ab^1 \times \Conf_{\Gt} \Ab^1) \times_{\Conf \Ab^1 \times \Conf_{\Gt} \Ab^1}(\Conf_{\Gt} \Ab^1 \times \Conf_{\Gt} \Ab^1)_{\Gt\circ}.$$
   Explicitly, it is given by 
   $$\{(z_i,w_j) \ : \  z_i\pm w_j\ne 0\}.$$
   It is easily checked that this action is equipped with an associativity map between composition of cobordisms (which is not an isomorphism, hence the structure is lax).

\end{proof}

We have an analogue of Proposition \ref{prop:SymplecticRanSpace} for Ran spaces:

\begin{prop}
  For $\Gt \in \{\Ot_{\textup{odd}},\Sp\}$ there is an equivalence 
  $$(\Conf_{\Gt}\Ab^1,ch_\Gt) \ \simeq \ (\Conf(\Ab^1//\pm),ch)$$
  of unordered $\Gt$-configuration spaces respecting the chiral structures. 
\end{prop}
\begin{proof}
  We explicitly write down the hyperplanes defining the chiral decomposition structure on the orthosymplectic configuration spaces. When $\Gt=\Ot_{\textup{odd}},\Sp$, the chiral structure \eqref{fig:GChiralStructure} is the complement of the diagonals and antidiagonals
  $$\{(z_i, w_j)=(z_1,\ldots ,z_n, w_1, \ldots , w_m) \ : \ z_i \pm w_j \ne 0\}.$$ 
  The image of these (anti)diagonals under the GIT quotient 
  \begin{center}
  \begin{tikzcd}[row sep = {30pt,between origins}, column sep = {20pt}]
  \tk_n \ar[r]  &\tk_n//W_n \\ 
   \Ab^n \ar[u,equals] \ar[r] & (\Ab^1//\pm)//\Sk_n \ar[u,equals] 
  \end{tikzcd}
  \end{center}
  are the diagonals in $\Ab^1//\pm$, giving the result.
\end{proof}

Note that the above Proposition is not true for ordered configurations spaces.

\subsubsection{Parabolic configuration spaces} Now we consider the infinite disjoint union of parabolics
\begin{center}
\begin{tikzcd}[row sep = {30pt,between origins}, column sep = {45pt, between origins}]
 &\Pt_{\GL \textup{-}\Gt}\ar[rd] \ar[ld]  & \\ 
 \GL \times \Gt& & \Gt
\end{tikzcd}
\end{center}
The \textit{orthosymplectic-linear configuration space} is then 
$$\Conf_{\GL \textup{-}\Gt}\Ab^1 \ = \ \Spec \Ht^\sbt(\BP) \ \simeq \ \Conf \Ab^1 \times \Conf_\Gt \Ab^1$$
and likewise the ordered configuration space; we can likewise define \textit{linear-linear configuration spaces} by taking the parabolic above for $\Gt=\GL$, with parabolic chiral structures defined in section \ref{sssec:Parabolic}, explicitly: 

\begin{lem}
The chiral decomposition structure on $\Conf_{\GL \textup{-}\GL}\Ab^1$ and $\Conf_{\GL \textup{-}\Gt}\Ab^1$  are given by
$$\{(z_i,w_j,z'_{i'},w_{j'}') \ : \ z_i-z_{i'}', \, z_i - w_{j'}',\, z_{i'}'- w_j, \, w_j - w_{j'}' \ne 0\}$$
and
$$\{(z_i,w_j,z'_{i'},w_{j'}') \ : \ z_i-z_{i'}', \, z_i \pm w_{j'}',\, z_{i'}'\pm w_j, \, w_j \pm w_{j'}' \ne 0\}.$$
\end{lem}

\begin{prop}
  The parbolic configuration space is a semidirect product 
  $$(\Conf_{\Pt}\Ab^1, ch_\Pt) \ \simeq \ (\Conf_\Gt \Ab^1, ch_\Gt)\rtimes (\Conf \Ab^1, ch).$$ 
\end{prop}
\begin{proof}
   As a space $\Conf_{\Pt}\Ab^1$ is a product, so it remains to show that the chiral decomposition structure is a semidirect product. This is clear if we twist the product chiral structure on $(\Conf_\Gt \Ab^1, ch_\Gt)\times (\Conf \Ab^1, ch)$ by the action \eqref{fig:ConfActionConfSp} of the linear configuration space on the orthosymplectic configuration space.
\end{proof}

\begin{cor}
  There are maps of decomposition spaces 
   $$ (\Conf_\Gt \Ab^1, ch_\Gt),\,  (\Conf \Ab^1, ch) \ \to \ (\Conf_{\Pt}\Ab^1, ch_\Pt) \ \to \  (\Conf \Ab^1, ch).$$
\end{cor}

 \subsection{Localised coproducts attached to hyperplane arrangement} 

 \subsubsection{} We now show that factorisation algebras, coalgebras, etc. over $\Gb$-configuration spaces are equivalent to algebras, algebras, etc. with poles along the hyperplanes $\Phi$.

\subsubsection{} A \textit{$\Gb$-localised coalgebra} is a collection of modules 
$$B_n \ \in \ \Sb_n \Md$$
over the algebra $\Sb_n=\Ol(\Conf_{\Gb,n}\Ab^1)$ for every colour  $n\in \pi_0(\Gb)$, together with structure maps of $\Sb_{n+m}$-modules
\begin{equation}
\label{eqn:LocalisedCoproduct} \Delta_{n,m} \ : \  B_{n+m} \ \to \ B_n \otimes B_m[\{1/f_\alpha\}], \hspace{15mm} \epsilon \ \in \  B_0^*
\end{equation}
for every $n,m$, satisfying coassociativity and counit conditions. On the right, we have localised by the defining functions $f_\alpha$ for the closed subvarieties $H_\alpha \in \varphi_{n+m} \smallsetminus \Phi_n \star \Phi_m$ of the hyperplane arrangement.

For instance, 
\begin{enumerate}
 \item When $\Gb=\BGL$, a \textit{localised coproduct} is a collection of $k[x_1,\ldots ,x_n]$-modules $A_n$ along with structure maps
 $$\Delta_A \ : \  A_{n+m} \ \to \ A_n \otimes A_m [(x_i \otimes 1 - 1 \otimes x_j)^{-1}],$$
 which are linear over $k[x_1,\ldots ,x_{n+m}]$ and satisfy coassociativity and counit conditions, where $x_i \in \Ht^2(\BGm)$ is a generator. This agrees with Davison's definition in \cite{Da}, where localised coproducts were introduced.
 \item When $\Gb=\BT_{\Sp}$, an \textit{ordered symplectic localised coalgebra} is a collection of $k[x_1,\ldots ,x_n]$-modules $B_n$ along with structure maps 
 $$\Delta_B \ : \  B_{n+m} \ \to \ B_n \otimes B_m [(x_i \otimes 1 \pm 1 \otimes x_j)^{-1}],$$
 which are linear over $k[x_1,\ldots ,x_{n+m}]$ and satisfy coassociativity and counit conditions.
 \item When $\Gb=\BSp$, an \textit{symplectic localised coalgebra} has structure map 
 $$\Delta \ : \  B_{n+m} \ \to \ B_n \otimes B_m[(y_i \otimes 1 - 1 \otimes y_j)^{-1}]$$
 of $k[y_1,\ldots ,y_{n+m}]^{\Sk_{n+m}}$-modules, where $y_i \in \Ht^4(\BSp_2)$ is a generator
 \item Next, when $\Gb=\BT_{\Pt_{\GL \textup{-}\Sp}}$, an \textit{ordered linear-orthosymplectic} localised coalgebra is a localised coalgebra $A$, a symplectic localised coalgebra $B$, and maps of $\Sb_{\Sp,n+m}$-modules
 $$\Delta_{AB} \ : \ B_{n+m} \ \to \ A_{2n} \otimes B_m[(x_i \otimes 1 \pm 1 \otimes x_j)^{-1}] $$
 which are compatible with $\Delta_A$ and $\Delta_B$, i.e. $A$- and $B$-colinear.
\end{enumerate}

We have a general geometric description of localised coalgebras:

\begin{theorem} \label{thm:LocalisedCoproducts}
  Let $\Gb$ be a moduli group stack.  There is an equivalence 
  $$\Psi_\Gb \ : \ \FactCoAlg^{\textup{flat}}_{\QCoh}(\Conf_{\Gb}\Ab^1,ch_{\Gb}) \ \stackrel{\sim}{\to} \ \textup{LocCoAlg}_{\Gb}$$
  from the category of factorisation coalgebras in flat quasicoherent sheaves and the category of $\Gb$-localised coalgebras. 
\end{theorem}
\begin{proof}
  To begin, notice that 
  $$\Sb_n \ = \ \Ol \left( \Conf_{\Gb,n}\Ab^1\right), \hspace{15mm} \Sb_{n,m} \ = \ \Ol \left( (\Conf_{\Gb,n}\Ab^1 \times \Conf_{\Gb,m}\Ab^1)_{\Gb \circ}\right).$$
Thus, if $\Bl \in \QCoh(\Conf_{\Gb}\Ab^1)$ is a chiral decomposition coalgebra, we have a map 
$$\Delta \ : \ (\cup j)^*\Bl \ \stackrel{}{\to} \  j^*(\Bl \boxtimes \Bl)$$
satisfying a coassociativity and counit condition. Flatness implies that
$$\Gamma(\Conf_{\Gb,n}\Ab^1 \times \Conf_{\Gb,m} \Ab^1, \Bl \boxtimes \Bl) \ \simeq \ B_n \otimes B_m$$
therefore we get precisely a collection $B_n=\Gamma(\Conf_{\Gb,n} \Ab^1,\Bl)$ of $\Sb_n$-modules, with structure maps $B_{n+m,\loc} \to (B_n \otimes B_m)_{\loc}$ pf $\Sb_{n,m}$-modules for every pair of integers $n,m$, satisfying a coassociativity and counit condition. This map is adjoint to the localised coproduct. 
\end{proof}

  Note that either from Corollary \ref{cor:FunctorialityConfigurationSpacesCoAlgebra} or directly, for any map of moduli group stacks $f$ we have a \textit{pushforward} functor
  $$f_* \ : \ \textup{LocCoAlg}_{\Gb} \ \to \ \textup{LocCoAlg}_{\Hb}$$
  which is the identity on vector spaces $B_n$, and sends coproducts to   
  $$f_* \Delta \ : \ B_{n+m} \ \stackrel{\Delta}{\to} \ B_n \otimes B_m[\{1/f_{\Gb,\alpha}\}] \ \to \ B_n \otimes B_n[\{1/f_{\Hb,\alpha}\}] $$
  because on the right side we are inverting strictly more functions. If $f$ is strict, then we are inverting the same set of functions and so we get an left adjoint \textit{pullback} functor
  $$ \textup{LocCoAlg}_{\Gb} \ \leftarrow \ \textup{LocCoAlg}_{\Hb} \ : \ f^*$$
  acting on vector spaces as $f^*B=B \otimes_{\Ht^\sbt(\Hb)}\Ht^\sbt(\Gb)$. ,

  \begin{cor}
    For $\Gt \in \{\GL, \Ot_{\textup{odd}}, \Ot_{\textup{ev}}, \Sp\}$ there are equivalences of categories such that
     \begin{center}
     \begin{tikzcd}[row sep = {30pt,between origins}, column sep = {20pt}]
     \FactCoAlg^{\textup{flat}}_{\QCoh}(\Conf^{\ord}\Ab^1) \ar[r,"\sim"] & \textup{LocCoAlg} \\ 
     \FactCoAlg^{\textup{flat}}_{\QCoh}(\Conf_{\GL \textup{-}\Gt}^{\ord} \Ab^1) \ar[r,"\sim"] \ar[u] \ar[d] & \textup{LocCoAlg}_{\GL \textup{-}\Gt}\ar[d] \ar[u] \\
     \FactCoAlg^{\textup{flat}}_{\QCoh}(\Conf_{\Gt}^{\ord} \Ab^1) \ar[r,"\sim"] & \textup{LocCoAlg}_{\Gt} 
     \end{tikzcd}
     \end{center} 
     commutes.
  \end{cor}
  \begin{proof}
    Follows from Theorem \ref{thm:LocalisedCoproducts}, which intertwines these functoriality morphisms with those of Corollary \ref{cor:FunctorialityConfigurationSpacesCoAlgebra}, which we apply to the strict maps $\BGL \to \BP_{\GL \textup{-}\Gt} \leftarrow \BG$.
  \end{proof}

As another example, we consider the pullback of a symplectic localised coalgebra to an ordered symplectic localised coalgebra, via
$$\Conf^{\ord}_{\Sp}\Ab^1 \ \to \ \Conf_{\Sp}\Ab^1$$
takes an unordered symplectic localised coalgebra
$$\Delta \ : \ B_{n+m} \ \to \ B_n \otimes B_m[(y_i \otimes 1 - 1 \otimes y_j)^{-1}]$$
to the ordered symplectic localised coalgebra $\tilde{B}_n= B_n \otimes_{\Ht^\sbt(\BSp_{2n})} \Ht^\sbt(\BGm^n)$, with structure map 
$$f^* \Delta \ : \ \tilde{B}_{n+m} \ \stackrel{\Delta}{\to} \ \tilde{B}_n \otimes \tilde{B}_m[(y_i \otimes 1 - 1 \otimes y_j)^{-1}] \ \hookrightarrow \ \tilde{B}_n \otimes \tilde{B}_m[(\sqrt{y}_i \otimes 1 \pm 1 \otimes \sqrt{y}_j)^{-1}]$$
where $\sqrt{y}_i \in \Ht^2(\BGm)$ is a generator.

\subsection{Localised bialgebras}

\subsubsection{} A \textit{$\Gb$-localised algebra} is a collection of modules 
$$B_n \ \in \ \Sb_n \Md$$
over the algebra $\Sb_n=\Ol(\Conf_{\Gb,n}\Ab^1)$ as $n$ varies over the connected components of $\Gb$, together with structure maps of $\Sb_{n+m}$-modules
\begin{equation}
\label{eqn:LocalisedProduct} m_{n,m} \ : \  B_n \otimes B_m \ \to \ B_{n+m}[\{1/f_\alpha\}], \hspace{15mm} 1 \ \in \  B_0
\end{equation}
for every $n,m$, satisfying associativity and unit conditions.

\subsubsection{} A \textit{$\Gb$-localised bialgebra} is a collection of modules $B_n$ with the structure of a $\Gb$-localised algebra and a $\Gb$-localised coalgebra, such that the following commutes:
\begin{center}
  \begin{tikzcd}[row sep = {30pt,between origins}, column sep = {20pt}]
    (B \otimes B \otimes B \otimes B)_{\loc(13,14,23,24)} \ar[r,"\sigma_{23}"] &(B \otimes B \otimes B \otimes B)_{\loc(12,14,32,34)} \ar[r,"m \otimes m"]  & (B \otimes B)_{\loc} \\ 
   (B \otimes B) \ar[u,"\Delta \otimes \Delta"]  \ar[rr,"m"] & & B \ar[u,"\Delta"] 
  \end{tikzcd}
  \end{center}
and that the (co)products and (co)units are compatible in the usual ways:
$$\Delta(1_B) \ = \ 1_B \otimes 1_B, \hspace{15mm} \epsilon \cdot m \ = \ \epsilon \otimes \epsilon, \hspace{15mm} \epsilon(1_B)\ = \ 1.$$
In an analogous way to how Theorem \ref{thm:LocalisedCoproducts} is proven, we have

  \begin{theorem} \label{thm:LocalisedBiAlgebras}
    Let $\Gb$ be a moduli group stack. There is a diagram of equivalences 
    \begin{center}
    \begin{tikzcd}[row sep = {30pt,between origins}, column sep = {20pt}]
      \FactCoAlg^{\textup{flat}}_{\QCoh}(\Conf_{\Gb}\Ab^1,ch_{\Gb}) \ar[r,"\sim"',"\Psi_\Gb"]&  \textup{LocCoAlg}_{\Gb}\\
      \FactBiAlg^{\textup{flat}}_{\QCoh}(\Conf_{\Gb}\Ab^1,ch_{\Gb}) \ar[r,"\sim"',"\Psi_\Gb"] \ar[u] \ar[d] &  \textup{LocBiAlg}_{\Gb} \ar[u] \ar[d] \\
      \FactAlg^{\textup{flat}}_{\QCoh}(\Conf_{\Gb}\Ab^1,ch_{\Gb}) \ar[r,"\sim"',"\Psi_\Gb"]&  \textup{LocAlg}_{\Gb}
    \end{tikzcd}
    \end{center}
    extending Theorem \ref{thm:LocalisedCoproducts}. In other words, there are equivalences between flat quasicoherent (co/bi/$\varnothing$)algebras and $\Gb$-localised (co/bi/$\varnothing$)algebras, compatible with forgetting the product or coproduct on a bialgebra.
  \end{theorem}

\newpage 
\section{Orthosymplectic vertex and factorisation algebras} \label{sec:OrthosymplecticAlgebras}

\noindent In this section, we explain how the ordinary theory of vertex algebras can be viewed as being attached to $\GL$, then generalise it to other groups in section \ref{ssec:OrthosymplecticVertexAlgebras}. 

For instance, our main example in section \ref{sec:JL} will be dual to a vertex algebra acting on an $\OSpt$ vertex algebra; in Theorem \ref{thm:OrthosymplecticFactorisationAlgebras} we explain how to view this as vertex algebra attached to a parabolic. 

To justify these definitions, in section \ref{ssec:OrthosymplecticFactorisationAlgebras} we show an equivalence of generalised vertex algebras
$$\FactAlg(\Ran_\Gt\Ab^1) \ \simeq \ \textup{VAlg}_\Gt,$$
with \textit{factorisation algebras} over \textit{generalised Ran spaces}, where $\Ran_\Gt\Ab^1$ is defined by gluing together Cartan subalgebras of $\gk_n$. This reduces to Beilinson-Drinfeld's characterisation of vertex algebras as factorisation algebras, see \cite{Bu,BDr,FG,La2}

Finally, we relate the configuration spaces of last section to these vertex algebras. In Theorem \ref{thm:LocalisedToVertex} we will define a functor 
$$\Phi \ : \ \textup{LocCoAlg}_{\Gb}^{\Ga} \ \to \ \textup{VCoAlg}_\Gb$$
producing $\Gt$-vertex algebras from localised coalgebras over $\Gt$-configuration spaces.

\subsection{Recollections on vertex and factorisation algebras}

\subsubsection{} Recall that a (\textit{nonlocal}) \textit{vertex algebra} is a vector space $A$ with endomorphism $T$ together with a $T$-equivariant\footnote{Here, $T$ acts as $T-\partial_z$ on the right and a $T\otimes\id$ on the left, i.e. $T$-equivariance is $[T,Y(\alpha,z)]=\partial_z Y(\alpha, z)$.} map
$$Y \ :\ A\otimes A \ \to \ A((z))$$
such that
\begin{equation}\label{eqn:WeakAssociativity}
  Y(z)(\id\otimes Y(w))\ =\ Y(w)(Y(z-w)\otimes\id).
\end{equation}
See \cite{FBZ} or Appendix A of \cite{JKL} for details about how to expand Laurent series in \eqref{eqn:WeakAssociativity}.

A vertex algebra is called \textit{local} if for $n \gg 0$,
$$(z-w)^n\left(Y(\alpha,z)Y(\beta,w) - Y(\beta,w)Y(\alpha,z)\right) \ = \ 0 $$
which is equivalent to the condition $Y(\alpha,z)\beta = e^{zT}Y(\beta \otimes \alpha,-z)$, see for instance section 3.2 of \cite{FBZ}. A vertex algebra is \textit{unital} if there is an element $|0\rangle$ such that $Y(|0\rangle,z)=\id$ and $Y(\alpha,z)|0\rangle$ has no negative powers of $z$.

\subsubsection{} There is an equivalent definition of vertex algebras discovered by Beilinson and Drinfeld \cite{BDr}, as algebraic factorisation algebras over $\Ab^1$. Relating ordinary algebras in symmetric monoidal categories to topological factorisation algebras over $\Rb$ was achieved by Lurie in \cite{Lu}.

\subsubsection{Rough description} We first give an informal description of what factorisation algebras are in the variants which we will use them. Loosely speaking, a (\textit{linear}) \textit{factorisation algebra} on $\Cb$ is a sheaf $\Al$ (a D-module, constructible sheaf, etc.) with parallel transport maps on its fibres 
\begin{center}
 \centering
 \begin{minipage}{.5\textwidth} \centering
  
  \begin{tikzpicture}
    \begin{scope} 
      \filldraw[black, fill=black, fill opacity = 0.2, pattern=north west lines] (-6*0.45,0) -- (-1*0.45,2*0.45) --  (6*0.45,0) -- (1*0.45,-2*0.45) -- cycle;
  
        \node[right] at (7*0.4,0) {$\Cb$};
 
        \draw[-,decorate,decoration={snake,amplitude=.4mm,segment length=2mm,post length=1mm},line width = 3pt,white] (-1+0.1,0) to[out=10, in=-160] (-0.3,0.7-0.1) ;
        \draw[->,decorate,decoration={snake,amplitude=.4mm,segment length=2mm,post length=1mm}]  (-1+0.1,0) to[out=10, in=-160] (-0.3,0.7-0.1) ;
 
 
        \filldraw (-1,0) circle (1.5pt);
        \node[above] at (-1,0) {$\Al_x$};
  
        \filldraw (-0.3,0.7) circle (1.5pt);
        \node[above] at (-0.3,0.7) {$\Al_y$};
 
        \filldraw (1.5,-0.2) circle (1.5pt);
        \node[above] at (1.5,-0.2) {$\Al_z$};
 
     \end{scope}
   \end{tikzpicture}
 \end{minipage}%
 \begin{minipage}{.5\textwidth} \centering
 
\begin{equation}\label{eqn:ParallelTransport}
  ``\nabla_{\GL,\gamma} \ : \ \Al_x \otimes \Al_y \ \to \  \Al_y"
 \end{equation}
 \end{minipage}
 \end{center}
along any path $\gamma:x \to y$. Likewise, an \textit{orthosymplectic factorisation algebra} on $\Cb$ is a sheaf $\Bl$ with parallel transport maps
 
\begin{center}
  \centering
  \begin{minipage}{.5\textwidth} \centering
   
\begin{tikzpicture}

  \begin{scope} 
    \filldraw[black, fill=black, fill opacity = 0.2, pattern=north west lines] (-6*0.45,0) -- (-1*0.45,2*0.45) --  (6*0.45,0) -- (1*0.45,-2*0.45) -- cycle;

      \node[right] at (7*0.4,0) {$\Cb$};

      \filldraw (0.3,0.2) circle (1.5pt);
      \node[above] at (0.3,0.2) {$\Bl_{\overbar{v}}$};

      \filldraw (-0.3,-0.2) circle (1.5pt);
      \node[above] at (-0.3,-0.2) {$\Bl_{-\overbar{v}}$};

      \filldraw (2,0.1) circle (1.5pt);
      \node[above] at (2,0.1) {$\Bl_{\overbar{u}}$};

      \filldraw (-2,-0.1) circle (1.5pt);
      \node[above] at (-2,-0.1) {$\Bl_{-\overbar{u}}$};

      \draw[-,decorate,decoration={snake,amplitude=.4mm,segment length=2mm,post length=1mm},line width = 3pt,white] (2-0.1,0.1) to[out=-160, in=-20] (-0.3+0.1,-0.2) ;
      \draw[->,decorate,decoration={snake,amplitude=.4mm,segment length=2mm,post length=1mm}]  (2-0.1,0.1) to[out=-160, in=-20] (-0.3+0.1,-0.2) ;

   \end{scope}
 \end{tikzpicture}
  \end{minipage}%
  \begin{minipage}{.5\textwidth} \centering
  
 \begin{equation}\label{eqn:ParallelTransport}
   ``\nabla_{\Gt,\gamma} \ : \ \Bl_u \otimes \Bl_v \ \to \  \Bl_v"
  \end{equation}
  \end{minipage}
  \end{center}
along any path $\gamma:u \to \pm v$ in $\Cb$.\footnote{In other words it is a factorisation algebra on $\Cb/\pm$, see Proposition \ref{prop:SymplecticRanSpace}.} An \textit{orthosymplectic-linear factorisation algebra} is both the above data $\Al,\Bl$ together with parallel transport maps
\begin{center}
 \centering
 \begin{minipage}{.5\textwidth} \centering
  
\begin{tikzpicture}

  \begin{scope} 
    \filldraw[black, fill=black, fill opacity = 0.2, pattern=north west lines] (-6*0.45,0) -- (-1*0.45,2*0.45) --  (6*0.45,0) -- (1*0.45,-2*0.45) -- cycle;

      \node[right] at (7*0.4,0) {$\Cb$};

      \filldraw (0.3,0.2) circle (1.5pt);
      \node[above] at (0.3,0.2) {$\Bl_{\overbar{v}}$};

      \filldraw (-0.3,-0.2) circle (1.5pt);
      \node[above] at (-0.3,-0.2) {$\Bl_{-\overbar{v}}$};

      \filldraw (2,0.1) circle (1.5pt);
      \node[above] at (2,0.1) {$\Bl_{\overbar{u}}$};

      \filldraw (-2,-0.1) circle (1.5pt);
      \node[above] at (-2,-0.1) {$\Bl_{-\overbar{u}}$};

      \filldraw (-1,0) circle (1.5pt);
      \node[above] at (-1,0) {$\Al_x$};

      \filldraw (-0.3,0.7) circle (1.5pt);
      \node[above] at (-0.3,0.7) {$\Al_y$};

      \filldraw (1.5,-0.2) circle (1.5pt);
      \node[above] at (1.5,-0.2) {$\Al_z$};
      
      \draw[-,decorate,decoration={snake,amplitude=.4mm,segment length=2mm,post length=1mm},line width = 3pt,white] (-1-0.1,0)  to[out=-160, in=-20] (-2+0.1,-0.1) ;
      \draw[->,decorate,decoration={snake,amplitude=.4mm,segment length=2mm,post length=1mm}]  (-1-0.1,0)  to[out=-160, in=-20] (-2+0.1,-0.1) ;

   \end{scope}
 \end{tikzpicture}
 \end{minipage}%
 \begin{minipage}{.5\textwidth} \centering
\begin{equation}\label{eqn:ParallelTransportGLOSp}
  ``\nabla_{\Pt,\gamma} \ : \ \Al_x \otimes \Bl_{\overbar{u}} \ \to \  \Bl_{\overbar{u}}"
 \end{equation}
 \end{minipage}
 \end{center}
 for every path $\gamma:x\to \pm u$ in $\Cb$, which are linear over $\nabla_\GL$ and $\nabla_\Gt$. As a consequence, $\Bl_0$ is a factorisation module over $\Al$ at $0$.

 \subsubsection{} We now give more details. The \textit{ordered Ran space} is defined as
$$\Ran^{\ord} X \ = \ \cup_{n \ge 0}X^n,$$
where the union (or colimit in $\PreStk$) is taken along all diagonal maps $X^n \to X^m$; the \textit{Ran space} $\Ran X$ is defined as the colimit of the diagonal and permutation maps.

Both Ran spaces have a \textit{decomposition} structure (see Appendix \ref{sec:Decomposition}): a correspondence
\begin{equation}
 \label{fig:RanCh} 
 \begin{tikzcd}[row sep = {30pt,between origins}, column sep = {45pt, between origins}]
  & ( \Ran X \times \Ran X)_\circ \ar[rd,"j\cup"] \ar[ld,"j"']  & \\
 \Ran X \times \Ran X & & \Ran X
 \end{tikzcd}
\end{equation}
satisfying an associativity condition, and a unit condition for $\varnothing \in \Ran X$. This correspondence is defined as the locus of \textit{disjoint} pairs of finite subsets.\footnote{Thus the roof \eqref{fig:RanCh} is defined as the colimit of $(X^n\times X^m)_\circ$, the open locus of disjoint tuples.} We may (as in Appendix \ref{sec:Decomposition} or \cite{Bu,FG,La2}) define a \textit{factorisation algebra} as an element $\Al$ in an appropriate category of (co)sheaves together with a map 
$$m \ : \  j^*(\Al \boxtimes \Al) \ \stackrel{\sim}{\to} \ (\cup j)^*\Al$$
satisfying an associativity condition, and a section $1_\Al\in \Gamma(\Al)$ satisfying a unit condition. 

We convert this structure into a parallel transport map as in \eqref{eqn:ParallelTransport} by taking the boundary map $\delta$ of the Mayer-Vietoris sequence of (co)sheaves
$$\Al\vert_{X^2}\ \to \  \jmath_*\jmath^*(\Al\vert_X \boxtimes \Al\vert_X) \ \stackrel{\delta}{\to}\ \iota_*\Al\vert_X[1]$$ 
for the restriction along 
$$X \ \stackrel{\iota}{\to} \ X^2 \ \stackrel{\jmath}{\leftarrow} \ X^2 \smallsetminus  X.$$  
See e.g. \cite{Bu,La2} for ordinary vertex algebras or section \ref{ssec:OrthosymplecticFactorisationAlgebras} for orthosymplectic variants. 

Finally, we remark that in the topological case Lurie gives an equivalence $\FactAlg(\Ran \Rb^n,\textup{CoSh}_{\textup{const}}(\Cl)) \simeq \Eb_n \Ag(\Cl)$ in \cite[5.5.4.10]{Lu} between factorisatable constructible cosheaves over $\Ran \Rb^n$ valued in a symmetric monoidal $\infty-$category $\Cl$ the category of $\Eb_n$-algebras in $\Cl$. The notion of vertex algebra should thus be viewed as a holomorphic or algebraic analogue of $\Eb_2$-algebra.

\subsection{Vertex algebras of classical and parabolic type} \label{ssec:OrthosymplecticVertexAlgebras}

\subsubsection{} One can show \cite{AB,Jo} that the structure of a vertex algebra on $A$ is equivalent to a collection of maps
$$X_n \ : \ A \otimes \cdots \otimes A \ \to \ A[[z_1,\ldots ,z_n]][\{(z_i-z_j)^{-1}\}]$$
for $n\ge 0$ satisfying an associativity, commutativity and unit conditions which we will detail below. We will now make definitions for other groups $\Gt$ in the same way.

For the rest of this section, we set $\Gt\in \{\GL,\Ot_{\textup{odd}},\Sp\}$. Writing $u_i$ for the coordinates on its Cartan Lie algebra, we have
$$\Ol(\tk_{\gk_n,\circ}^\wedge)\ \simeq \ k[[u_1, \ldots ,u_n]][\{f_\alpha^{-1}\}]$$
where $f_\alpha \in \{u_i\pm u_j\}$ are functions defining unipotent hyperplanes as in section \ref{sssec:UnipotentRoots}, and where $\wedge$ denotes the formal neighbourhood at the origin.

\subsubsection{} \label{sssec:GVA} A (\textit{local}) \textit{$\Gt$-vertex algebra} is a vector space with involution $(B,\tau)$ and a collection of maps 
$$X_{n} \ = \ X_{\Gt,n} \ : \ B \otimes \cdots \otimes B \ \to \ B[[w_1,\ldots , w_n]][ \{f_\alpha^{-1}\}]$$
from the $n\ge 0$ fold tensor product, which satisfies an equivariance condition with respect to elements of the Weyl group $\sigma \in W$
\begin{equation}\label{fig:WeylEquivariance}
\begin{tikzcd}[row sep = {30pt,between origins}, column sep = {20pt}]
B \otimes \cdots  \otimes B \ar[d,"\tilde{\sigma}"] \ar[r,"X_n"]  & B[[w_1,\ldots , w_n]][ \{f_\alpha^{-1}\}] \ar[d,"\sigma"] \\ 
B \otimes \cdots  \otimes B \ar[r,"X_n"]  & B[[w_1,\ldots , w_n]][ \{f_\alpha^{-1}\}]
\end{tikzcd}
\end{equation}
and an associativity condition: 
\begin{equation}
  \label{eqn:GAssociativity} 
  X_r\cdot (X_{n_1} \otimes \cdots \otimes X_{n_r}) \ = \ X_{n_1 + \cdots + n_r}.
\end{equation}
Here, 
\begin{itemize}
 \item In \eqref{fig:WeylEquivariance} we have on the left the action of an element of the Weyl group $\tilde{\sigma} \in W_n\simeq (\Zb/2)^n\rtimes \Sk_n$ by permutations and the involution, and $\sigma$ is its image under the surjection 
 $$W_n \ \twoheadrightarrow \ \Zb/2\ \simeq \ \langle \tau\rangle$$
 acting trivially on $\Sk_n$ and the identity on each $\Zb/2$ factor. 
 \item The variables in \eqref{eqn:GAssociativity} are
 $$X_r(z_i)\cdot (X_{n_1}(u_j^{(1)}) \otimes \cdots \otimes X_{n_r}(u_j^{(r)})) \ = \ X_{n_1 + \cdots + n_r}(z_i+u_j^{(i)})$$
 where the $z_i+u_j^{(i)}$ are lexicographically ordered: 
 $$\left(z_1+u_1^{(1)}, \ z_1+ u_2^{(1)}, \ \ldots , z_1+u_{n_1}^{(1)}\Big| \ z_2+u_1^{(2)}, \ \ldots  \ \Big| \  \ldots , z_r+u^{(r)}_{n_r}\right).$$
 In the following we will suppress notation in the same way, with the variables always adding like this.
\end{itemize}
We define \textit{nonlocal} $\Gt$-vertex algebras by dropping the involution and $W$-equivariance from the definition.

\begin{prop} \label{prop:GVertexAlgebraExplicit}
  A $\Gt$-vertex algebra is equivalent to a pointed vector space with endomorphism $(B,|0\rangle,T)$, an involution $\tau$ with $\tau |0\rangle=|0\rangle$ and $\tau T=-T\tau$, and a $T$-equivariant\footnote{Here, $T$ acts as $T-\partial_z$ on the right and a $T\otimes\id$ on the left, i.e. $T$-equivariance is $[T,Y(\alpha,w)]=\partial_w Y(\alpha, w)$.} map 
  $$Y_B(w) \ : \ B \otimes B \ \to \ B((w))$$
  satisfying $W$-locality 
  \begin{align*}
    \left((w_1+w_2)(w_1-w_2)\right)^n \left( Y_B(\beta_1,w_1)Y_B(\beta_2,w_2)- Y_B(\beta_2,w_2)Y_B(\beta_1,w_1) \right)& \ = \ 0 \\
    Y_{B}(w)\cdot (\tau \otimes \id),  \hspace{5mm}  Y_{B}(w)\cdot (\id \otimes \tau) & \ = \   \tau \cdot Y_{B}(-w)
  \end{align*}
  for $n \gg 0$, such that $Y(|0\rangle,w)=\id$ and $Y(\beta,w)|0\rangle$ has no negative powers of $w$.
\end{prop}
\begin{proof}
  Given a $\Gt$-vertex algebra we set
  \begin{equation}\label{eqn:XToY}
    |0\rangle \ = \ X_0(1), \hspace{15mm} e^{uT} \ = \ X_1(u), \hspace{15mm} Y_B(w) \ = \ X_2(w,0)
  \end{equation}
  where the condition that $X_1(u)\cdot X_1(v)=X_1(u+v)$ implies that $X_1(u)$ is the exponential of an endomorphism. Compatibility of $|0\rangle,T,\tau$ with each other and $Y_B(w)$ follows from \eqref{eqn:GAssociativity} and $W$-locality follows as in \cite{AB}.

  Conversely, given the data in the Proposition, setting \eqref{eqn:XToY}, the other $X_n$ are then determined by the associativity condition. Associavity \eqref{eqn:GAssociativity} and commutativity follows from the $W$-locality condition as in the proof of associativity in Theorem 3.2.1 and 3.2.5 of \cite{FBZ}.
 \end{proof}

\subsubsection{Parabolic vertex algebras and $\Gt$-vertex modules} \label{sssec:PVA} We now write down the analogous definition attached to the standard family of parabolics 
\begin{center}
\begin{tikzcd}[row sep = {30pt,between origins}, column sep = {45pt, between origins}]
 &\Pt \ar[rd] \ar[ld]  & \\ 
 \GL \times \Gt& & \Gt
\end{tikzcd}
\end{center}
A \textit{$\Pt_{\GL \textup{-}\Gt}$-vertex algebra} or \textit{linear-orthosymplectic vertex algebra} $(A,B]$ is a local vertex algebra $A$, a $\Gt$-vertex algebra $B$, and a map
$$ X_{AB,n,m} \ : \ A^{\otimes n}\otimes B^{\otimes m} \ \to \ B[[z_1,\ldots ,z_{n}, w_1,\ldots ,w_m]][\{f_\alpha^{-1}\}]$$
for each $n,m \ge 0$, where $f_\alpha\in \{z_i\pm w_j\}$ are the unipotent hyperplanes for $\Pt_{\GL \textup{-}\Gt}$, which satisfies a linearity condition over $A$ and $B$: 
\begin{equation}\label{eqn:PAssociativity}
  X_{AB,r,s}\cdot (X_{A,n_1}\otimes \cdots \otimes X_{A,n_r}\otimes X_{B,m_1} \otimes \cdots  \otimes X_{B,m_s}) \ = \ X_{AB,n_1+ \cdots +n_r,m_1+ \cdots +m_s}
\end{equation}
where the suppressed variables are as in \eqref{eqn:GAssociativity}, and which satisfies $W$-locality: the diagram
\begin{equation}
  \label{fig:WeylEquivarianceP}
  \begin{tikzcd}[row sep = {30pt,between origins}, column sep = {20pt}]
  A^{\otimes n}\otimes B^{\otimes m} \ar[d,"\tilde{\sigma}"] \ar[r,"X_{AB,n,m}"]  &[10pt] B[[z_1,\ldots ,z_{n}, w_1,\ldots ,w_m]][\{f_\alpha^{-1}\}] \ar[d,"\sigma"] \\
  A^{\otimes n}\otimes B^{\otimes m} \ar[r,"X_{AB,n,m}"]  & B[[z_1,\ldots ,z_{n}, w_1,\ldots ,w_m]][\{f_\alpha^{-1}\}]
  \end{tikzcd} 
\end{equation}
commutes whenever $\tilde{\sigma}\in W_{P_{n,m}}\simeq \Sk_n\times W_m$ and its image under $\Sk_n \times W_m  \twoheadrightarrow  \Zb/2$ is denoted $\sigma$. By an argument as in Proposition \ref{prop:GVertexAlgebraExplicit}, we have the following result.

\begin{prop}
  A linear-orthosymplectic vertex algebra is equivalent to the data of a vertex algebra $A$, $\Gt$-vertex algebra $B$ (as in Proposition \ref{prop:GVertexAlgebraExplicit}), and a linear map
 $$Y_{AB}(z,w) \ : \ A \otimes B \ \to \ B(((z\pm w)^{-1}))$$ 
 satisfying linearity conditions over $A$ and $B$:
 \begin{align}
  \begin{split}
   Y_{AB}(z_1+z_2,w_2)\cdot (\id\otimes Y_{AB}(z_2,w_2))&\ =\ Y_{AB}(z_2,w_2)\cdot (Y_{A}(z_1)\otimes\id),\\
   Y_{AB}(z_1+z_2,w_1)\cdot (\id\otimes Y_{B}(z_2))&\ =\ Y_{B}(z_2)\cdot (Y_{AB}(z_1,w_1)\otimes\id),
  \end{split}
 \end{align}
 compatible with the involution
 \begin{equation}
  Y_{AB}(z,w)\cdot (\id \times \tau) \ = \ \tau\cdot Y_{AB}(z,-w),
 \end{equation}
 and such that $Y(|0_A\rangle,z,w)=\id$ and $Y(\alpha,z,w)|0_B\rangle$ has no negative powers of $z\pm w$.
\end{prop}

For $A$ a vertex algebra, an \textit{orthosymplectic vertex $A$-module} $B$ is the above data and conditions, except we only allow $m=1$ in the definition of $X_{AB,n,m}$

\subsubsection{} Likewise, a \textit{$\Pt$-vertex algebra} $(A_1|\cdots|A_{k}|B]$ for the family of standard parabolics 
\begin{center}
  \begin{tikzcd}[row sep = {30pt,between origins}, column sep = {45pt, between origins}]
   &\Pt \ar[rd] \ar[ld]  & \\ 
   \GL^{\times k} \times \Gt& & \Gt
  \end{tikzcd}
  \end{center}
is a collection $A_1,\ldots ,A_k$ of vertex algebras, a $\Gt$-vertex algebra $B$, and maps 
\begin{align*}
  X_{A_iA_{i+1},n,m} \ : \ A_{i}^{\otimes n} \otimes A_{i+1}^{\otimes m} &\ \to \ A_{i+1} [[z_1, \ldots  , z_{n+m}]][\{f_\alpha^{-1}\}] \\
  X_{A_kB,n,m} \ : \ A_{k}^{\otimes n} \otimes B^{\otimes m} &\ \to \ B [[z_1, \ldots  , z_{n},w_1, \ldots ,w_m]][\{f_\alpha^{-1}\}] 
\end{align*}
where $f_\alpha$ are the functions defining hyperplanes in $\Ol(\tk_{\glk_n}\oplus \tk_{\glk_m})$ and $\Ol(\tk_{\glk_n} \oplus\tk_{\gk_m})$ respectively, satisfying associativity and $W$-equivariance conditions.

\subsubsection{Gradings} If $\Gt$ is of classical type, we define a \textit{graded} $\Gt$-vertex algebra to be a collection of vector spaces $B_r$ for $r\ge 0$ and maps 
$$B_{r_1} \otimes \cdots  \otimes B_{r_n} \ \to \ B_{r_1+\cdots + r_n}[[w_1,\ldots ,w_n]][\{f_\alpha^{-1}\}]$$
satisfying the same conditions as in the definition of a $\Gt$-vertex algebra. Likewise, if $\Pt$ is standard parabolic type, then we define a \textit{graded} $\Pt$-vertex algebra to be a collection of vector spaces $A_{1,r},\ldots ,A_{k,r},B_r$ for $r\ge 0$ with graded analogues of the structure maps and conditions for a $\Pt$-vertex algebra.

\subsubsection{Dual notions} Each of the above definitions admits a dual version of (graded) $\Gt$-vertex coalgebra; we do not spell out the definition. We write $\textup{VAlg}_\Gt$ for the category of graded $\Gt$-vertex algebras.




\subsubsection{Nonsingular examples} If $A$ is a commutative algebra with derivation $T$, then 
$$Y(a,z)a' \ = \ (e^{z T}a)\cdot a'$$
makes $A$ into a vertex algebra.  These are called \textit{holomorphic}\footnote{As opposed to meromorphic.} because these structure maps have no polar terms.

\begin{prop} \label{prop:HolomorphicClassicalVertexAlgebra}
  Let $A$ be an associative algebra with derivation $T$. Then $X_0=1_A$ and 
  $$X_n(u_1,\ldots ,u_n)a_1 \otimes \cdots \otimes a_n \ = \ (e^{u_1T}a_1)\cdots (e^{u_nT}a_n)$$
  define a nonlocal $\Gt$-vertex algebra structure, for any $\Gt \in \{ \GL, \Ot_{\textup{odd}}, \Sp\}$ of classical type. If in addition $A$ is commutative with an algebra involution $\tau$ anticommuting with $T$, then this is a local $\Gt$-vertex algebra.
\end{prop}
\begin{proof}
  Since $T$ is a derivation, we have 
  $$e^{u_1T}\cdot e^{u_2T} \ = \  e^{(u_1+u_2)T}$$
  which implies the associativity condition \eqref{eqn:GAssociativity}, and hence $A$ forms a nonlocal $\Gt$-vertex algebra. In the second case, the diagram \eqref{fig:WeylEquivariance} commutes for any element in $\Sk_n  \subseteq W_n$ by commutativity and any element $\Zb/2  \subseteq W_n$ by
  $$X_n(u_1,\ldots ,u_n)a_1 \otimes \cdots \otimes \tau(a_i)\otimes \cdots \otimes a_n \ = \ (e^{u_1T}a_1) \cdots \tau(e^{-u_iT}a_i)\cdots (e^{u_nT}a_n)$$
  since as the involution anticommutes with the derivation we have $e^{uT}\tau = \tau e^{-uT}$.
\end{proof}

\begin{prop}
  Let
   $$A_1 \ \stackrel{f_1}{\to} \  A_2 \ \stackrel{f_2}{\to} \  \ldots \ \stackrel{f_{k-1}}{\to} \  A_k \ \stackrel{f_k}{\to} \ A_{k+1} \ = \  B$$
   be maps of commutative algebras with derivations $T_i,T_{k+1}=T_B$, the last having an involution $\tau_B$ anticommuting with $T_B$. Then
   $$X_{A_iA_j,n,m}(z_i,w_j) \otimes a_{i,n_i} \otimes \otimes a_{j,m_j} \ = \ \smprod (e^{z_iT_i}a_{i,n_i})\cdot \smprod(e^{w_jT_j}a_{j,m_j})$$
    for $1 \le i \le j \le k+1$ defines a $\Pt_{\GL^k\times \Gt}$-vertex algebra structure on $(A_1|\cdots |A_k|B]$.
\end{prop}
\begin{proof}
  Proved as Proposition \ref{prop:HolomorphicClassicalVertexAlgebra}. 
\end{proof}

\subsubsection{Example}  As an example of Proposition \ref{prop:HolomorphicClassicalVertexAlgebra} we take $A = \Cb[t]$ equipped and $T=\partial_t$ with structure map
$$Y(a(t),z)b(t)\ =\ (e^{z\partial_t }a(t))b(t) \ = \ a(t+z)b(t).$$
and involution $\tau(t)=-t$. We then have 
 $$\tau \left(a(t+z)b(t) \right) \ = \ a(-t+z)b(-t) \ = \ a(-(t-z)) b(-t) \ = \  Y(\tau a(t),-z)\tau b(t).$$

\subsection{Categories of finite sets} \label{ssec:FiniteSets}

\subsubsection{} We begin with the $\infty$-category $\FinSet$ of finite sets, likewise $\FinSet^{\surj}$ the category of finite sets where morphisms are surjections. We often use the notation $\FinSet=\FinSet_{\GL}$.

\subsubsection{} In what follows we will define categories of finite sets equipped with various extra data. Writing $\Dl_n$ for the $\infty$-category of all possible data attached to the finite set $[n]$, we ask that this extends to give a functor 
$$F_\Dl \ : \ \FinSet \ \to \ \Cat, \hspace{15mm} [n] \ \mapsto \ \Dl_n$$
and we then define the category of finite $\Dl$-sets to be the Grothendieck construction \cite[Thm. 3.2.0.1]{Lu}  applied to this functor: 
\begin{center}
\begin{tikzcd}[row sep = {30pt,between origins}, column sep = {20pt}]
  \FinSet_\Dl^{\surj} \ar[r] \ar[d] &\FinSet_\Dl \ar[d,"p"] & \\ 
 \FinSet^{\surj} \ar[r] & \FinSet \ar[r,"F_\Dl"] &\Cat
\end{tikzcd}
\end{center}
So the objects of $\FinSet_\Dl^{\surj}$ as the pullback, so its objects $([n],d_n)$ are pairs of finite sets $[n]$ with extra data $d_n \in \Dl_n$ and morphisms are surjections of sets $f:[n]\to [m]$ preserving this data: $F_\Dl(f):d_n \mapsto  d_m$.

If in addition $F_\Dl$ is (symmetric) monoidal, i.e. we have a compatible family of equivalences
$$\alpha_{n,m} \ : \ \Dl_{[n]\sqcup [m]}\ \simeq\ \Dl_{[n]}\times \Dl_{[m]}$$
(and $\alpha_{n,m}\simeq \alpha_{m,n}$) satisfying higher coherences, then $\FinSet_\Dl$ inherits a (symmetric) monoidal structure such that $p$ is (symmetric) monoidal.

\subsubsection{Remark} Since $p$ is coCartesian, we have a preferred (coCartesian) 1-morphism 
$$\tilde{f}\ : \ ([n],d_n) \ \to \ ([m],f(d_n)) $$
in $\FinSet_\Dl$ lifting any map $f:[n] \to [m]$ of finite sets. We often abuse notation and write $f$ for this 1-morphism.

\subsubsection{} \label{sssec:FiniteSetsEx} Thus we define the categories of 
\begin{itemize}
  \item  \textit{ordered finite sets} $\FinSet^{\ord}$ to be finite sets with a total order, or equivalently $\Dl_n$ is the category of embeddings of $[n]$
 \item \textit{orthosymplectic finite sets} $\FinSet_{\Gt}$ whose data is defined by
 $$\Dl_{\Sp,n} \ = \ \Dl_{\Ot_{\textup{odd}},n} \ = \ (\Bt \Zb/2)^n, \hspace{15mm} \Dl_{\Ot_{\textup{ev}},n}\ = \ (\Bt \Zb/2)^{n-1},$$
 \item \textit{linear-linear finite sets} $\FinSet_{\GL \textup{-}\GL} = \FinSet_{\Pt_{\GL^{\times 2}}}$ whose data is defined by partitions 
 $$\obj \Dl_n \ = \ \{[n]\simeq (n_1|n_2) \ : \ n= n_1+n_2\}$$
 and morphisms 
 $$\left( \begin{smallmatrix} 
  f_{11} & \\ 
  f_{21} & f_{22}
  \end{smallmatrix} \right) \ : \  (n_1|n_2) \ \to \ (m_1|m_2)$$
 are set maps $f:[n_1+n_2]\to [m_1+m_2]$ sending each each stratum of the partition to itself or a lower stratum,\footnote{To the right, i.e. with higher numbered index.}  giving maps $f_{ij}:[n_i]\to [m_j]$.
 \item \textit{Linear-$\cdots$-linear finite sets} $\FinSet_{\GL^{\times k}}=\FinSet_{\Pt_{\GL^{\times k}}}$ where 
  $$\obj \Dl_n \ = \ \{[n]\simeq (n_1|\cdots |n_k) \ : \ n= n_1+\cdots +n_k\}$$
  and morphisms 
  $$\left( \begin{smallmatrix} 
    f_{11} & & \\ 
    \vdots & \ddots & \\ 
    f_{1k} &\cdots & f_{kk}
    \end{smallmatrix} \right) \ : \  (n_1|\cdots |n_k) \ \to \ (m_1|\cdots |m_k)$$
    are set maps nonincreasing the stratification.
  \item \textit{Linear-orthosymplectic finite sets} $\FinSet_{\GL \textup{-}\Gt}= \FinSet_{\Pt_{\GL \times \Gt}}$ where $\obj \Dl_n$ consists of a partition $[n]\simeq (n_1|n_2)$ with orthosymplectic structures on $n$ and $n_2$ such that the map $[n_2]\hookrightarrow [n]$ respects the orthosymplectic structure, i.e. 
  $$\FinSet_{\GL \textup{-}\Gt} \ = \ \FinSet_{\GL\times \GL}\times_{\Fun(\sbt \to \sbt,\FinSet)} \Fun(\sbt \to \sbt,\FinSet_{\Gt})$$
  where $\Fun(\sbt \to \sbt,\FinSet_{(\Gt)})$ is the category of two finite ($\Gt$-)sets with a map between them.
  \item Likewise we define \textit{linear-$\cdots$-linear-orthosymplectic finite sets} $\FinSet_{\GL^{\times k-1}\textup{-}\Gt}=\FinSet_{\Pt_{\GL^{\times k-1}\times \Gt}}$ to be finite sets with partition $[n]\simeq (n_1|\cdots |n_k)$ with compatible orthosymplectic structures on $n$ and on $n_k$.
\end{itemize}
Finally, we define $\FinSet_{\Dl}^{\ord}$ to be the category of finite $\Dl$-sets with a total ordering which are required to refine the above partitions.\footnote{e.g. for $(n_1|\cdots|n_k)$, every element in $[n_1] \subseteq [n]$ less than $[n_2]\subseteq [n]$, and so on.}

\begin{lem} \label{lem:UnionMonoidalStructure}
  The above categories are all symmetric monoidal with product denoted 
  $$\cup_{\Dl} \ : \ \FinSet_{\Dl}\times \FinSet_{\Dl} \ \to \ \FinSet_{\Dl}$$ 
  and all ordered categories as monoidal with product denoted
  $$\cup_{\Dl}^{\ord} \ : \ \FinSet_{\Dl}^{\ord}\times \FinSet_{\Dl}^{\ord} \ \to \ \FinSet_{\Dl}^{\ord}$$
 such that the forgetful functor $\FinSet_\Dl^{\surj,\ord} \to \FinSet_\Dl^{\surj}$ is monoidal.
\end{lem}

\subsubsection{Example} The category of symplectic finite sets has one isomorphism class of object $\pm[n]$ for each integer $n \ge 0$. The one-endomorphisms correspond to elements of the Weyl group 
$$\Hom(\pm [n],\pm[n])\ \simeq \ W_{\Sp_{2n}}$$
and the set of one-morphisms $\pm[n] \to \pm [m]$ is non-canonically identified with the product $W_{\Sp_{2m}}\times \Maps^{\surj}([n],[m])$. More generally, 

\begin{lem} \label{lem:GFiniteSetsMorphismsDescription}
  Let $\Gb$ be a group of classical type (i.e. $\GL, \Ot_{\textup{odd}}, \Ot_{\textup{ev}}, \Sp$) or a standard parabolic $\Pt=\Pt_{\GL^{\times k-1}\times \Gt}$. Then the objects of $\FinSet^{\surj,(\ord)}_{\Gb}$ are (resp. ordered) finite sets with a partition, the group of one-endomorphisms
  $$(n_1|\cdots| n_k) \ \stackrel{\sim}{\to} \ (n_1|\cdots|n_k)$$
  are identified with the Weyl group $W_\Gb$ (resp. is trivial), and the one morphisms  
  $$(n_1|\cdots| n_k) \ \twoheadrightarrow \ (m_1|\cdots|m_k)$$
  are compositions of one-endomorphisms and the order-preserving \emph{adjacency maps}
  $$\varphi_{i,j} \ : \ (n_1|\cdots|n_{i}+1|\cdots|n_k) \ \twoheadrightarrow \ (n_1|\cdots |n_i|\cdots|n_k)$$
  for $j=0,1,\ldots, n_i$, given by merging the $(j-1)$th and $j$th element in $[n_i]$, where the $-1$th is defined as the last element of $[n_{i-1}]$, and is injective on all other elements.
\end{lem}

\subsection{Orthosymplectic Ran spaces} \label{ssec:GeneralisedRanSpaces}

\subsubsection{} Recall that the ordinary \textit{Ran space} of the affine line is defined as a colimit of $\Ab^n$ over all diagonal and symmetric maps, which we suggestively denote by
\begin{center}
  \begin{tikzcd}[row sep = 20pt, column sep = 20pt]
    \Ran_{\GL} \Ab^1\ =\  \colim\bigg( &[-20pt] \tk_{\glk_1}\ar[r,std]\arrow[loop, distance=2em, in=50, out=130, looseness=5,"\Sk_1"] & \tk_{\glk_2}\ar[r,{yshift=3pt},std]\ar[r,{yshift=-3pt},std] \arrow[loop, distance=2em, in=50, out=130, looseness=5,"\Sk_2"] & \tk_{\glk_3} \arrow[loop, distance=2em, in=50, out=130, looseness=5,"\Sk_3"] \ar[r,{yshift=6pt},std] \ar[r,{yshift=0pt},std]\ar[r,{yshift=-6pt},std] & \cdots\ \ \bigg)
  \end{tikzcd}
  \end{center} 
  where the colimit is taken in $\PreStk$, see e.g. \cite{FG} for more details on the Ran space. We can likewise define the \textit{ordered Ran space} of the affine line as 
  \begin{center}
    \begin{tikzcd}[row sep = 20pt, column sep = 20pt]
      \Ran_{\GL}^{\ord} \Ab^1\ =\  \colim\bigg( &[-20pt] \tk_{\glk_1}\ar[r,std] & \tk_{\glk_2}\ar[r,{yshift=3pt},std]\ar[r,{yshift=-3pt},std]  & \tk_{\glk_3}  \ar[r,{yshift=6pt},std] \ar[r,{yshift=0pt},std]\ar[r,{yshift=-6pt},std] & \cdots\ \ \bigg)
    \end{tikzcd}
    \end{center} 


Let $\Gt \in \{\GL, \Ot_{\textup{ev}}, \Ot_{\textup{odd}},\Sp\}$ be a family of classical groups, and $\Pt=\Pt_{\GL^{\times k-1} \textup{-}\Gt}$ the standard family of parabolics. Let $\Gb$ denote either $\Gt$ or $\Pt$ and $X$ be a prestack with involution $\tau$. 

We define a functor 
$$X^{(-)}\ : \ \FinSet_{\Gb}^{\surj,\textup{op}} \ \to \ \PreStk$$
defined on objects by $I\mapsto X^I$ and on compositions of adjacency maps $\varphi:I \twoheadrightarrow J$ and isomorphisms $\psi:I \stackrel{\sim}{\to}I$  by 
$$\Delta_\varphi \ : \  X^I \ \to \ X^J, \hspace{15mm} \psi \ : \ X^I \ \stackrel{\sim}{\to} \ X^I,$$
compositions of the associated adjacency diagonal maps $\Delta_{i,j}=\Delta_{\varphi_{i,j}}$ and elements of the Weyl group $W_{\Gb}$, the symmetric group factors acting by permutations and the $\Zb/2$ factors acting by $\tau$ on the appropriate factor of $X^I$.

\begin{defn}
  The (\textup{ordered}) \textit{$\Gb$-Ran space} of $\Ab^1$ is the colimit
  $$\Ran_\Gt X \ = \ \colim_{I \in \FinSet_\Gb^{\surj,(\ord),\textup{op}}} X^I.$$
\end{defn}

\begin{prop}
  If $\Gb$ is of classical type or a standard parabolic, we have 
\begin{center}
  \begin{tikzcd}[row sep = 20pt, column sep = 20pt]
    \Ran_{\Gb} \Ab^1\ =\  \colim\bigg( &[-20pt] \tk_{\gk_1}\ar[r,std]\arrow[loop, distance=2em, in=50, out=130, looseness=5,"W_{\Gb_1}"] & \tk_{\gk_2}\ar[r,{yshift=3pt},std]\ar[r,{yshift=-3pt},std] \arrow[loop, distance=2em, in=50, out=130, looseness=5,"W_{\Gb_2}"] & \tk_{\gk_3} \arrow[loop, distance=2em, in=50, out=130, looseness=5,"W_{\Gb_3}"] \ar[r,{yshift=6pt},std] \ar[r,{yshift=0pt},std]\ar[r,{yshift=-6pt},std] & \cdots\ \ \bigg)
  \end{tikzcd}
  \end{center}
  where $\gk_n$ is the Lie algebra and $W_{\Gb_n}$ is the Weyl group of $\Gb_n$, and the horizontal maps are the adjacent diagonal maps.
\end{prop}
\begin{proof}
  This follows from Lemma \ref{lem:GFiniteSetsMorphismsDescription} describing the morphisms in $\FinSet^{\surj}_\Gb$.
\end{proof}

\subsubsection{Remark} In the classical case 
$$W_{\Sp_{2n}}\ \simeq\  W_{\Ot_{2n+1}}\ \simeq\ (\Zb/2)^n \rtimes \Sk_n, \hspace{15mm}  W_{\Ot_{2n}}\ \simeq\ (\Zb/2)^{n-1} \rtimes \Sk_n,$$
 where the symmetric group acts by permutation on the rank $n$ vector space $\tk_{\spk_{2n}}\simeq \tk_{\ok_{2n+1}}\simeq \tk_{\ok_{2n}}$ and the other factor acts by negation on each of the factors, where $(\Zb/2)^{n-1} \subseteq  (\Zb/2)^n$ is the subspace whose product is trivial.
 
 \subsubsection{Remark} In the parabolic case we have Weyl groups 
 $$W_\Pt \ = \  W_{\GL}\times \cdots \times W_{\GL}\times W_{\Gt}$$
 and Cartans $\tk_{\pk_{n_1, \ldots ,n_k}}=\tk_{\glk_{n_1}}\oplus\cdots \oplus \tk_{\glk_{n_{k-1}}} \oplus\tk_{\gk_{n_k}}$. For instance, the $\Pt_{\GL \textup{-}\GL}$ ordered Ran space is the colimit of $\Ab^{n,m}=\Ab^n \times \Ab^m$ over iterated compositions of the adjacency diagonal maps adjacent diagonal maps
 \begin{align*}
   (x_1,\ldots ,x_n| y_1, \ldots , y_m) & \ \mapsto \ (x_1,\ldots, x_i,x_i, \ldots ,x_n| y_1, \ldots , y_m), \\
    (x_1,\ldots ,x_n| y_1, \ldots , y_m) & \ \mapsto \ (x_1,\ldots, x_n|x_n, y_1, \ldots ,y_m), \\
    (x_1,\ldots ,x_n| y_1, \ldots , y_m) & \ \mapsto \ (x_1,\ldots, x_n| y_1, \ldots ,y_i,y_i, \ldots ,y_m).
 \end{align*}
which define maps from $\Ab^{n,m}$ to $\Ab^{n+1,m}, \Ab^{n,m+1}$, and $\Ab^{n,m+1}$ respetively.

\begin{prop}\label{prop:SymplecticRanSpace}
  If $\Gt=\Sp$ or $\Gt=\Ot_{\textup{odd}}$, we have
  $$\Ran_{\Gt}X \ \simeq \ \Ran(X/\tau).$$ 
\end{prop}
\begin{proof}
This follows from the twisted Fubini colimit formula, e.g. Theorem 4.2 of \cite{PT}. Indeed, we have that 
$$\colim_{I\in \FinSet^{\surj,\textup{op}}_{\Sp}}(-) \ \simeq \ \colim_{I\in \FinSet^{\surj,\textup{op}}}(\colim_{x\in \Phi(I)}(-))$$
and applying this to the functor $X^{(-)}$ gives the result:
$$\Ran_{\Sp}X \ \simeq \ \colim_{I\in \FinSet^{\surj,\textup{op}}}(X/\tau)^I \ = \ \Ran(X/\tau)$$
since the colimit of $X$ with the arrows $\id,\tau$ is $X/\tau$.
\end{proof}

\subsubsection{Remark} 
The above proof should be viewed as a kind of multiplicativity for colimits under the (co)limit diagram of categories
\begin{equation}
  \label{fig:FinSetFibreDiagram}
  \begin{tikzcd}[row sep = {30pt,between origins}, column sep = {20pt}]
    \sqcup_{n \ge 0} \Bt \langle \pm \rangle^n \ar[r] \ar[d]  & \FinSet^{\surj}_{\Sp}\ar[d]  \\ 
    \sqcup_{n \ge 0}\, 1 \ar[r] & \FinSet^{\surj}
  \end{tikzcd}
\end{equation} 
which itself should be viewed as a categorification of the short exact sequence $\langle \pm\rangle^n \to W_{\Sp_{2n}}\to \Sk_n$. Indeed, short exact sequence induces the (co)limit diagram 
\begin{center}
\begin{tikzcd}[row sep = {30pt,between origins}, column sep = {20pt}]
 \Bt \langle \pm \rangle^n \ar[r] \ar[d]  & \Bt W_{\Sp_{2n}} \ar[d]  \\ 
 1 \ar[r]  & \Bt\Sk_n 
\end{tikzcd}
\end{center} 
which are subcategories of \eqref{fig:FinSetFibreDiagram}. Thus, the above Proposition and its proof should be viewed as a generalisation of the fact that $X^n/W_{\Sp_{2n}}=(X^n/\langle \pm\rangle^n)/\Sk_n$, which may also be proven using the twisted Fubini colimit Theorem.

\begin{lem}\label{lem:CupDecompositionStructure}
  There is a commutative monoid structure
 $$\cup \ : \  \Ran_\Gb\Ab^1 \times \Ran_\Gb\Ab^1 \ \to \  \Ran_\Gb\Ab^1.$$
 If $\Gb$ is classical type, there is an associative monoid structure
 $$\cup_{\ord} \ : \  \Ran_\Gb^{\ord}\Ab^1 \times \Ran_\Gb^{\ord}\Ab^1 \ \to \  \Ran_\Gb^{\ord}\Ab^1$$
  such that the natural map $\Ran_\Gb^{\ord}\Ab^1 \to \Ran_\Gb\Ab^1$ is a map of monoids.
\end{lem}
\begin{proof}
 Follows from Lemma \ref{lem:UnionMonoidalStructure}.
\end{proof}

\begin{prop} \label{prop:ChiralDecompositionStructure}
  There is a commutative decomposition structure $ch_{\Gb}$
  \begin{center}
  \begin{tikzcd}[row sep = {30pt,between origins}, column sep = {65pt, between origins}]
   &(\Ran_\Gt\Ab^1 \times \Ran_\Gt\Ab^1 )_\circ \ar[ld,"j"'] \ar[rd,"\cup j"]  & \\ 
   \Ran_\Gt\Ab^1  \times \Ran_\Gt\Ab^1 & &\Ran_\Gt\Ab^1
  \end{tikzcd}
  \end{center}
  and there is an associative decomposition structure $ch_{\Gb}^{\ord}$
  \begin{center}
  \begin{tikzcd}[row sep = {30pt,between origins}, column sep = {65pt, between origins}]
   &(\Ran_\Gt^{\ord}\Ab^1 \times \Ran_\Gt^{\ord}\Ab^1 )_\circ \ar[ld,"j"'] \ar[rd,"\cup j"]  & \\ 
   \Ran_\Gt^{\ord}\Ab^1  \times \Ran_\Gt^{\ord}\Ab^1 & &\Ran_\Gt^{\ord}\Ab^1  
  \end{tikzcd}
  \end{center} 
  such that the natural map $\Ran_\Gt^{\ord}\Ab^1 \to \Ran_\Gt\Ab^1$ is a map of decomposition spaces.
\end{prop}
\begin{proof}
   In both cases, the roof of the composition is defined as the colimit of 
   \begin{center}
   \begin{tikzcd}[row sep = {30pt,between origins}, column sep = {45pt, between origins}]
    &(\tk_{\gk_n}\times \tk_{\gk_m})_\circ \ar[rd] \ar[ld]  & \\ 
    \tk_{\gk_n}\times \tk_{\gk_m}& & \tk_{\gk_{n+m}}
   \end{tikzcd}
   \end{center}
   where $(\tk_{\gk_n}\times \tk_{\gk_m})_\circ=\tk_{\gk_n}\times \tk_{\gk_m}\smallsetminus \cup H_\alpha$ is the complement of the unipotent hyperplanes $\alpha \in \oplus^*\Phi \smallsetminus (\Phi \star \Phi)$. Associativity
   $$((\tk_{\gk_n}\times \tk_{\gk_m})_\circ \times \tk_{\gk_l})_\circ \ \simeq \ (\tk_{\gk_n}\times (\tk_{\gk_m}\times \tk_{\gk_l})_\circ)_\circ$$
    follows as in the Proposition \ref{prop:GChiralStructureConf}.
\end{proof}

We call the structures in Lemma \ref{lem:CupDecompositionStructure} and Proposition \ref{prop:ChiralDecompositionStructure} the \textit{$\star$-} and \textit{chiral decomposition structures} on the $\Gb$-Ran space. This should be compared to the analogous definition for configuration spaces in section \ref{ssec:ConfigurationSpacesModuliRootData}. For an account of decomposition algebras in the category of prestacks with correspodences, see \cite{La2}, where they are referred to as \textit{factorisation spaces}.

\begin{prop} \label{prop:FunctorialityRanSpaces}
  If $f:\Gb \to \Gb'$ is a map of moduli group stacks of classical type or standard parabolic type  , there is a lax map of decomposition spaces
  $$f \ : \ \Ran_\Gb \Ab^1 \ \to \ \Ran_{\Gb'}\Ab^1$$ 
  which is strict if $f$ is a strict map of moduli group stacks.
\end{prop}
\begin{proof}
  Follows by a similar argument to Lemma \ref{lem:FunctorialityConfigurationSpaces}. 
\end{proof}

\begin{cor}
  If $\Gt$ is of classical type there are maps 
  $$f_{\GL,\Gt} \ : \ \Ran \Ab^1 \ \to \  \Ran_\Gt \Ab^1, \hspace{15mm} f_{\Gt, \GL} \ : \ \Ran_\Gt \Ab^1 \ \to \ \Ran_\GL \Ab^1$$
   of decomposition spaces, the first lax and the second strict. 
\end{cor}
\begin{proof}
  These maps are induced by Proposition \ref{prop:FunctorialityRanSpaces} from the maps of groups
  $$\GL \ \to \ \Gt, \hspace{15mm} \Gt \ \to \ \GL,$$
  which on moduli stacks are given by 
  $$\BGL \ \to \ \BG, \hspace{15mm} \BG \ \to \ \BGL$$
  the first taking a vector bundle $V$ to $V, V \otimes (\Cb^2,\omega), V \otimes (\Cb^2, g), V \otimes(\Cb^2 , g) \oplus (\Cb,g)$, where $\omega,g$ are the standard symplectic and orthogonal forms, and the second taking vector bundle with bilinear form $(V,h)$ to $V$. The first class of maps is not strict and the second is.
\end{proof}

\begin{cor} \label{cor:ActionRanSpaces}
  There is a lax action of $(\Ran \Ab^1, ch)$ on $(\Ran_\Gt \Ab^1,ch)$ for $\Gt$ of classical type. 
\end{cor}

We write down the action in Corollary \ref{cor:ActionRanSpaces} explicitly in the case that $\Gt\in \{\Ot_{\textup{ev}},\Ot_{\textup{odd}},\Sp\}$. It is given by 
\begin{center}
\begin{tikzcd}[row sep = {30pt,between origins}, column sep = {65pt, between origins}]
 & (\Ran \Ab^1 \times \Ran_\Gt \Ab^1)_{\Gt\circ}\ar[rd,"\cup j_{\Gt}"] \ar[ld,"j"']  & \\ 
 \Ran \Ab^1 \times \Ran_\Gt \Ab^1 & & \Ran_\Gt \Ab^1
\end{tikzcd}
\end{center}
where the roof is defined as the colimit of
$$\tk_{\glk_n}\times \tk_{\gk_m} \smallsetminus \pm \Delta_{ij}$$
where for coordinates $z_i$ and $w_j$ on $\tk_{\glk_n}, \tk_{\gk_m}$, the antidiagonal $\pm \Delta_{ij}$ is given by $z_i\pm w_j \ne 0$. Indeed, this is the pullback of the $\Gt$-chiral decomposition structure along $f_{\GL, \Gt}\times \id$.

\subsection{Orthosymplectic vertex algebras as factorisation algebras} \label{ssec:OrthosymplecticFactorisationAlgebras}

\subsubsection{} It is well-known since e.g. \cite{BDr} that the category of vertex algebras is equivalent to the category of factorisation algebras on the Ran space of the affine line. In this section we give the analogue of this for orthosymplectic vertex algebras.

\subsubsection{} A \textit{factorisation} or \textit{orthogonal factorisation algebra} is then a D-module $\Al\in \Dl\Md(\Ran \Ab^1)$ or $\Bl \in\Dl\Md(\Ran (\Ab^1/\pm))$ with a map of D-modules 
 $$m_A \ : \ j_A^*(\Al \boxtimes \Al) \ \stackrel{\sim}{\to} \ (\cup j)_A^*\Al, \hspace{15mm} m_B \ : \ j_B^*(\Bl \boxtimes \Bl) \ \stackrel{\sim}{\to} \ (\cup j)_B^*\Bl$$
 satisfying associativity and unit axioms. 
 
 A \textit{linear-orthogonal factorisation algebra} is  D-module $\Al \Bl$ on $\Ran_{\Sp \textup{-} \GL}\Ab^1$ with a map of D-modules
$$m_{AB} \ : \ j_{AB}^*(\Al\Bl \boxtimes \Al\Bl) \ \stackrel{\sim}{\to} \ (\cup j)_{AB}^*\Al\Bl$$
satisfying associativity and unit axioms.

By the following Lemma, this should be thought of as a factorisation algebra acting on an orthogonal factorisation algebra:

\begin{lem}
There are maps of factorisation spaces
$$\Ran (\Ab^1) \ \stackrel{a}{\to} \ \Ran_{\Sp \textup{-} \GL}\Ab^1 \ \stackrel{b}{\leftarrow} \ \Ran (\Ab^1/\pm)$$
such that for any linear-orthogonal factorisation algebra $\Al\Bl$, the restrictions 
$$\Al \ \defeq \ a^! \Al\Bl, \hspace{15mm} \Bl \ \defeq \ b^! \Al\Bl$$
are factorisation and orthogonal factorisation algebras respectively.
\end{lem}
\begin{proof}
 The the map $a$ sends $S$ to $T = \varnothing \subseteq S$ and the map $b$ sends $T$ to $\pm T \subseteq \pm T=S$ respectively. These maps preserve the decomposition structure on the spaces, and pulling back takes factorisation algebras to factorisation algebras.
\end{proof}

We may ask for these structures to be (partially) translation-equivariant, see section \ref{ssec:TranslationEquivariance}.

\begin{theorem} \label{thm:OrthosymplecticFactorisationAlgebras}
 There are equivalences of categories 
 \begin{center}
 \begin{tikzcd}[row sep = {30pt,between origins}, column sep = {20pt}]
   \FactAlg^{\textup{st}}_{\Ga}(\Ran \Ab^1) \ar[r,"\sim",std]&\textup{VA}  \\ 
   \FactAlg^{\textup{st}}_{\Ga}(\Ran_{\Sp \textup{-} \GL} \Ab^1)  \ar[r,"\sim",std]\ar[d] \ar[u]&\textup{VA}_{\GL\textup{-}\Ot} \ar[d] \ar[u]  \\ 
   \FactAlg^{\textup{st}}_{\Ga}(\Ran (\Ab^1/\pm)) \ar[r,"\sim",std]&\textup{VA}_\Ot
 \end{tikzcd}
 \end{center}
 where the vertical arrows are the forgetful functors. These lift to equivalences between categories of commutative and/or unital structures.
\end{theorem}
\begin{proof}
 We spell out the proof for the bottom row, with the rest following similarly. Let $\Bl$ be an orthogonal factoristaion algebra. As a D-module on $\Ran (\Ab^1/\pm)$ it is determined by its restrictions to $\Ab^n$, which we refer to as $\Bl_n$, plus compatible isomorphisms $\Delta_{\varphi}^*\Bl_m \simeq \Bl_n$ whenever $\varphi:\Ab^n \to \Ab^m$ is a compositition of diagonal embeddings and sign actions. 

 Let us consider the component of the Ran space chiral structure over $\Ab^2$: 
 \begin{center}
 \begin{tikzcd}[row sep = {30pt,between origins}, column sep = {45pt, between origins}]
  &(\Ab^1 \times \Ab^1 \smallsetminus  \pm \Delta) \ar[rd,"\cup j"] \ar[ld,"j"']  & \\ 
  \Ab^1 \times \Ab^1 & & \Ab^2
 \end{tikzcd}
 \end{center}
 Now, $\Bl_2$ is uniquely determined by its restrictions to the open $j^*\Bl_2$ and its closed complement $i^*\Bl_2$, along with a boundary map 
 $$\delta \ : \  j_*j^*\Bl_2 \ \to \ i_*i^*\Bl_2$$
 whose fibre is then $\Bl_2$. By the factorisation product on $\Bl$ we have an isomorphism $j^*(\Bl_1 \boxtimes\Bl_1)\stackrel{\sim}{\to}(\cup j)^*\Bl_2$. On the right, we need to use partial translation equivariance to conclude
 $$\Bl_1 \ = \ B \otimes \Ol_{\Ab^1}, \hspace{15mm} i^*\Bl_2 \ = \ \Bl\vert_{\pm \Delta} \ = \ B \otimes\Ol_{\pm \Delta}$$ 
 where the first factor is equipped with an involution $\tau + \textup{sgn}_z$ and the generating vector field acting via $T+\partial_z$. It follows that $\delta$ is a map 
  $$\Gamma(\delta) \ : \  B \otimes B [w_1,w_2,(w_1\pm w_2)^{-1}] \ \to \ B \otimes\delta_{\pm \Delta}$$
 Which by a variant of 9.2.4 of \cite{But}, one shows is equivalent to a map 
 $$Y(w_1,w_2) \ : \  B \otimes B \ \to \ B[[w_1,w_2]][(w_1+w_2)^{-1},(w_1-w_2)^{-1}] \ = \ B((w_1\pm w_2))$$
 satisfying the above conditions. The associativity of factorisation product on $\Bl$ implies weak associativity for $Y$, likewise for commutativity and unitality; the proof proceeds as in \cite{La2} for the $\GL$ and $\GL$-$\Ot$ case.
\end{proof}

The fibres of $\Al\Bl$ above $\Ab^{1,0},\Ab^{0,1}\simeq \Ab^1$ are
\begin{equation*}
 \label{fig:SymplecticFactorisationAlgebra} 
 \begin{tikzpicture}[scale=0.8]


   \begin{scope} 
     [xshift=-5cm]
     \draw[black, line width = 1.5pt] (-2,0) -- (2,0);
     \filldraw[fill=black] (1,0) circle (2pt);
     \node[above] at (1,0) {$A$};
     \end{scope}


   \begin{scope} 
    [xshift=0cm]
    \draw[black, line width = 1.5pt] (-2,0) -- (2,0);
    \filldraw[fill=black] (1,0) circle (2pt);
    \node[above] at (1,0) {$B$};
    \end{scope}
  \end{tikzpicture}
\end{equation*}
and the fibres above $\Ab^{2,0},\Ab^{1,1},\Ab^{0,2}\simeq \Ab^2$ are
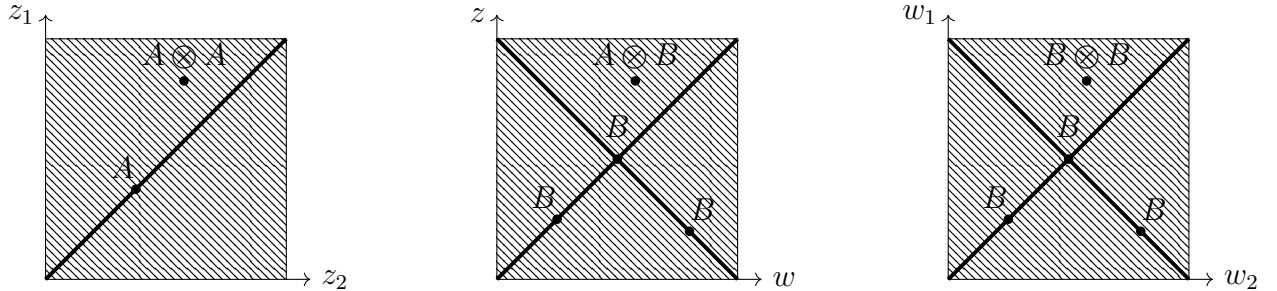
\begin{equation*}
 \label{fig:SymplecticFactorisationAlgebra} 
 \begin{tikzpicture}[scale=0.8]


   \begin{scope} 
    [xshift=-15cm]
   \draw[black, fill=black, fill opacity = 0.2, pattern=north west lines] (-2,-2) rectangle (2,2);

    \draw[black, line width = 1.5pt] (-2,-2) -- (2,2);
    \filldraw[draw=black,fill=black] (-0.5,-0.5) circle (2pt);
   \filldraw[draw=black,fill=black] (0.3,1.3) circle (2pt);
    \node[above,xshift=-5pt] at (-0.5,-0.5) {$A$};
    \node[above,yshift=1pt] at (0.3,1.3) {$ A \otimes A$};
    \draw[->] (-2,2) -- (-2,2.4); 
    \node[left] at (-2,2.4) {$z_1$};
    \draw[->] (2,-2) -- (2.4,-2);
    \node[right] at (2.4,-2) {$z_2$};
    \end{scope}


   \begin{scope} 
    [xshift=-7.5cm]

  \draw[black, fill=black, fill opacity = 0.2, pattern=north west lines] (-2,-2) rectangle (2,2);

  \draw[black, line width = 1.5pt] (-2,-2) -- (2,2);
  \draw[black, line width = 1.5pt] (-2,2) -- (2,-2);
  \filldraw[draw=black,fill=black] (0,0) circle (2pt);
  \filldraw[draw=black,fill=black] (0.3,1.3) circle (2pt);
  \filldraw[fill=black] (-1,-1) circle (2pt);
  \filldraw[fill=black] (1.2,-1.2) circle (2pt);
  \node[above,yshift=4pt] at (0,0) {$B$};
  \node[above,xshift=-5pt] at (-1,-1) {$B$};
  \node[above,xshift=5pt] at (1.2,-1.2) {$B$};
  \node[above,yshift=1pt] at (0.3,1.3)  {$ A \otimes B$};
   \draw[->] (-2,2) -- (-2,2.4); 
   \node[left] at (-2,2.4) {$z$};
   \draw[->] (2,-2) -- (2.4,-2);
   \node[right] at (2.4,-2) {$w$};
    \end{scope}


  \draw[black, fill=black, fill opacity = 0.2, pattern=north west lines] (-2,-2) rectangle (2,2);

   \draw[black, line width = 1.5pt] (-2,-2) -- (2,2);
   \draw[black, line width = 1.5pt] (-2,2) -- (2,-2);
   \filldraw[draw=black,fill=black] (0,0) circle (2pt);
   \filldraw[draw=black,fill=black] (0.3,1.3) circle (2pt);
   \filldraw[fill=black] (-1,-1) circle (2pt);
   \filldraw[fill=black] (1.2,-1.2) circle (2pt);
   \node[above,yshift=4pt] at (0,0) {$B$};
   \node[above,xshift=-5pt] at (-1,-1) {$B$};
   \node[above,xshift=5pt] at (1.2,-1.2) {$B$};
   \node[above,yshift=1pt] at (0.3,1.3) {$ B \otimes B$};
   \draw[->] (-2,2) -- (-2,2.4); 
   \node[left] at (-2,2.4) {$w_1$};
   \draw[->] (2,-2) -- (2.4,-2);
   \node[right] at (2.4,-2) {$w_2$};
  \end{tikzpicture}
\end{equation*}
Notice that $\Al\Bl$ is locally constant along the hyperplanes  of $\ok_4$, which are the diagonal $z_1=z_2$ and the antidiagonal $z_1=-z_2$.

\subsubsection{Remark} 
The map $Y_B(w_1,w_2)$ defining the orthogonal vertex structure on $B$ should be thought of parallel transport along $w_1,w_2 \to 0$. We may of course also first parallel transport first along the path $w_1-w_2 \to 0$ then along $w_1+w_2 \to 0$, 
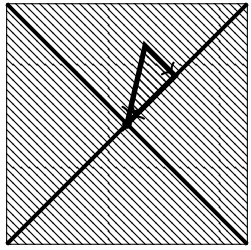
\begin{center}
 \begin{tikzpicture}[scale=0.8]
    
  \draw[black, fill=black, fill opacity = 0.2, pattern=north west lines] (-2,-2) rectangle (2,2);
   \draw[black, line width = 1.5pt] (-2,-2) -- (2,2);
   \draw[black, line width = 1.5pt] (-2,2) -- (2,-2);

   \draw[-,black, line width = 2pt] (0.3,1.3) -- (0.8,0.8) -- (0,0) -- cycle;
   \fill[fill opacity = 0.4, pattern=north west lines] (0.3,1.3) -- (0.8,0.8) -- (0,0) -- cycle;
   \draw[->,black,thick, shorten >=1pt, shorten <=1pt] (0.3,1.3) -- (0.8,0.8);
   \draw[->,black,thick, shorten >=3pt, shorten <=3pt] (0.3,1.3) -- (0,0);
   \draw[->,black,thick, shorten >=3pt, shorten <=3pt] (0.8,0.8) -- (0,0);

    \filldraw[fill=black] (0,0) circle (2pt);

  \end{tikzpicture}
  \end{center} 
If we write 
$$Y_-(w_1,w_2) \ : \ B \otimes B \ \to \ B((w_1\pm w_2))$$
for the gluing map along the diagonal $\Delta \smallsetminus  0 \hookrightarrow \Ab^2 \smallsetminus  -\Delta$, it follows from this that $ e^{-(w_1+w_2)T/2}\cdot Y_-(w_1,w_2)= Y_B(w_1,w_2)$, and similarly for the antidiagonal, and so these gluing maps in fact supply no extra data. See Appendix A of \cite{La2} for more details.

\subsubsection{} Finally, we note that if we fix $\Bl$ at the origin then it defines a factorisation module for $\Al$:

\begin{prop}
 Let $\Al\Bl$ be a linear-orthogonal factorisation algebra. Then there is a map 
 $$\Ran_0(\Ab^1) \ \stackrel{ab_0}{\to} \ \Ran_{\Sp \textup{-} \GL}\Ab^1$$
 such that the pullback $\Ml=(ab_0)^!\Al\Bl$ is a factorisation module for $\Al=a^!\Al\Bl$ at the origin. Its fibre at the origin is $M=B$.
\end{prop}
\begin{proof}
 Recall that $\Ran_0\Ab^1$ parametrises finite subsets of $\Ab^1$ containing the origin. Thus the map $ab_0$ sends $S$ to $0 \subseteq S$. The chiral decomposition structure on $\Ran_{\Sp \textup{-} \GL}\Ab^1$ pulls back to 
 \begin{center}
 \begin{tikzcd}[row sep = {30pt,between origins}, column sep = {45pt, between origins}]
  & (\Ran_0\Ab^1 \times \Ran \Ab^1)_\circ \ar[ld] \ar[rd]  & \\ 
  \Ran_0 \Ab^1 \times \Ran \Ab^1 & & \Ran_0 \Ab^1 
 \end{tikzcd}
 \end{center}
 and so the factorisation algebra structure on $\Al\Bl$ pulls back to a factorisation $\Al$-module structure on $\Ml$.
\end{proof}

\subsection{Localised-to-vertex functor} \label{ssec:LocalisedToVertex}

\subsubsection{} In this section, we explain how to produce orthosymplectic vertex algebras from localised algebraic structures; see section \ref{ssec:Axiomatic} for an information sketch of the construction of our functor $\Phi$ of Theorem \ref{thm:LocalisedToVertex}. This expands on the work in \cite{JKL} for the $\GL$ case.

\subsubsection{} Let $\Gt$ be of classical or classical parabolic type. The group action
$$\Gm \times \Tt_\Gt \ \to \ \Tt_\Gt$$
by scaling the maximal torus induces an action of the translation group 
$$a \ : \ \Ga \times \Conf_{\Gt}^{\ord}\Ab^1 \ \to \ \Conf_{\Gt}^{\ord}\Ab^1.$$

\begin{theorem} \label{thm:LocalisedToVertex}
 There is a \textbf{$\Gt$-localised-to-vertex} functor 
 $$\varphi_{\Gt} \ : \ \textup{LocCoAlg}_{\Gt}^{\Ga} \ \to \ \textup{VertexCoAlg}_{\Gt}$$
 from partially translation-equivariant factorisation algebras on $\Conf_\Gt^{\ord}\Ab^1$ to the category of graded nonlocal $\Gt$-vertex coalgebras. 
\end{theorem}
\begin{proof}
Let $B$ be a localised $\Gt$-coalgebra, which by Theorem \ref{thm:LocalisedCoproducts} is equivalent to a flat quasicoherent factorisation coalgebra on the (un)ordered $\Gt$-configuration space. By pulling back along $\Conf_\Gt^{\ord}\Ab^1 \to \Conf_{\Gt}\Ab^1$, without loss of generality we have a partially transation equivariant factorisation coalgebra over the ordered configuration space.

Next, we show that if we restrict to the punctured formal disk at infinity 
$$\Dt^\times_{u^{-1}} \ \hookrightarrow \ \Ga$$
where $u$ is a coordinate on $\Ga$, then the translation action takes $(\Conf_\Gt^{\ord}\Ab^1)^2$ to the $\Gt$-disjoint locus:
\begin{equation} 
\begin{tikzcd}[row sep = {40pt,between origins}, column sep = {10pt}]
  \Dt_{u^{-1}_1}^\times \hat{\times}\,\Dt_{u^{-1}_2}^\times \hat{\times}\,(\Conf_\Gt^{\ord} \Ab^1\times \Conf_\Gt^{\ord} \Ab^1 )
  \ar[rd,"j"] \ar[d,dashed,"\Tay"]
  & \\
  (\Ga \times \Ga \times  \Conf_\Gt^{\ord} \Ab^1\times \Conf_\Gt^{\ord} \Ab^1)_{\Gt\circ} \ar[d] \ar[r] &   \ar[d,"a \times a"]  \Ga\times \Ga \times  \Conf_\Gt^{\ord} \Ab^1\times  \Conf_\Gt^{\ord} \Ab^1 
  \\ 
   (\Conf_\Gt^{\ord} \Ab^1 \times \Conf_\Gt^{\ord} \Ab^1)_{\Gt\circ} \ar[r] &\Conf_\Gt^{\ord} \Ab^1 \times \Conf_\Gt^{\ord} \Ab^1
\end{tikzcd}
\end{equation}
where the bottom square is defined as the pullback. The map $\Tay$ is a map between ind-affine schemes, so it is enough to define it on functions. Write $x_i,y_j$ for coordinate functions on the first and second $\Conf_\Gt^{\ord}\Ab^1$ factors, and $f_\alpha$ for the linear functions the complement of whose zero locus is $(\Conf_\Gt^{\ord}\Ab^1)^2_{\Gt\circ}$. In particular,
$$(a \times a)^*f_\alpha \ = \ f_\alpha + a_{\alpha}u_1 + b_{\alpha}u_2$$
 are the functions the complement of whose zero locus is $(\Ga \times \Conf_\Gt^{\ord}\Ab^1)^2_{\Gt\circ}$, where $a_\alpha,b_\beta$ do not both vanish. We then define $\Tay$ by sending $1/f_\alpha$ to its Taylor expansion in $u_1^{-1},u_2^{-1}$:
$$\frac{F(u_1,u_2,x_i,y_j)}{\prod_\alpha f_\alpha^r} \ \mapsto \  F(u_1,u_2,x_i,y_j)\cdot \prod_\alpha \left(\sum_{n \ge 0}\frac{f_\alpha^n}{(a_{\alpha}u_1 + b_{\alpha}u_2)^{n+1}} \right)^r.$$

Thus we take the factorisation coproduct 
$$\Delta \ : \ (\cup j)^* \Bl \ \to \  j^*(\Bl \boxtimes \Bl)$$
in  $\QCoh \left((\Conf^{\ord}_\Gt \Ab^1 \times \Conf^{\ord}_\Gt \Ab^1)_{\Gt\circ}\right)$ to the map of $k((u_1^{-1},u_2^{-1}))$-modules
$$\Gamma(\Tay^*\Delta) \ : \ \Gamma(\Tay^*(\cup j)^* \Bl) \ \to \  \Gamma(\Tay^*j^*(\Bl \boxtimes \Bl))$$
which defines the $\Gt$-vertex coproduct. The vertex axioms are inherited from the factorisation coalgebra axioms on $\Bl$ as in \cite{JKL}. 
\end{proof}

\begin{prop}
  For any map $\Gt \to \Gt'$ of classical or parabolic classical groups, the following diagram commutes: 
 \begin{center}
 \begin{tikzcd}[row sep = {30pt,between origins}, column sep = {20pt}]
   \textup{LocCoAlg}_{\Gt}^{\Ga} \ar[r,"\varphi_{\Gt}"] & \textup{VertexCoAlg}_{\Gt} \\ 
   \textup{LocCoAlg}_{\Gt'}^{\Ga}\ar[r,"\varphi_{\Gt'}"] \ar[u,"f_*"]  & \textup{VertexCoAlg}_{\Gt'} \ar[u,"f_*"]  
 \end{tikzcd}
 \end{center}
 If this map is strict, then there is a commutative diagram 
 \begin{center}
 \begin{tikzcd}[row sep = {30pt,between origins}, column sep = {20pt}]
   \textup{LocCoAlg}_{\Gt}^{\Ga} \ar[r,"\varphi_{\Gt}"] & \textup{VertexCoAlg}_{\Gt} \\ 
   \textup{LocCoAlg}_{\Gt'}^{\Ga}\ar[r,"\varphi_{\Gt'}"] \ar[u,"f^*"',<-]  & \textup{VertexCoAlg}_{\Gt'} \ar[u,"f^*"',<-]  
 \end{tikzcd}
 \end{center}
\end{prop}

\begin{cor}
 For $\Gt \in \{\GL, \Ot_{\textup{odd}}, \Ot_{\textup{ev}}, \Sp\}$ there is a commutative diagram of functors
 \begin{center}
 \begin{tikzcd}[row sep = {30pt,between origins}, column sep = {20pt}]
  \textup{LocCoAlg}^{\Ga} \ar[r,"\Phi"] & \textup{VertexCoAlg}_{\GL} \\ 
  \textup{LocCoAlg}_{\GL \textup{-}\Gt}^{\Ga}\ar[d] \ar[u]\ar[r,"\varphi_{\GL \textup{-}\Gt}"]  & \textup{VertexCoAlg}_{\GL \textup{-}\Gt}^{\Ga}\ar[u] \ar[d]  \\
 \textup{LocCoAlg}_{\Gt}^{\Ga} \ar[r,"\varphi_{\Gt}"] & \textup{VertexCoAlg}_{\Gt} 
 \end{tikzcd}
 \end{center} 
 in each row sending $\Bl \mapsto B_n = \Gamma(\Conf_{-,n}^{\ord}\Ab^1,\Bl)$.
\end{cor}
 
For instance, if we have an symplectic localised coalgebra for $\Gt=\Sp$:
$$\Delta \ : \  B_{n+m} \ \to \ B_n \otimes B_m [x_1^{-1},x_2^{-1},(x_1\pm x_2)^{-1}],$$
we compose with the pullbacks $\act_{w_1}^* \otimes \act_{w_2}^*$ to get 
$$\Delta(w_1,w_2) \ : \  B_{n+m} \ \to \ B_n \otimes B_m [(x_1+w_1)^{-1},(x_2+w_2)^{-1},(x_1 + w_1 \pm x_2 \pm w_2)^{-1}].$$
Taylor expanding this then gives an orthosymplectic vertex coproduct:
$$\Delta(w_1,w_2) \ : \  B_{n+m} \ \to \ B_n \otimes B_m ((w_1^{-1},w_2^{-1},(w_1 \pm w_2)^{-1})).$$

\newpage 
 \section{Orthosymplectic vertex Yetter-Drinfeld modules}  \label{sec:JL}

\noindent In \cite{Jo} Joyce discovered that the cohomology of moduli stacks of abelian categories have \textit{vertex coalgebra} structures, which was generalised to Borel-Moore and critical cohomology for certain Calabi-Yau-three categories in \cite{JKL,Li}. Then in \cite{La1,JKL,Li} it was shown that together with the CoHA, this forms a \textit{vertex quantum group} - the vertex analogue of quasitriangular bialgebra.

In this section we show that the critical cohomology of $\Ml^\tau$ forms a \textit{vertex comodule} for $\Ht^\sbt(\Ml,\varphi)$, and compatibly with this $\Ht^\sbt(\Ml^\tau, \varphi^\tau)$ itself has an \textit{orthosymplectic vertex coalgebra} structure; we interpret this as 
$$\Ht^\sbt(\SES^{\OSpt},\varphi) \ \simeq \ \Ht^\sbt(\Ml^\tau, \varphi^\tau) \oplus \Ht^\sbt(\Ml, \varphi)$$
having a $\Pt_{\GL \textup{-}\OSpt}$-vertex coalgebra structure, where $\Pt_{\GL \textup{-}\OSpt}$ is a linear-orthosymplectic parabolic. Then, we show that together with the CoHA module structure of Theorem \ref{thmX:CoHAVA}, $\Ht^\sbt(\Ml^\tau, \varphi^\tau)$ forms a \textit{twisted vertex Yetter-Drinfeld} module for the vertex quantum group $\Ht^\sbt(\Ml,\varphi)$.

\subsubsection*{Heuristics} The vertex structures on critical cohomology arise from the direct sum structure on the category and its moduli stack $\Ml$
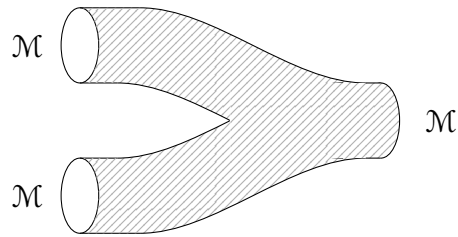
\begin{center}
  \begin{tikzpicture}
    [rotate = 90]
    \begin{scope}
      [xshift=1cm]
      \draw (0,4) ellipse (0.5 and 0.25);
      \draw (-0.5,3.5) -- (-0.5,4);
      \draw (0.5,3.2) -- (0.5,4);
    \end{scope}
    \begin{scope}
      [xshift=-1cm]
      \draw (0,4) ellipse (0.5 and 0.25);
      \draw (-0.5,3.2) -- (-0.5,4);
      \draw (0.5,3.5) -- (0.5,4);
    \end{scope}
\fill[pattern=north east lines, pattern color=black!30]
(1-0.5,3.5) to[out=-90,in=90,looseness=0.8] (-0.5,0.5) to[out=90,in=-90,looseness=0.8]  (1-0.5,3.5) -- (1-0.5,4) arc (180:360:0.5 and 0.25)  (1+0.5,4) --(1+0.5,3.2) to[out=-90,in=90,looseness=0.8] (0.5,0.2) -- (0.5,0) -- (-0.5,0)  -- (-0.5,0.5) to[out=90,in=-90,looseness=0.8]  (1-0.5,3.5) -- cycle ;
  
\fill[pattern=north east lines, pattern color=black!30]
(-1-0.5,3.2) to[out=-90,in=90,looseness=0.8] (-0.5,0.2) to[out=90,in=-90,looseness=0.8]  (-1-0.5,3.2) -- (-1-0.5,4) arc (180:360:0.5 and 0.25)  (-1+0.5,4) --(-1+0.5,3.5) to[out=-90,in=90,looseness=0.8] (0.5,0.5) -- (0.5,0) -- (-0.5,0)  -- (-0.5,0.2) to[out=90,in=-90,looseness=0.8]  (-1-0.5,3.2) -- cycle ;

    \draw (-1-0.5,3.2) to[out=-90,in=90,looseness=0.8] (-0.5,0.2);
    \draw (-1+0.5,3.5) to[out=-90,in=90,looseness=0.8] (0.5,0.5);
    \draw (1-0.5,3.5) to[out=-90,in=90,looseness=0.8] (-0.5,0.5);
    \draw (1+0.5,3.2) to[out=-90,in=90,looseness=0.8] (0.5,0.2);

    \filldraw[white,opacity=1] (-0.49,0) rectangle (0.49,2);
    \draw[dashed] (0.5,0) arc (0:180:0.5 and 0.25);

    \fill[white,opacity=0.7] (-0.5,2) -- (-0.5,0) arc (180:360:0.5 and 0.25) -- (0.5,2) -- cycle;
    \fill[pattern=north east lines, pattern color=black!30] (-0.5,2) -- (-0.5,0) arc (180:360:0.5 and 0.25) -- (0.5,2) -- cycle;

    \draw (-0.5,0) arc (180:360:0.5 and 0.25);
    \draw (-0.5,0) -- (-0.5,0.2);
    \draw (0.5,0) -- (0.5,0.2);
    
     \node[] at (1,4.7) {$\Ml$};
     \node[] at (-1,4.7) {$\Ml$};
     \node[] at (0,-0.8) {$\Ml$};
\end{tikzpicture} 
\end{center}
This map is commutative and $\tau$-equivariant, so we get a commutative product $\oplus_{\OSpt}$ on $\Ml^\tau$ and an action $\oplus_{\GL \textup{-}\OSpt}$ of $\Ml$ on $\Ml^\tau$ by taking invariants of the above:
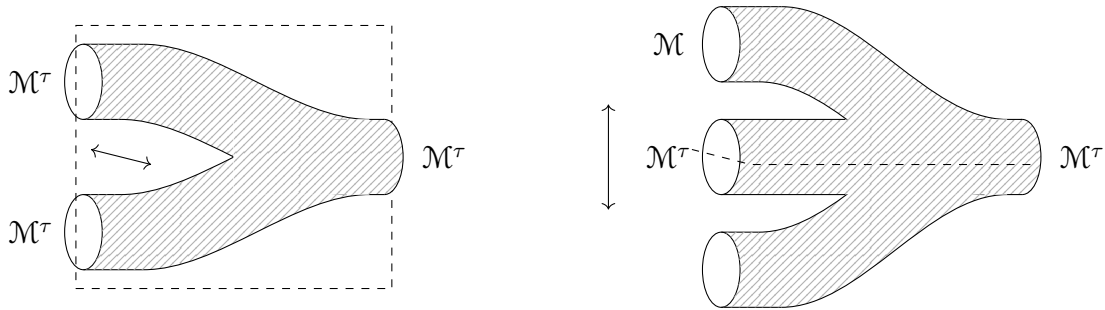
\begin{center}
  \centering
  \begin{minipage}{.5\textwidth} \centering
   
   \begin{tikzpicture}
     [rotate = 90]
     \begin{scope}
       [xshift=1cm]
       \draw (0,4) ellipse (0.5 and 0.25);
       \draw (-0.5,3.5) -- (-0.5,4);
       \draw (0.5,3.2) -- (0.5,4);
     \end{scope}
     \begin{scope}
       [xshift=-1cm]
       \draw (0,4) ellipse (0.5 and 0.25);
       \draw (-0.5,3.2) -- (-0.5,4);
       \draw (0.5,3.5) -- (0.5,4);
     \end{scope}
 \fill[pattern=north east lines, pattern color=black!30]
 (1-0.5,3.5) to[out=-90,in=90,looseness=0.8] (-0.5,0.5) to[out=90,in=-90,looseness=0.8]  (1-0.5,3.5) -- (1-0.5,4) arc (180:360:0.5 and 0.25)  (1+0.5,4) --(1+0.5,3.2) to[out=-90,in=90,looseness=0.8] (0.5,0.2) -- (0.5,0) -- (-0.5,0)  -- (-0.5,0.5) to[out=90,in=-90,looseness=0.8]  (1-0.5,3.5) -- cycle ;
   
 \fill[pattern=north east lines, pattern color=black!30]
 (-1-0.5,3.2) to[out=-90,in=90,looseness=0.8] (-0.5,0.2) to[out=90,in=-90,looseness=0.8]  (-1-0.5,3.2) -- (-1-0.5,4) arc (180:360:0.5 and 0.25)  (-1+0.5,4) --(-1+0.5,3.5) to[out=-90,in=90,looseness=0.8] (0.5,0.5) -- (0.5,0) -- (-0.5,0)  -- (-0.5,0.2) to[out=90,in=-90,looseness=0.8]  (-1-0.5,3.2) -- cycle ;
 
     \draw (-1-0.5,3.2) to[out=-90,in=90,looseness=0.8] (-0.5,0.2);
     \draw (-1+0.5,3.5) to[out=-90,in=90,looseness=0.8] (0.5,0.5);
     \draw (1-0.5,3.5) to[out=-90,in=90,looseness=0.8] (-0.5,0.5);
     \draw (1+0.5,3.2) to[out=-90,in=90,looseness=0.8] (0.5,0.2);
 
     \filldraw[white,opacity=1] (-0.49,0) rectangle (0.49,2);
     \draw[dashed] (0.5,0) arc (0:180:0.5 and 0.25);
     
     \draw[dashed] (-1.75,4+0.1) -- (1.75,4+0.1) -- (1.75,-0.1) -- (-1.75,-0.1) -- cycle;
     \draw[<->] (0.1,3.5+0.4) -- (-0.1,3.5-0.4);

     \fill[white,opacity=0.7] (-0.5,2) -- (-0.5,0) arc (180:360:0.5 and 0.25) -- (0.5,2) -- cycle;
     \fill[pattern=north east lines, pattern color=black!30] (-0.5,2) -- (-0.5,0) arc (180:360:0.5 and 0.25) -- (0.5,2) -- cycle;
 
     \draw (-0.5,0) arc (180:360:0.5 and 0.25);
     \draw (-0.5,0) -- (-0.5,0.2);
     \draw (0.5,0) -- (0.5,0.2);
     
      \node[] at (1,4.7) {$\Ml^\tau$};
      \node[] at (-1,4.7) {$\Ml^\tau$};
      \node[] at (0,-0.8) {$\Ml^\tau$};
 \end{tikzpicture}
 
  \end{minipage}%
  \begin{minipage}{.5\textwidth} \centering
  
    \begin{tikzpicture}[rotate = 90]
      \begin{scope}
        [xshift=1.5cm]
        \draw (0,4) ellipse (0.5 and 0.25);
        \draw (-0.5,3.5) -- (-0.5,4);
        \draw (0.5,3.2) -- (0.5,4);
      \end{scope}
      \begin{scope}
        [xshift=-1.5cm]
        \draw (0,4) ellipse (0.5 and 0.25);
        \draw (-0.5,3.2) -- (-0.5,4);
        \draw (0.5,3.5) -- (0.5,4);
      \end{scope}
    \fill[pattern=north east lines, pattern color=black!30]
    (1.5-0.5,3.5) to[out=-90,in=90,looseness=0.8] (-0.5,0.5) to[out=90,in=-90,looseness=0.8]  (1.5-0.5,3.5) -- (1.5-0.5,4) arc (180:360:0.5 and 0.25)  (1.5+0.5,4) --(1.5+0.5,3.2) to[out=-90,in=90,looseness=0.8] (0.5,0.2) -- (0.5,0) -- (-0.5,0)  -- (-0.5,0.5) to[out=90,in=-90,looseness=0.8]  (1.5-0.5,3.5) -- cycle ;
    
    \fill[pattern=north east lines, pattern color=black!30]
    (-1.5-0.5,3.2) to[out=-90,in=90,looseness=0.8] (-0.5,0.2) to[out=90,in=-90,looseness=0.8]  (-1.5-0.5,3.2) -- (-1.5-0.5,4) arc (180:360:0.5 and 0.25)  (-1.5+0.5,4) --(-1.5+0.5,3.5) to[out=-90,in=90,looseness=0.8] (0.5,0.5) -- (0.5,0) -- (-0.5,0)  -- (-0.5,0.2) to[out=90,in=-90,looseness=0.8]  (-1.5-0.5,3.2) -- cycle ;
    
      \draw (-1.5-0.5,3.2) to[out=-90,in=90,looseness=0.8] (-0.5,0.2);
      \draw (-1.5+0.5,3.5) to[out=-90,in=90,looseness=0.8] (0.5,0.5);
      \draw (1.5-0.5,3.5) to[out=-90,in=90,looseness=0.8] (-0.5,0.5);
      \draw (1.5+0.5,3.2) to[out=-90,in=90,looseness=0.8] (0.5,0.2);
    
      \filldraw[white,opacity=1] (-0.49,0) rectangle (0.49,2.4);
      \draw[dashed] (0.5,0) arc (0:180:0.5 and 0.25);
      \filldraw[white,opacity=0.5] (-0.49,0) rectangle (0.49,1.9);
    
    \fill[pattern=north east lines, pattern color=black!30] (-0.5,4) -- (-0.5,0) arc (180:360:0.5 and 0.25) -- (0.5,4) arc (360:180:0.5 and 0.25);
    
      \draw (0,4) ellipse (0.5 and 0.25);
      \draw (-0.5,0) arc (180:360:0.5 and 0.25);
      \draw (-0.5,2.35) -- (-0.5,4);
      \draw (0.5,2.35) -- (0.5,4);
      \draw (-0.5,0) -- (-0.5,0.2);
      \draw (0.5,0) -- (0.5,0.2);
    
      \draw[dashed] (-0.1,0-0.25+0.14) -- (0-0.1,4-0.25-0.14) -- (0+0.1,4+0.25+0.14) ;
      \draw[<->] (0.7,5.5) -- (-0.7,5.5);

      \node[] at (0,4.7) {$\Ml^\tau$};
      \node[] at (1.5,4.7) {$\Ml$};
      \node[] at (0,-0.8) {$\Ml^\tau$};
    \end{tikzpicture}
  \end{minipage}
  \end{center}
Thus roughly speaking, these structures should be viewed as attached to $\Zb/2$-equivariant topological cobordisms; see section \ref{sec:Cherednik} for more details.

\subsubsection*{Assumptions} To define the vertex structures, we do not need any smoothness assumptions on our stacks, only mild structures to exist. The main restriction is that we only consider vanishing cycle sheaves.

\begin{customass}{\textup{StkVA}}\label{ass:StkVA}
  $\Ml$ is a stack and $W:\Ml\to\Ab^1$ is a function, with a point and map 
  $$0 \ \in \  \Ml , \hspace{15mm}  \oplus \ : \  \Ml \times \Ml \ \to \  \Ml$$
  satisfying unit, commutativity and associativity axioms, also $W\vert_0=0$ and $\oplus^*W=W \boxplus W$. Second there is an action of the group object $\BGm$
  $$\act \ : \ \BGL \times \Ml \ \to \ \Ml$$
  compatible with the above structures: it acts trivially on $0$, has $\oplus \cdot \act  =(\act \times \act)\cdot \oplus$ and $\act^*W=0 \boxplus W$. Finally, let $\tau: \Ml \stackrel{\sim}{\to}\Ml$ be an involution leaving $W$ and $0$ invariant, and such that 
  \begin{equation}
    \label{eqn:TauEquivarianceOplusAct} 
    \tau \cdot \oplus \ = \ \oplus \cdot (\tau \times \tau), \hspace{15mm} \tau \cdot \act \ = \ \act \cdot (\iota \times \tau)
  \end{equation}
  where $\iota$ is an involution of $\BGL$ sending a vector bundle to its dual.
\end{customass}

As described in section \ref{sssec:VectEx}, there are two natural choices for the cocycle isomorphism of the involution $\iota$ on $\BGL$, namely $\alpha_{\Ot}=\ev$ and $\alpha_{\Sp}=-\ev$, which give us 
$$\BGL^\iota \ = \ \BO, \hspace{15mm} \BGL^\iota \ = \ \BSp$$
respectively. Thus, let us make a choice and denote by $\Gt\in \{\Ot,\Sp\}$ the associated sequence of classical groups.

\subsection{Orthosymplectic-linear algebras from direct sum} \label{ssec:HolomorphicJoyceLiu}

\subsubsection{} To begin with, the fixed stack $\Ml^\tau$ has a commutative monoid structure 
$$ \oplus_{\OSpt}  \ : \ \Ml^\tau \times \Ml^\tau \ \to \  \Ml^\tau$$
by including $\Ml^\tau \times \Ml^\tau\to (\Ml\times \Ml)^\tau$ and composing with the restriction of the direct sum map to $\tau$-invariants. Note that this is well-defined by the $\tau$-equivariance property \eqref{eqn:TauEquivarianceOplusAct} of the direct sum map.

\subsubsection{} Next, we have a map of commutative monoids
$$ \iota_{\GL \textup{-}\OSpt} \ : \ \Ml \ \to \ \Ml^\tau,$$
by taking invariants of the direct sum map for the involution $\sigma \cdot (\tau \times \tau)$ acting on $\Ml \times \Ml$. Note that this is well-defined because the direct sum map is commutative. 

\subsubsection{} Both of the above structures are $\tau$-equivariant for the action of $\tau$ on $\Ml$ and the trivial action on $\Ml^\tau$.

\subsubsection{} It follows that there is a module structure 
  $$\oplus_{\textup{\GL}\textup{-}\OSpt} \ : \ \Ml \times \Ml^\tau \ \to \ \Ml^\tau$$
  which is $\tau$-equivariant and commutes with the action of $\Ml^\tau$ on itself. We may equivalently define it as the composition $\oplus_{\OSpt}\cdot \iota_{\GL \textup{-}\OSpt} $, or as the fixed points of the triple direct sum map
  $$\Ml \times\Ml^\tau \ \simeq \ (\Ml \times\Ml \times \Ml)^{\Zb/2} \ \to \ \Ml^\tau$$
  under the involution sending $(a,b,c) \mapsto (\tau(c),\tau(b),\tau(a))$.

  \begin{prop}\label{prop:OplusStructureOSp}
    The critical cohomologies form two graded coalgebras and there is a graded coaction respecting both coproducts
    \begin{align}
      \begin{split} \label{eqn:OplusStructureOSp}
      \oplus^* \ : \ \Ht^\sbt(\Ml,\varphi)& \ \to \ \Ht^\sbt(\Ml,\varphi) \, \hat{\otimes}\, \Ht^\sbt(\Ml, \varphi), \\
      \oplus^*_{\OSpt} \ : \ \Ht^\sbt(\Ml^\tau,\varphi^\tau)& \ \to \ \Ht^\sbt(\Ml^\tau,\varphi^\tau) \, \hat{\otimes}\, \Ht^\sbt(\Ml^\tau,\varphi^\tau),\\
      \oplus^*_{\GL \textup{-} \OSpt}\ : \ \Ht^\sbt(\Ml^\tau,\varphi^\tau) &\ \to \ \Ht^\sbt(\Ml^\tau, \varphi^\tau)\, \hat{\otimes}\,\Ht^\sbt(\Ml,\varphi).
      \end{split}
    \end{align}
     All the above structures are $\tau$-equivariant: 
     $$(\tau \times \tau)\cdot \oplus^* \ = \ \oplus^* \cdot \tau, \hspace{15mm} (\id \times \tau)\cdot\oplus^*_{\GL \textup{-} \OSpt}\ = \ \oplus^*_{\GL \textup{-} \OSpt} .$$
  \end{prop}
  \begin{proof}
    The first, which is originally due to \cite{JKL}, follows from the fact that $\oplus^*W=W \boxplus W$, functoriality of the vanishing cycles functor and the Thom-Sebastiani isomorphism; see \cite{JKL} for details. The coalgebra axioms follow from Proposition \ref{prop:OplusStructureOSp}. The second follows by the same argument applied to $\oplus^*_{\OSpt}W_{\OSpt} = W_{\OSpt} \boxplus W_{\OSpt}$ and the third by the fact that $\iota_{\OSpt}^*W_{\OSpt}=W+\tau^*W=2W$ and that we have $\varphi_{2W}=\varphi_W$. 
  \end{proof}

  \begin{cor}
    There is a $\tau$-equivariant map of graded coalgebras 
    $$\iota^*_{\GL \textup{-} \OSpt}\ : \ \Ht^\sbt(\Ml^\tau,\varphi^\tau)\ \to \ \Ht^\sbt(\Ml,\varphi).$$
  \end{cor}
\begin{proof}
  Follows by composing the coaction from Proposition \ref{prop:OplusStructureOSp} with the counit in $\Ht^\sbt(\Ml^\tau,\varphi^\tau)$. 
\end{proof}

\subsection{Nonsingular vertex structures} 

\subsubsection{} Recall the construction of a holomorphic vertex coproduct from \cite{Jo,JKL} works by considering the action 
$$\BGm \times \Ml \ \to \  \Ml,$$
showing that the generator $x \in \Ht^2(\BGm)$ thus induces a coderivation $T$ on $\Ht^\sbt(\Ml,\varphi)$, and defining 
$$\Delta_0(\alpha,z) \ = \ (e^{zT}\otimes \id)\cdot \oplus^*(\alpha).$$
To generalise to the orthosymplectic case, we need to consider the action of all $\BGL$:

\begin{lem} \label{lem:HopfCoactionAct}
 The map $\act$ induces a Hopf coaction 
$$\act^* \ : \ \Ht^\sbt(\Ml,\varphi) \ \to \ \Ht^\sbt(\Ml,\varphi) \otimes \Ht^\sbt(\BGL).$$
  for the Hopf structure on $\Ht^\sbt(\BGL)$ with its cup product, coproduct $\otimes^*$, and antipode $\tau$ given by dualising vector bundles. 
\end{lem}
\begin{proof}
  This follows formally because $\act$ is an action for the monoid object with involution $\BGL$. This is because, together with the diagonal map, it forms a Hopf algebra in $\PreStk^{\textup{corr}}$.
\end{proof}

\begin{cor} \label{lem:Coderivation}
  There is a degree $-2$ coderivation $T: \Ht^\sbt(\Ml,\varphi) \to \Ht^\sbt(\Ml,\varphi)$ satisfying $\tau T=-T\tau$. 
\end{cor}
\begin{proof}
  Thus, writing $\gamma$ for the tautological line bundle over $\BGm$, the element $e(\gamma)\in \Ht^\sbt(\BGL)$ is primitive and so 
  $$T(\alpha) \ = \ \langle e(\gamma)^\vee, \act^*\alpha \rangle$$
  defines a coderivation as the Hopf action respects the cocommutative graded coproduct on $\Ht^\sbt(\Ml,\varphi)$. It anticommutes with $\tau$ because $\tau\cdot \act=\act \cdot (\tau \times \tau)$.
\end{proof}

So far this is identical to the usual argument constructing a holomorphic graded vertex coalgebra structure.

\subsubsection{} We now write down the orthosymplectic variant. For $\Gt \in \{\Ot, \Sp\}$ as in \ref{ass:StkVA} we have 
$$\act_{\Gt} \ : \ \BG \times\Ml^\tau \ \to \ \Ml^\tau$$
and hence a bialgebra coaction
$$\act_{\Gt}^* \ : \ \Ht^\sbt(\Ml^\tau,\varphi^\tau) \ \to \ \Ht^\sbt(\Ml^\tau,\varphi^\tau) \otimes \Ht^\sbt(\BG).$$
We get a degree $-4$ endomorphism by pairing with the dual to the Euler class $e(\El)\in \Ht^4(\Gt_2)$ of the tautological rank two bundle, however this will not be a coderivation. Indeed, $\BG_2$ is not a group stack. 

\subsubsection{} Instead, we need to pass to the maximal tori $\Tt  \subseteq \GL$ and $\Tt_{\Gt}\to \Gt$, and consider the actions 
$$\act \ : \ \BT \times \Ml \ \to \ \Ml, \hspace{15mm} \act_{\BT_\Gt} \ : \ \BT_\Gt \times \Ml^\tau \ \to \ \Ml^\tau.$$
Using these actions we get
\begin{prop} \label{prop:HolomorphicOrthosymplecticVertexStructure}
  There is the structure of a graded vertex coalgebra, (orthosymplectic) graded vertex coalgebra, and a graded vertex comodule structure respecting the vertex coproducts 
  \begin{align}
    \begin{split} \label{eqn:HolomorphicOrthosymplecticVertexStructure}
    \Delta_0(z) \ : \ \Ht^\sbt(\Ml,\varphi)& \ \to \ \Ht^\sbt(\Ml, \varphi) \, \hat{\otimes}\, \Ht^\sbt(\Ml, \varphi)[z]\\
    \Delta_{\OSpt,0}(w_1,w_2) \ : \ \Ht^\sbt(\Ml^\tau, \varphi^\tau)& \ \to \  \Ht^\sbt(\Ml^\tau, \varphi^\tau) \, \hat{\otimes}\, \Ht^\sbt(\Ml^\tau, \varphi^\tau)[w_1^2,w_2^2]\\
    \Delta_{\GL \textup{-}\OSpt,0}(z,w) \ : \ \Ht^\sbt(\Ml^\tau,\varphi^\tau) & \ \to \   \Ht^\sbt(\Ml^\tau, \varphi^\tau) \, \hat{\otimes}\, \Ht^\sbt(\Ml,\varphi)[z,w^2]
    \end{split}
  \end{align}
 All the above structures are $\tau$-equivariant:
  $${\Delta_0}(z)\cdot \tau \ = \  (\tau \boxtimes \tau) \cdot {\Delta_0}(-z), \hspace{15mm} {\Delta}_{\OSpt \textup{-}\GL,0}(z) \ = \  (\tau \boxtimes \id) \cdot {\Delta}_{\GL \textup{-}\OSpt,0}(-z).$$
  In the language of section \ref{sssec:PVA}, this forms a $\Pt_{\GL \textup{-}\Gt}$-vertex coalgebra.  
\end{prop}
\begin{proof}
We compose the maps \eqref{eqn:OplusStructureOSp} with pullbacks via the rank one maximal tori $\BGm  \subseteq \BGL$ and $\BGm\simeq \BT_{\Gt_2}  \to \BG$ in both factors as appropriate, for instance we define 
$$\Delta_{\GL \textup{-}\OSpt,0}(z,w) \ = \ (\act_{\BGm,z}^* \otimes \act_{\BT_{\Gt_2},w}^*)\cdot \oplus^*_{\GL \textup{-}\OSpt}.$$
Colinearity follows from the colinearity of the structure maps \eqref{eqn:OplusStructureOSp} and the compatibility between the action maps and direct sum:
$$(\act\times \act)\cdot \oplus \ = \ \oplus \cdot \act,$$
$$ (\act_{\BGm}\times \act_{\BT_\Gt})\cdot \oplus_{\GL \textup{-}\OSpt} \ = \ \oplus_{\GL \textup{-}\OSpt}\cdot \act_{\BT_{\Gt}}, \hspace{5mm} (\act_{\BT_\Gt} \times \act_{\BT_{\Gt}})\cdot \oplus_{\OSpt} \ = \ \oplus_{\OSpt}\cdot \act_{\BT_{\Gt}}.$$
Finally, as the final two maps in \eqref{eqn:HolomorphicOrthosymplecticVertexStructure} factor through the maps 
$$k[w^2]\ \simeq \ \Ht^\sbt(\BG_2)\ \to \ \Ht^\sbt(\BGm)\ \simeq \ k[w]$$
in the $w$-variable, they are valued in $k[w^2] \subseteq k[w]$.
\end{proof}

If $q$ is a smooth morphism, we have an isomorphism 
$$\Ht^\sbt(\Ml,\varphi)\simeq\Ht^\sbt(\Ml^\tau, \varphi^\tau)\oplus\Ht^\sbt(\SES_3^\tau,\varphi^{\SES^\tau})$$
where $\varphi^{\SES^\tau}$ is the vanishing cycle sheaf of $W_{\SES}\vert_{\SES_3^\tau}$, and the above structures \eqref{eqn:HolomorphicOrthosymplecticVertexStructure} are equivalent to a holomorphic $\Pt$-vertex coalgebra structure $\act_{\BT_{\Pt}}^*\oplus_{\SES_3^\tau}^*$ induced by pulling back using the action of the maximal torus $\Tt$ inside the linear-orthosymplectic parabolic $\Pt$.

\subsection{Orthosymplectic Joyce-Liu vertex comodules} \label{ssec:OrthosymplecticJoyceLiu}

\subsubsection{} Assume $\Ml$ is a stack satisfying assumption \ref{ass:StkVA}, and is equipped with an associative $\tau$-antiequivariant correspondence $\SES$ as in section \ref{sec:OrthosymplecticStacks}; for this second assumption it is sufficient that $\Ml$ is the moduli stack of objects in an abelian category.

\begin{theorem}\label{thm:OrthosymplecticJoyceLiuLocalised}
  The maps 
  \begin{equation} \label{eqn:OrthosymplecticJoyceLiuLocalised}
    \Delta\ =\ e(\Nt_s)\cdot \oplus^*, \hspace{10mm} \Delta_{\GL \textup{-}\OSpt}\ =\ e(\Nt_{s_3}^\tau)\cdot \oplus_{\GL \textup{-}\OSpt}^*
  \end{equation}
  define a localised coproduct and orthosymplectic localised comodule, internal to the (localised) braided monoidal category $\Ht^\sbt(\Ml)\Md$ with (localised) braided module category $\Ht^\sbt(\Ml^\tau)\Md$ defined in Proposition \ref{prop:OrthosymplecticBraidedFactorisation}. They are braided cocommutative.
\end{theorem}
\begin{proof}The maps \eqref{eqn:OrthosymplecticJoyceLiuLocalised} are easily checked to be linear over $\Ht^\sbt(\Ml^{(\tau)})$ for the direct sum coproduct, e.g. the following commutes 
  \begin{center}
  \begin{tikzcd}[row sep = {30pt,between origins}, column sep = {20pt}]
  \Ht^\sbt(\Ml^\tau, \varphi^\tau) \ar[r,"\Delta_{\GL \textup{-}\OSpt}"] \ar[d,"\alpha"]   & \Ht^\sbt(\Ml, \varphi) \hatotimes \Ht^\sbt(\Ml^\tau, \varphi^\tau) \ar[d,"\oplus^*_{\GL \textup{-}\OSpt}\alpha"] \\ 
   \Ht^\sbt(\Ml^\tau, \varphi^\tau) \ar[r,"\Delta_{\GL \textup{-}\OSpt}"]& \Ht^\sbt(\Ml, \varphi) \hatotimes \Ht^\sbt(\Ml^\tau, \varphi^\tau) 
  \end{tikzcd}
  \end{center}
  commutes for any cohomology class $\alpha \in \Ht^\sbt(\Ml^\tau)$ acting by cup product on sheaf cohomology, and likewise for $\Delta$. It follows that \eqref{eqn:OrthosymplecticJoyceLiuLocalised} are internal to the  monoidal category $\Ht^\sbt(\Ml)$ and its module category $\Ht^\sbt(\Ml^{\tau})\Md$. That it is braided cocommutative for their braidings is immediate from their definitions \eqref{eqn:DefinitionsBraidings}.
  
 Next, notice that by Corollaries \ref{cor:EulerGLHexagon} and \ref{cor:EulerGLOSpModuleHexagon}
\begin{equation}
  \label{eqn:EulerCherednik} 
  S \ = \ e(\Nt_s), \hspace{15mm} T  \ = \ e(\Nt_{s_3}^\tau)
\end{equation}
satisfy the Cherednik hexagon relations.\footnote{See section \ref{sssec:OrthosymplecticModuleHexagonRelations} for a definition.} The Theorem follows by applying the Cherednik Borcherds twist construction of Proposition \ref{prop:OrthosymplecticBorcherdsTwistMain}.
\end{proof}

The statement about $\Delta$ is due to \cite{JKL}. We will now construct the background categories in which our vertex structures will live.

\begin{prop} \label{prop:OrthosymplecticBraidedFactorisation}
  The categories 
$$\Al \ = \ (\Ht^\sbt(\Ml),\cup)\Md \ \simeq \ \QCoh(\Conf_\Ml\Ab^1), \hspace{10mm} \Bl \ = \ (\Ht^\sbt(\Ml^\tau),\cup) \Md\ \simeq \ \QCoh(\Conf_{\Ml^\tau}\Ab^1)$$
have the structure of a a braided linear-orthosymplectic factorisation category over $\Conf_\Ml\Ab^1$ and $\Conf_{\Ml^\tau}\Ab^1$.
\end{prop}
\begin{proof}
Note that $\Al,\Bl$ are equivalent to the category of quasicoherent sheaves over $\Conf_\Ml\Ab^1$ and $\Conf_{\Ml^\tau}\Ab^1$ respectively. The linear-orthosymplectic (factorisation) monoidal structure is induced by the direct sum maps on cohomology: \eqref{eqn:OplusStructureOSp} for $W=0$.There is an involution 
$$\tau \ : \ \Al \ \stackrel{\sim}{\to} \ \Al$$
induced by the involution of the moduli stack. Finally, the  linear-orthosymplectic factorisation braiding is given by the endomorphisms of
\begin{equation}
  \label{eqn:DefinitionsBraidings} 
  \beta \ = \ \frac{e(\sigma^*\Nt_s)}{e(\Nt_s)}\cdot \sigma, \hspace{10mm} \kappa \ = \ \frac{e((\tau\times \id)^*\Nt_{s_3}^\tau)}{e(\Nt_{s_3}^\tau)}\cdot (\tau \times \id), 
\end{equation}
    of $(M_1 \otimes M_2)_{\loc}$, $(M_1 \otimes N)_{\loc}$, where $M_i \in \Al$, $N\in \Bl$ and $\loc$ denotes the localisation at cohomology classes such that the Euler classes are defined, define linear-orthosymplectic braidings. 
\end{proof}

By applying the functor \ref{thm:LocalisedToVertex} to Theorem \ref{thm:OrthosymplecticJoyceLiuLocalised}, we get a vertex structure

\begin{cor}\label{cor:OrthosymplecticJoyceLiuVertex}
  The maps
  $$\Delta(z) \ = \ e(\Nt_s,z)\cdot \act_{1,z}^*\oplus^*, \hspace{15mm} \Delta_{\GL \textup{-}\OSpt}(z,w) \ = \ e(\Nt_{s_3^\tau},z,w)\cdot (\act_{z}^{*}\times \act_{\Gt_2,w}^*)\oplus_{\GL \textup{-}\OSpt}^*$$
  define a orthosymplectic vertex comodule structure on $\Ht^\sbt(\Ml^\tau, \varphi^\tau)$ for the vertex coalgebra $\Ht^\sbt(\Ml, \varphi)$. These vertex coproducts are linear over $\Ht^\sbt(\Ml^{(\tau)})$ and braided local.
\end{cor}

Here we have denoted the  \textbf{Joyce twists}
\begin{align}
  \begin{split}
   S(z) &\ = \ e(\Nt_s,z) \ \defeq \ \sum_{k \ge 0} z^{\rank \theta -k}c_k(\theta),\\
   T(z,w) &\ = \ e(\Nt_{s_3^\tau},z,w) \ \defeq \ \sum_{k \ge 0} z^{\rank\theta_{\GL \textup{-}\OSpt}/2 -k}w^{\rank\theta_{\GL \textup{-}\OSpt}/2-l}c_{k,l}(\theta_{\GL \textup{-}\OSpt})
  \end{split}
 \end{align}
which satisfy the spectral Cherednik reflection equation \cite{Che}. The Corollary follows since $S(z),T(z,w)$ are given by putting back $S,T$ along the the actions of $\BGm$ on $\Ml$ and of $\BT_{\Gt_2}$ on $\Ml^\tau$.

\subsection{Compatibility with the CoHA: twisted Yetter-Drinfeld vertex module}

\subsubsection{}Let $\Ml$ a stack satisfying Assumptions \ref{ass:StkVA} and \ref{ass:StkCoHAM}. In the main result of this section, Theorem \ref{thm:YD}, we show that the CoHA action and the Joyce-Liu vertex coaction of $\Ht^\sbt(\Ml, \varphi)$ on $\Ht^\sbt(\Ml^\tau, \varphi^\tau)$ are compatible: they form a twisted vertex Yetter-Drinfeld module.

See \ref{ssec:Heuristics} for an explanation of why we might a priori have expected this.

\subsubsection{} 
If $A$ is a bialgebra with comodule $M$ inside background orthosymplectic module category $(\Al,\Bl]$, as in Proposition \ref{prop:OrthosymplecticBraidedFactorisation}, then it is called a \textit{twisted Yetter-Drinfeld module} if the following two compositions agree:
\begin{center}
  \begin{tikzpicture}[scale = 0.9]

  \newcommand{\mydrawingsmall}{
  \draw[black, ultra thick] (-1.5-0.2,2) .. controls +(0,-0.75) and +(0,0.75) .. (0,0.2);

  \draw[black, line width=5pt] (0,-2) -- (0,2);
  }

  \mydrawingsmall

  \begin{scope}[xscale=-1]
      \mydrawingsmall
  \end{scope}

  \begin{scope}[yscale=-1]
      \mydrawingsmall
  \end{scope}

  \begin{scope}[xscale=-1, yscale=-1]
      \mydrawingsmall
  \end{scope}

  \fill[white,opacity=0.9] (0,-2) rectangle (1.5+0.8,2);
  \draw[black, line width=5pt] (0,-2) -- (0,2);

  \node[] at (3.5,0) {$=$};

  \node[] at (-4,-1.8) {$A \otimes M$};
  \node[] at (-4,0) {$M$};
  \node[] at (-4,1.8) {$A \otimes M$};

  \draw[->,shorten >=10pt, shorten <=10pt] (-4,-1.8) -- (-4,0);
  \draw[->,shorten >=10pt, shorten <=10pt] (-4,0) -- (-4,1.8);


  \begin{scope} 
   [xshift=7cm]

   \newcommand{\mydrawing}{%
   \draw[black, ultra thick] (-1.5-0.2,-2) -- (-1.5-0.2,-1);
   \draw[black, ultra thick] (-1.5-0.2,-2) .. controls +(0,0.25) and +(0,-0.25) .. (-1.5-0.7,-1.5);
   \draw[black, ultra thick] (-1.5-0.2,-2) .. controls +(0,0.25) and +(0,-0.25) .. (-1.5+0.3,-1.5);
   
   \draw[black, ultra thick] (-1.5-0.2,-1.5) -- (-1.5-0.2,-1);
   \draw[black, ultra thick] (-1.5-0.2,-1) -- (-1.5-0.2,-0.5);
   \draw[black, ultra thick] (-1.5+0.3,-1.5) -- (-1.5+0.3,-1);
   \draw[black, ultra thick] (-1.5-0.7,-1.5) -- (-1.5-0.7,0);
 
   \draw[white,line width=5pt] (-1.5+0.3,-1) .. controls +(0,1.5) and +(0,-1.5) .. (1.5-0.3,3-2);
   \draw[black, ultra thick] (-1.5+0.3,-1) .. controls +(0,1.5) and +(0,-1.5) .. (1.5-0.3,3-2);
   \draw[white,line width=5pt] (-1.5-0.2,-0.5) .. controls +(0,1) and +(0,-1) .. (0,2);
   \draw[black, ultra thick] (-1.5-0.2,-0.5) .. controls +(0,1) and +(0,-1) .. (0,2);

   \draw[white, line width=7pt] (0,-0.1) -- (0,0.1);
   }

   \mydrawing
 
   \begin{scope}[yscale=-1]
       \mydrawing
   \end{scope}

   \begin{scope}[xscale=-1]
    \mydrawing
\end{scope}

   \begin{scope}[xscale=-1, yscale=-1]
       \mydrawing
   \end{scope}

   \fill[white,opacity=0.9] (0,-2) rectangle (1.5+0.8,2);
   \draw[black, line width=5pt] (0,-2) -- (0,2);
 
   \node[] at (4,-1.8) {$A \otimes M$};
   \node[] at (4,0.6) {$A^{\otimes 3} \otimes A \otimes M$};
   \node[] at (4,-0.6) {$A^{\otimes 3} \otimes A \otimes M$};
   \node[] at (4,1.8) {$A \otimes M$};

   \draw[->,shorten >=10pt, shorten <=10pt] (4,-1.8) -- (4,-0.6);
    \draw[->,shorten >=10pt, shorten <=10pt] (4,-0.6) -- (4,0.6);
    \draw[->,shorten >=10pt, shorten <=10pt] (4,0.6) -- (4,1.8);

   \end{scope}

  \end{tikzpicture}
  \end{center}
or to be explicit, 
\begin{equation}
  \Delta_{M}\cdot m_M \ = \ (m_3 \otimes m_M) \cdot \beta_{34} \beta_{23}\kappa_4 \beta_{34} \cdot (\Delta_3 \otimes\Delta_M)
\end{equation}
where $(-)_M$ is the (co)action and $(-)_3$ is the threefold (co)product, $\beta$ is the braiding and $\kappa$ is the orthosymplectic braiding in $\El$.

\begin{theorem} \label{thm:YD}
  The the CoHA action and localised coaction of $\Ht^\sbt(\Ml,\varphi)$ makes $\Ht^\sbt(\Ml^\tau, \varphi^\tau)$ forms a $\tau$-twisted localised Yetter-Drinfeld module. 
\end{theorem}
\begin{proof}
  Following through the definitions, the twisted Yetter-Drinfeld condition is equivalent to the commutativity of the diagram 
  \begin{equation} \label{fig:YDConditionLocalised}
  \begin{tikzcd}[row sep = {15pt,between origins}, column sep = {20pt}]
   (A \otimes M)_{\loc} \ar[r,equals]  & (A \otimes M)_{\loc}\\
   & \\
    &  (A^{\otimes 3} \otimes A \otimes M)_{\loc} \ar[uu]  \\ 
    M \ar[uuu] &  \\
  & (A^{\otimes 3} \otimes A \otimes M)_{\loc} \ar[uu]  \\
  & \\
  A \otimes M \ar[uuu]  \ar[r,equals]  & A \otimes M \ar[uu] 
  \end{tikzcd}
  \end{equation}
  of modules over $\Ht^\sbt(\Ml^{\times 4} \times \Ml^\tau)$.

  To begin, we notice that both compositions arise from correspondences from $\Ml \times \Ml^\tau$ to itself, and there is a map $i$ between these correspondences: 
\begin{equation}\label{fig:YDSpaceLevelDiagram}
\begin{tikzcd}[row sep = {30pt,between origins}, column sep = {65pt, between origins}]
  & &\SES_{3} \times \SES^\tau \ar[d,dashed,"i"] \ar[rdd,bend left = 30,"p_3 \times p^{\OSpt}"]\ar[rdd,bend left = 30,-,ultra thick] \ar[lldd, bend right = 20," \sigma_{34} \sigma_{23}\tau_4 \sigma_{34} \cdot (q_3 \times q^{\OSpt})"'] \ar[lldd, bend right = 20,-,ultra thick]  & & \\
  & & \SES^\tau \times_{\Ml^\tau}(\Ml\times \Ml^\tau) \ar[rd,"\overbar{p}^{\OSpt}"'] \ar[ld,"\overbar{\oplus}_{\OSpt}"] & & \\
 \Ml^3 \times(\Ml\times \Ml^\tau) \ar[d,"\oplus_3\times \oplus_{\OSpt}"'] \ar[d,-,ultra thick]   & \SES^\tau \ar[rd,"p^{\OSpt}"'] \ar[rd,-,ultra thick] \ar[ld,"q^{\OSpt}"] \ar[ld,-,ultra thick]  & & \Ml\times \Ml^\tau \ar[ld,"\oplus_{\OSpt}"]\ar[ld,-,ultra thick] & \\ 
 \Ml \times\Ml^\tau & & \Ml^\tau & &  
 \end{tikzcd} 
\end{equation}
The lower correspondence gives rise to the left side of \eqref{fig:YDConditionLocalised} by pull-push-pull then multiplication by the Joyce twist, and the upper correspondence gives rise to the right side.

The map $i$ is defined by taking the triple extension and orthosymplectic extension
\begin{equation}
\begin{tikzcd}[row sep={30pt,between origins}, column sep={10pt}]
a_1 \ar[d,equals] \ar[r] & \ker \ar[r] \ar[d] & b \ar[d] &&& a_2 \ar[d,equals] \ar[r] & \ker \ar[r] \ar[d] & d \ar[d] \\ 
a_1 \ar[r] & c \ar[r] \ar[d] & \coker \ar[d] &&& a_2 \ar[r] & e \ar[r] \ar[d] & \coker \ar[d] \\ 
0 & a_3^* \ar[r,equals] & a_3^* &&& 0 & a_2^* \ar[r,equals] & a_2^*
\end{tikzcd}
\end{equation}
where $e\simeq e^*$ and $d \simeq d^*$, to 
\begin{equation}\label{fig:DirectSumSESSESOSp}
\begin{tikzcd}[row sep = {30pt,between origins}, column sep = {10pt}]
a_1 \oplus a_2 \oplus a_3 \ar[d,equals] \ar[r]  & \ker \ar[r] \ar[d] & b \oplus d \oplus b^*\ar[d]&[10pt] \\ 
a_1 \oplus a_2 \oplus a_3\ar[r]  &(c \oplus e \oplus c^*)\ar[r] \ar[d]  &\coker\ar[d]  \\
0  & a_3^* \oplus a_2^* \oplus a_1^* \ar[r,equals] & a_3^* \oplus a_2^* \oplus a_1^* &
\end{tikzcd}
\end{equation}
where the maps in \eqref{fig:DirectSumSESSESOSp} are induced by the orthosymplectic embeddings $a_1 \oplus a_3 \to c \oplus c^*$ and $a_2 \to x$. Note that this is just the action from Proposition \ref{prop:MActionOnMt} of $\Ml_\Cl$ on $\Ml_{\Cl}^\tau$ in the case where $\Cl$ is the stack of length three flags. 

We now apply torus localisation to compare the two sides of \eqref{fig:YDConditionLocalised}. We have 
\begin{align*}
  \textup{left side of }\eqref{fig:YDConditionLocalised} & \ = \    (J^{\OSpt}) \oplus^{\OSpt\, *}\cdot p^{\OSpt}_* q^{\OSpt\, *}\\
  & \ = \ (J^{\OSpt}) \cdot \overbar{p}^{\OSpt}_*\cdot  \overbar{\oplus}^{\OSpt\, *}\cdot q^{\OSpt\, *}\\
  & \ \stackrel{\star}{=} \ (J^{\OSpt}) \cdot \overbar{p}^{\OSpt}_* \cdot \left( i_*\frac{i^*}{e(\Nt_i)} \right)\cdot  \overbar{\oplus}^{\OSpt\, *}\cdot q^{\OSpt\, *}\\
  & \ =\ (p_3 \times p^{\OSpt})_* \cdot \left( \frac{1}{e(\Nt_i)}\right)\cdot  (q_3 \times q^{\OSpt})^* \cdot (\sigma_{34} \sigma_{23}\tau_4 \sigma_{34})^* \cdot (\oplus_3 \times \oplus^{\OSpt})^*\\
 & \ \stackrel{\dagger}{=} \ (p_3 \times p^{\OSpt})_* (q_3 \times q^{\OSpt})^* \cdot (\beta_{34} \beta_{23}\kappa_4 \beta_{34}) \cdot (J_3 \otimes J^{\OSpt})(\oplus_3 \times \oplus^{\OSpt})^*\\
 & \ = \  \textup{right side of }\eqref{fig:YDConditionLocalised} 
\end{align*}
where in $\star$ we used Atiyah-Bott torus localisation, see \cite{La3} for details, and 
$$\beta_{ij} \ = \ R_{ij}\cdot \sigma_{ij}, \hspace{15mm} \kappa_i \ = \ K_i \cdot \tau_i.$$

It remains to show $\dagger$, which reduces to proving an equality of localised cohomology classes on $\Ml^{\times 4}\times \Ml^\tau$, i.e. functions on
$$\Ol \left( ((\Conf_\Ml \Ab^1)^{\times 4}\times \Conf_{\Ml^\tau}\Ab^1)_\circ\right).$$
We need to show that the products of the Euler class factor $e(\Nt_i)$ and the Joyce twist $J^{(\OSpt)}$ factors give precisely the class $R_{34} R_{23} K_4 R_{34}$. We note that
$$e(\Nt_i) \ = \ e(\Nt_{p_3\times p^{\OSpt}}) \big/  i^* \overbar{\oplus}^{\OSpt\, *} e(\Nt_{p^{\OSpt}})$$
by the distinguished triangle for normal complexes of a composition. The proof now finishes by a computation, which we either do by hand using the Cherednik hexagon relations as in the shuffle case of section \ref{ssec:ShuffleYetDrin}. 

\end{proof}

\subsection{Graphical proof of Yetter-Drinfeld condition}

\begin{proof}[Second proof: pictoral] We may complete the crucial computation $\dagger$ of the first proof in another way, which makes the reason for the twisted Yetter-Drinfeld condition clearer.

  To set up notation, we write down the top correspondence in \eqref{fig:YDSpaceLevelDiagram} defines a morphism in the category $\PreStk^{\corr}$ of prestacks with morphisms correspondences, which we draw as follows

  
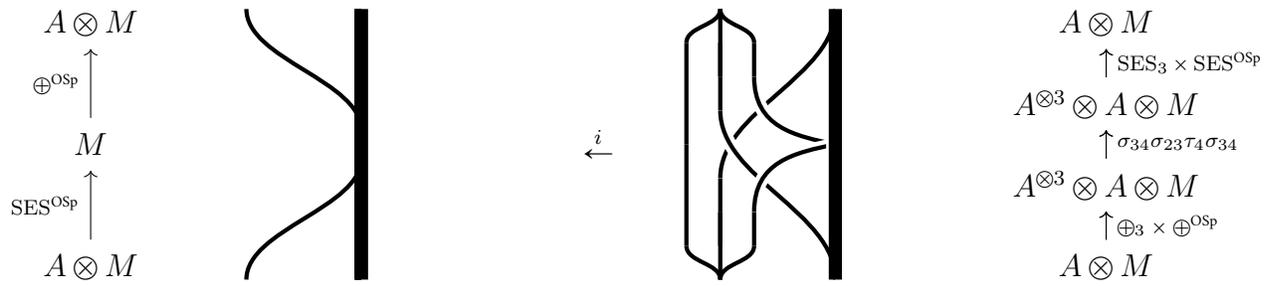
\begin{center}
    \begin{tikzpicture}[scale = 0.9]

    \newcommand{\mydrawingsmall}{
    \draw[black, ultra thick] (-1.5-0.2,2) .. controls +(0,-0.75) and +(0,0.75) .. (0,0.2);
  
    \draw[black, line width=5pt] (0,-2) -- (0,2);
    }

    \mydrawingsmall
  
    \begin{scope}[xscale=-1]
        \mydrawingsmall
    \end{scope}
  
    \begin{scope}[yscale=-1]
        \mydrawingsmall
    \end{scope}
  
    \begin{scope}[xscale=-1, yscale=-1]
        \mydrawingsmall
    \end{scope}

    \fill[white,opacity=0.9] (0,-2) rectangle (1.5+0.8,2);
    \draw[black, line width=5pt] (0,-2) -- (0,2);
  
    \node[] at (3.5,0) {$\stackrel{i}{\leftarrow}$};
  
    \node[] at (-4,-1.8) {$A \otimes M$};
    \node[] at (-4,0) {$M$};
    \node[] at (-4,1.8) {$A \otimes M$};
  
    \draw[->,shorten >=10pt, shorten <=10pt] (-4,-1.8) -- (-4,0) node[midway,left] {\tiny $\SES^{\OSpt}$};
    \draw[->,shorten >=10pt, shorten <=10pt] (-4,0) -- (-4,1.8) node[midway,left] {\tiny $\oplus^{\OSpt}$};

  
    \begin{scope} 
     [xshift=7cm]

     \newcommand{\mydrawing}{%
     \draw[black, ultra thick] (-1.5-0.2,-2) -- (-1.5-0.2,-1);
     \draw[black, ultra thick] (-1.5-0.2,-2) .. controls +(0,0.25) and +(0,-0.25) .. (-1.5-0.7,-1.5);
     \draw[black, ultra thick] (-1.5-0.2,-2) .. controls +(0,0.25) and +(0,-0.25) .. (-1.5+0.3,-1.5);
     
     \draw[black, ultra thick] (-1.5-0.2,-1.5) -- (-1.5-0.2,-1);
     \draw[black, ultra thick] (-1.5-0.2,-1) -- (-1.5-0.2,-0.5);
     \draw[black, ultra thick] (-1.5+0.3,-1.5) -- (-1.5+0.3,-1);
     \draw[black, ultra thick] (-1.5-0.7,-1.5) -- (-1.5-0.7,0);
   
     \draw[white,line width=5pt] (-1.5+0.3,-1) .. controls +(0,1.5) and +(0,-1.5) .. (1.5-0.3,3-2);
     \draw[black, ultra thick] (-1.5+0.3,-1) .. controls +(0,1.5) and +(0,-1.5) .. (1.5-0.3,3-2);
     \draw[white,line width=5pt] (-1.5-0.2,-0.5) .. controls +(0,1) and +(0,-1) .. (0,2);
     \draw[black, ultra thick] (-1.5-0.2,-0.5) .. controls +(0,1) and +(0,-1) .. (0,2);

     \draw[white, line width=7pt] (0,-0.1) -- (0,0.1);
     }

     \mydrawing
   
     \begin{scope}[yscale=-1]
         \mydrawing
     \end{scope}

     \begin{scope}[xscale=-1]
      \mydrawing
  \end{scope}
  
     \begin{scope}[xscale=-1, yscale=-1]
         \mydrawing
     \end{scope}
  
     \fill[white,opacity=0.9] (0,-2) rectangle (1.5+0.8,2);
     \draw[black, line width=5pt] (0,-2) -- (0,2);
   
     \node[] at (4,-1.8) {$A \otimes M$};
     \node[] at (4,0.6) {$A^{\otimes 3} \otimes A \otimes M$};
     \node[] at (4,-0.6) {$A^{\otimes 3} \otimes A \otimes M$};
     \node[] at (4,1.8) {$A \otimes M$};
  
     \draw[->,shorten >=10pt, shorten <=10pt] (4,-1.8) -- (4,-0.6);
      \draw[->,shorten >=10pt, shorten <=10pt] (4,-0.6) -- (4,0.6);
      \draw[->,shorten >=10pt, shorten <=10pt] (4,0.6) -- (4,1.8);
  
      \draw[->,shorten >=10pt, shorten <=10pt] (4,-1.8) -- (4,-0.6) node[midway,right] {\tiny $\oplus_3 \times \oplus^{\OSpt}$};
      \draw[->,shorten >=10pt, shorten <=10pt] (4,-0.6) -- (4,0.6) node[midway,right] {\tiny $\sigma_{34}\sigma_{23}\tau_4\sigma_{34}$};
       \draw[->,shorten >=10pt, shorten <=10pt] (4,0.6) -- (4,1.8) node[midway,right] {\tiny $\SES_3 \times \SES^{\OSpt}$};
   
     \end{scope}
  
    \end{tikzpicture}
    \end{center}
The map $i$ becomes a two-morphism in $\PreStk^{\corr}$ between the two above morphisms. The euler class $e(\Nt_i)$ is a product of factors, which we represent by
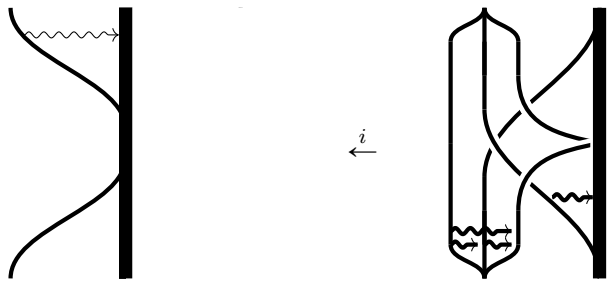
\begin{center}
  \begin{tikzpicture}[scale = 0.9]

  \newcommand{\mydrawingsmall}{
  \draw[black, ultra thick] (-1.5-0.2,2) .. controls +(0,-0.75) and +(0,0.75) .. (0,0.2);

  \draw[black, line width=5pt] (0,-2) -- (0,2);
  }

  \mydrawingsmall

  \begin{scope}[xscale=-1]
      \mydrawingsmall
  \end{scope}

  \begin{scope}[yscale=-1]
      \mydrawingsmall
  \end{scope}

  \begin{scope}[xscale=-1, yscale=-1]
      \mydrawingsmall
  \end{scope}

  \fill[white,opacity=0.9] (0,-2) rectangle (1.5+0.8,2);
  \draw[black, line width=5pt] (0,-2) -- (0,2);

  \node[] at (3.5,0) {$\stackrel{i}{\leftarrow}$};

  \draw[->,decorate,decoration={snake,amplitude=.4mm,segment length=2mm,post length=1mm}]  (-1.5,1.6)  -- (-0.1,1.6);


  \begin{scope} 
   [xshift=7cm]

   \newcommand{\mydrawing}{%
   \draw[black, ultra thick] (-1.5-0.2,-2) -- (-1.5-0.2,-1);
   \draw[black, ultra thick] (-1.5-0.2,-2) .. controls +(0,0.25) and +(0,-0.25) .. (-1.5-0.7,-1.5);
   \draw[black, ultra thick] (-1.5-0.2,-2) .. controls +(0,0.25) and +(0,-0.25) .. (-1.5+0.3,-1.5);
   
   \draw[black, ultra thick] (-1.5-0.2,-1.5) -- (-1.5-0.2,-1);
   \draw[black, ultra thick] (-1.5-0.2,-1) -- (-1.5-0.2,-0.5);
   \draw[black, ultra thick] (-1.5+0.3,-1.5) -- (-1.5+0.3,-1);
   \draw[black, ultra thick] (-1.5-0.7,-1.5) -- (-1.5-0.7,0);
 
   \draw[white,line width=5pt] (-1.5+0.3,-1) .. controls +(0,1.5) and +(0,-1.5) .. (1.5-0.3,3-2);
   \draw[black, ultra thick] (-1.5+0.3,-1) .. controls +(0,1.5) and +(0,-1.5) .. (1.5-0.3,3-2);
   \draw[white,line width=5pt] (-1.5-0.2,-0.5) .. controls +(0,1) and +(0,-1) .. (0,2);
   \draw[black, ultra thick] (-1.5-0.2,-0.5) .. controls +(0,1) and +(0,-1) .. (0,2);

   \draw[white, line width=7pt] (0,-0.1) -- (0,0.1);
   }

   \mydrawing
 
   \begin{scope}[yscale=-1]
       \mydrawing
   \end{scope}

   \begin{scope}[xscale=-1]
    \mydrawing
\end{scope}

   \begin{scope}[xscale=-1, yscale=-1]
       \mydrawing
   \end{scope}

   \fill[white,opacity=0.9] (0,-2) rectangle (1.5+0.8,2);
   \draw[black, line width=5pt] (0,-2) -- (0,2);

  \draw[->,decorate,decoration={snake,amplitude=.4mm,segment length=2mm,post length=1mm}]  (-1.5-0.7,-1.5)  -- (-1.5-0.2-0.1,-1.5);
  \draw[->,decorate,decoration={snake,amplitude=.4mm,segment length=2mm,post length=1mm}]  (-1.5-0.2,-1.5)  -- (-1.5+0.3-0.1,-1.5);
  \draw[->,decorate,decoration={snake,amplitude=.4mm,segment length=2mm,post length=1mm}]  (-1.5-0.7,-1.3) -- (-1.5+0.3-0.1,-1.3);
  \draw[-,ultra thick,decorate,decoration={snake,amplitude=.4mm,segment length=2mm,post length=1mm}]  (-1.5-0.7,-1.5)  -- (-1.5-0.2-0.1,-1.5);
  \draw[-,ultra thick,decorate,decoration={snake,amplitude=.4mm,segment length=2mm,post length=1mm}]  (-1.5-0.2,-1.5)  -- (-1.5+0.3-0.1,-1.5);
  \draw[-,ultra thick,decorate,decoration={snake,amplitude=.4mm,segment length=2mm,post length=1mm}]  (-1.5-0.7,-1.3) -- (-1.5+0.3-0.1,-1.3);

  \draw[->,decorate,decoration={snake,amplitude=.4mm,segment length=2mm,post length=1mm}]  (-0.7,-0.8)  -- (-0.1,-0.8);
  \draw[-,ultra thick,decorate,decoration={snake,amplitude=.4mm,segment length=2mm,post length=1mm}]  (-0.7,-0.8)  -- (-0.1,-0.8);
   \end{scope}

  \end{tikzpicture}
  \end{center}
where as in the notation introduced in section \ref{sssec:OrthosymplecticModuleHexagonRelations} each horizontal wavy arrow represents a localised cohomology class on the space corresponding to that horizontal slice of the diagram. The thin wavy arrows represent $e(\Nt_s^{(\OSpt)})$ and the thick wavy arrows represent their inverses. Thus, $e(\Nt_i)$ is then the product of these factors, pulled back to $\Ml^{\times 4}\times \Ml^\tau$.

We may pull back the classes along $i$, allowing us to work just on the right hand side diagram (which we vertically stretch for convenience). To begin computing $e(\Nt_i)$, we use that by \ref{eqn:EulerCherednik} the Euler classes satisfy orthosymplectic module hexagon relations of section \ref{sssec:OrthosymplecticModuleHexagonRelations}:
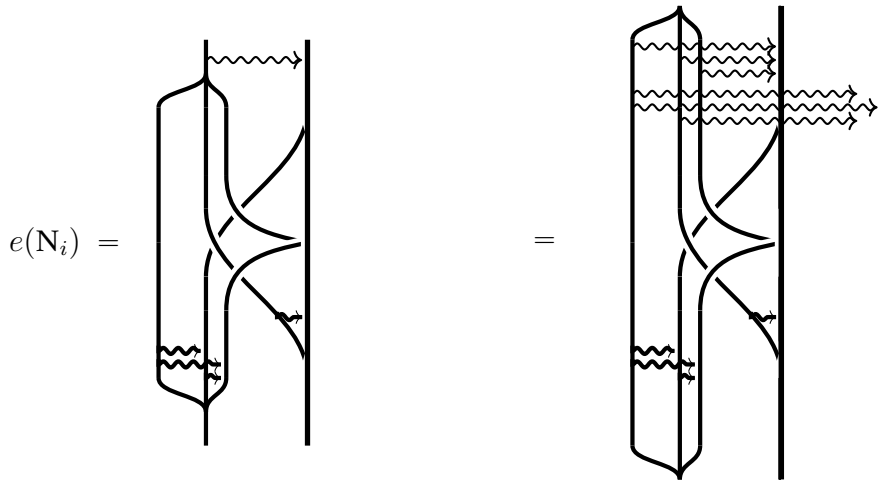
 \begin{center}
   \begin{tikzpicture}[scale = 0.9]
 
   \newcommand{\mydrawing}{%
 
   \draw[black, ultra thick] (-1.5,-3) -- (-1.5,-2);
   \draw[black, ultra thick] (-1.5,-2.5) .. controls +(0,0.25) and +(0,-0.25) .. (-1.5-0.7,-2);
   \draw[black, ultra thick] (-1.5,-2.5) .. controls +(0,0.25) and +(0,-0.25) .. (-1.5+0.3,-2);
 
   \draw[black, ultra thick] (-1.5,-2) -- (-1.5,-1);
   \draw[black, ultra thick] (-1.5,-1) -- (-1.5,-0.5);
   \draw[black, ultra thick] (-1.5+0.3,-2) -- (-1.5+0.3,-1);
   \draw[black, ultra thick] (-1.5-0.7,-2) -- (-1.5-0.7,0);
 
   \draw[white,line width=5pt] (-1.5+0.3,-1) .. controls +(0,1.5) and +(0,-1.5) .. (1.5-0.3,3-2);
   \draw[black, ultra thick] (-1.5+0.3,-1) .. controls +(0,1.5) and +(0,-1.5) .. (1.5-0.3,3-2);
   \draw[white,line width=5pt] (-1.5,-0.5) .. controls +(0,1) and +(0,-1) .. (0,2);
   \draw[black, ultra thick] (-1.5,-0.5) .. controls +(0,1) and +(0,-1) .. (0,2);
 
   \draw[white, line width=5pt] (0,-0.5) -- (0,0.5);
   }

   \mydrawing
 
   \begin{scope}[xscale=-1]
       \mydrawing
   \end{scope}
 
   \begin{scope}[yscale=-1]
       \mydrawing
   \end{scope}
 
   \begin{scope}[xscale=-1, yscale=-1]
       \mydrawing
   \end{scope}

   \fill[white,opacity=0.9] (0,-3.1) rectangle (1.5+0.8,3.1);
   \draw[black, line width=2pt] (0,-3) -- (0,3);

   \node[left] at (-2.5,0) {$e(\Nt_i) \ = \ $};
 
   \draw[->,decorate,decoration={snake,amplitude=.4mm,segment length=2mm,post length=1mm}, shorten >=2pt, shorten <=2pt]  (-1.5-0.7,-1.6)  -- (-1.5,-1.6);
   \draw[->,decorate,decoration={snake,amplitude=.4mm,segment length=2mm,post length=1mm}, shorten >=2pt, shorten <=2pt]  (-1.5-0.7,-1.8)  -- (-1.5+0.3,-1.8);
   \draw[->,decorate,decoration={snake,amplitude=.4mm,segment length=2mm,post length=1mm}, shorten >=2pt, shorten <=2pt]  (-1.5,-2.0)  -- (-1.5+0.3,-2.0);
 
   \draw[-,decorate,decoration={snake,amplitude=.4mm,segment length=2mm,post length=1mm},ultra thick, shorten >=2pt, shorten <=2pt]  (-1.5-0.7,-1.6)  -- (-1.5,-1.6);
   \draw[-,decorate,decoration={snake,amplitude=.4mm,segment length=2mm,post length=1mm},ultra thick, shorten >=2pt, shorten <=2pt]  (-1.5-0.7,-1.8)  -- (-1.5+0.3,-1.8);
   \draw[-,decorate,decoration={snake,amplitude=.4mm,segment length=2mm,post length=1mm},ultra thick, shorten >=2pt, shorten <=2pt]  (-1.5,-2.0)  -- (-1.5+0.3,-2.0);

   \draw[->,decorate,decoration={snake,amplitude=.4mm,segment length=2mm,post length=1mm}, shorten >=2pt, shorten <=2pt]  (-0.45,-1.1)  --  (0,-1.1);

   \draw[-,decorate,decoration={snake,amplitude=.4mm,segment length=2mm,post length=1mm},ultra thick, shorten >=2pt, shorten <=2pt] (-0.45,-1.1)  --  (0,-1.1);

   
   \draw[->,decorate,decoration={snake,amplitude=.4mm,segment length=2mm,post length=1mm}, thick, shorten >=2pt, shorten <=2pt]  (-1.5,2.7)  -- (0,2.7);

   \node[] at (3.5,0) {$=$};

   \begin{scope} 
    [xshift=7cm]
    
   \newcommand{\mydrawinglong}{%
 
   \draw[black, ultra thick] (-1.5,-3.5) -- (-1.5,-2);
   \draw[black, ultra thick] (-1.5,-3.5) .. controls +(0,0.25) and +(0,-0.25) .. (-1.5-0.7,-3);
   \draw[black, ultra thick] (-1.5,-3.5) .. controls +(0,0.25) and +(0,-0.25) .. (-1.5+0.3,-3);
 
   \draw[black, ultra thick] (-1.5,-2) -- (-1.5,-1);
   \draw[black, ultra thick] (-1.5,-1) -- (-1.5,-0.5);
   \draw[black, ultra thick] (-1.5+0.3,-3) -- (-1.5+0.3,-1);
   \draw[black, ultra thick] (-1.5-0.7,-3) -- (-1.5-0.7,0);
 
   \draw[white,line width=5pt] (-1.5+0.3,-1) .. controls +(0,1.5) and +(0,-1.5) .. (1.5-0.3,3-2);
   \draw[black, ultra thick] (-1.5+0.3,-1) .. controls +(0,1.5) and +(0,-1.5) .. (1.5-0.3,3-2);
   \draw[white,line width=5pt] (-1.5,-0.5) .. controls +(0,1) and +(0,-1) .. (0,2);
   \draw[black, ultra thick] (-1.5,-0.5) .. controls +(0,1) and +(0,-1) .. (0,2);
 
   \draw[white, line width=5pt] (0,-0.5) -- (0,0.5);
   \draw[black, line width=2pt] (0,-3.5) -- (0,3.5);
 
   }
    
   \mydrawinglong
 
   \begin{scope}[xscale=-1]
       \mydrawinglong
   \end{scope}
 
   \begin{scope}[yscale=-1]
       \mydrawinglong
   \end{scope}
 
   \begin{scope}[xscale=-1, yscale=-1]
       \mydrawinglong
   \end{scope}

   \fill[white,opacity=0.9] (0,-3.6) rectangle (1.5+0.8,3.6);
   \draw[black, line width=2pt] (0,-3.5) -- (0,3.5);

   \draw[->,decorate,decoration={snake,amplitude=.4mm,segment length=2mm,post length=1mm}, shorten >=2pt, shorten <=2pt]  (-1.5-0.7,-1.6)  -- (-1.5,-1.6);
   \draw[->,decorate,decoration={snake,amplitude=.4mm,segment length=2mm,post length=1mm}, shorten >=2pt, shorten <=2pt]  (-1.5-0.7,-1.8)  -- (-1.5+0.3,-1.8);
   \draw[->,decorate,decoration={snake,amplitude=.4mm,segment length=2mm,post length=1mm}, shorten >=2pt, shorten <=2pt]  (-1.5,-2.0)  -- (-1.5+0.3,-2.0);
 
   \draw[-,decorate,decoration={snake,amplitude=.4mm,segment length=2mm,post length=1mm},ultra thick, shorten >=2pt, shorten <=2pt]  (-1.5-0.7,-1.6)  -- (-1.5,-1.6);
   \draw[-,decorate,decoration={snake,amplitude=.4mm,segment length=2mm,post length=1mm},ultra thick, shorten >=2pt, shorten <=2pt]  (-1.5-0.7,-1.8)  -- (-1.5+0.3,-1.8);
   \draw[-,decorate,decoration={snake,amplitude=.4mm,segment length=2mm,post length=1mm},ultra thick, shorten >=2pt, shorten <=2pt]  (-1.5,-2.0)  -- (-1.5+0.3,-2.0);

   \draw[->,decorate,decoration={snake,amplitude=.4mm,segment length=2mm,post length=1mm}, shorten >=2pt, shorten <=2pt]  (-0.45,-1.1)  --  (0,-1.1);
   \draw[-,decorate,decoration={snake,amplitude=.4mm,segment length=2mm,post length=1mm},ultra thick, shorten >=2pt, shorten <=2pt] (-0.45,-1.1)  --  (0,-1.1);

   \draw[->,decorate,decoration={snake,amplitude=.4mm,segment length=2mm,post length=1mm}, thick, shorten >=2pt, shorten <=2pt]  (-1.5-0.7,2.9)  -- (0,2.9);
   \draw[->,decorate,decoration={snake,amplitude=.4mm,segment length=2mm,post length=1mm}, thick, shorten >=2pt, shorten <=2pt]  (-1.5,2.7) -- (0,2.7);
   \draw[->,decorate,decoration={snake,amplitude=.4mm,segment length=2mm,post length=1mm}, thick, shorten >=2pt, shorten <=2pt]  (-1.5+0.3,2.5)  -- (0,2.5);
   \draw[->,decorate,decoration={snake,amplitude=.4mm,segment length=2mm,post length=1mm}, thick, shorten >=2pt, shorten <=2pt]  (-1.5-0.7,2.2)  -- (1.5-0.3,2.2);
   \draw[->,decorate,decoration={snake,amplitude=.4mm,segment length=2mm,post length=1mm}, thick, shorten >=2pt, shorten <=2pt]  (-1.5-0.7,2) -- (1.5,2);
   \draw[->,decorate,decoration={snake,amplitude=.4mm,segment length=2mm,post length=1mm}, thick, shorten >=2pt, shorten <=2pt]  (-1.5,1.8) -- (1.5-0.3,1.8);

    \end{scope}
 
   \end{tikzpicture}
   \end{center}
We cancel off pairs of thick and thin arrows (representing inverse cohomology classes), apply the Cherednik hexagon relations to the remaining arrows, then cancel off more pairs and drag arrows to the location of the the braid swaps (which we colour).

 
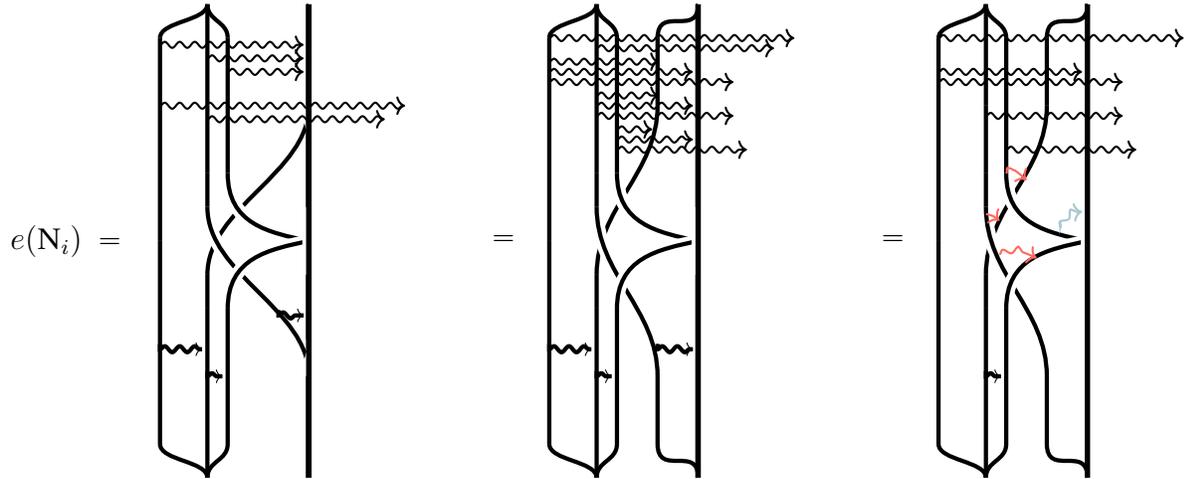
\begin{center}
  \begin{tikzpicture}[scale = 0.9]

  \newcommand{\mydrawing}{%

  \draw[black, ultra thick] (-1.5,-3) -- (-1.5,-2);
  \draw[black, ultra thick] (-1.5,-2.5) .. controls +(0,0.25) and +(0,-0.25) .. (-1.5-0.7,-2);
  \draw[black, ultra thick] (-1.5,-2.5) .. controls +(0,0.25) and +(0,-0.25) .. (-1.5+0.3,-2);

  \draw[black, ultra thick] (-1.5,-2) -- (-1.5,-1);
  \draw[black, ultra thick] (-1.5,-1) -- (-1.5,-0.5);
  \draw[black, ultra thick] (-1.5+0.3,-2) -- (-1.5+0.3,-1);
  \draw[black, ultra thick] (-1.5-0.7,-2) -- (-1.5-0.7,0);

  \draw[white,line width=5pt] (-1.5+0.3,-1) .. controls +(0,1.5) and +(0,-1.5) .. (1.5-0.3,3-2);
  \draw[black, ultra thick] (-1.5+0.3,-1) .. controls +(0,1.5) and +(0,-1.5) .. (1.5-0.3,3-2);
  \draw[white,line width=5pt] (-1.5,-0.5) .. controls +(0,1) and +(0,-1) .. (0,2);
  \draw[black, ultra thick] (-1.5,-0.5) .. controls +(0,1) and +(0,-1) .. (0,2);

  \draw[white, line width=5pt] (0,-0.5) -- (0,0.5);
  \draw[black, line width=2pt] (0,-3) -- (0,3);

  }

  \newcommand{\mydrawinglong}{%

  \draw[black, ultra thick] (-1.5,-3.5) -- (-1.5,-2);
  \draw[black, ultra thick] (-1.5,-3.5) .. controls +(0,0.25) and +(0,-0.25) .. (-1.5-0.7,-3);
  \draw[black, ultra thick] (-1.5,-3.5) .. controls +(0,0.25) and +(0,-0.25) .. (-1.5+0.3,-3);

  \draw[black, ultra thick] (-1.5,-2) -- (-1.5,-1);
  \draw[black, ultra thick] (-1.5,-1) -- (-1.5,-0.5);
  \draw[black, ultra thick] (-1.5+0.3,-3) -- (-1.5+0.3,-1);
  \draw[black, ultra thick] (-1.5-0.7,-3) -- (-1.5-0.7,0);

  \draw[white,line width=5pt] (-1.5+0.3,-1) .. controls +(0,1.5) and +(0,-1.5) .. (1.5-0.3,3-2);
  \draw[black, ultra thick] (-1.5+0.3,-1) .. controls +(0,1.5) and +(0,-1.5) .. (1.5-0.3,3-2);
  \draw[white,line width=5pt] (-1.5,-0.5) .. controls +(0,1) and +(0,-1) .. (0,2);
  \draw[black, ultra thick] (-1.5,-0.5) .. controls +(0,1) and +(0,-1) .. (0,2);

  \draw[white, line width=5pt] (0,-0.5) -- (0,0.5);
  \draw[black, line width=2pt] (0,-3.5) -- (0,3.5);

  }
   
  \mydrawinglong

  \begin{scope}[xscale=-1]
      \mydrawinglong
  \end{scope}

  \begin{scope}[yscale=-1]
      \mydrawinglong
  \end{scope}

  \begin{scope}[xscale=-1, yscale=-1]
      \mydrawinglong
  \end{scope}

  \fill[white,opacity=0.9] (0,-3.6) rectangle (1.5+0.8,3.6);
  \draw[black, line width=2pt] (0,-3.5) -- (0,3.5);

  \node[left] at (-2.5,0) {$e(\Nt_i) \ = \ $};

  \draw[->,decorate,decoration={snake,amplitude=.4mm,segment length=2mm,post length=1mm}, shorten >=2pt, shorten <=2pt]  (-1.5-0.7,-1.6)  -- (-1.5,-1.6);
  \draw[->,decorate,decoration={snake,amplitude=.4mm,segment length=2mm,post length=1mm}, shorten >=2pt, shorten <=2pt]  (-1.5,-2.0)  -- (-1.5+0.3,-2.0);

  \draw[-,decorate,decoration={snake,amplitude=.4mm,segment length=2mm,post length=1mm},ultra thick, shorten >=2pt, shorten <=2pt]  (-1.5-0.7,-1.6)  -- (-1.5,-1.6);
  \draw[-,decorate,decoration={snake,amplitude=.4mm,segment length=2mm,post length=1mm},ultra thick, shorten >=2pt, shorten <=2pt]  (-1.5,-2.0)  -- (-1.5+0.3,-2.0);

   \draw[->,decorate,decoration={snake,amplitude=.4mm,segment length=2mm,post length=1mm}, shorten >=2pt, shorten <=2pt]  (-0.45,-1.1)  --  (0,-1.1);
   \draw[-,decorate,decoration={snake,amplitude=.4mm,segment length=2mm,post length=1mm},ultra thick, shorten >=2pt, shorten <=2pt] (-0.45,-1.1)  --  (0,-1.1);

  \draw[->,decorate,decoration={snake,amplitude=.4mm,segment length=2mm,post length=1mm}, thick, shorten >=2pt, shorten <=2pt]  (-1.5-0.7,2.9)  -- (0,2.9);
  \draw[->,decorate,decoration={snake,amplitude=.4mm,segment length=2mm,post length=1mm}, thick, shorten >=2pt, shorten <=2pt]  (-1.5,2.7) -- (0,2.7);
  \draw[->,decorate,decoration={snake,amplitude=.4mm,segment length=2mm,post length=1mm}, thick, shorten >=2pt, shorten <=2pt]  (-1.5+0.3,2.5)  -- (0,2.5);
  
   \draw[->,decorate,decoration={snake,amplitude=.4mm,segment length=2mm,post length=1mm}, thick, shorten >=2pt, shorten <=2pt]  (-1.5-0.7,2) -- (1.5,2);
   \draw[->,decorate,decoration={snake,amplitude=.4mm,segment length=2mm,post length=1mm}, thick, shorten >=2pt, shorten <=2pt]  (-1.5,1.8) -- (1.5-0.3,1.8);

  \node[] at (2.875,0) {$=$};

  \node[] at (8.625,0) {$=$};

 \newcommand{\mydrawinglonger}{%
   
 \draw[black, ultra thick] (-1.5,-3.5) -- (-1.5,-2);
 \draw[black, ultra thick] (-1.5,-3.5) .. controls +(0,0.25) and +(0,-0.25) .. (-1.5-0.7,-3);
 \draw[black, ultra thick] (-1.5,-3.5) .. controls +(0,0.25) and +(0,-0.25) .. (-1.5+0.3,-3);
 
 \draw[black, ultra thick] (-1.5,-2) -- (-1.5,-1);
 \draw[black, ultra thick] (-1.5,-1) -- (-1.5,-0.5);
 \draw[black, ultra thick] (-1.5+0.3,-3) -- (-1.5+0.3,-1);
 \draw[black, ultra thick] (-1.5-0.7,-3) -- (-1.5-0.7,0);

 \draw[white,line width=5pt] (-1.5+0.3,-1) .. controls +(0,1.5) and +(0,-1.5) .. (1.5-0.3,3-2);
 \draw[black, ultra thick] (-1.5+0.3,-1) .. controls +(0,1.5) and +(0,-1.5) .. (1.5-0.3,3-2);
 \draw[white,line width=5pt] (-1.5,-0.5) .. controls +(0,1) and +(0,-1) .. (-0.6,2);
 \draw[black, ultra thick] (-1.5,-0.5) .. controls +(0,1) and +(0,-1) .. (-0.6,2);

 \draw[black, ultra thick]  (-0.6,2) --  (-0.6,3);
 \draw[black, ultra thick] (-0.6,3) .. controls +(0,0.5) and +(0,-0.5) .. (0,3.5);

 \draw[white, line width=5pt] (0,-0.5) -- (0,0.5);
 \draw[black, line width=2pt] (0,-3.5) -- (0,3.5);

 }

\begin{scope} 
 [xshift=5.75cm]

 \mydrawinglonger

 \begin{scope}[xscale=-1]
     \mydrawinglonger
 \end{scope}

 \begin{scope}[yscale=-1]
     \mydrawinglonger
 \end{scope}

 \begin{scope}[xscale=-1, yscale=-1]
     \mydrawinglonger
 \end{scope}

   \fill[white,opacity=0.9] (0,-3.6) rectangle (1.5+0.8,3.6);
   \draw[black, line width=2pt] (0,-3.5) -- (0,3.5);

  \draw[->,decorate,decoration={snake,amplitude=.4mm,segment length=2mm,post length=1mm}, shorten >=2pt, shorten <=2pt]  (-1.5-0.7,-1.6)  -- (-1.5,-1.6);
  \draw[->,decorate,decoration={snake,amplitude=.4mm,segment length=2mm,post length=1mm}, shorten >=2pt, shorten <=2pt]  (-1.5,-2.0)  -- (-1.5+0.3,-2.0);

  \draw[-,decorate,decoration={snake,amplitude=.4mm,segment length=2mm,post length=1mm},ultra thick, shorten >=2pt, shorten <=2pt]  (-1.5-0.7,-1.6)  -- (-1.5,-1.6);
  \draw[-,decorate,decoration={snake,amplitude=.4mm,segment length=2mm,post length=1mm},ultra thick, shorten >=2pt, shorten <=2pt]  (-1.5,-2.0)  -- (-1.5+0.3,-2.0);

  \draw[->,decorate,decoration={snake,amplitude=.4mm,segment length=2mm,post length=1mm}, shorten >=2pt, shorten <=2pt]  (-0.6,-1.6)  -- (0,-1.6);
  \draw[-,decorate,decoration={snake,amplitude=.4mm,segment length=2mm,post length=1mm},ultra thick, shorten >=2pt, shorten <=2pt]  (-0.6,-1.6)  -- (0,-1.6);

  
   \draw[->,decorate,decoration={snake,amplitude=.4mm,segment length=2mm,post length=1mm}, thick, shorten >=2pt, shorten <=2pt]  (-1.5-0.7,3.0) -- (1.5,3.0);
   \draw[->,decorate,decoration={snake,amplitude=.4mm,segment length=2mm,post length=1mm}, thick, shorten >=2pt, shorten <=2pt]  (-1.5,2.85) -- (1.5-0.3,2.85);

   \draw[->,decorate,decoration={snake,amplitude=.4mm,segment length=2mm,post length=1mm}, thick, shorten >=2pt, shorten <=2pt]  (-1.5-0.7,2.65) -- (-0.5,2.65);
   \draw[->,decorate,decoration={snake,amplitude=.4mm,segment length=2mm,post length=1mm}, thick, shorten >=2pt, shorten <=2pt]  (-1.5-0.7,2.50) -- (0,2.50);
   \draw[->,decorate,decoration={snake,amplitude=.4mm,segment length=2mm,post length=1mm}, thick, shorten >=2pt, shorten <=2pt]  (-1.5-0.7,2.35) -- (0.6,2.35);

   \draw[->,decorate,decoration={snake,amplitude=.4mm,segment length=2mm,post length=1mm}, thick, shorten >=2pt, shorten <=2pt]  (-1.5,2.15) -- (-0.5,2.15);
   \draw[->,decorate,decoration={snake,amplitude=.4mm,segment length=2mm,post length=1mm}, thick, shorten >=2pt, shorten <=2pt]  (-1.5,2.0) -- (0,2.0);
   \draw[->,decorate,decoration={snake,amplitude=.4mm,segment length=2mm,post length=1mm}, thick, shorten >=2pt, shorten <=2pt]  (-1.5,1.85) -- (0.6,1.85);

   \draw[->,decorate,decoration={snake,amplitude=.4mm,segment length=2mm,post length=1mm}, thick, shorten >=2pt, shorten <=2pt]  (-1.5+0.3,1.65) -- (-0.6,1.65);
   \draw[->,decorate,decoration={snake,amplitude=.4mm,segment length=2mm,post length=1mm}, thick, shorten >=2pt, shorten <=2pt]  (-1.5+0.3,1.5) -- (0,1.5);
   \draw[->,decorate,decoration={snake,amplitude=.4mm,segment length=2mm,post length=1mm}, thick, shorten >=2pt, shorten <=2pt]  (-1.5+0.3,1.35) -- (0.8,1.35);
 \end{scope}

 \begin{scope} 
  [xshift=11.5cm]

  \mydrawinglonger
 
  \begin{scope}[xscale=-1]
      \mydrawinglonger
  \end{scope}
 
  \begin{scope}[yscale=-1]
      \mydrawinglonger
  \end{scope}
 
  \begin{scope}[xscale=-1, yscale=-1]
      \mydrawinglonger
  \end{scope}

    \fill[white,opacity=0.9] (0,-3.6) rectangle (1.5+0.8,3.6);
    \draw[black, line width=2pt] (0,-3.5) -- (0,3.5);
 
   \draw[->,decorate,decoration={snake,amplitude=.4mm,segment length=2mm,post length=1mm}, shorten >=2pt, shorten <=2pt]  (-1.5,-2.0)  -- (-1.5+0.3,-2.0);
 
   \draw[-,decorate,decoration={snake,amplitude=.4mm,segment length=2mm,post length=1mm},ultra thick, shorten >=2pt, shorten <=2pt]  (-1.5,-2.0)  -- (-1.5+0.3,-2.0);
 

 
  \draw[->,decorate,decoration={snake,amplitude=.4mm,segment length=2mm,post length=1mm}, thick, shorten >=2pt, shorten <=2pt]  (-1.5-0.7,3.0) -- (1.5,3.0);

  \draw[->,decorate,decoration={snake,amplitude=.4mm,segment length=2mm,post length=1mm}, thick, shorten >=2pt, shorten <=2pt]  (-1.5-0.7,2.50) -- (0,2.50);
  \draw[->,decorate,decoration={snake,amplitude=.4mm,segment length=2mm,post length=1mm}, thick, shorten >=2pt, shorten <=2pt]  (-1.5-0.7,2.35) -- (0.6,2.35);

  \draw[->,decorate,decoration={snake,amplitude=.4mm,segment length=2mm,post length=1mm}, thick, shorten >=2pt, shorten <=2pt]  (-1.5,1.85) -- (0.6,1.85);

  \draw[->,decorate,decoration={snake,amplitude=.4mm,segment length=2mm,post length=1mm}, thick, shorten >=2pt, shorten <=2pt]  (-1.5+0.3,1.35) -- (0.8,1.35);

   \draw[->,decorate,decoration={snake,amplitude=.4mm,segment length=2mm,post length=1mm}, thick,pastelred]  (-1.5+0.3,1.0)  to[out=45,in=180-45] (-0.9,0.9);
   \draw[->,decorate,decoration={snake,amplitude=.4mm,segment length=2mm,post length=1mm}, thick,pastelred]  (-1.5,0.4)  to[out=25,in=180-25] (-1.3,0.3);
   \draw[->,decorate,decoration={snake,amplitude=.4mm,segment length=2mm,post length=1mm}, thick,pastelred]  (-1.3,-0.2)  to[out=25,in=180-25] (-0.75,-0.25);

   \draw[->,decorate,decoration={snake,amplitude=.4mm,segment length=2mm,post length=1mm}, thick, shorten >=2pt, shorten <=2pt,pastelblue]  (-0.4,0.2) to[out=65,in=180] (0,0.44);

  \end{scope}
  \end{tikzpicture}
  \end{center}

    Likewise, if we denote by $J$ the product of the cohomology class arising from $J_3 \otimes J^{\OSpt}$ and the pullback of $J^{-1}$ to $\Ml^{\times 4}\times \Ml^\tau$, by reflection-symmetry of the twistedYetter-Drinfeld string diagram we have

    
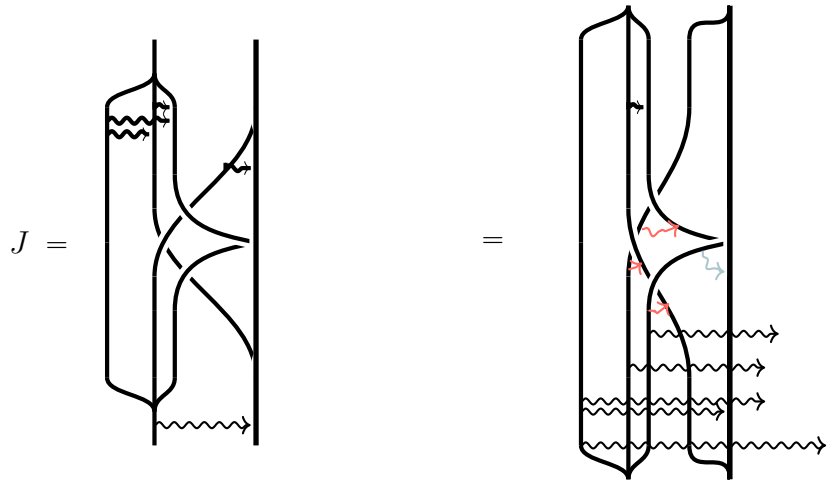
\begin{center}
      \begin{tikzpicture}[scale = 0.9]
    
      \begin{scope} 
       [yscale=-1]
   \newcommand{\mydrawing}{%
 
   \draw[black, ultra thick] (-1.5,-3) -- (-1.5,-2);
   \draw[black, ultra thick] (-1.5,-2.5) .. controls +(0,0.25) and +(0,-0.25) .. (-1.5-0.7,-2);
   \draw[black, ultra thick] (-1.5,-2.5) .. controls +(0,0.25) and +(0,-0.25) .. (-1.5+0.3,-2);
 
   \draw[black, ultra thick] (-1.5,-2) -- (-1.5,-1);
   \draw[black, ultra thick] (-1.5,-1) -- (-1.5,-0.5);
   \draw[black, ultra thick] (-1.5+0.3,-2) -- (-1.5+0.3,-1);
   \draw[black, ultra thick] (-1.5-0.7,-2) -- (-1.5-0.7,0);
 
   \draw[white,line width=5pt] (-1.5+0.3,-1) .. controls +(0,1.5) and +(0,-1.5) .. (1.5-0.3,3-2);
   \draw[black, ultra thick] (-1.5+0.3,-1) .. controls +(0,1.5) and +(0,-1.5) .. (1.5-0.3,3-2);
   \draw[white,line width=5pt] (-1.5,-0.5) .. controls +(0,1) and +(0,-1) .. (0,2);
   \draw[black, ultra thick] (-1.5,-0.5) .. controls +(0,1) and +(0,-1) .. (0,2);
 
   \draw[white, line width=5pt] (0,-0.5) -- (0,0.5);
   }

   \mydrawing
 
   \begin{scope}[xscale=-1]
       \mydrawing
   \end{scope}
 
   \begin{scope}[yscale=-1]
       \mydrawing
   \end{scope}
 
   \begin{scope}[xscale=-1, yscale=-1]
       \mydrawing
   \end{scope}

   \fill[white,opacity=0.9] (0,-3.1) rectangle (1.5+0.8,3.1);
   \draw[black, line width=2pt] (0,-3) -- (0,3);

   \draw[->,decorate,decoration={snake,amplitude=.4mm,segment length=2mm,post length=1mm}, shorten >=2pt, shorten <=2pt]  (-1.5-0.7,-1.6)  -- (-1.5,-1.6);
   \draw[->,decorate,decoration={snake,amplitude=.4mm,segment length=2mm,post length=1mm}, shorten >=2pt, shorten <=2pt]  (-1.5-0.7,-1.8)  -- (-1.5+0.3,-1.8);
   \draw[->,decorate,decoration={snake,amplitude=.4mm,segment length=2mm,post length=1mm}, shorten >=2pt, shorten <=2pt]  (-1.5,-2.0)  -- (-1.5+0.3,-2.0);
 
   \draw[-,decorate,decoration={snake,amplitude=.4mm,segment length=2mm,post length=1mm},ultra thick, shorten >=2pt, shorten <=2pt]  (-1.5-0.7,-1.6)  -- (-1.5,-1.6);
   \draw[-,decorate,decoration={snake,amplitude=.4mm,segment length=2mm,post length=1mm},ultra thick, shorten >=2pt, shorten <=2pt]  (-1.5-0.7,-1.8)  -- (-1.5+0.3,-1.8);
   \draw[-,decorate,decoration={snake,amplitude=.4mm,segment length=2mm,post length=1mm},ultra thick, shorten >=2pt, shorten <=2pt]  (-1.5,-2.0)  -- (-1.5+0.3,-2.0);

   \draw[->,decorate,decoration={snake,amplitude=.4mm,segment length=2mm,post length=1mm}, shorten >=2pt, shorten <=2pt]  (-0.45,-1.1)  --  (0,-1.1);

   \draw[-,decorate,decoration={snake,amplitude=.4mm,segment length=2mm,post length=1mm},ultra thick, shorten >=2pt, shorten <=2pt] (-0.45,-1.1)  --  (0,-1.1);

   
   \draw[->,decorate,decoration={snake,amplitude=.4mm,segment length=2mm,post length=1mm}, thick, shorten >=2pt, shorten <=2pt]  (-1.5,2.7)  -- (0,2.7); 
  \end{scope}

      \node[] at (3.5,0) {$=$};
      \node[] at (7,0) {$=$};
    
      \node[left] at (-2.5,0) {$J \ = \ $};

      \begin{scope} 
       [xshift=7cm]

     \newcommand{\mydrawinglonger}{%
       
     \draw[black, ultra thick] (-1.5,-3.5) -- (-1.5,-2);
     \draw[black, ultra thick] (-1.5,-3.5) .. controls +(0,0.25) and +(0,-0.25) .. (-1.5-0.7,-3);
     \draw[black, ultra thick] (-1.5,-3.5) .. controls +(0,0.25) and +(0,-0.25) .. (-1.5+0.3,-3);

     \draw[black, ultra thick] (-1.5,-2) -- (-1.5,-1);
     \draw[black, ultra thick] (-1.5,-1) -- (-1.5,-0.5);
     \draw[black, ultra thick] (-1.5+0.3,-3) -- (-1.5+0.3,-1);
     \draw[black, ultra thick] (-1.5-0.7,-3) -- (-1.5-0.7,0);
    
     \draw[white,line width=5pt] (-1.5+0.3,-1) .. controls +(0,1.5) and +(0,-1.5) .. (1.5-0.3,3-2);
     \draw[black, ultra thick] (-1.5+0.3,-1) .. controls +(0,1.5) and +(0,-1.5) .. (1.5-0.3,3-2);
     \draw[white,line width=5pt] (-1.5,-0.5) .. controls +(0,1) and +(0,-1) .. (-0.6,2);
     \draw[black, ultra thick] (-1.5,-0.5) .. controls +(0,1) and +(0,-1) .. (-0.6,2);
    
     \draw[black, ultra thick]  (-0.6,2) --  (-0.6,3);
     \draw[black, ultra thick] (-0.6,3) .. controls +(0,0.5) and +(0,-0.5) .. (0,3.5);

     \draw[white, line width=5pt] (0,-0.5) -- (0,0.5);
     \draw[black, line width=2pt] (0,-3.5) -- (0,3.5);
    
     }

     \mydrawinglonger
    
     \begin{scope}[xscale=-1]
         \mydrawinglonger
     \end{scope}
    
     \begin{scope}[yscale=-1]
         \mydrawinglonger
     \end{scope}
    
     \begin{scope}[xscale=-1, yscale=-1]
         \mydrawinglonger
     \end{scope}

       \fill[white,opacity=0.9] (0,-3.6) rectangle (1.5+0.8,3.6);
       \draw[black, line width=2pt] (0,-3.5) -- (0,3.5);

          \draw[->,decorate,decoration={snake,amplitude=.4mm,segment length=2mm,post length=1mm}, shorten >=2pt, shorten <=2pt]  (-1.5,2.0)  -- (-1.5+0.3,2.0);
    
          \draw[-,decorate,decoration={snake,amplitude=.4mm,segment length=2mm,post length=1mm},ultra thick, shorten >=2pt, shorten <=2pt]  (-1.5,2.0)  -- (-1.5+0.3,2.0);
        

        
         \draw[->,decorate,decoration={snake,amplitude=.4mm,segment length=2mm,post length=1mm}, thick, shorten >=2pt, shorten <=2pt]  (-1.5-0.7,-3.0) -- (1.5,-3.0);

         \draw[->,decorate,decoration={snake,amplitude=.4mm,segment length=2mm,post length=1mm}, thick, shorten >=2pt, shorten <=2pt]  (-1.5-0.7,-2.50) -- (0,-2.50);
         \draw[->,decorate,decoration={snake,amplitude=.4mm,segment length=2mm,post length=1mm}, thick, shorten >=2pt, shorten <=2pt]  (-1.5-0.7,-2.35) -- (0.6,-2.35);

         \draw[->,decorate,decoration={snake,amplitude=.4mm,segment length=2mm,post length=1mm}, thick, shorten >=2pt, shorten <=2pt]  (-1.5,-1.85) -- (0.6,-1.85);

         \draw[->,decorate,decoration={snake,amplitude=.4mm,segment length=2mm,post length=1mm}, thick, shorten >=2pt, shorten <=2pt]  (-1.5+0.3,-1.35) -- (0.8,-1.35);

          \draw[->,decorate,decoration={snake,amplitude=.4mm,segment length=2mm,post length=1mm}, thick,pastelred]  (-1.5+0.3,-1.0)  to[out=-45,in=-180+45] (-0.9,-0.9);
          \draw[->,decorate,decoration={snake,amplitude=.4mm,segment length=2mm,post length=1mm}, thick,pastelred]  (-1.5,-0.4)  to[out=-25,in=-180+25] (-1.3,-0.3);
          \draw[->,decorate,decoration={snake,amplitude=.4mm,segment length=2mm,post length=1mm}, thick,pastelred]  (-1.3,0.2)  to[out=-25,in=-180+25] (-0.75,0.25);

          \draw[->,decorate,decoration={snake,amplitude=.4mm,segment length=2mm,post length=1mm}, thick, shorten >=2pt, shorten <=2pt,pastelblue]  (-0.4,-0.2) to[out=-65,in=180] (0,-0.44);

       \end{scope}
     
      \end{tikzpicture}
      \end{center}
It follows that  as localised cohomology classes on $\Ml^{\times 4}\times \Ml^\tau$, the product $(J_3 \otimes J_3^{\OSpt})\big/ e(\Nt_i)_{\loc}\cdot J^{\OSpt}$ is equal to

 \begin{center}
  \begin{tikzpicture}[scale = 0.9]
   
 \newcommand{\mydrawinglong}{%
   
 \draw[black, ultra thick] (-1.5,-3.5) -- (-1.5,-2);
 \draw[black, ultra thick] (-1.5,-3.5) .. controls +(0,0.25) and +(0,-0.25) .. (-1.5-0.7,-3);
 \draw[black, ultra thick] (-1.5,-3.5) .. controls +(0,0.25) and +(0,-0.25) .. (-1.5+0.3,-3);

 \draw[black, ultra thick] (-1.5,-2) -- (-1.5,-1);
 \draw[black, ultra thick] (-1.5,-1) -- (-1.5,-0.5);
 \draw[black, ultra thick] (-1.5+0.3,-3) -- (-1.5+0.3,-1);
 \draw[black, ultra thick] (-1.5-0.7,-3) -- (-1.5-0.7,0);

 \draw[white,line width=5pt] (-1.5+0.3,-1) .. controls +(0,1.5) and +(0,-1.5) .. (1.5-0.3,3-2);
 \draw[black, ultra thick] (-1.5+0.3,-1) .. controls +(0,1.5) and +(0,-1.5) .. (1.5-0.3,3-2);
 \draw[white,line width=5pt] (-1.5,-0.5) .. controls +(0,1) and +(0,-1) .. (-0.6,2);
 \draw[black, ultra thick] (-1.5,-0.5) .. controls +(0,1) and +(0,-1) .. (-0.6,2);

 \draw[black, ultra thick]  (-0.6,2) --  (-0.6,3);
 \draw[black, ultra thick] (-0.6,3) .. controls +(0,0.5) and +(0,-0.5) .. (0,3.5);

 \draw[white, line width=5pt] (0,-0.5) -- (0,0.5);
 \draw[black, line width=2pt] (0,-3.5) -- (0,3.5);

 }
   
  \mydrawinglong

  \begin{scope}[xscale=-1]
      \mydrawinglong
  \end{scope}

  \begin{scope}[yscale=-1]
      \mydrawinglong
  \end{scope}

  \begin{scope}[xscale=-1, yscale=-1]
      \mydrawinglong
  \end{scope}

  \fill[white,opacity=0.9] (0,-3.6) rectangle (1.5+0.8,3.6);
  \draw[black, line width=2pt] (0,-3.5) -- (0,3.5);

          \draw[->,decorate,decoration={snake,amplitude=.4mm,segment length=2mm,post length=1mm}, shorten >=2pt, shorten <=2pt]  (-1.5,2.0)  -- (-1.5+0.3,2.0);
    
          \draw[-,decorate,decoration={snake,amplitude=.4mm,segment length=2mm,post length=1mm},ultra thick, shorten >=2pt, shorten <=2pt]  (-1.5,2.0)  -- (-1.5+0.3,2.0);
        

        
         \draw[->,decorate,decoration={snake,amplitude=.4mm,segment length=2mm,post length=1mm}, thick, shorten >=2pt, shorten <=2pt]  (-1.5-0.7,-3.0) -- (1.5,-3.0);

         \draw[->,decorate,decoration={snake,amplitude=.4mm,segment length=2mm,post length=1mm}, thick, shorten >=2pt, shorten <=2pt]  (-1.5-0.7,-2.50) -- (0,-2.50);
         \draw[->,decorate,decoration={snake,amplitude=.4mm,segment length=2mm,post length=1mm}, thick, shorten >=2pt, shorten <=2pt]  (-1.5-0.7,-2.35) -- (0.6,-2.35);

         \draw[->,decorate,decoration={snake,amplitude=.4mm,segment length=2mm,post length=1mm}, thick, shorten >=2pt, shorten <=2pt]  (-1.5,-1.85) -- (0.6,-1.85);

         \draw[->,decorate,decoration={snake,amplitude=.4mm,segment length=2mm,post length=1mm}, thick, shorten >=2pt, shorten <=2pt]  (-1.5+0.3,-1.35) -- (0.8,-1.35);

          \draw[->,decorate,decoration={snake,amplitude=.4mm,segment length=2mm,post length=1mm}, thick,pastelred]  (-1.5+0.3,-1.0)  to[out=-45,in=-180+45] (-0.9,-0.9);
          \draw[->,decorate,decoration={snake,amplitude=.4mm,segment length=2mm,post length=1mm}, thick,pastelred]  (-1.5,-0.4)  to[out=-25,in=-180+25] (-1.3,-0.3);
          \draw[->,decorate,decoration={snake,amplitude=.4mm,segment length=2mm,post length=1mm}, thick,pastelred]  (-1.3,0.2)  to[out=-25,in=-180+25] (-0.75,0.25);

          \draw[->,decorate,decoration={snake,amplitude=.4mm,segment length=2mm,post length=1mm}, thick, shorten >=2pt, shorten <=2pt,pastelblue]  (-0.4,-0.2) to[out=-65,in=180] (0,-0.44);

        \draw[->,decorate,decoration={snake,amplitude=.4mm,segment length=2mm,post length=1mm}, shorten >=2pt, shorten <=2pt]  (-1.5,-2.0)  -- (-1.5+0.3,-2.0);
        

        \draw[->,decorate,decoration={snake,amplitude=.4mm,segment length=2mm,post length=1mm}, thick, shorten >=2pt, shorten <=2pt]  (-1.5-0.7,3.0) -- (1.5,3.0);
        \draw[-,decorate,decoration={snake,amplitude=.4mm,segment length=2mm,post length=1mm},ultra thick, shorten >=2pt, shorten <=2pt]  (-1.5-0.7,3.0) -- (1.5,3.0);

        \draw[->,decorate,decoration={snake,amplitude=.4mm,segment length=2mm,post length=1mm}, thick, shorten >=2pt, shorten <=2pt]  (-1.5-0.7,2.50) -- (0,2.50);
        \draw[-,decorate,decoration={snake,amplitude=.4mm,segment length=2mm,post length=1mm},ultra thick, shorten >=2pt, shorten <=2pt]  (-1.5-0.7,2.50) -- (0,2.50);

        \draw[->,decorate,decoration={snake,amplitude=.4mm,segment length=2mm,post length=1mm}, thick, shorten >=2pt, shorten <=2pt]  (-1.5-0.7,2.35) -- (0.6,2.35);
        \draw[-,decorate,decoration={snake,amplitude=.4mm,segment length=2mm,post length=1mm},ultra thick, shorten >=2pt, shorten <=2pt]  (-1.5-0.7,2.35) -- (0.6,2.35);

        \draw[->,decorate,decoration={snake,amplitude=.4mm,segment length=2mm,post length=1mm}, thick, shorten >=2pt, shorten <=2pt]  (-1.5,1.85) -- (0.6,1.85);
        \draw[-,decorate,decoration={snake,amplitude=.4mm,segment length=2mm,post length=1mm},ultra thick, shorten >=2pt, shorten <=2pt]  (-1.5,1.85) -- (0.6,1.85);

        \draw[->,decorate,decoration={snake,amplitude=.4mm,segment length=2mm,post length=1mm}, thick, shorten >=2pt, shorten <=2pt]  (-1.5+0.3,1.35) -- (0.8,1.35);
        \draw[-,decorate,decoration={snake,amplitude=.4mm,segment length=2mm,post length=1mm},ultra thick, shorten >=2pt, shorten <=2pt]  (-1.5+0.3,1.35) -- (0.8,1.35);

        \draw[->,decorate,decoration={snake,amplitude=.4mm,segment length=2mm,post length=1mm}, thick,pastelred]  (-1.5+0.3,1.0)  to[out=45,in=180-45] (-0.9,0.9);
        \draw[-,decorate,decoration={snake,amplitude=.4mm,segment length=2mm,post length=1mm},ultra thick,pastelred]  (-1.5+0.3,1.0)  to[out=45,in=180-45] (-0.9,0.9);
        \draw[->,decorate,decoration={snake,amplitude=.4mm,segment length=2mm,post length=1mm}, thick,pastelred]  (-1.5,0.4)  to[out=25,in=180-25] (-1.3,0.3);
        \draw[-,decorate,decoration={snake,amplitude=.4mm,segment length=2mm,post length=1mm},ultra thick,pastelred]  (-1.5,0.4)  to[out=25,in=180-25] (-1.3,0.3);
        \draw[->,decorate,decoration={snake,amplitude=.4mm,segment length=2mm,post length=1mm}, thick,pastelred]  (-1.3,-0.2)  to[out=25,in=180-25] (-0.75,-0.25);
        \draw[-,decorate,decoration={snake,amplitude=.4mm,segment length=2mm,post length=1mm},ultra thick,pastelred]  (-1.3,-0.2)  to[out=25,in=180-25] (-0.75,-0.25);

        \draw[->,decorate,decoration={snake,amplitude=.4mm,segment length=2mm,post length=1mm}, thick, shorten >=2pt, shorten <=2pt,pastelblue]  (-0.4,0.2) to[out=65,in=180] (0,0.44);
        \draw[-,decorate,decoration={snake,amplitude=.4mm,segment length=2mm,post length=1mm},ultra thick, shorten >=2pt, shorten <=2pt,pastelblue]  (-0.4,0.2) to[out=65,in=180] (0,0.44);

  \node[left] at (-2.5,0) {$J/e(\Nt_i) \ = \ $};

  \node[] at (2.875,0) {$=$};

  \node[] at (8.625,0) {$=$};

  \begin{scope} 
   [xshift=5.75cm] 
   
  \mydrawinglong

  \begin{scope}[xscale=-1]
      \mydrawinglong
  \end{scope}

  \begin{scope}[yscale=-1]
      \mydrawinglong
  \end{scope}

  \begin{scope}[xscale=-1, yscale=-1]
      \mydrawinglong
  \end{scope}

  \fill[white,opacity=0.9] (0,-3.6) rectangle (1.5+0.8,3.6);
  \draw[black, line width=2pt] (0,-3.5) -- (0,3.5);

          \draw[->,decorate,decoration={snake,amplitude=.4mm,segment length=2mm,post length=1mm}, shorten >=2pt, shorten <=2pt]  (-1.5,2.0)  -- (-1.5+0.3,2.0);
    
          \draw[-,decorate,decoration={snake,amplitude=.4mm,segment length=2mm,post length=1mm},ultra thick, shorten >=2pt, shorten <=2pt]  (-1.5,2.0)  -- (-1.5+0.3,2.0);


         \draw[->,decorate,decoration={snake,amplitude=.4mm,segment length=2mm,post length=1mm}, thick, shorten >=2pt, shorten <=2pt]  (-1.5,-1.85) -- (0.6,-1.85);

         \draw[->,decorate,decoration={snake,amplitude=.4mm,segment length=2mm,post length=1mm}, thick, shorten >=2pt, shorten <=2pt]  (-1.5+0.3,-1.35) -- (0.8,-1.35);

          \draw[->,decorate,decoration={snake,amplitude=.4mm,segment length=2mm,post length=1mm}, thick,pastelred]  (-1.5+0.3,-1.0)  to[out=-45,in=-180+45] (-0.9,-0.9);
          \draw[->,decorate,decoration={snake,amplitude=.4mm,segment length=2mm,post length=1mm}, thick,pastelred]  (-1.5,-0.4)  to[out=-25,in=-180+25] (-1.3,-0.3);
          \draw[->,decorate,decoration={snake,amplitude=.4mm,segment length=2mm,post length=1mm}, thick,pastelred]  (-1.3,0.2)  to[out=-25,in=-180+25] (-0.75,0.25);

          \draw[->,decorate,decoration={snake,amplitude=.4mm,segment length=2mm,post length=1mm}, thick, shorten >=2pt, shorten <=2pt,pastelblue]  (-0.4,-0.2) to[out=-65,in=180] (0,-0.44);

        \draw[->,decorate,decoration={snake,amplitude=.4mm,segment length=2mm,post length=1mm}, shorten >=2pt, shorten <=2pt]  (-1.5,-2.0)  -- (-1.5+0.3,-2.0);
        

        \draw[->,decorate,decoration={snake,amplitude=.4mm,segment length=2mm,post length=1mm}, thick, shorten >=2pt, shorten <=2pt]  (-1.5,1.85) -- (0.6,1.85);
        \draw[-,decorate,decoration={snake,amplitude=.4mm,segment length=2mm,post length=1mm},ultra thick, shorten >=2pt, shorten <=2pt]  (-1.5,1.85) -- (0.6,1.85);

        \draw[->,decorate,decoration={snake,amplitude=.4mm,segment length=2mm,post length=1mm}, thick, shorten >=2pt, shorten <=2pt]  (-1.5+0.3,1.35) -- (0.8,1.35);
        \draw[-,decorate,decoration={snake,amplitude=.4mm,segment length=2mm,post length=1mm},ultra thick, shorten >=2pt, shorten <=2pt]  (-1.5+0.3,1.35) -- (0.8,1.35);

        \draw[->,decorate,decoration={snake,amplitude=.4mm,segment length=2mm,post length=1mm}, thick,pastelred]  (-1.5+0.3,1.0)  to[out=45,in=180-45] (-0.9,0.9);
        \draw[-,decorate,decoration={snake,amplitude=.4mm,segment length=2mm,post length=1mm},ultra thick,pastelred]  (-1.5+0.3,1.0)  to[out=45,in=180-45] (-0.9,0.9);
        \draw[->,decorate,decoration={snake,amplitude=.4mm,segment length=2mm,post length=1mm}, thick,pastelred]  (-1.5,0.4)  to[out=25,in=180-25] (-1.3,0.3);
        \draw[-,decorate,decoration={snake,amplitude=.4mm,segment length=2mm,post length=1mm},ultra thick,pastelred]  (-1.5,0.4)  to[out=25,in=180-25] (-1.3,0.3);
        \draw[->,decorate,decoration={snake,amplitude=.4mm,segment length=2mm,post length=1mm}, thick,pastelred]  (-1.3,-0.2)  to[out=25,in=180-25] (-0.75,-0.25);
        \draw[-,decorate,decoration={snake,amplitude=.4mm,segment length=2mm,post length=1mm},ultra thick,pastelred]  (-1.3,-0.2)  to[out=25,in=180-25] (-0.75,-0.25);

        \draw[->,decorate,decoration={snake,amplitude=.4mm,segment length=2mm,post length=1mm}, thick, shorten >=2pt, shorten <=2pt,pastelblue]  (-0.4,0.2) to[out=65,in=180] (0,0.44);
        \draw[-,decorate,decoration={snake,amplitude=.4mm,segment length=2mm,post length=1mm},ultra thick, shorten >=2pt, shorten <=2pt,pastelblue]  (-0.4,0.2) to[out=65,in=180] (0,0.44);

   \end{scope}

   \begin{scope} 
    [xshift=11.5cm] 
    
   \mydrawinglong
 
   \begin{scope}[xscale=-1]
       \mydrawinglong
   \end{scope}
 
   \begin{scope}[yscale=-1]
       \mydrawinglong
   \end{scope}
 
   \begin{scope}[xscale=-1, yscale=-1]
       \mydrawinglong
   \end{scope}

   \fill[white,opacity=0.9] (0,-3.6) rectangle (1.5+0.8,3.6);
   \draw[black, line width=2pt] (0,-3.5) -- (0,3.5);


           \draw[->,decorate,decoration={snake,amplitude=.4mm,segment length=2mm,post length=1mm}, thick,pastelred]  (-1.5+0.3,-1.0)  to[out=-45,in=-180+45] (-0.9,-0.9);
           \draw[->,decorate,decoration={snake,amplitude=.4mm,segment length=2mm,post length=1mm}, thick,pastelred]  (-1.5,-0.4)  to[out=-25,in=-180+25] (-1.3,-0.3);
           \draw[->,decorate,decoration={snake,amplitude=.4mm,segment length=2mm,post length=1mm}, thick,pastelred]  (-1.3,0.2)  to[out=-25,in=-180+25] (-0.75,0.25);

           \draw[->,decorate,decoration={snake,amplitude=.4mm,segment length=2mm,post length=1mm}, thick, shorten >=2pt, shorten <=2pt,pastelblue]  (-0.4,-0.2) to[out=-65,in=180] (0,-0.44);

         \draw[->,decorate,decoration={snake,amplitude=.4mm,segment length=2mm,post length=1mm}, thick,pastelred]  (-1.5+0.3,1.0)  to[out=45,in=180-45] (-0.9,0.9);
         \draw[-,decorate,decoration={snake,amplitude=.4mm,segment length=2mm,post length=1mm},ultra thick,pastelred]  (-1.5+0.3,1.0)  to[out=45,in=180-45] (-0.9,0.9);
         \draw[->,decorate,decoration={snake,amplitude=.4mm,segment length=2mm,post length=1mm}, thick,pastelred]  (-1.5,0.4)  to[out=25,in=180-25] (-1.3,0.3);
         \draw[-,decorate,decoration={snake,amplitude=.4mm,segment length=2mm,post length=1mm},ultra thick,pastelred]  (-1.5,0.4)  to[out=25,in=180-25] (-1.3,0.3);
         \draw[->,decorate,decoration={snake,amplitude=.4mm,segment length=2mm,post length=1mm}, thick,pastelred]  (-1.3,-0.2)  to[out=25,in=180-25] (-0.75,-0.25);
         \draw[-,decorate,decoration={snake,amplitude=.4mm,segment length=2mm,post length=1mm},ultra thick,pastelred]  (-1.3,-0.2)  to[out=25,in=180-25] (-0.75,-0.25);
 
         \draw[->,decorate,decoration={snake,amplitude=.4mm,segment length=2mm,post length=1mm}, thick, shorten >=2pt, shorten <=2pt,pastelblue]  (-0.4,0.2) to[out=65,in=180] (0,0.44);
         \draw[-,decorate,decoration={snake,amplitude=.4mm,segment length=2mm,post length=1mm},ultra thick, shorten >=2pt, shorten <=2pt,pastelblue]  (-0.4,0.2) to[out=65,in=180] (0,0.44);

    \end{scope}

  \end{tikzpicture}
  \end{center}
  and so the middle endomorphism of $(\Ht^\sbt(\Ml,\varphi)^{\otimes 4} \otimes \Ht^\sbt(\Ml^\tau,\varphi^\tau))_{\loc}$ is
  $$R_{34}\sigma_{34}\cdot R_{23}\sigma_{23}\cdot  K_{45}\tau_4 \cdot R_{34}\sigma_{34} \ = \ \beta_{34}\cdot \beta_{23}\cdot \kappa_{4}\cdot \beta_{34}$$
  which finishes the proof of Yetter-Drinfeld condition that \eqref{fig:YDConditionLocalised} commutes.
\end{proof}

  \begin{cor}
    Consider the vertex bialgebra structure on $\Ht^\sbt(\Ml,\varphi)$ coming from the Joyce-Liu vertex coproduct and CoHA product. Then its action and vertex coaction on $\Ht^\sbt(\Ml^\tau,\varphi^\tau)$ makes it into a vertex Yetter-Drinfeld module, i.e. the diagram 
    \begin{equation} \label{fig:YDConditionVertex}
      \begin{tikzcd}[row sep = {20pt,between origins}, column sep = {10pt}]
       (\Ht^\sbt(\Ml,\varphi) \otimes \Ht^\sbt(\Ml^\tau,\varphi^\tau))((z,w)) \ar[<-,r,"\iota"]  & (\Ht^\sbt(\Ml,\varphi) \otimes \Ht^\sbt(\Ml^\tau,\varphi^\tau))((z_1,\ldots ,z_4,w))\\
       & \\
        &  (\Ht^\sbt(\Ml,\varphi)^{\otimes 3} \otimes \Ht^\sbt(\Ml,\varphi) \otimes \Ht^\sbt(\Ml^\tau,\varphi^\tau))((z_1,\ldots ,z_4,w)) \ar[uu]  \\ 
        M \ar[uuu] &  \\
      & (\Ht^\sbt(\Ml,\varphi)^{\otimes 3} \otimes \Ht^\sbt(\Ml,\varphi) \otimes \Ht^\sbt(\Ml^\tau,\varphi^\tau))((z_1,\ldots ,z_4,w)) \ar[uu,"\beta(z_2-z_3)\kappa(z_3)\beta(z_2-z_3)"]  \\
      & \\
      \Ht^\sbt(\Ml,\varphi) \otimes \Ht^\sbt(\Ml^\tau,\varphi^\tau) \ar[uuu]  \ar[r,equals]  & \Ht^\sbt(\Ml,\varphi) \otimes \Ht^\sbt(\Ml^\tau,\varphi^\tau) \ar[uu] 
      \end{tikzcd}
      \end{equation}
    commutes, where the map $\iota$ sends $z_i \mapsto  z,$ $w_i  \mapsto w$.
  \end{cor}

\subsection{Final remarks: a twisted bialgebra module structure}
\label{ssec:FinalRemarks}
\subsubsection{} We suspect that the twisted Yetter-Drinfeld module is the most natural structure on $\Ht^\sbt(\Ml^\tau, \varphi^\tau)$, however we note that it is \textit{almost} a bialgebra module for $\Ht^\sbt(\Ml,\varphi)$. 

Recall that $M$ is a bialgebra module for bialgebra $A$ if it has an action and coaction with
$$\Delta_M(a \cdot m) \ = \ \Delta(a)\cdot \Delta_M(m)$$
where on the right hand side, we used the braiding of the background category to permute the second and third factors of $\Delta(a)\otimes \Delta_M(m)$ before multiplying. 

Instead, we get a \textit{twisted bialgebra module} structure: the following commutes
\begin{center}
  \begin{tikzcd}[row sep = {-30pt,between origins}, column sep = {30pt}]
  A \otimes M \ar[r,"m_M"] \ar[d,"\Delta \otimes \Delta_M"]   &    M \ar[dd,"\Delta_M"] \\ 
  A \otimes A \otimes A \otimes M \ar[d,"R_{2\tau(3)}\cdot R_{23}\cdot \sigma_{23}"]  &\\
  A \otimes A \otimes A \otimes M \ar[r, "m \otimes m_M"]  & M \otimes M
  \end{tikzcd}
  \end{center}
where $R_{2\tau(3)}=(\id\times \tau)R_{23}(\id\times \tau)$. 

\begin{prop}
  The CoHA and COHM satisfy the above, for $R=S_{23}/S_{32}$.
\end{prop}

\begin{proof}[First proof: algebraic]
  The Theorem will follow from applying torus localisation the diagram 
  \begin{equation}
     \label{fig:BigDiagramStack}
  \begin{tikzcd}[row sep = {30pt,between origins}, column sep = {65pt, between origins}]
   & &\SES_{12} \times\SES_{34}^\tau \ar[d,dashed,"i"] \ar[rdd,bend left = 30,"p \times p^{\OSpt}"] \ar[lldd, bend right = 20,"q \times q^{\OSpt}"']   & & \\
   & & \SES^\tau \times_{\Ml^\tau}(\Ml\times \Ml^\tau) \ar[rd,"\overbar{p}^{\OSpt}"] \ar[ld] & & \\
  \Ml^2_{12} \times(\Ml_3\times \Ml^\tau_4) \ar[d,"\sigma_{23}\cdot (\oplus \times \oplus_{\OSpt})"']  & \SES^\tau \ar[rd,"p^{\OSpt}"] \ar[ld] & & \Ml\times \Ml^\tau \ar[ld,"\oplus_{\OSpt}"] & \\ 
  \Ml \times\Ml^\tau & & \Ml^\tau & &  
  \end{tikzcd}
  \end{equation}
  As an aside for the sake of concreteness, the pullback parameterises extension squares
  \begin{equation}\label{fig:ExtensionSquareProofCoHAVA}
    \begin{tikzcd}[row sep = {30pt,between origins}, column sep = {10pt}]
    a \ar[d,equals] \ar[r]  & \ker \ar[r] \ar[d] & b\ar[d]&[10pt] \\ 
    a\ar[r]  &(e_1 \oplus e_2 \oplus e_1^*)\ar[r] \ar[d]  &\coker\ar[d]  \\
    0  & a^* \ar[r,equals] &a^* &
    \end{tikzcd}
  \end{equation}
  where $e_2 \simeq e_2^*$ and $b\simeq b^*$, plus the data of the split factors $e_1,e_2$. Its images under the above various maps are 
  \begin{center}
  \begin{tikzcd}[row sep = {30pt,between origins}, column sep = {45pt, between origins}]
    & & \textup{\eqref{fig:ExtensionSquareProofCoHAVA}} \ar[ld,|->] \ar[rd,|->]   & \\
   & \textup{\eqref{fig:ExtensionSquareProofCoHAVA}} \ar[ld,|->] \ar[rd,|->]  & &(e_1,e_2) \ar[ld,|->] \\ 
   (a,b)& & e_1 \oplus e_2 \oplus e_1^*&
  \end{tikzcd}
  \end{center}
  where under the left map the split factors are forgotten.  As a second aside, the split locus $\SES \times \SES^\tau$ then parametrises pairs
  \begin{equation}\label{fig:SplitExtensionSquareProofCoHAVA}
    \begin{tikzcd}[row sep = {30pt,between origins}, column sep = {10pt}]
    &&a_2 \ar[d,equals] \ar[r]  & \ker \ar[r] \ar[d] & b_2\ar[d]&[10pt] \\ 
    a_1 \to e_1 \to b_1 &&a_2\ar[r]  & e_2 \ar[r] \ar[d]  &\coker\ar[d]  \\
    &&0  & a^*_2 \ar[r,equals] &a^*_2 &
    \end{tikzcd}
  \end{equation}
  Both asides will only be useful in the proof insofar as they help us compute the normal bundle to $p \times p^{\OSpt}$. Continuing on - using torus localisation (see \cite{La3} for full details), the above diagram reduces to 
  \begin{equation}
    \label{fig:LocalisedBialgebraDiagram} 
    \begin{tikzcd}[row sep = {30pt,between origins}, column sep = {40pt}]
    (A_1\otimes A_2)\otimes (A_3\otimes M_4)_{\loc}\ar[r,"U",dashed]  &[55pt] A_1\otimes A_2\otimes A_3\otimes M_{4,\loc} \ar[r,"m \otimes m_{\OSpt}"]  &[-10pt] A\otimes M_{\loc} \\ 
    (A\otimes A)\otimes (A\otimes M)_{\loc}\ar[u,"e\otimes e^{\OSpt}"]  \ar[r," \frac{1}{e(\Nt_i)_{\loc}}\cdot \sigma_{23}"]  & A\otimes A\otimes A\otimes M_{\loc} \ar[r,"m \otimes m_{\OSpt}"] \ar[u,"(\oplus\times\, \oplus_{\OSpt})^*e^{\OSpt}"] & A\otimes M_{\loc} \ar[u,"e^{\OSpt}"] \\
    A\otimes M \ar[rr] \ar[u,"\oplus^*\otimes\, \oplus_{\OSpt}^*"]  & & M \ar[u,"\oplus_{\OSpt}^*"]  
    \end{tikzcd}
  \end{equation}
The normal complex is computed as a fibre
\begin{equation}\label{eqn:NormalComplexComputationProof}
  (s\times s^{\OSpt})^*\Nt_i \ = \ (s\times s^{\OSpt})^*\fib \left( \Nt_{p \times p^{\OSpt}} \to i^* \Nt_{\overbar{p}^{\OSpt}}\right) \ = \ \fib\left(\Nt_s\boxplus\Nt_{s_3^\tau} \to \sigma_{23}^*(\oplus\times\,\oplus_{\OSpt}^*)\Nt_{s_3^\tau}\right)[1]
\end{equation}
 and so its localised Euler class is a quotient 
 \begin{align*}
  \frac{1}{e(\Nt_i)_{\loc}} &\ \stackrel{\ddagger}{=} \ \frac{e(\Nt_s\boxplus\Nt_{s_3^\tau})_{\loc}}{e(\sigma_{23}^*(\oplus\times\,\oplus_{\OSpt}^*)\Nt_{s_3^\tau})_{\loc}} \ \stackrel{\star}{=} \ \frac{S_{12}T_{34}}{\sigma_{23}^*\left(S_{13}T_{14}S_{1\tau(3)} S_{23}T_{24}S_{2\tau(3)}S_{1\tau(2)}\right)}\\
  & \ \stackrel{\dagger}{=} \ \frac{S_{12}T_{34}}{S_{12}T_{14}S_{1\tau(2)} S_{32}T_{34}S_{3\tau(2)}S_{1\tau(3)}} \\
  & \ = \  \frac{1}{T_{14}S_{1\tau(2)} S_{32}S_{3\tau(2)}S_{1\tau(3)}} .
 \end{align*}
We thus conclude that the diagram commutes if and only if 
\begin{align*}
  U \cdot S_{12}T_{34} &\ = \  \left(S_{13}T_{14}S_{1\tau(3)} S_{23}T_{24}S_{2\tau(3)}S_{1\tau(2)}\right)\cdot  \frac{1}{T_{14}S_{1\tau(2)} S_{32}S_{3\tau(2)}S_{1\tau(3)}}\\
  & \ = \    \left(S_{13} S_{23}T_{24}S_{2\tau(3)}\right)\cdot  \frac{1}{ S_{32}S_{3\tau(2)}}\\
  & \ = \   \frac{S_{23}}{S_{32}}\cdot \frac{S_{2\tau(3)}}{ S_{3\tau(2)}} \cdot  S_{13}T_{24} 
\end{align*}
For instance, 
$$U  \ = \  \frac{S_{23}}{S_{32}}\cdot \frac{S_{2\tau(3)}}{ S_{3\tau(2)}} \cdot \sigma_{23}$$ suffices, which is precisely the Theorem. 
 \end{proof}

\begin{proof}[Second proof: pictoral] As in the proof of Theorem \ref{thm:YD}, the localised Euler class is 
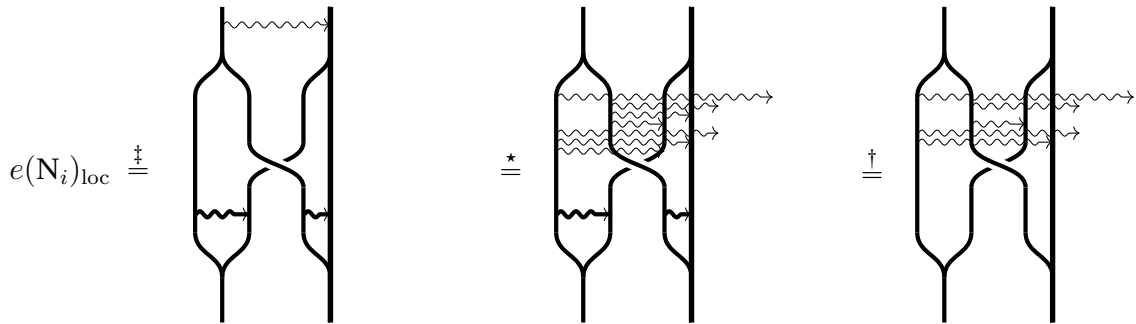
\begin{center}
\begin{tikzpicture}[scale = 1.2]
  \newcommand{\mydrawing}{%
  \draw[black,ultra thick] (0,1) .. controls +(0,0.25) and +(0,-0.25) .. (0+0.3,1.5);
  \draw[black,ultra thick] (0+0.3,1.5) -- (0+0.3,2);
  \draw[black,ultra thick] (0+0.3,2.5) -- (0+0.3,3);
  \draw[black,ultra thick] (+0.3,3) .. controls +(0,0.25) and +(0,-0.25) .. (0,3.5);
  \draw[black,ultra thick] (0,3.5) -- (0,4);

  \draw[black,ultra thick] (1.2+0,0.5) -- (1.2+0,1);
  \draw[black,ultra thick] (1.2,1) .. controls +(0,0.25) and +(0,-0.25) .. (1.2-0.3,1.5);
  \draw[black,ultra thick] (1.2,1) .. controls +(0,0.25) and +(0,-0.25) .. (1.2+0.3,1.5);
  \draw[black,ultra thick] (1.2-0.3,1.5) -- (1.2-0.3,2);
  \draw[black,ultra thick] (1.2-0.3,2.5) -- (1.2-0.3,3);
  \draw[black,ultra thick] (1.2+0.3,1.5) -- (1.2+0.3,3);
  \draw[black,ultra thick] (1.2-0.3,2) .. controls +(0,0.25) and +(0,-0.25) .. (0+0.3,2.5);
  \draw[white,line width = 5pt] (+0.3,2).. controls +(0,0.25) and +(0,-0.25) .. (1.2-0.3,2.5);
  \draw[black,ultra thick] (+0.3,2).. controls +(0,0.25) and +(0,-0.25) .. (1.2-0.3,2.5);
  \draw[black,ultra thick] (1.2+0.3,3) .. controls +(0,0.25) and +(0,-0.25) .. (1.2,3.5);
  \draw[black,ultra thick] (1.2-0.3,3) .. controls +(0,0.25) and +(0,-0.25) .. (1.2,3.5);
  \draw[black,ultra thick] (1.2,3.5) -- (1.2,4);

  \draw[black,line width = 2pt]  (0,0.5) -- (0,4);
  }

  \mydrawing

  \begin{scope} 
   [xscale=-1] 
    \mydrawing
   \end{scope}
  
   \node[left] at (-1.8,2.25) {$e(\Nt_i)_{\loc} \ \stackrel{\ddagger}{=} \ $};

   \fill[white,opacity=0.9] (0,0.4) rectangle (1.6,4.1);
   \draw[black,line width = 2pt]  (0,0.5) -- (0,4);

   \draw[->,decorate,decoration={snake,amplitude=.4mm,segment length=2mm,post length=1mm}] (-1.2,3.8) -- (0,3.8);
    
   \draw[-,decorate,decoration={snake,amplitude=.4mm,segment length=2mm,post length=1mm},ultra thick]  (-1.2-0.3,1.7)  -- (-1.2+0.3,1.7);
   \draw[->,decorate,decoration={snake,amplitude=.4mm,segment length=2mm,post length=1mm}]   (-1.2-0.3,1.7)  -- (-1.2+0.3,1.7);

   \draw[-,decorate,decoration={snake,amplitude=.4mm,segment length=2mm,post length=1mm},ultra thick]  (0-0.3,1.7)  -- (0,1.7);
   \draw[->,decorate,decoration={snake,amplitude=.4mm,segment length=2mm,post length=1mm}]  (0-0.3,1.7)  -- (0,1.7);;

   \node[] at (2,2.25) {$\stackrel{\star}{=}$};
   \node[] at (6,2.25) {$\stackrel{\dagger}{=}$};

   \begin{scope} 
    [xshift=4cm] 

    \mydrawing
  \begin{scope} 
   [xscale=-1] 
    \mydrawing
   \end{scope}

   \fill[white,opacity=0.9] (0,0.4) rectangle (1.6,4.1);
   \draw[black,line width = 2pt]  (0,0.5) -- (0,4);

   \draw[->,decorate,decoration={snake,amplitude=.4mm,segment length=2mm,post length=1mm}] (-1.2-0.3,2.4) -- (-0.3,2.4);
    \draw[->,decorate,decoration={snake,amplitude=.4mm,segment length=2mm,post length=1mm}] (-1.2-0.3,2.5) -- (0,2.5);
    \draw[->,decorate,decoration={snake,amplitude=.4mm,segment length=2mm,post length=1mm}] (-1.2-0.3,2.6) -- (0.3,2.6);
    \draw[->,decorate,decoration={snake,amplitude=.4mm,segment length=2mm,post length=1mm}] (-1.2+0.3,2.7) -- (-0.3,2.7);
    \draw[->,decorate,decoration={snake,amplitude=.4mm,segment length=2mm,post length=1mm}] (-1.2+0.3,2.8) -- (0,2.8);
    \draw[->,decorate,decoration={snake,amplitude=.4mm,segment length=2mm,post length=1mm}] (-1.2+0.3,2.9) -- (0.3,2.9);
    \draw[->,decorate,decoration={snake,amplitude=.4mm,segment length=2mm,post length=1mm}] (-1.2-0.3,3.0) -- (1.2-0.3,3.0);
    
   \draw[-,decorate,decoration={snake,amplitude=.4mm,segment length=2mm,post length=1mm},ultra thick]  (-1.2-0.3,1.7)  -- (-1.2+0.3,1.7);
   \draw[->,decorate,decoration={snake,amplitude=.4mm,segment length=2mm,post length=1mm}]   (-1.2-0.3,1.7)  -- (-1.2+0.3,1.7);

   \draw[-,decorate,decoration={snake,amplitude=.4mm,segment length=2mm,post length=1mm},ultra thick]  (0-0.3,1.7)  -- (0,1.7);
   \draw[->,decorate,decoration={snake,amplitude=.4mm,segment length=2mm,post length=1mm}]  (0-0.3,1.7)  -- (0,1.7);;

    \end{scope}

    \begin{scope} 
      [xshift=8cm] 
  
      \mydrawing
    \begin{scope} 
     [xscale=-1] 
      \mydrawing
     \end{scope}

   \fill[white,opacity=0.9] (0,0.4) rectangle (1.6,4.1);
   \draw[black,line width = 2pt]  (0,0.5) -- (0,4);

     \draw[->,decorate,decoration={snake,amplitude=.4mm,segment length=2mm,post length=1mm}] (-1.2-0.3,2.5) -- (0,2.5);
     \draw[->,decorate,decoration={snake,amplitude=.4mm,segment length=2mm,post length=1mm}] (-1.2-0.3,2.6) -- (0.3,2.6);
     \draw[->,decorate,decoration={snake,amplitude=.4mm,segment length=2mm,post length=1mm}] (-1.2+0.3,2.7) -- (-0.3,2.7);
     \draw[->,decorate,decoration={snake,amplitude=.4mm,segment length=2mm,post length=1mm}] (-1.2+0.3,2.9) -- (0.3,2.9);
     \draw[->,decorate,decoration={snake,amplitude=.4mm,segment length=2mm,post length=1mm}] (-1.2-0.3,3.0) -- (1.2-0.3,3.0);
      \end{scope}
 \end{tikzpicture}
 \end{center}

Likewise, the Joyce twists in the diagram induce the cohomology class
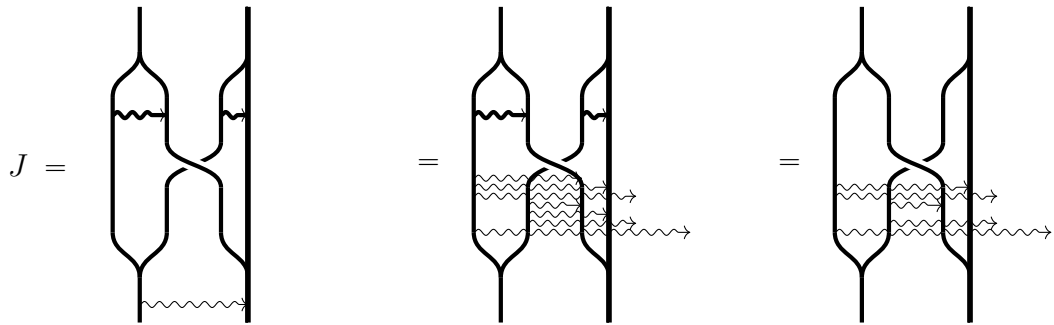
\begin{center}
  \begin{tikzpicture}[scale = 1.2]
    \newcommand{\mydrawing}{%
    \draw[black,ultra thick] (0,1) .. controls +(0,0.25) and +(0,-0.25) .. (0+0.3,1.5);
    \draw[black,ultra thick] (0+0.3,1.5) -- (0+0.3,2);
    \draw[black,ultra thick] (0+0.3,2.5) -- (0+0.3,3);
    \draw[black,ultra thick] (+0.3,3) .. controls +(0,0.25) and +(0,-0.25) .. (0,3.5);
    \draw[black,ultra thick] (0,3.5) -- (0,4);

    \draw[black,ultra thick] (1.2+0,0.5) -- (1.2+0,1);
    \draw[black,ultra thick] (1.2,1) .. controls +(0,0.25) and +(0,-0.25) .. (1.2-0.3,1.5);
    \draw[black,ultra thick] (1.2,1) .. controls +(0,0.25) and +(0,-0.25) .. (1.2+0.3,1.5);
    \draw[black,ultra thick] (1.2-0.3,1.5) -- (1.2-0.3,2);
    \draw[black,ultra thick] (1.2-0.3,2.5) -- (1.2-0.3,3);
    \draw[black,ultra thick] (1.2+0.3,1.5) -- (1.2+0.3,3);
    \draw[black,ultra thick] (1.2-0.3,2) .. controls +(0,0.25) and +(0,-0.25) .. (0+0.3,2.5);
    \draw[white,line width = 5pt] (+0.3,2).. controls +(0,0.25) and +(0,-0.25) .. (1.2-0.3,2.5);
    \draw[black,ultra thick] (+0.3,2).. controls +(0,0.25) and +(0,-0.25) .. (1.2-0.3,2.5);
    \draw[black,ultra thick] (1.2+0.3,3) .. controls +(0,0.25) and +(0,-0.25) .. (1.2,3.5);
    \draw[black,ultra thick] (1.2-0.3,3) .. controls +(0,0.25) and +(0,-0.25) .. (1.2,3.5);
    \draw[black,ultra thick] (1.2,3.5) -- (1.2,4);

    \draw[black,line width = 2pt]  (0,0.5) -- (0,4);
    }
  
    \mydrawing
  
    \begin{scope} 
     [xscale=-1] 
      \mydrawing
     \end{scope}
    
     \node[left] at (-1.8,2.25) {$J \ = \ $};
     \node[] at (2,2.25) {$=$};
     \node[] at (6,2.25) {$=$};

     \begin{scope} 
      [xshift=4cm] 
      \mydrawing
    \begin{scope} 
     [xscale=-1] 
      \mydrawing
     \end{scope}
      \end{scope}
  
      \begin{scope} 
        [xshift=8cm] 
        \mydrawing
      \begin{scope} 
       [xscale=-1] 
        \mydrawing
       \end{scope}
        \end{scope}

   \fill[white,opacity=0.9] (0,0.4) rectangle (1.6,4.1);
   \draw[black,line width = 2pt]  (0,0.5) -- (0,4);

    \begin{scope} 
     [yscale=-1,yshift=-4.5cm] 
   \draw[->,decorate,decoration={snake,amplitude=.4mm,segment length=2mm,post length=1mm}] (-1.2,3.8) -- (0,3.8);
    
   \draw[-,decorate,decoration={snake,amplitude=.4mm,segment length=2mm,post length=1mm},ultra thick]  (-1.2-0.3,1.7)  -- (-1.2+0.3,1.7);
   \draw[->,decorate,decoration={snake,amplitude=.4mm,segment length=2mm,post length=1mm}]   (-1.2-0.3,1.7)  -- (-1.2+0.3,1.7);
   \draw[-,decorate,decoration={snake,amplitude=.4mm,segment length=2mm,post length=1mm},ultra thick]  (0-0.3,1.7)  -- (0,1.7);
   \draw[->,decorate,decoration={snake,amplitude=.4mm,segment length=2mm,post length=1mm}]  (0-0.3,1.7)  -- (0,1.7);;

      \begin{scope} 
       [xshift=4cm]

   \fill[white,opacity=0.9] (0,0.4) rectangle (1.6,4.1);
   \draw[black,line width = 2pt]  (0,0.5) -- (0,4);

       \draw[->,decorate,decoration={snake,amplitude=.4mm,segment length=2mm,post length=1mm}] (-1.2-0.3,2.4) -- (-0.3,2.4);
       \draw[->,decorate,decoration={snake,amplitude=.4mm,segment length=2mm,post length=1mm}] (-1.2-0.3,2.5) -- (0,2.5);
       \draw[->,decorate,decoration={snake,amplitude=.4mm,segment length=2mm,post length=1mm}] (-1.2-0.3,2.6) -- (0.3,2.6);
       \draw[->,decorate,decoration={snake,amplitude=.4mm,segment length=2mm,post length=1mm}] (-1.2+0.3,2.7) -- (-0.3,2.7);
       \draw[->,decorate,decoration={snake,amplitude=.4mm,segment length=2mm,post length=1mm}] (-1.2+0.3,2.8) -- (0,2.8);
       \draw[->,decorate,decoration={snake,amplitude=.4mm,segment length=2mm,post length=1mm}] (-1.2+0.3,2.9) -- (0.3,2.9);
       \draw[->,decorate,decoration={snake,amplitude=.4mm,segment length=2mm,post length=1mm}] (-1.2-0.3,3.0) -- (1.2-0.3,3.0);
       
      \draw[-,decorate,decoration={snake,amplitude=.4mm,segment length=2mm,post length=1mm},ultra thick]  (-1.2-0.3,1.7)  -- (-1.2+0.3,1.7);
      \draw[->,decorate,decoration={snake,amplitude=.4mm,segment length=2mm,post length=1mm}]   (-1.2-0.3,1.7)  -- (-1.2+0.3,1.7);
   
      \draw[-,decorate,decoration={snake,amplitude=.4mm,segment length=2mm,post length=1mm},ultra thick]  (0-0.3,1.7)  -- (0,1.7);
      \draw[->,decorate,decoration={snake,amplitude=.4mm,segment length=2mm,post length=1mm}]  (0-0.3,1.7)  -- (0,1.7);;
       \end{scope}

       \begin{scope} 
        [xshift=8cm] 

   \fill[white,opacity=0.9] (0,0.4) rectangle (1.6,4.1);
   \draw[black,line width = 2pt]  (0,0.5) -- (0,4);

     \draw[->,decorate,decoration={snake,amplitude=.4mm,segment length=2mm,post length=1mm}] (-1.2-0.3,2.5) -- (0,2.5);
     \draw[->,decorate,decoration={snake,amplitude=.4mm,segment length=2mm,post length=1mm}] (-1.2-0.3,2.6) -- (0.3,2.6);
     \draw[->,decorate,decoration={snake,amplitude=.4mm,segment length=2mm,post length=1mm}] (-1.2+0.3,2.7) -- (-0.3,2.7);
     \draw[->,decorate,decoration={snake,amplitude=.4mm,segment length=2mm,post length=1mm}] (-1.2+0.3,2.9) -- (0.3,2.9);
     \draw[->,decorate,decoration={snake,amplitude=.4mm,segment length=2mm,post length=1mm}] (-1.2-0.3,3.0) -- (1.2-0.3,3.0);
        \end{scope}

     \end{scope}
   \end{tikzpicture}
   \end{center}
Notice that this is the rotation of $e(\Nt_i)_{\loc}$ about the vertical axis, and so the reason the quotient does not fully cancel is because of the braidings. Thus the quotient is 

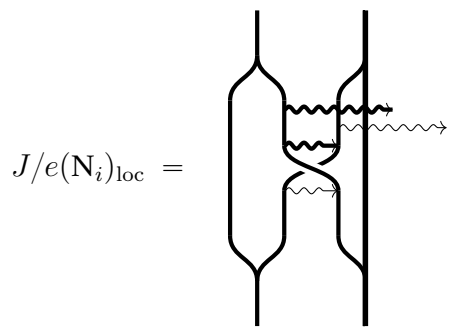
\begin{center}
  \begin{tikzpicture}[scale = 1.2]
    \newcommand{\mydrawing}{%
    \draw[black,ultra thick] (0,1) .. controls +(0,0.25) and +(0,-0.25) .. (0+0.3,1.5);
    \draw[black,ultra thick] (0+0.3,1.5) -- (0+0.3,2);
    \draw[black,ultra thick] (0+0.3,2.5) -- (0+0.3,3);
    \draw[black,ultra thick] (+0.3,3) .. controls +(0,0.25) and +(0,-0.25) .. (0,3.5);
    \draw[black,ultra thick] (0,3.5) -- (0,4);

    \draw[black,ultra thick] (1.2+0,0.5) -- (1.2+0,1);
    \draw[black,ultra thick] (1.2,1) .. controls +(0,0.25) and +(0,-0.25) .. (1.2-0.3,1.5);
    \draw[black,ultra thick] (1.2,1) .. controls +(0,0.25) and +(0,-0.25) .. (1.2+0.3,1.5);
    \draw[black,ultra thick] (1.2-0.3,1.5) -- (1.2-0.3,2);
    \draw[black,ultra thick] (1.2-0.3,2.5) -- (1.2-0.3,3);
    \draw[black,ultra thick] (1.2+0.3,1.5) -- (1.2+0.3,3);
    \draw[black,ultra thick] (1.2-0.3,2) .. controls +(0,0.25) and +(0,-0.25) .. (0+0.3,2.5);
    \draw[white,line width = 5pt] (+0.3,2).. controls +(0,0.25) and +(0,-0.25) .. (1.2-0.3,2.5);
    \draw[black,ultra thick] (+0.3,2).. controls +(0,0.25) and +(0,-0.25) .. (1.2-0.3,2.5);
    \draw[black,ultra thick] (1.2+0.3,3) .. controls +(0,0.25) and +(0,-0.25) .. (1.2,3.5);
    \draw[black,ultra thick] (1.2-0.3,3) .. controls +(0,0.25) and +(0,-0.25) .. (1.2,3.5);
    \draw[black,ultra thick] (1.2,3.5) -- (1.2,4);

    \draw[black,line width = 2pt]  (0,0.5) -- (0,4);
    }
  
    \mydrawing
  
    \begin{scope} 
     [xscale=-1] 
      \mydrawing
     \end{scope}
    
     \node[left] at (-1.8,2.25) {$J/e(\Nt_i)_{\loc} \ = \ $};

   \fill[white,opacity=0.9] (0,0.4) rectangle (1.6,4.1);
   \draw[black,line width = 2pt]  (0,0.5) -- (0,4);

     \draw[-,ultra thick,decorate,decoration={snake,amplitude=.4mm,segment length=2mm,post length=1mm}] (-1.2+0.3,2.5) -- (-0.3,2.5);
     \draw[-,ultra thick,decorate,decoration={snake,amplitude=.4mm,segment length=2mm,post length=1mm}] (-1.2+0.3,2.9) -- (0.3,2.9);
     \draw[->,decorate,decoration={snake,amplitude=.4mm,segment length=2mm,post length=1mm}] (-1.2+0.3,2.5) -- (-0.3,2.5);
     \draw[->,decorate,decoration={snake,amplitude=.4mm,segment length=2mm,post length=1mm}] (-1.2+0.3,2.9) -- (0.3,2.9);

     \draw[->,decorate,decoration={snake,amplitude=.4mm,segment length=2mm,post length=1mm}]  (-0.3,2.7) -- (1.2-0.3,2.7);

     \draw[->,decorate,decoration={snake,amplitude=.4mm,segment length=2mm,post length=1mm}] (-1.2+0.3,2.0) -- (-0.3,2.0);

   \end{tikzpicture}
   \end{center}
which is precisely the Yang-Baxter matrix $S_{23}S_{2\tau(3)}/S_{32}S_{3\tau(2)}$, if we number the four strands going left to right and below the swap. 

\end{proof}

\subsubsection{Remark} \label{sssec:FinalRemarksBialg} For completeness, here is the analogous proof from \cite{La1,Li} that $\Ht^\sbt(\Ml,\varphi)$ forms a vertex bialgebra. In that case for the associated diagram we have 
\begin{center}
  \begin{tikzpicture}[scale = 1.2]

  \newcommand{\mydrawing}{%
  \draw[black,ultra thick] (0,0.5) -- (0,1);
  \draw[black,ultra thick] (0,1) .. controls +(0,0.25) and +(0,-0.25) .. (0-0.3,1.5);
  \draw[black,ultra thick] (0,1) .. controls +(0,0.25) and +(0,-0.25) .. (0+0.3,1.5);
  \draw[black,ultra thick] (0-0.3,1.5) -- (0-0.3,3);
  \draw[black,ultra thick] (0+0.3,1.5) -- (0+0.3,2);
  \draw[black,ultra thick] (0+0.3,2.5) -- (0+0.3,3);
  \draw[black,ultra thick] (+0.3,3) .. controls +(0,0.25) and +(0,-0.25) .. (0,3.5);
  \draw[black,ultra thick] (-0.3,3) .. controls +(0,0.25) and +(0,-0.25) .. (0,3.5);
  \draw[black,ultra thick] (0,3.5) -- (0,4);

  \draw[black,ultra thick] (1.2+0,0.5) -- (1.2+0,1);
  \draw[black,ultra thick] (1.2,1) .. controls +(0,0.25) and +(0,-0.25) .. (1.2-0.3,1.5);
  \draw[black,ultra thick] (1.2,1) .. controls +(0,0.25) and +(0,-0.25) .. (1.2+0.3,1.5);
  \draw[black,ultra thick] (1.2-0.3,1.5) -- (1.2-0.3,2);
  \draw[black,ultra thick] (1.2-0.3,2.5) -- (1.2-0.3,3);
  \draw[black,ultra thick] (1.2+0.3,1.5) -- (1.2+0.3,3);
  \draw[black,ultra thick] (1.2-0.3,2) .. controls +(0,0.25) and +(0,-0.25) .. (0+0.3,2.5);
  \draw[white, line width = 5pt] (+0.3,2).. controls +(0,0.25) and +(0,-0.25) .. (1.2-0.3,2.5);
  \draw[black,ultra thick] (+0.3,2).. controls +(0,0.25) and +(0,-0.25) .. (1.2-0.3,2.5);
  \draw[black,ultra thick] (1.2+0.3,3) .. controls +(0,0.25) and +(0,-0.25) .. (1.2,3.5);
  \draw[black,ultra thick] (1.2-0.3,3) .. controls +(0,0.25) and +(0,-0.25) .. (1.2,3.5);
  \draw[black,ultra thick] (1.2,3.5) -- (1.2,4);
  } 
  
  \mydrawing

  \node[white] at (-2.0,2.25) {a};
  \node[left] at (-0.8,2.25) {$e(\Nt_i)_{\loc} \ = \ $};
  \node[] at (2.1,2.25) {$=$};
  \node[] at (5.2,2.25) {$=$};

    \draw[->,decorate,decoration={snake,amplitude=.4mm,segment length=2mm,post length=1mm}] (0,3.8) -- (1.2,3.8);
    
    \draw[-,decorate,decoration={snake,amplitude=.4mm,segment length=2mm,post length=1mm},ultra thick]  (1.2-0.3,1.7)  -- (1.2+0.3,1.7);
    \draw[-,decorate,decoration={snake,amplitude=.4mm,segment length=2mm,post length=1mm},ultra thick]  (0-0.3,1.7)  -- (0+0.3,1.7);
    \draw[->,decorate,decoration={snake,amplitude=.4mm,segment length=2mm,post length=1mm}]  (1.2-0.3,1.7)  -- (1.2+0.3,1.7);
    \draw[->,decorate,decoration={snake,amplitude=.4mm,segment length=2mm,post length=1mm}]  (0-0.3,1.7)  -- (0+0.3,1.7);

    \begin{scope} 
      [xshift=3cm]
      \mydrawing

    \draw[-,decorate,decoration={snake,amplitude=.4mm,segment length=2mm,post length=1mm},ultra thick]  (1.2-0.3,1.7)  -- (1.2+0.3,1.7);
    \draw[-,decorate,decoration={snake,amplitude=.4mm,segment length=2mm,post length=1mm},ultra thick]  (0-0.3,1.7)  -- (0+0.3,1.7);
    \draw[->,decorate,decoration={snake,amplitude=.4mm,segment length=2mm,post length=1mm}]  (1.2-0.3,1.7)  -- (1.2+0.3,1.7);
    \draw[->,decorate,decoration={snake,amplitude=.4mm,segment length=2mm,post length=1mm}]  (0-0.3,1.7)  -- (0+0.3,1.7);
    
    \draw[->,decorate,decoration={snake,amplitude=.4mm,segment length=2mm,post length=1mm}]  (+0.3,2.6)  -- (1.2-0.3,2.6);
    \draw[->,decorate,decoration={snake,amplitude=.4mm,segment length=2mm,post length=1mm}]  (-0.3,2.7)  -- (1.2+0.3,2.7);
    \draw[->,decorate,decoration={snake,amplitude=.4mm,segment length=2mm,post length=1mm}]  (0+0.3,2.8)  -- (1.2+0.3,2.8);
    \draw[->,decorate,decoration={snake,amplitude=.4mm,segment length=2mm,post length=1mm}]  (0-0.3,2.9)  -- (1.2-0.3,2.9);
     \end{scope}

     \begin{scope} 
      [xshift=6cm]
      \mydrawing  

    \draw[->,decorate,decoration={snake,amplitude=.4mm,segment length=2mm,post length=1mm}]  (+0.3,2.6)  -- (1.2-0.3,2.6);
    \draw[->,decorate,decoration={snake,amplitude=.4mm,segment length=2mm,post length=1mm}]  (0-0.3,2.7)  -- (1.2+0.3,2.7);

     \end{scope}

  \end{tikzpicture}
  \end{center}
  and 
  \begin{center}
    \begin{tikzpicture}[scale = 1.2,yscale=-1]
  
      \newcommand{\mydrawing}{%
      \draw[black,ultra thick] (0,0.5) -- (0,1);
      \draw[black,ultra thick] (0,1) .. controls +(0,0.25) and +(0,-0.25) .. (0-0.3,1.5);
      \draw[black,ultra thick] (0,1) .. controls +(0,0.25) and +(0,-0.25) .. (0+0.3,1.5);
      \draw[black,ultra thick] (0-0.3,1.5) -- (0-0.3,3);
      \draw[black,ultra thick] (0+0.3,1.5) -- (0+0.3,2);
      \draw[black,ultra thick] (0+0.3,2.5) -- (0+0.3,3);
      \draw[black,ultra thick] (+0.3,3) .. controls +(0,0.25) and +(0,-0.25) .. (0,3.5);
      \draw[black,ultra thick] (-0.3,3) .. controls +(0,0.25) and +(0,-0.25) .. (0,3.5);
      \draw[black,ultra thick] (0,3.5) -- (0,4);

      \draw[black,ultra thick] (1.2+0,0.5) -- (1.2+0,1);
      \draw[black,ultra thick] (1.2,1) .. controls +(0,0.25) and +(0,-0.25) .. (1.2-0.3,1.5);
      \draw[black,ultra thick] (1.2,1) .. controls +(0,0.25) and +(0,-0.25) .. (1.2+0.3,1.5);
      \draw[black,ultra thick] (1.2-0.3,1.5) -- (1.2-0.3,2);
      \draw[black,ultra thick] (1.2-0.3,2.5) -- (1.2-0.3,3);
      \draw[black,ultra thick] (1.2+0.3,1.5) -- (1.2+0.3,3);
      \draw[black,ultra thick] (1.2-0.3,2) .. controls +(0,0.25) and +(0,-0.25) .. (0+0.3,2.5);
      \draw[white, line width = 5pt] (+0.3,2).. controls +(0,0.25) and +(0,-0.25) .. (1.2-0.3,2.5);
      \draw[black,ultra thick] (+0.3,2).. controls +(0,0.25) and +(0,-0.25) .. (1.2-0.3,2.5);
      \draw[black,ultra thick] (1.2+0.3,3) .. controls +(0,0.25) and +(0,-0.25) .. (1.2,3.5);
      \draw[black,ultra thick] (1.2-0.3,3) .. controls +(0,0.25) and +(0,-0.25) .. (1.2,3.5);
      \draw[black,ultra thick] (1.2,3.5) -- (1.2,4);
      } 

      \mydrawing

      \node[white] at (-2.0,2.25) {a};
    \node[left] at (-0.8,2.25) {$J \ = \ $};
    \node[] at (2.1,2.25) {$=$};
    \node[] at (5.2,2.25) {$=$};

      \draw[->,decorate,decoration={snake,amplitude=.4mm,segment length=2mm,post length=1mm}] (0,3.8) -- (1.2,3.8);
      
      \draw[-,decorate,decoration={snake,amplitude=.4mm,segment length=2mm,post length=1mm},ultra thick]  (1.2-0.3,1.7)  -- (1.2+0.3,1.7);
      \draw[-,decorate,decoration={snake,amplitude=.4mm,segment length=2mm,post length=1mm},ultra thick]  (0-0.3,1.7)  -- (0+0.3,1.7);
      \draw[->,decorate,decoration={snake,amplitude=.4mm,segment length=2mm,post length=1mm}]  (1.2-0.3,1.7)  -- (1.2+0.3,1.7);
      \draw[->,decorate,decoration={snake,amplitude=.4mm,segment length=2mm,post length=1mm}]  (0-0.3,1.7)  -- (0+0.3,1.7);

      \begin{scope} 
        [xshift=3cm]
\mydrawing

      \draw[-,decorate,decoration={snake,amplitude=.4mm,segment length=2mm,post length=1mm},ultra thick]  (1.2-0.3,1.7)  -- (1.2+0.3,1.7);
      \draw[-,decorate,decoration={snake,amplitude=.4mm,segment length=2mm,post length=1mm},ultra thick]  (0-0.3,1.7)  -- (0+0.3,1.7);
      \draw[->,decorate,decoration={snake,amplitude=.4mm,segment length=2mm,post length=1mm}]  (1.2-0.3,1.7)  -- (1.2+0.3,1.7);
      \draw[->,decorate,decoration={snake,amplitude=.4mm,segment length=2mm,post length=1mm}]  (0-0.3,1.7)  -- (0+0.3,1.7);
      
      \draw[->,decorate,decoration={snake,amplitude=.4mm,segment length=2mm,post length=1mm}]  (+0.3,2.6)  -- (1.2-0.3,2.6);
      \draw[->,decorate,decoration={snake,amplitude=.4mm,segment length=2mm,post length=1mm}]  (-0.3,2.7)  -- (1.2+0.3,2.7);
      \draw[->,decorate,decoration={snake,amplitude=.4mm,segment length=2mm,post length=1mm}]  (0+0.3,2.8)  -- (1.2+0.3,2.8);
      \draw[->,decorate,decoration={snake,amplitude=.4mm,segment length=2mm,post length=1mm}]  (0-0.3,2.9)  -- (1.2-0.3,2.9);

       \end{scope}

       \begin{scope} 
        [xshift=6cm]
\mydrawing

      \draw[->,decorate,decoration={snake,amplitude=.4mm,segment length=2mm,post length=1mm}]  (+0.3,2.6)  -- (1.2-0.3,2.6);
      \draw[->,decorate,decoration={snake,amplitude=.4mm,segment length=2mm,post length=1mm}]  (0-0.3,2.7)  -- (1.2+0.3,2.7);

       \end{scope}
  
    \end{tikzpicture}
    \end{center}
It follows that

\begin{center}
  \begin{tikzpicture}[scale = 1.2,yscale=-1]

    \newcommand{\mydrawing}{%
    \draw[black,ultra thick] (0,0.5) -- (0,1);
    \draw[black,ultra thick] (0,1) .. controls +(0,0.25) and +(0,-0.25) .. (0-0.3,1.5);
    \draw[black,ultra thick] (0,1) .. controls +(0,0.25) and +(0,-0.25) .. (0+0.3,1.5);
    \draw[black,ultra thick] (0-0.3,1.5) -- (0-0.3,3);
    \draw[black,ultra thick] (0+0.3,1.5) -- (0+0.3,2);
    \draw[black,ultra thick] (0+0.3,2.5) -- (0+0.3,3);
    \draw[black,ultra thick] (+0.3,3) .. controls +(0,0.25) and +(0,-0.25) .. (0,3.5);
    \draw[black,ultra thick] (-0.3,3) .. controls +(0,0.25) and +(0,-0.25) .. (0,3.5);
    \draw[black,ultra thick] (0,3.5) -- (0,4);

    \draw[black,ultra thick] (1.2+0,0.5) -- (1.2+0,1);
    \draw[black,ultra thick] (1.2,1) .. controls +(0,0.25) and +(0,-0.25) .. (1.2-0.3,1.5);
    \draw[black,ultra thick] (1.2,1) .. controls +(0,0.25) and +(0,-0.25) .. (1.2+0.3,1.5);
    \draw[black,ultra thick] (1.2-0.3,1.5) -- (1.2-0.3,2);
    \draw[black,ultra thick] (1.2-0.3,2.5) -- (1.2-0.3,3);
    \draw[black,ultra thick] (1.2+0.3,1.5) -- (1.2+0.3,3);
    \draw[black,ultra thick] (1.2-0.3,2) .. controls +(0,0.25) and +(0,-0.25) .. (0+0.3,2.5);
    \draw[white, line width = 5pt] (+0.3,2).. controls +(0,0.25) and +(0,-0.25) .. (1.2-0.3,2.5);
    \draw[black,ultra thick] (+0.3,2).. controls +(0,0.25) and +(0,-0.25) .. (1.2-0.3,2.5);
    \draw[black,ultra thick] (1.2+0.3,3) .. controls +(0,0.25) and +(0,-0.25) .. (1.2,3.5);
    \draw[black,ultra thick] (1.2-0.3,3) .. controls +(0,0.25) and +(0,-0.25) .. (1.2,3.5);
    \draw[black,ultra thick] (1.2,3.5) -- (1.2,4);
    } 

    \mydrawing

  \node[left] at (-0.8,2.25) {$J/e(\Nt_i) \ = \ $};

  \draw[->,decorate,decoration={snake,amplitude=.4mm,segment length=2mm,post length=1mm}]  (+0.3,2.6)  -- (1.2-0.3,2.6);

  \draw[-,decorate,decoration={snake,amplitude=.4mm,segment length=2mm,post length=1mm},ultra thick] (+0.3,1.9) --  (1.2-0.3,1.9) ;
  \draw[->,decorate,decoration={snake,amplitude=.4mm,segment length=2mm,post length=1mm}] (+0.3,1.9) --  (1.2-0.3,1.9) ;

  \end{tikzpicture}
  \end{center}
is multiplication by the $R$-matrix $S_{23}/S_{32}$.

\subsection{Orthosymplectic-linear vertex structures} 

\subsubsection{} The above considered orthosymplectic vertex module structure arose from the Borcherds twist coming from 
\begin{center}
\begin{tikzcd}[row sep = {30pt,between origins}, column sep = {45pt, between origins}]
 &\SES \ar[rd] \ar[ld]  & & & & \SES_3^\tau\ar[rd] \ar[ld] & \\
 \Ml \times \Ml& & \Ml & & \Ml \times \Ml^\tau & & \Ml^\tau
\end{tikzcd}
\end{center}
In this section, we will define a different vertex structure by \textit{restricting} the Borcherds twist given by the Ext complex $\Nt_s=\Ext$. It will be an orthosymplectic-linear vertex algebra as opposed to just an orthosymplectic module for a vertex algebra, and is attached to 
\begin{center}
\begin{tikzcd}[row sep = {30pt,between origins}, column sep = {45pt, between origins}]
 &\SES \ar[rd] \ar[ld]  & & & & \SES\times_{\Ml^{\times 2}}{\Ml^\tau}^{\times 2}\ar[rd] \ar[ld] & &&& \SES\times_{\Ml^{\times 3}}{\Ml^\tau}^{\times 3} \ar[rd] \ar[ld] & \\
 \Ml \times \Ml& & \Ml & & \Ml \times \Ml^\tau & & \Ml^\tau & & \Ml^\tau \times \Ml^\tau && \Ml^\tau
\end{tikzcd}
\end{center}
In other words, we consider stacks parametrising short exact sequences whose terms are required to carry orthosymplectic structure (but the morphisms need not respect this).

\subsubsection{} Again, we fix background linear-orthosymplectic braided monoidal factorisation category 
$$\Al \ = \ (\Ht^\sbt(\Ml),\cup) \Md , \hspace{15mm} \Bl \ = \ (\Ht^\sbt(\Ml^\tau),\cup) \Md$$
this time using the braidings 
$$\beta \ = \ \frac{e(\sigma^*\Nt_s)}{e(\Nt_s)}\cdot \sigma, \hspace{10mm} \kappa \ = \ \frac{e((\tau\times \id)^*\Nt_{s})}{e(\Nt_{s})}\cdot (\tau \times \id), \hspace{10mm} \beta_{\OSpt} \ = \ \frac{e(\sigma^*\Nt_{s})}{e(\Nt_{s})}\cdot \sigma$$
where in each we have pulled back the localised cohomology class $e(\Nt_s)$ to $\Ml^\tau$ if necessary. 
    \begin{prop}
      The maps 
      \begin{equation}
        \tilde{\Delta} \ = \ e(\Nt_s)\cdot \oplus^*, \hspace{10mm} \tilde{\Delta}_{\GL \textup{-}\OSpt} \ = \ e(\Nt_{s})\cdot \oplus_{\GL \textup{-}\OSpt}^*, \hspace{10mm} \tilde{\Delta}_{\OSpt} \ = \ e(\Nt_{s})\cdot \oplus_{\OSpt}^* 
      \end{equation}
      define an orthosymplectic-linear localised coproduct internal to $(\Al,\Bl]$.
    \end{prop}

If we define the \textit{linear-orthosymplectic} Joyce twists as 
\begin{align*}
  \tilde{S}(z) &\ = \ e(\Nt_s,z) \ \defeq \ \sum_{k \ge 0} z^{\rank \Nt_s -k}c_k(\Nt_s),\\
  \tilde{T}(z,w) &\ = \ e(\Nt_{s_3^\tau},z,w) \ \defeq \ \sum_{k \ge 0} z^{\rank\Nt_s/2 -k}w^{\rank\Nt_s/2-l}c_{k,l}(\Nt_s)\\ 
  \tilde{S}_{\OSpt}(w_1,w_2)& \ = \ e(\Nt_s,w_1,w_2) \ \defeq \ \sum_{k \ge 0} w_1^{\Nt_s/2 -k}w_2^{\rank \Nt_2/2-l}c_{k,l}(\Nt_s)
\end{align*}
 then 

\begin{cor}
      The maps 
      \begin{align} \label{eqn:OrthosymplecticJoyceLiuVertex}
        \tilde{\Delta}(z)&\ =\ e(\Nt_s,z)\cdot \act_{1,z}^*\oplus^*,\\
        \tilde{\Delta}_{\GL \textup{-}\OSpt}(z,w)&\ =\ e(\Nt_s,z,w)\cdot (\act_{1,z}\times \act_{\Gt,2,w})^*\oplus_{\GL \textup{-}\OSpt}^*\\
        \tilde{\Delta}_{\OSpt}(w_1,w_2) & \ = \ e(\Nt_s,w_1,w_2)\cdot \act_{\Gt,1,w_1}^*\oplus^*_{\OSpt}
      \end{align}
    define a local linear-orthosymplectic vertex coalgebra structure on $\Ht^\sbt(\Ml, \varphi)$ and $\Ht^\sbt(\Ml^\tau, \varphi^\tau)$ internal to $\El$.
    \end{cor}

    \subsubsection{Example} When $\Ml=\BGL$ is the the moduli stack of finite-dimensional vector spaces and $\Ml= \BSp$, the coaction on cohomology is
    $$\oplus^*_{\GL \textup{-}\OSpt} \ : \ \Ht^\sbt(\BSp) \ \to \ \Ht^\sbt(\BSp) \otimes \Ht^\sbt(\BGL), \hspace{15mm} y_i \ \mapsto \ y_i \otimes 1 - 1 \otimes x_i^2$$
    where $y_i \in \Ht^4(\BSp_2)$ is a symplectic chern root, i.e. the Euler class of the tautological bundle $\El_{\Sp}$, and $x_i \in \Ht^2(\BGL_1)$ is a chern root. Indeed, we have 
    $$\oplus^*_{\Sp \textup{-}\GL} \El_{\Sp} \ = \ \El_{\Sp} \boxplus \left( \El \oplus \El^\vee\right).$$
    The above choice of twist then gives
    \begin{align*}
      T (z,w) & \ = \ \act_{1,z}^*\act_{2,w}^*e(\Nt_s)\\
       & \ = \act_{1,z}^*\act_{2,w}^*e(\El_{\Sp}^\vee \boxtimes(\El \oplus\El^\vee))\\
      & \ = \ \act_{1,z}^*\act_{2,w}^*\prod_{i,j} (y_i \otimes 1 - 1 \otimes x_j^2)(-y_i \otimes 1 - 1 \otimes x_j^2)\\
      & \ = \ \prod_{i,j} (\pm_1\sqrt{y}_i \otimes 1 \pm_1 w \pm_2 1 \otimes x_j \pm_2 z)
    \end{align*}
    where $\sqrt{y}_i$ is a square root in $\Ht^2(\BT_{\Sp_2})$ and 
    where we take the product over all sign choices $\pm_1,\pm_2$. Note that the poles of the above lie along $z\pm w =0$. Likewise 
    \begin{align*}
     S(z_1,z_2)& \ = \ \prod_{i,j}(x_i \otimes 1 + z_1 - 1 \otimes x_j - z_2),\\
     S_{\OSpt}(w_1,w_2) & \ = \ \prod_{i,j}(\pm_1\sqrt{y}_i \otimes 1 \pm_1 w_1 \pm_2 1 \otimes \sqrt{y}_j \pm_2 w_2).
    \end{align*}

\newpage
\section{Example: perverse-coherent sheaves} \label{sec:PervCoh}

\noindent In this section, we will consider orthosymplectic perverse-coherent sheaves on smooth algebraic surfaces. The derived duality functor 
$$\El\ \mapsto\ \Dt \El$$
on the category $D^b(\Coh(X))$ does not preserve the heart of coherent sheaves. For example, if $k_{\mathrm{pt}}$ is a skyscraper sheaf supported at a point then $\Dt k_{\mathrm{pt}} = k_{\mathrm{pt}}[\dim X]$. 

On the other hand, when $X$ is a surface with no coherent sheaves of pure dimension 1 then if we tilt the heart $\Coh(X)$ using the torsion pair $(\Tl, \Fl)$ of dimension zero sheaves to obtain the tilted heart 
\[ \PervCoh(X) = \langle \Fl, \Tl[1]\rangle \] 
then $\PervCoh(X)$ is preserved under $\Dt$.

We then consider examples of specific surfaces. To begin with, we spell out what our formulas for ADHM quivers give for framed perverse-coherent sheaves on $\Pb^2$. Next, we write down the same results for framed orthosymplectic perverse-coherent sheaves on ALE spaces. 

\subsection{Examples} 

We compute some examples of orthosymplectic objects in $D^b(\Coh(X))$ in general. 

\subsubsection{} 
For instance, if $\El$ is a vector bundle, then an orthosymplectic structure same as a fibrewise orthogonal form  
$$\El\otimes \El\ \to\ \Ol.$$
There are no orthosymplectic structures on shifted vector bundles $\El[n]$ unless $n=0$, because $\El[n]$ and $\El^*[-n]$ are not abstractly isomorphic. Instead the relevant involution in this case is 
$$\El\ \mapsto\ \El^*[2n].$$

\subsubsection{} For instance, let $\El_D$ is the restriction of a vector bundle to the zero locus $D$ of a regular function $x$. We have a free Koszul resolution 
$$\Ol\ \stackrel{x}{\to}\ \Ol\ \to\ \Ol_D.$$
Note that there cannot be an isomorphism 
$$(\Ol\ \stackrel{x}{\to}\ \Ol)\ \simeq\ (\Ol\ \stackrel{x}{\to}\ \Ol)^*\ =\ (\Ol\ \stackrel{x}{\to}\ \Ol)[1]$$
because the left hand side has tor-amplitude in $[-1,0]$, while the right hand side has tor-amplitude in $[0,1]$. By a similar argument there is no orthosymplectic structure on $\El_D$ unless $\El=0$. There are only orthosymplectic structures with respect to the involution
$$\El\ \mapsto\ \El^*[-1].$$

\subsubsection{} In codimension two, consider the common vanishing locus $V$ of $x,y$ such that they form a regular sequence. The Koszul resolution in this case is 
$$\Ol \ \stackrel{y,-x}{\rightrightarrows}\ \Ol^2\ \stackrel{x,y}{\rightrightarrows}\ \Ol\ \to\ \Ol_V$$
\begin{center}
\begin{tikzcd}[row sep = {13pt,between origins}, column sep = {20pt},arrows = std]
 & \Ol \ar[rd,"x"]& & \\ 
\Ol\ar[ru,"y"]\ar[rd,"-x"'] & \oplus &\Ol \ar[r] & \Ol_V \\
 & \Ol\ar[ru,"y"']& &
\end{tikzcd}
\end{center}
We thus have a Koszul resolution 
\begin{center}
\begin{tikzcd}[row sep = {13pt,between origins}, column sep = {20pt},arrows = std]
 & \El \ar[rd,"x"]& & \\ 
\El\ar[ru,"y"]\ar[rd,"-x"'] & \oplus &\El \ar[r] & \El_V \\
 & \El\ar[ru,"y"']& &
\end{tikzcd}
\end{center}
so that $\El_V^*[2]$ is 
\begin{center}
\begin{tikzcd}[row sep = {13pt,between origins}, column sep = {20pt},arrows = std]
 & \El^* \ar[rd,"y^*"]& \\ 
\El^*\ar[ru,"x^*"]\ar[rd,"y^*"'] & \oplus &\El^*  \\
 & \El^*\ar[ru,"-x^*"']&
\end{tikzcd}
\end{center}

If we apply a shift, the complex resolving $\Cb_0[-1]$ is concentrated in degrees $[-1,1]$, and likewise for the dual complex resolving $(\Cb_0[1])^*$, which is 
\begin{center}
\begin{tikzcd}[row sep = {10pt,between origins}, column sep = {20pt},arrows = std]
 & \Ol \ar[rd,"y"]& \\ 
\Ol\ar[ru,"x"]\ar[rd,"y"'] & &\Ol \\
 & \Ol\ar[ru,"-x"']& 
\end{tikzcd}
\end{center}
In particular, a pairing $\Cb_0[-1]\simeq (\Cb_0[-1])^*$ is the same thing as a symmetric pairing on $\Ol\oplus \Ol$ and antisymmetric pairings on the outer $\Ol$'s.

\subsubsection{Example} Let $\Al=\Coh(X)^\heartsuit$ be the category of coherent sheaves on any $X$. An orthosymplectic structure on a coherent sheaf $\El$ is the same as a fibrewise bilinear (i.e. symplectic) form 
$$ \El\otimes\El\ \to\ \Ol.$$
However, an orthosymplectic structure on $\El[1]$ is the same as a fibrewise \textit{antisymmetric} bilinear (i.e. orthogonal) form 
$$\El\wedge \El\ \to\ \Ol.$$
One can likewise write down the functor of points for $\Ml_{\Coh(X)}^{\OSpt}$.

\begin{lem}\label{defn:OSp}
  The stack of orthosymplectic coherent sheaves is the substack
  $$\Ml^{\textup{\OSpt}} \ \subseteq\ \Hom(\Ul\otimes\Ul,\Ol)$$
  cut out by the non-degenerate symmetric bilinear forms. Here, $\Ul$ is the universal coherent sheaf on $\Ml\times X$, and $\Hom$ refers to the pushforward of the Hom sheaf along $\Ml\times X\to \Ml$.
\end{lem}

\subsection{Compactification of stack of classical type bundles}

\subsubsection{Stability on vertial wall} 
Now we restrict to case of abelian subcategories of the $D^b(\Coh(X))$ which are preserved under $\Dt$ which will allow for the construction of proper moduli stacks of orthosyplectic objects. 

Let $(S,H)$ be a polarized smooth projective surface.

Given a class $B \in N^1(S)_{\Rb}$ define $\ch^B = e^{-B} \ch$ and $\mu_B(E)$ the twisted slope of $E \in K_{num}(S)$ , which is the slope of the formal class $[E\otimes \Ll_{-B}] \in K_{num}(S)_{\Rb}$ with chern class $\ch_B(E)$. 

Let $\alpha > 0$ and $\beta\in \Rb$ be fixed and consider the Bridgeland stability conditions $\sigma_{\alpha, \beta}$ with heart $\Coh^\beta(S) = \langle \Fl_\beta[1], \Tl_\beta\rangle$ where $\Tl_\beta$ is generated by slope semistable sheaves $F$ of slope $\mu_B(F) > \beta$ and $\Fl_\beta$ is genrated by slope semistable sheaves of slope $\mu_B(F) \le \beta$. 

Let $\sigma_{\alpha,\beta} \in \operatorname{Stab}^\dagger(S)$ denote a Bridgeland stability condition on the vertical wall from \cite{Ta} whose heart $\Pl(0,1]$ is $\Coh^\beta(S)$ and whose stability function is 
\[ Z_{\alpha, \beta}(E) = -\int_X(e^{-\beta H + i \alpha H} \ch^B(E)).\]

Given $v$, we fix $\alpha> 0$ and 
\[ \beta = \beta_0 = \frac{ H\cdot \ch_1^B(v)}{H^2 \ch_0(v)} \] 
where the phases of (shifts of) dimension zero sheaves agree with those of sheaves of class $v = (r,0,d)$. Let $\sigma = \sigma_{\alpha, \beta_0}(v)$ denote the resulting stability condition.

 The stack $M_{\sigma}(v)$ is preserved under the derived duality functor $\Dt: \Ml_{\Al_\sigma}(r, 0, d) \to \Ml_{\Al_\sigma}(r, 0, d)$, giving rise to an action of $\Zb/2$ on $M_{\sigma}(v)$.   

\begin{defn} 
Let $(\El,\eta) \in \Al$ denote an orthosymplectic object. An \emph{isotropic} subobject $\Ll$ is an object such that the induced map $\Ll\otimes \Ll\to \El\otimes \El \to \Ol_S$ vanishes. 
\end{defn} 

Let $\Ml_{\sigma}(r,0,d)^{\Dt}$ denote the stack of orthosymplectic objects in $\Ml_{\sigma}(r,0,d)$
and $\Ml_{\sigma}(r,0,d)^{[1] \circ \Dt\circ [1] }$ the set of shifted orthosymplectic objects. Concretely we have an idenitifcation of $T$-points of $\Mt_{\sigma}(r,0,d)^{\Dt}$ with 
\[ \Ml_{\sigma}(r,0,d)^{\Dt}(T) = \{ (\El, \eta)  \mid \El \in \Al_{\sigma,T}, \eta: \El \xrightarrow{\sim} \Dt(\El) \text{ antisymmetric } \} \] and an identification  
\[ \Ml_{\sigma}(r,0,d)^{[1]\circ\Dt\circ[1]}(T) = \{ (\El, \eta)  \mid \El \in \Al_{\sigma, T}, \eta: \El \xrightarrow{\sim} \Dt(\El) \text{ symmetric }. \} \]

\begin{defn} 
An orthosymplectic object $\El\in M_{\sigma}(v)^{\Dt} $ is $\sigma$-(semi)stable if for every isotropic subobject $\Ll \to \El$ we have $\varphi(\Ll) < (\le) \varphi(\El)$.
\end{defn} 
Let \[\Ml_{\sigma}(v)^{\Dt,{ss}} \subset \Ml_{\sigma}(v)^{\Dt}\] denote the substack of semistable objects. This is an open substack becau destabilizing sequences are preserved under specialization on account of the fact that the specialization of an isotropic subobject is isotropic. 

For classical type vector bundles, we likewise have notions of $\mu$-(semi)stable $\Gt$-bundles where destabilizing subbundles are required to be isotropic. 

\begin{prop} \label{prop:bridgeland_polystable}
    If $E$ is a $\sigma$-polystable orthosympleictic object in $\Ml_\sigma(r,0,d)^{\Dt}$ then $E$ is isomorphic to 
    \[ E \ \simeq \ (\oplus_{i = 1}^k  G_i \oplus G_i^\vee) \bigoplus H \bigoplus \Ol_Y[-1] \] where $\Gt_i$ are $\mu$-stable locally free sheaves, $H$ is a $\mu$-stable locally free symplectic sheaf and $\Ol_Y[-1]$ is an orthosymplectic torsion sheaf, so that $H^0(\Ol_Y)$ carries an orthogonal form. 
\end{prop}

A sign shifted result holds for $\Ml_\sigma(r,0,d)^{\Dt}$. We expect that $\Ml_\sigma(r,0,d)^{\Dt}$ and $\Ml_\sigma(r,0,d)^{[1]\circ \Dt\circ [1]}$ are $\Theta$-reducive and $S$-complete and so admit good moduli spaces $M_\sigma(r,0,d)^{\Dt}$ and $M_\sigma(r,0,d)^{[1]\circ \Dt\circ [1]}$. Proposition \ref{prop:bridgeland_polystable} therefore provides a description of the closed points of $M_\sigma(r,0,d)^{\Dt}$ and $M_\sigma(r,0,d)^{[1]\circ \Dt\circ [1]}$ which agree with the closed points of the Uhlenbeck compactificaiton of $\Gt$-bundles for classical type $\Gt$. 

\subsection{Action of dimension zero CoHA}

\subsubsection{} 
Let $ \Hl_{\Sp(2r)} =  \Ht^{\BM}_\sbt(\bigsqcup_d \Ml_{\Al_{\sigma}}(2r,0,d)^{\Dt},ss)$ and let $\Hl_{\Ot(k)} =  \Ht^{\BM}_\sbt(\bigsqcup_d \Ml_{\Al_{\sigma}}(2r,0,d)^{[1]\circ \Dt\circ[1]},ss).$ Let $(\Hl_{0}, \star)$ denote the cohomological Hall algebra of dimension zero sheaves. 

In the context of moduli of orthosymplectic perverse coherent sheaves, our construction provides the following. 

\begin{theorem} \label{thm:Surface}
The isotropic extension correspondence gives $\Hl_{G}$ the structure of a left $\Hl_{0}$-module for $G \in \{ \Sp(2r), \Ot(2r), \Ot(2r + 1) \}$. 
\end{theorem} 
\begin{proof} 
We only need to show that properness of $p$ is preserved when we pass to the open substack $\Ht^{\BM}_\sbt(\bigsqcup_d \Ml_{\Al_{\sigma}}(2r,0,d)^{\Dt},ss)$ of semistable orthosympleictc perverse coherent sheaves.  But because $\sigma$ lies on the wall in which dimension zero sheaves have the same phase as all objects in $\Ml_{\Al_{\sigma}}(2r,0,d)^{\Dt}$ it is impossible for isotropic extensions by dimension zero sheaves to destabilize terms in the extension. 
\end{proof}

\subsection{ADHM construction of orthosymplectic instantons}

\subsubsection{Orthsymplectic tautological complex} \label{sssec:OSpTautologicalComplex} Consider the category $\Al$ of perverse-coherent sheaves on $\Ab^2$, or equivalently representations of the ADHM quiver
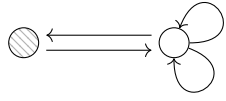
\begin{center}
  \begin{tikzpicture}[rotate=90]
    \node[circle, draw, inner sep=4pt] (dot1) at (0,0) {};
    \node[circle, draw, inner sep=4pt, pattern=north west lines, pattern color=black!30] (dot2) at (0,2) {};
  
    \draw[->, in=-10, out=-80, looseness=12] (dot1) to (dot1);
    \draw[->, out=-100, in=-170, looseness=12] (dot1) to (dot1);
  
  \draw[->] (0.1,0.3) -- (0.1,1.7);
  \draw[<-] (-0.1,0.3) -- (-0.1,1.7);
  \end{tikzpicture}
\end{center}
We have the following taulogical perfect complex over $\Ml$. Let $\mathcal V$ and $\mathcal W$ be the tautological vector bundles on $\Ml^{\OSpt, fr}$, and let 
\[ C^\bullet = [\mathcal V \to \mathcal V\oplus \mathcal V \oplus  \mathcal W \to \mathcal V ] \] 
be the tautological complex. 

\subsubsection{Linear tautological complex} 
In the non-orthosympletic case, let $C^\bullet_{\GL}$ be the similar complex on $\Ml^{\textup{fr}}$. Then the coefficients of 
\[ \varphi(w) := \frac{c_{-1/w}(C^\bullet_{\GL})}{c_{-1/w}(\hbar\otimes C^\bullet_{\GL})} \]
act via the loop cartan generators of the extended CoHA.

\subsubsection{} We now want to write the analogue for the orthosympletic case.

Let $\Ml_0$ be the stack of zero-dimensional torsion sheaves. On $\Ml_0$ we have the same complex $C^\bullet_{\GL}$ with $V = 0$. Define  
\[ C^{\bullet}_{\OSpt} = C^{\bullet}_{\GL} \oplus \Dt(C^{\bullet}_{\GL}), \] 
where we have implicitly restricted to $\Ml^{OSp,fr}_0$.

The action of the the CoHA on orthsymplectic modules does not extend to the usual extended CoHA but rather to the algebra extended by the action of coefficients of 
\[ \varphi(w) := \frac{c_{-1/w}(C^\bullet_{\OSpt})}{c_{-1/w}(\hbar\otimes C^\bullet_{\OSpt})}. \]

\newpage
\section{Shuffle modules and Yetter-Drinfeld condition} \label{sec:TwistedYangians} 

\noindent In this section, we will give shuffle formulas for the CoHA action and vertex coaction on $\Ht^\sbt(\Ml^{\textup{\OSpt}},\varphi)$ for quivers using equivariant torus localisation. Using an intertwining map from the orthosymplectic CoHA action in the cases of a quiver with potential and for the preprojective algebra we give shuffle formulas for the orthosymplctic action in these cases.

We then give evidence for a conjecture relating these structures to twisted Yangians.

\subsubsection{Remarks on the relationship to $\iota$-Hall algebras and $\iota$-quiver representation} 

Twisted Yangians and other $\iota$-quantum groups arising in geometric representation theory are coideal subalgebras $\Il \subseteq \Al$ of Hopf algebras $\Al$, satsifying $\Delta(\Il) \subset \Il \otimes \Al$. Dually, the orthosymplectic representations in this paper are instead representations of bialgebras and therefore correspond to representations of a \emph{quotient} of a (vertex) bialgebra $\Al$. 

While we don't resolve the discrepancy between our narrative the more familiar one, we now make a brief remark in Lemma \ref{lem:HopfRadical} on algebra ideals that arise in the presence of a coideal subalgebra and a Hopf pairing.

To set up notation for Lemma \ref{lem:HopfRadical}, let $(H, \cdot, \Delta, S, \epsilon ,1)$ be a Hopf algebra with a Hopf pairing $\langle\ ,\ \rangle$, i.e. a bilinear pairing satisfying
\[ \langle f \cdot g, h \rangle = \langle f\otimes g, \Delta h\rangle \] 
with the pairing extended to $H \otimes H$ multiplicatively. Note that a Hopf pairing is equivalent to a choice of quasitriangular structure on the Hopf algebra $H$. Suppose $T \subset H$ is a left coideal subalgebra so that $\Delta|_{T} : T \to H \otimes T$ gives the coideal structure.  Finally, write $K$ for the right radical of $T$, so that 
\[ K = \{ f \in H \ :\  \langle t, f \rangle = 0\ \forall t \in T \}\]

\begin{lem} \label{lem:HopfRadical}
  The quotient $M = H/K$ has the structure of an $H$-module naturally dual to $T$. 
\end{lem} 
\begin{proof} 
The radical  $K$ is a left $H$-ideal because 
\[ \langle t, h \cdot f \rangle = \langle \Delta t , h\otimes f \rangle =\langle t_1 \otimes t_2 , h\otimes f \rangle = 0\]
since $t_2\in T$. The pairing $\langle-,-\rangle$ induces the pairing between $T$ and $M$. 
\end{proof}

\subsection{Notation: equivariant quivers with potential}

\subsubsection{} 

Let $Q=(Q_0,Q_1)$ be a quiver with an involution $\tau$; any such may be drawn as e.g.
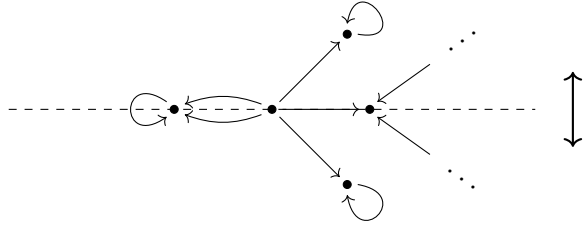
\begin{center}
  \begin{tikzpicture}
  \begin{scope} 
   [rotate=-90] 
   \draw[dashed] (0,-3.5) -- (0,3.5);
   \begin{scope} 
   \clip (0,0) circle (2.2);
   \filldraw[] (0,0) circle (1.5pt);
   \filldraw[] (1,1) circle (1.5pt);
   \filldraw[] (-1,1) circle (1.5pt);
   \filldraw[] (0,1.3) circle (1.5pt);
   \filldraw[] (0,-1.3) circle (1.5pt);
   \draw[->] (0.1,0.1) -- (0.9,0.9);
   \draw[->] (-0.1,0.1) -- (-0.9,0.9);
   \draw[->] (0,0.14) -- (0,1.16);
   \draw[->] (0.07,-0.14) to[out=-90+20,in=90-20] (0.07,-1.16);
   \draw[->] (-0.07,-0.14) to[out=-90-20,in=90+20] (-0.07,-1.16);
   \draw[->] (0.1+0.5,1.4+0.7) -- (0.1,1.4);
   \draw[->] (-0.1-0.5,1.4+0.7) -- (-0.1,1.4);
   \draw[->,looseness=10] (1,1+0.14) to[out=80,in=10] (1+0.14,1);
   \draw[->,looseness=10] (-1,1+0.14) to[out=80,in=180-10] (-1-0.14,1);
   \draw[->,looseness=10]  (-0.1,-1.3-0.1) to[out=-135+10,in=-45-10]  (0.1,-1.3-0.1);
 
    \end{scope}
    \node[] at (0.9,2.5) {\rotatebox{145}{$\cdots$}};
    \node[] at (-0.9,2.5) {\rotatebox{215}{$\cdots$}};
  \draw[<->,thick] (0.5,4) -- (-0.5,4);
  \draw[<->,thick,white] (0.5,-4) -- (-0.5,-4);
   \end{scope}
  
   \end{tikzpicture}
   \end{center}
where $\tau$-fixed nodes and edges are drawn along the reflection axis. In addition, let $\mathrm{wt}:Q_1 \to \Zb^r$ be a weighting of the edges the conditions in \cite{YZ}. Note that these conditions are preserved under $\tau$. Finally, let $W\in \Cb Q/[\Cb Q,\Cb Q]$ be a potential, which is invariant under the involution, i.e. 
$$\tau^*W\ =\ W.$$

\subsubsection{Remark} One may view the above as a groupoid-valued quiver $Q/\tau=(Q_0/\tau,Q_1/\tau)$, i.e. a pair of groupoids with source and target maps. This raises interesting questions about groupoid quivers which are not global quotients, which we will not address.

\subsubsection{Unbalanced involutions} \label{ssec:unbalanced_involutions}

Forseeing applications to folding of non-simply laced quivers, and serving as the most efficient way to write down contributions from arrows even in the balanced case, we study a generalization of of the involution $\tau$. We let $Q^{\textup{op}}$ denote the opposite quiver and consider an isomorphism 
\[ \theta: Q \to Q^{\textup{op}}.\] The composition 
\[ \tau = \Dt\circ \theta: {\Rep Q}\to {\Rep Q}\]
defines an weak involution. 

All of the balanced involutions $\theta : Q\to Q$ arise via this construction and the involutive balancing isomorphism $b:Q \simeq Q^{\textup{op}}$ so we assume in the balanced case a fixed balancing isoorphism and define $\theta = \tau\circ b$.

\subsubsection{Representations} Consider a quiver $Q  = (Q_0, Q_1)$ which arises as the double or triple of a quiver and an involution $\tau$. Assume $Q^\tau$ is either bipartite or the Jordan quiver so that $\tau$ defines an duality map on $Q$, and define a sign map $\textup{sgn}: Q_0^\tau \to \Zb/2$. We may pull back any quiver representation along the involution and compose with dualisation  and define an involution on $\Rep Q$:
$$\tau \ :\ \Ml_{\Rep Q}\ \stackrel{\sim}{\to}\ \Ml_{\Rep Q}, \hspace{15mm} V \ \mapsto \ (\tau^* V)^\vee.$$
We define a \textit{representation} of a $\tau$-equivariant quiver $Q$ to be a point in the fixed stack $\Ml_{\Rep Q}^\tau$. Note that $\tau^*\tr W=\tr W$ and so we may define the invariant the vanishing cycle functor as $\varphi_W^\tau= \varphi_{\tr W \vert_{\Ml^\tau}}$.

  As a stack, we have 
  $$ \Ml_{\Rep Q} \ = \ \prod_{i \stackrel{e}{\to} j} \Hom(\El_i,\El_j) \ = \ \prod_{d \in \Nb^{|Q_0|}, \, e} \Hom(\Cb^{d_i},\Cb^{d_j})/ \smprod_i \GL_{d_i}$$
   as the total space of a vector bundle over $\prod_{i\in Q_0}\BGL$, where again $\El_i$ is the tautological vector bundle over the $i$th factor.

   
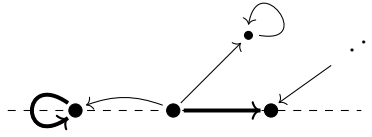
\begin{equation}\label{fig:QuotientQuiver}
    \begin{tikzpicture}
    \begin{scope} 
      [rotate=-90]
      \draw[dashed] (0,-2.2) -- (0,2.5);
      \begin{scope} 
      \clip (0,0) circle (2.2);
      \filldraw[] (0,0) circle (2.5pt);
      \filldraw[] (-1,1) circle (1.5pt);
      \filldraw[] (0,1.3) circle (2.5pt);
      \filldraw[] (0,-1.3) circle (2.5pt);
      \draw[->] (-0.1,0.1) -- (-0.9,0.9);
      \draw[-,line width = 1.5pt] (0,0.14) -- (0,1.16);
      \draw[->,thick] (0,0.14) -- (0,1.16);
      \draw[->] (-0.07,-0.14) to[out=-90-20,in=90+20] (-0.07,-1.16);
      \draw[->] (-0.1-0.5,1.4+0.7) -- (-0.1,1.4);
      \draw[->,looseness=10] (-1,1+0.14) to[out=80,in=180-10] (-1-0.14,1);
      \draw[-,line width = 1.5pt,looseness=10]  (-0.1,-1.3-0.1) to[out=-135+10,in=-45-10]  (0.1,-1.3-0.1);
      \draw[->,thick,looseness=10]  (-0.1,-1.3-0.1) to[out=-135+10,in=-45-10]  (0.1,-1.3-0.1);
    
       \end{scope}
       \node[] at (-0.9,2.5) {\rotatebox{215}{$\cdots$}};
    \end{scope}
     \end{tikzpicture}
     \end{equation}

   The stack of $\tau$-equivariant representations is thus
  $$\Ml_{\Rep Q}^\tau \ = \ \prod_{i/\tau \stackrel{e}{\to} j/\tau} \Hom(\El_{i/\tau},\El_{j/\tau})^\tau$$
  where the product is over $\tau$-orbits of edges $e$. Noteice that 
  $$\Ml^\tau_{\Rep Q} \ = \ \Rep(Q)^{\OSpt}/ \smprod_{i}\Gt_i \ $$
  where $\Gt_i$ is the disjoint union over a family of classical groups $\Gt_{i,0} \subseteq \Gt_{i,1} \subseteq \cdots$ attached to every node $i\in Q_0/\tau$, where the classical sequence is the $\GL(k)$ sequence if the node $i$ is not $\tau$-fixed and is one of the sequences $\Ot(2k), \Ot(2k+1), \Spt(2k)$ depending on $\textup{sgn}(i)$ and a choice of dimension parity if $\textup{sgn}(i) = 0$. 
  
   



   We likewise define for an unbalanced involution from \ref{ssec:unbalanced_involutions} the stack $\Ml^\tau_{\Rep Q}$.

   \subsubsection{Euler classes and extensions} Let $\Lt_{i,a,b}$ denote the Levi and $\Pt_{i,a,b}$ the parabolic of a pair of classical groups $\Gt_{i,a} \subset \Gt_{i,b}$, so that $\Lt_{i,a,b} \simeq \Gt_{i,a} \times \GL((b -a)/2)\subset \Gt_{i,b}$.   The moduli stack of extensions of a $\tau$-equivariant representation by an ordinary representation is then 
   $$\Ml_{\SES}^\tau \ = \ \Rep_{\SES}(Q)^{\OSpt}/\Pt $$
   where $\Rep_{\SES}(Q)^{\OSpt}$ is the stack of $\tau$-invariant 3-step exact sequences of representations of $Q$.

   We now give an equation for the equivariant Euler class to the tangent complex of the projection $q:\Ml_{\SES}^\tau\to \Ml^\tau\times \Ml$. Given an arrow $e \in Q_1/\tau$ let ${}_e Q$ denote the quiver ${}_e Q = (Q_1, \{e , \overline{(\theta(e))}\}) \subseteq Q$ consisting of $e$ and the opposite of $\theta(e)$. The quiver ${}_e Q$ is the smallest quiver with unbalanced involution containing $e$ and all the vertices of $Q$. 

   \begin{lem}
     The equivariant Euler class of $\Tt_q$ is  
     \begin{equation}
       \label{eqn:EulerClassSimple} 
       e(\Tt_q) \ = \ \prod_{e\in Q_1/\tau}e( \Rep_{\SES}({}_e Q)^{\OSpt} ).
     \end{equation}
   \end{lem}
   \begin{proof}
    We consider the commutative diagram
     \begin{center}
     \begin{tikzcd}[row sep = {30pt,between origins}, column sep = {20pt}]
      \Ml_{\SES}^\tau \ar[r,"q"] \ar[d,"\pi_{\SES}"]  &\Ml^\tau \times \Ml \ar[d,"\pi^\tau \times \pi"] \\
      \Bt \Pt \ar[r,"\overbar{q}"] & \Bt \Gt \times \BGL
     \end{tikzcd}
     \end{center} 
     where the vertical maps are projections to the base of tautological vector bundls. Note that all maps are isomorphisms on cohomology, and so we may without ambiguity write 
     $$e(\Tt_q) \ = \ e(\Tt_{\overbar{q}})\cdot \left(e(\pi_{\SES})\cdot e(\Tt_{\pi^\tau}\boxplus \Tt_{\pi})^{-1}\right).$$
     This euler class factors over the contributions of orbits of edges, which agree with the contributions from ${}_e Q$. 
   \end{proof}

   \subsubsection{Explicit edge contribtuions} 
   For a dimension vector $\db = (d_i)_{i \in Q_0/\sigma}$ for the orthosymplectic quiver, we define 
  $$\Gt_\db \ = \  \prod_{i \in Q_0/ \sigma} \Gt_{i,d_i}.$$ Likewise define $\Lt_{\db',\db''}$ and $\Pt_{\db',\db''}$, with $i,a,b$ suppressed from the notation. We will let $\GL_{\db'} = \prod_{i \in Q_0} \GL(d_i')$ denote the general linear factor of $\Lt_{\db', \db''}$.  We may thus succinctly write
   $$\Ml^{\OSpt}_{\db} \ = \ \coprod_{\db}\Rep_{\db}^{\OSpt}(Q)/ \Gt_{\db}, \hspace{7.5mm}  \Ml_{\db'} = \coprod_{\db'}\Rep_{\db}(Q)/\GL_{\db'}, \hspace{7.5mm} \Ml_{\SES,\db,\db'}^{\OSpt} \ = \ \coprod_{\db,\db'}\Rep_{\db', \db''}^{\OSpt}(Q)/P_{\db', \db''} $$
   where $\Rep_{\db', \db''}^{\OSpt}(Q) \subset \Rep_{\db}^{\OSpt}(Q)$ consists of those representations $(V,e)$ where the first $\db'$ corrdinate subspaces form an isotropic subrepresentation $(V', e')$ and the qutotient by the first $\db''$ coordinate subspaces forms a quotient respresentation isomorphic to $\tau(V', e') = (V_{\theta(i)}^{'\vee}, \pm \theta(e')^\vee)$.

   Let $a$ be an arrow in $Q_1$. Fix $\db'$ for the ordinary quiver and $\db''$ for the orthsymplectic quiver, and an identification  $\rho: T_{\db'} \times T_{\db''} \xrightarrow{\sim} T_{\db}$. Let $x_{(1)}$ be the coordinates on $T_{\db'}$ and $x_{(2)}$ on $T_{\db''}$.  Consider the arrow contribution associated to 
   $ {}_aQ $ so that 
   \begin{equation*} 
   e(Q_{1}^{\OSpt})_{\db', \db''} := \prod_{\underline a \in Q_1/ \tau} e({}_aQ)
   \end{equation*} 
   where the class $e({}_a Q)$ depends on the type of the orbit of $a$ based on the following cases, which we assume to always hold for the arrow $a$. Notice that we draw the diagram for the actual extensions while the normal bundle to the space of extensions is the complement of this and the diagonal in the space of all maps, so the edge contributions have opposite torus weights than the drawn arrows. 
   \begin{itemize} 
    \item[Case O:] In this case $\mathbf{s(a), t(a)} \not\in Q_0^\theta$. In an orthosymplectic extension the isotropic subobject is a pair $(a\xrightarrow{f}b)$ and $(c \xrightarrow{g} d)$ where $a = V_{ s(a)}, b = V_{t(a)}$, $c = t(\theta(a)), d = s(\theta(a))$ while the orthosymplectic quotient is a pair as above but with $x = a, y = b$ and $c = y^\vee, d= x^\vee$ with $g = f^\vee$. 
    The diagram of an extension is given by 
\[\begin{tikzcd}
	a & b && c & d \\
	x & y && {y^\vee} & {x^\vee} \\
	{d^\vee} & {c^\vee} && {b^\vee} & {a^\vee}
	\arrow[from=1-1, to=1-2]
	\arrow[from=1-4, to=1-5]
	\arrow["\eta", from=2-1, to=1-2]
	\arrow[from=2-1, to=2-2]
	\arrow["{\xi^\vee}", from=2-4, to=1-5]
	\arrow[from=2-4, to=2-5]
	\arrow["\mu"{pos=0.3}, from=3-1, to=1-2]
	\arrow["\xi"', from=3-1, to=2-2]
	\arrow[from=3-1, to=3-2]
	\arrow["{\mu^\vee}"{pos=0.2}, from=3-4, to=1-5]
	\arrow["{\eta^\vee}"', from=3-4, to=2-5]
	\arrow[from=3-4, to=3-5]
\end{tikzcd}\]
    
    so that the contribution of the arrow is determined by the content of $\eta, \xi$ and $\mu$ and thus given by 
    \begin{align} e({}_a Q)_{\db', \db''} &=     
        \prod_{\substack{1\le i \le \db'_{t(a)} \\ 1 \le j \le \db''_{s(a)} }} 
        \left( -x_{t(a), (1), i} + x_{s(a), (2), j} - \mathrm{wt}(a) \right)\nonumber \\
        &\prod_{\substack{1\le j \le \db'_{\theta(s(a))} \\ 1 \le i \le \db'_{t(a)} }} 
         \left( -x_{t(a), (1), i} - x_{\theta(s(a)), (1), j } - \mathrm{wt}(a) \right) \label{eq:type0} \\
       & \prod_{\substack{ 1\le j \le \db'_{\theta(s(a))} \\ 1 \le i \le \db''_{t(a)} }} 
         \left( -x_{t(a), (2), i} - x_{\theta(s(a)), (1), j} - \mathrm{wt}(a) \right)\nonumber
      \end{align}
    \item[Case I:] In this case $s(a) = t(a) \in Q_0^\theta, \theta(a) = a$ so that 
\begin{equation} \label{eq:caseI}e({}_a Q)_{\db', \db''} = \prod_{\alpha \in R(P_{\db', \db''}/ L_{\db', \db''})} (-\alpha(\rho(x_{(1)}, x_{(2)})) - \mathrm{wt}(a)). \end{equation} 
 
\item[Case II:] In this case we have $s(a) \neq t(a)$ with $\theta(s(a)) = t(a)$ and $\theta(a) = \overline{a}$. An orthosymplectic extension with isotropic subobject $c = (c_1 \xrightarrow{f} c_2)$ and orthosymplecitc quotient $(x \xrightarrow{\phi} x^\vee)$ is represented by the following diagram 
\[\begin{tikzcd}
	{c_1} & {c_2} \\
	x & {x^\vee} \\
	{\theta(c_1)^\vee} & {\theta(c_2)^\vee}
	\arrow["f", from=1-1, to=1-2]
	\arrow["\eta", from=2-1, to=1-2]
	\arrow["\phi", from=2-1, to=2-2]
	\arrow["\mu"{pos=0.2}, from=3-1, to=1-2]
	\arrow["{\eta^\vee}"'{pos=0.7}, from=3-1, to=2-2]
	\arrow["{f^\vee}", from=3-1, to=3-2]
\end{tikzcd}.\]
 If the parity of $s(a)$ is $p \in \Zb/2$ then we must have that $\phi^\vee = -(-1)^p \alpha_{t(a)}\phi$ and by the bipartite assumption on $Q$ we have $\textup{sgn}(t(a)) = p+1$ so that $\phi^{\vee} = (-1)^p \phi$. 

   so that the arrow contribution is 
   \begin{equation}\label{eq:typeII} e({}_a Q)_{\db', \db''} =      \prod_{\substack{1\le i \le \db'_{t(a)} \\ 1 \le j \le \db'_{s(a)} }} (-x_{t(a), (1), i} + x_{s(a), (2), i}  - \mathrm{wt}(a)) \\ 
     \prod_{1 \le i \le (\lneq) j \le \db'_{t(a)}} -x_{t(a),(1),i} - x_{t(a),(1),j} - \mathrm{wt}(a) \end{equation} 
     where the second product ranges over $i \le j$ or $i\lneq j$  depending on the parity of $s(a)$. 

   \item[Case III:] In this case we have 
   $s(a) \neq t(a)$
    with
     $s(a) \not\in Q_0^\theta$ but $t(a) \in Q_0^\theta$.
      Then $s(\theta(a)) = t(a)$ 
      and $t(\theta(a)) = \theta(s(a))$
       so we have a directed $A_2$ subquiver containing all the edges of ${}_aQ$.  An orthosymplectic extension with isotropic subobject 
       $c = (c_1 \xrightarrow{f} c_2\xrightarrow{g}c_3)$ 
       and orthosymplecitc quotient 
       $(x \xrightarrow{\phi} y \xrightarrow{\phi^\vee} x^\vee)$ 
       is represented by the following diagram 
\begin{center}
\begin{tikzcd}[row sep = {30pt,between origins}, column sep = {20pt}]
  {c_1} & {c_2} & {c_3} \\
	x & y & {x^\vee} \\
	{\theta(c_1)^\vee} & {\theta(c_2)^\vee} & {\theta(c_3)^\vee}
	\arrow[from=1-1, to=1-2]
	\arrow[from=1-2, to=1-3]
	\arrow[from=2-1, to=1-2]
	\arrow[from=2-1, to=2-2]
	\arrow[from=2-2, to=1-3]
	\arrow[from=2-2, to=2-3]
	\arrow[from=3-1, to=1-2]
	\arrow[from=3-1, to=2-2]
	\arrow[from=3-1, to=3-2]
	\arrow[from=3-2, to=1-3]
	\arrow[from=3-2, to=2-3]
	\arrow[from=3-2, to=3-3]
\end{tikzcd}
\end{center}
which as a 9$\times$9 block matrix can be expressed as 
\[
\begin{blockarray}{cccccccccc}
& c_3 & c_2 & c_1 & x^\vee & y & x & \theta(c_3)^\vee & \theta(c_2)^\vee & \theta(c_1)^\vee \\
\begin{block}{c(ccc|ccc|ccc)}
  c_3    & & g&  & & \xi & & & \mu^\vee & \\
  c_2    & &  & f& & & \eta & & &  \mu  \\
  c_1    & &  &  & & & & & &  \\
  \cmidrule(lr){2-10}
  x^\vee & &  &  & & \phi^\vee& & & \eta^\vee &  \\
  y      & &  &  & & & \phi & & &  \xi^\vee \\
  x      & &  &  & & & & & & \\
  \cmidrule(lr){2-10}
  \theta(c_3)^\vee & &  &  & & & & & g^\vee & \\
  \theta(c_2)^\vee & &  &  & & & & & &f^\vee\\
  \theta(c_1)^\vee  & &  &  & & & & & &\\
\end{block}
\end{blockarray}
 \]

     so that the arrow contribution is 
     \begin{align} e({}_a Q)_{\db', \db''} &=     
       \prod_{\substack{1\le i \le \db'_{t(a)} \\ 1 \le j \le \db''_{s(a)} }} 
       \left( -x_{t(a), (1), i} + x_{s(a), (2), j} - \mathrm{wt}(a) \right)\nonumber \\
       &\prod_{\substack{1\le i \le \db'_{t(a)} \\ 1 \le j \le \db'_{\theta(s(a))} }} 
        \left( -x_{t(a), (1), i} - x_{\theta(s(a)), (1), j } - \mathrm{wt}(a) \right) \label{eq:typeIII} \\
      & \prod_{\substack{ 1\le j \le \db'_{\theta(s(a))} \\ 1 \le i \le \db''_{t(a)} }} 
        \left( -x_{\theta(s(a)), (1), i} + x_{t(a), (2), j} - \mathrm{wt}(a) \right)\nonumber
     \end{align}
   from the contributions of $\eta$, $\mu$ and $\xi$ respectively. 
\item[Case III':] In this case we have 
$s(a) \neq t(a)$
 with
  $s(a) \in Q_0^\theta$ but $t(a) \not\in Q_0^\theta$.
   Then $s(\theta(a)) = s(a)$ 
   and $t(\theta(a)) = \theta(t(a))$. In this case $\overline{\theta(a)}$ is a Case III arrow so that 
   \[ e({}_a Q)_{\db', \db''} =
    e({}_{\overline{\theta(a)}}Q)_{\db', \db''}\] is calculated from \eqref{eq:typeIII}.

    \item[Case IV:] In this case we have $s(a), t(a) \in Q_0^\theta$ with opposite parity. Then $\theta(a) = \overline{a}$ and an orthosymplectic extensions consists of an isotropic subobject with objets $c = V_{s(a)}, d  = V_{t(a)}$ arrows $f: c\to d$ and $g: d\to c$, and an orthosymplectic quotient of a pair $(x,y)$ with a map $\phi:x\to y$ such that $\rho(\overline{a}) = \phi^\times$ is the adjoint of $\phi$ so that $\langle -, \phi -\rangle_y = \langle \phi^\times - ,- \rangle_x$. The diagram of the orthosympplectic extension is given by 
    \begin{center}
    \begin{tikzcd}[row sep = {30pt,between origins}, column sep = {20pt}]
      c & d \\
      x & y \\
      {c^\vee} & {d^\vee}
      \arrow[shift right, from=1-1, to=1-2]
      \arrow[shift right, from=1-2, to=1-1]
      \arrow[from=2-1, to=1-2]
      \arrow[shift right=2, from=2-1, to=2-2]
      \arrow[from=2-2, to=1-1]
      \arrow[from=2-2, to=2-1]
      \arrow[from=3-1, to=1-2]
      \arrow[from=3-1, to=2-2]
      \arrow[shift right, from=3-1, to=3-2]
      \arrow[from=3-2, to=1-1]
      \arrow[from=3-2, to=2-1]
      \arrow[shift right, from=3-2, to=3-1]
    \end{tikzcd}
    \end{center}
     and corresponds to the block matrix 
\[
\begin{blockarray}{ccccccc}
& d & c & y & x & d^\vee & c^\vee \\
\begin{block}{c(cc|cc|cc)}
  d &  &f&&\eta&& \mu \\
  c &  g&&\xi&&\mu^\vee &\\
  \cmidrule(lr){2-7}
  y     &&&&\phi&&\xi^\vee \\
  x     &&&\phi^\times&&\eta^\vee& \\
  \cmidrule(lr){2-7}
  d^\vee&&&&&&g^\vee\\
  c^\vee&&&&&f^\vee&\\
\end{block}
\end{blockarray}
 \]
 so that the arrow contribution is 
 \begin{align} e({}_a Q)_{\db', \db''} &=     
   \prod_{\substack{1\le i \le \db'_{t(a)} \\ 1 \le j \le \db''_{s(a)} }} 
   \left( -x_{t(a), (1), i} + x_{s(a), (2), j} - \mathrm{wt}(a) \right)\nonumber \\
   &\prod_{\substack{1\le i \le \db'_{t(a)} \\ 1 \le j \le \db'_{s(a)} }} 
    \left( -x_{t(a), (1), i} - x_{s(a), (1), j } - \mathrm{wt}(a) \right) \label{eq:typeIV} \\
  & \prod_{\substack{ 1\le j \le \db'_{s(a)} \\ 1 \le i \le \db''_{t(a)} }} 
    \left( -x_{s(a), (1), i} + x_{t(a), (2), j} - \mathrm{wt}(a) \right)\nonumber
 \end{align}
from the contributions of $\eta$, $\mu$ and $\xi$ respectively. 
    \end{itemize}

  \subsection{Shuffle formulas orthosymplectic CoHAs}

  \subsubsection{}
  We first compute the action on the orthosymplectic COHM in the case that the potential is $W=0$.  Via localisation, the cohomological Hall algebra $\Ht^{\sbt}(\Ml)$ of a quiver has a formula given in \cite{KS,BDa,YZ} using the identification 
  \begin{equation} \label{eq:cohomology_is_invariants} \Ht^\sbt(\Ml_{\db}) \simeq \Ht^\sbt(\BGL_{\db}) = k[x_{i, k}]^{W_{\GL_{\db}}}_{i \in Q_0, k = 1, \ldots, \db_i} \end{equation}
given $f \in  \Ht^\sbt(\Ml_{\db}), g\in \Ht^\sbt(\Ml_{\db'})$ we have  

\[ f\star g = \sum_{\sigma \in W(\GL_{\db' + \db''})} \frac{ e(Q_1)_{\db', \db''}(\sigma(x_{(1)}), \sigma(x_{(2)}))}{e(Q_0)_{\db', \db''}(\sigma(x_{(1)}), \sigma(x_{(2)}))} f(\sigma(x_1))g(\sigma(x_2)) \] 
where $\sigma \in W \GL_{\db}$ acts on the variables as in \ref{eq:cohomology_is_invariants} under the identification of the maximal torus of $\GL_{\db}$ with that of $\GL_{\db'}\times \GL_{\db''}$. We will omit the argument of $e(Q_a)$ when it is clear.

\begin{theorem} \label{thm:ShuffleNoPotential}
  The action of $\Ht^{\sbt}(\Ml)$ on $\Ht^{\sbt}(\Ml^{\OSpt})$ is given by 
  \begin{equation}
    \label{eqn:COHMVect} 
    f(x_{(1)})\cdot g(x_{(2)}) \ = \  \sum_{\tilde{\sigma}\in W_{\Gt_{n+m}}} \frac{e(Q_1^{\OSpt})}{e(Q_0^{\OSpt})}f(\sigma(x_{(1)}))g(\sigma(x_{(2)}))
  \end{equation}
  where the equivariant Euler class is
  \begin{equation}
    e(Q^{\OSpt}_0) \ =\   \prod_{\alpha \in R(P_{\db', \db''}/L_{\db', \db''})} \alpha (x_{(1)}, x_{(2)}), \hspace{15mm}  e(Q_1^{\OSpt}) \ =\  \prod_{a \in Q_1} e(Q_a)_{\db', \db''}
  \end{equation}
  as above. 
\end{theorem}
\begin{proof}
  Follows from a near-identical proof to Theorem 2 of \cite{KS}. 
\end{proof}

\subsubsection{Orthsymplectic COHM with potential}  We now relate orthosymplecitc CoHA modules for the CoHA of a quiver with potential to those without potential, allowing the study of modules for the CoHA with $W \neq 0$ using the above shuffle module formula. 
 
Apply the natural transformation for the vanishing cycle transformation and pullback by the closed locus $i:\Ml\hookrightarrow \tilde{\Ml}$, giving the map of sheaves
$$\varphi_Wk[-1] \ \to \ i_*i^*k \ \leftarrow\ k.$$
This, and the orthosymplectic analogue, gives maps on cohomology
\begin{align*}
   \xi\ :\ \Ht^\sbt(\Ml,\varphi_W) &\ \to \ \Ht^\sbt(\Ml) \ \stackrel[\raisebox{2pt}{$\sim$}]{i^*}{\leftarrow} \ \Ht^\sbt(\tilde{\Ml})\\
   \xi^{\OSpt}\ :\ \Ht^\sbt(\Ml^{\OSpt},\varphi_{W^{\OSpt}}) &\ \to \ \Ht^\sbt(\Ml^{\OSpt}) \ \stackrel[\raisebox{2pt}{$\sim$}]{i^*}{\leftarrow} \ \Ht^\sbt(\tilde{\Ml}^{\OSpt}).
\end{align*} 

\begin{prop} \label{prop:Intertwine}
The maps $\xi, \xi^{\OSpt}$ intertwine the actions of the CoHA with $W \neq 0$  on the orthosymplectic CoHA module with those with trivial potential, so the following diagram commutes where the horizontal arrows are the action map. 
\begin{center}
  \begin{tikzcd}[row sep = {30pt,between origins}, column sep = {20pt}]
    {\Ht^\bullet(\Ml, \underline{\Qb})\otimes \Ht^\bullet(\Ml^{\OSpt}, \underline{\Qb})} & {\Ht^\bullet(\Ml^{\OSpt}, \underline{\Qb})} \\
    {\Ht^\bullet(\Ml,\varphi_W)\otimes \Ht^\bullet(\Ml^{\OSpt},\varphi_{W^\OSpt})} & {\Ht^\bullet(\Ml^{\OSpt}, \varphi_{W^{\OSpt}})}
    \arrow[from=1-1, to=1-2]
    \arrow["{\xi\otimes\xi^{\OSpt}}", from=2-1, to=1-1]
    \arrow[from=2-1, to=2-2]
    \arrow["{\xi^{\OSpt}}", from=2-2, to=1-2]
  \end{tikzcd}
  \end{center}
  
\end{prop}
\begin{proof} 
  The argument in \cite{BDa} carries over exactly to the present situation. 
\end{proof} 

Via the dimensional reduction isomorphisms from Proposition \ref{prop:DimensionalReduction}, we have an isomorphisms
$$
\Ht^{\BM}_\bullet(\Ml_{\Pi}) \ \stackrel[\raisebox{2pt}{$\sim$}]{\textup{dr}}{\to} \ \Ht^\bullet(\Ml,\varphi_W ) \hspace{15mm} 
\Ht^{\BM}_\bullet(\Ml_{\Pi}^{\OSpt}) \ \stackrel[\raisebox{2pt}{$\sim$}]{\textup{dr}}{\to} \ \Ht^\bullet(\Ml^{\OSpt},\varphi_W^{\OSpt} ).
$$
We obtain an intertwining map 
\begin{cor} \label{cor:preproj_intertwiner}
The composition of dimensional reduction isomorphisms and $\xi^{(\OSpt)}$ in \eqref{eq:preproj_intertwiner} intertwine the action of $\Ht^{\BM}_\bullet( \Ml_{\Pi_Q})$ on $\Ht^{\BM}_\bullet(\Ml_{\Pi_Q}^{\OSpt})$ with those of $\Ht^\bullet(\Ml)$ on $\Ht^{\bullet}(\Ml^{\OSpt})$:
\begin{equation}\label{eq:preproj_intertwiner}
\begin{tikzcd}[row sep = {30pt,between origins}, column sep = {30pt}]
  {\Ht^{\BM}_\bullet(\Ml_{\Pi_Q}^\OSpt)} \ar[r,"\xi \cdot \textup{dr}"]  &[20pt] {\Ht^\bullet(\Ml^{\OSpt})} \\
	{\Ht^{\BM}_\bullet (\Ml_{\Pi_Q})\otimes \Ht^{\BM}_{\bullet}(\Ml_{\Pi_Q}^\OSpt)} \ar[r,"\xi \cdot \textup{dr} \otimes \xi^{\OSpt} \cdot \textup{dr}^{\OSpt}"] \ar[u]  & {\Ht^\bullet (\Ml)\otimes \Ht^{\bullet}(\Ml^\OSpt)}\ar[u] 
\end{tikzcd}
\end{equation}
commutes.
\end{cor}

\subsection{Action on orthosymplectic framed moduli stacks in $\Sp$, $\Ot_\textup{ev}$ cases }

\subsubsection{} In this section, we give a formula for the action of the preprojective CoHA on the Borel-Moore homology of the stack of $\Gt$-orthosymplectic framed quiver representations.

\subsubsection{} 
Given the sequence of groups and a framing vector $\wb \in Q_0/\sigma^{\mathbf{Z}_{\ge 0}}$ for the groups $\Gt_{\db}$ we consider the dual framing groups $\Gt_{\wb} = \prod_{i \in Q_0/\sigma} G_{\textup{fr}, w_i}$ where $\Gt_{\textup{fr}, w_i}$ preserves a form of the opposite parity from $\Gt^i_{d_i}$, as in \cite{Na1}.

We consider the framed moduli  
\[ \Ml^{\OSpt}_{\db, \wb} = \Rep_{(\db,1)}^{\OSpt}(Q^{\textup{fr}}_\wb)/G_{\db}, \qquad \Ml_{\db'} = \Rep_{(\db,1)}(Q^{\textup{fr}}_{\wb}) \]
where $Q^{\textup{fr}}_{\wb}$ is the framed quiver. Let $\Ml^{\OSpt}_{\wb} = \bigsqcup_{\db} \Ml^{\OSpt}_{\db, \wb} $. The orthosymplectic extension correspondence is given by the quotient 
\[ \Rep_{(\db',1), (\db'',1)}^{\OSpt}(Q^{\textup{fr}}_{\wb})/P_{\db', \db''} \]

We also consider the preprojective example.
\begin{lem}
  There is an action of $\Ht^{\BM}_\sbt(\Ml_{\Pi_Q})$ on $\Ht^{\BM}_\sbt(\Ml_{\Pi_{{Q}^{\textup{fr}}}}^{\OSpt}(w))$.
\end{lem}
\begin{proof}
   The inclusion of quivers $Q \hookrightarrow Q^{\textup{fr}}$ gives a map on algebras  $\Ht^{\BM}_\sbt(\Ml_{\Pi_Q}) \to \Ht^{\BM}_\sbt(\Ml_{\Pi_Q^{\textup{fr}}})$, which acts on $\Ht^{\BM}_\sbt(\Ml_{\Pi_{{Q}^{\textup{fr}}}}^{\OSpt})$. This action preserves the framing dimension vector $w$.
\end{proof}
Now using Corollary \ref{cor:preproj_intertwiner} gives a shuffle algebra formula for the action of action of the cohomological Hall algebra $\Ht^{\BM}_{\sbt}(\Ml_{\Pi_Q})$ on $\Ht(\Ml_{\Pi_{{Q}^{\textup{fr}}_{\wb}}}^{\OSpt})$. Note is that the $\tau$-fixed points of the stack of representations of the framed quiver impose a pairing on the framing space of opposite pairing of the pairing at the node.

\subsubsection{Edge contributions} 

We calculate the contributions from the framing edges explicitly.

Let $a$ be an arrow in $Q_1$. Fix $\db'$ for the ordinary quiver and $\db''$ for the orthsymplectic quiver, and an identification  $\rho: T_{\db'} \times T_{\db''} \xrightarrow{\sim} T_{\db}$. Let $x_{(1)}$ be the coordinates on $T_{\db'}$ and $x_{(1)}$ on $T_{\db''}$. Pick coorditnates $u_{i,j}, j = 1, \ldots, \wb_i$ on the framing torus at node $i$.

When the node $i$ in the orthosymplectic quiver is not in $Q_0^{\sigma}$ the contribution to the normal bundle of $\Rep_{(\db',1), (\db'',1)}^{\OSpt}(Q^{\textup{fr}}_{\wb})$ in $\Rep_{(\db,1)}(Q^{\textup{fr}}_{\wb})$ is 
\begin{align*} 
  e(Q_{\wb}^{\textup{fr}, i}) & = \prod_{j =1 }^{\wb_i}\prod_{k = 1}^{\db'_i} (u_j - x_{i, k} - \mathrm{wt}(a_{i,j}))\prod_{j =1 }^{\wb_i}\prod_{k = 1}^{\db'_i} (-u_j + x_{\sigma(i), k} - \mathrm{wt}(\overline{a_{i,j}}) ) \\ 
  & = \prod_{j =1 }^{\wb_i}\prod_{k = 1}^{\db'_i} (u_j - x_{i, k} - \mathrm{wt}(a_{i,j}))\prod_{j =1 }^{\wb_i}\prod_{k = 1}^{\db'_i} (-u_j + x_{\sigma(i), k} + \mathrm{wt}(a_{i,j}) ) 
\end{align*}
where $a_{i,j}$ is the $j$th framing arrow from $i$ to $w_i$ and $\overline{a_{i,j}}$ is the opposite way which are assumed to have opposite weights.  

If the $i$th vertex is in $Q_0^\sigma$ the contribution is 
\[ e(Q_{\wb}^{\textup{fr}, i})  = \prod_{j =1 }^{\wb_i}\prod_{k = 1}^{\db'_i} (-u_j - x_{i, k} + \mathrm{wt}(a_{i,j})) (u_j - x_{i, k} + \mathrm{wt}(a_{i,j})). \]

\begin{prop} \label{prop:COHMVectFraming}
The action of $\Al_{Q}$ on $\Ht^\sbt(\Ml^{\OSpt}_{\wb})$ is given by

\begin{equation}
  \label{eqn:COHMVect_framing} 
  f(x_{(1)})\cdot g(x_{(2)}) \ = \  \sum_{\tilde{\sigma}\in W_{\Gt_{n+m}}} \frac{e(Q_1^{\OSpt})\prod_{i \in Q_0}e(Q_{\wb}^{\textup{fr}, i}) }{e(Q_0^{\OSpt})}f(\sigma(x_{(1)}))g(\sigma(x_{(2)})).
\end{equation}
\end{prop} 
\begin{proof}
  This follows from Theorem \ref{thm:ShuffleNoPotential} and the above calculations.  
\end{proof}

By combining the action with Corollary \ref{cor:preproj_intertwiner} we obtain the following relation between the preprojective and plain cohomology actions:

\begin{cor} \label{cor:preproj_w_intertwiner}
There is an intertwining map 
\begin{center}
\begin{tikzcd}[row sep = {30pt,between origins}, column sep = {20pt}]
	{\Ht^\sbt (\Ml)\otimes \Ht^\sbt(\Ml^\OSpt_\wb)} && {\Ht^\sbt(\Ml^{\OSpt}_{\wb})} \\
	{\Ht^{\BM}_\bullet (\Ml_{\Pi_Q})\otimes \Ht^{\BM}_{\bullet}(\Ml_{\Pi_{Q_{\wb}^{\textup{fr}}}}^\OSpt)} && {\Ht^{\BM}_\bullet(\Ml_{\Pi_{Q_{\wb}^{\textup{fr}}}}^\OSpt)}
	\arrow[from=1-1, to=1-3]
	\arrow[from=2-1, to=1-1]
	\arrow[from=2-1, to=2-3]
	\arrow[from=2-3, to=1-3]
\end{tikzcd}
\end{center}
where the action on the first row is given by \eqref{eqn:COHMVect_framing}. 
\end{cor}

\subsection{Coaction} \label{ssec:coaction_shuffle}
\subsubsection{} 

Let $\Al_Q$ denote the CoHA and $\Hl^{\OSpt}$ the orthosymplectic module. More generally, $\Hl^{\OSpt}_{\pmb w}$ the framed orthosymplectic module. 

Given an orthosymplecitc dimension vector $\db$ it corresponds to a $\tau$-invariant dimension vector for regular representations for $Q$. Given $\db \in \pi_0(\Ml)$ an ordinary dimension vector we let $\theta(\db)$ denote the connected component of $\tau((V_i, \rho))$ for $(V_i, \rho) \in \Ml_{\db}$. 

The localised coaction $\Delta_L$ from Section \ref{ssec:OrthosymplecticJoyceLiu} is a Borcherds twist of the pullback along the map $\oplus^{\OSpt}$. The Borcherds twist (arising essentially from localised convolution with respect to the OSp extension diagram read the opposite way ) does not change the connected components. Thus we have a decomposition 
\[ \Delta_L = \sum_{(\db', \db'') \in \delta_{\db}} (\Delta_L)_{\db'} \] 
where \[\delta_{\db} = \{( \db' , \db'') \mid \db' + \theta(\db') + \db'' = \db \} \] 
is the set of classes of potential orthosymplectic decompositions of an element of $\Ml^{\OSpt}_{\db}$. 

We may think of $\Delta_L$ as localised convolution along the orthosymplectic extension correspondence with extensions going the opposite way. Thus if the orthosymplectic extension correspondence has maps  
\[ (a,x) \mapsfrom ( a \to y \to y/a \to \tau(a))\mapsto y \] the maps defining the coproduct are 
\[ (a,x) \mapsfrom (\tau(a) \to y \to y/\tau(a) \to a)\mapsto y. \]

Given $f(x_{\db}) \in \Ht^\bullet(\Ml^{\OSpt}) \simeq k[\tk_{\db}]^{W^{\OSpt}_{\db}}$ and a decomposition $(\db', \db'') \in \Delta_{\db}$ we write 
\[ f(x_{\db'}\otimes x_{\db''}) \in k[\tk_{\db'}]^{W_{\db'}}\otimes k[\tk_{\db''}]^{W^{\OSpt}_{\db'}}\] 
for the function pulled back along the map $\tk_{\db'}\times \tk_{\db''} \to \tk_{\db}$. 

The maps 
\[ \Delta_L \ : \  \Hl^{\OSpt}_{\wb} \ \to \  (\Al_Q\hat{\otimes}\Hl^{\OSpt}_{\wb})_{\loc}, \hspace{15mm}  \Delta_R \ : \  \Hl^{\OSpt}_{\wb} \ \to \  (\Hl^{\OSpt}_{\wb}\hat{\otimes}\Al_Q)_{\loc} ,\] 
given on $f \in \Hl^{\OSpt}_{\wb, k}$ by
\begin{equation}\label{eq:coactionL} 
\Delta_L(f) \ = \ 
 \sum_{(\db' , \db'') \in \delta_{\db}}  \frac{f(x_{\db'} \otimes x_{\db''})}{e(Q_0^{\OSpt})_{\db', \db''}/e(Q^{ \OSpt}_1)_{\db', \db''}}, \hspace{15mm}  \Delta_R(f) = \sum_{(\db' , \db'') \in \delta_{\db} }  \frac{f(x_{\db'} \otimes x_{\db''})}{e(Q_0^{\OSpt})_{\db'', \db'}/e(Q^{\OSpt}_1)_{\db'', \db'}}.
\end{equation}
define left and right coactions. We will only use the left coaction.

\subsection{Loop Cartan generators} 

\subsubsection{Extended shuffle algebra} 

For this section we will restrict to the case of the doubled quiver underlying the preprojective algebra and assume that the edge weights satisfy one of the cases of Assumption 5.7 of \cite{YZ}. For another account of bosonisation, also see \cite{JKL}.

Letting $\Al = \Ht^\sbt(\Ml)$ denote the shuffle algebra, the extended shuffle algebra $\Al^{ext}$ is the algebra $\Al\otimes \Al^0$ where $\Al^0$ is generated by the coefficients of the series 
\[ \phi_i(w) = \sum_{n \ge 0 } \varphi_{i,n} w^{-n} \]
and the algebra structure is determined by the relation for $ f \in \Ht^\sbt(\Ml_{\db})$ given by 
\begin{equation}\label{eq:loop_cartan_commutator}\phi_i(w) f \varphi_{i}(w)^{-1} = f \cup \varphi_{i, \db}(w) \end{equation}
where 
\begin{equation} \label{eq:phi}
\varphi_{i, \db}(w) = \frac{ [e(Q_1)_{\delta_i, \db} /e(Q_0)_{\delta_i, \db}](w, z_{(1)})}{[e(Q_1)_{\db, \delta_i} /e(Q_0)_{\db, \delta_i}](z_{(1)}, w)}
\end{equation} 
where $z_{(1)} \in \Ht^\sbt(\Ml_{\db})$ is the set of variables under the equivalence \eqref{eq:cohomology_is_invariants} and $\delta_i$ is dimension vector $1$ at vertex $i$, $0$ elsewhere. The expansion in \eqref{eq:phi} is given in terms of negative powers of $w$. 

Owing to our assumption on edge weights we have an identification 
\begin{equation} 
\varphi_{i, \db} \ = \ \frac{c_{-1/w}(C_i)}{c_{-1/w}(\hbar C_i)}
\end{equation} 
of $\Phi$ with the ratio of total chern classes of the tuatological complex of section \ref{sssec:OSpTautologicalComplex} at the $i$th node.

Likewise, the action on $\Ht^\bullet(\Ml_\wb)$ is given by 
\begin{equation} \label{eq:phi_action}\phi_i (w) \gamma \ = \ \frac{c_{-1/w}(C_{i, \wb})}{c_{-1/w}(\hbar C_{i, \wb})}\cup \gamma. \end{equation}
We have a coproduct on $\Al^0$ defined on generating series by 
\begin{equation}\label{eq:extended_coproduct} \Delta(\phi_i(w)) \ = \  \phi_i(w) \otimes \phi_i(w). \end{equation}

\subsubsection{Orthosymplectic loop cartan generators} 

In the orthosymplectic case only a subalgebra $\Al^{0'}$ of $\Al^0$ acts on the framed orthosymplectic modules $\Ht^\bullet(\Ml_\wb)$.

We introduce additional generating series 
\[ \psi_i(w) = \sum_{n\ge 0} \psi_{i,n} w^{-n} \] 
for $i \in Q_0/\sigma$ and define an action on $\Ht^\bullet(\Ml_\wb^{\OSpt})$ so that for $\gamma \in \Ht^\bullet(\Ml_\wb^{\OSpt})$ we have 
\begin{equation} \label{eq:psi_action}
\psi_i(w) \gamma \ = \  \frac{c_{-1/w}(C_{i, \wb})}{c_{-1/w}(\hbar C_{i, \wb})}\cup \gamma.
\end{equation} 

Choose a fixed representative $i$ for each class $i \in Q_0/\sigma$. 
There is a coaction 
\begin{equation} \label{eq:loop_cartan_coaction}
\Delta \ : \  \Al^{0'} \ \to \  \Al^0 \otimes \Al^{0'}
\end{equation} 
defined on generators by the equation
\begin{equation} \label{eq:loop_cartan_coaction_gens}
  \Delta(\psi_{i}(w)) \ = \  \varphi_{i}(w) \varphi_{\sigma(i)}(-w)^{-1} \otimes \psi_i(w).
  \end{equation} 

Via equations \eqref{eq:phi_action} and \eqref{eq:psi_action} the coaction on $\Al^{0'}$ can be thought of as being induced by the coaction $\Delta: \Ht^\bullet(\BG)\to \Ht^\bullet(\BGL)\otimes \Ht^\bullet(\BG)$ given by pullback along the map $\oplus^{\OSpt}$ where $\GL$ is the limit of groups $\GL_{\db}$ over dimension vectors $\db$ and $\Gt$ is the limit of $\Gt_{\db}$ along orthosymplectic dimension vectors $\db$.  

The action of the first tensor factor of \eqref{eq:loop_cartan_coaction_gens} acts as an expansion of cup product with the class 
\[ \frac{c_{-1/w}(C_{i, \wb} \oplus \hbar \tau^*(C_{i, \wb}))}{c_{-1/w}(\hbar C_{i, \wb} \oplus \hbar \tau^*(\hbar C_{i, \wb}))}.\]

We give an explicit formula for the OSP-extended preprojective CoHA $ H_\bullet^{\BM}(\Ml_{\Pi_Q})\otimes \Al^{0'}$, in which formula
\eqref{eq:loop_cartan_commutator} still holds with cup product replaced with the action of cohomology on Borel-Moore homology. We study this using Corollary \ref{cor:preproj_w_intertwiner}. 

For the framed module $\Ht^\sbt(\Ml^{\OSpt}_\wb)$ the action of the loop cartan generators on the vacuum $|\wb\rangle \in \Ht^\sbt(\Ml_{0}^{\OSpt})$ which is given by 
\begin{equation} 
\psi_{i}(w) |\mathbf{w}\rangle = \begin{cases} \prod_{j = 1}^{w_i} \frac{(w-u_i)(w + u_i  ) }{(w-u_i - \hbar)(w+u_i-\hbar ) } |\mathbf{w}\rangle \text{ if } i \in Q_0^\sigma \\ 
  \prod_{j = 1}^{w_i} \frac{w-u_i}{w-u_i - \hbar} |\mathbf{w}\rangle\text{ if  } i \not\in Q_0^\sigma \end{cases} 
\end{equation}

\begin{prop} 
  There is an action of $\Al^{0'}\otimes \Al$ on $\Ht^\bullet(\Ml^{\OSpt}(\wb))$ for any $\wb$. 
  \end{prop}
  
\begin{proof} 
The proof for the vertices which are not in $Q_0^\sigma$ is identical to the usual argument. For $i \in Q_0^\sigma$ we note that if $\varphi_{i}(w)$ acts on $g \in \Ht^\sbt(\Ml^{\OSpt}_{\db}(\wb))$ via 

\[ \psi_i(w) g = \prod_{j = 1}^{w_i} \frac{(w-u_i)(w + u_i  ) }{(w-u_i - \hbar)(w+u_i-\hbar ) }\frac{ [e(Q_1)^{\OSpt}_{\delta_i, \db} /e(Q_0)^{\OSpt}_{\delta_i, \db}](w, z_{(1)})}{[e(Q_1)^{\OSpt}_{\db, \delta_i} /e(Q_0)^{\OSpt}_{\db, \delta_i}](z_{(1)}, w)} \cup g \] 
and it is a straightforward coputation that this gives an action of $\Al\otimes \Al^{0'}$.
\end{proof} 

We will occasionally identify $\Ht^\bullet(\BG)$ as a subalgebra of $\Ht^{\bullet}(\BGL)$ via pullback along the map $ \BGL \to BG$ defined by $E\mapsto E \oplus \tau^*(E)$.

\subsection{Module bosonisation} 

\subsubsection{} 
We recall the Radford-Majid bosonisation construction \cite{Mdb}. One version of this construction starts with a bialgebra object $(B,m,\Delta)$ in a strict braided monoidal category $\Cl = H\Md$  with braiding given by 
\[ \beta(v\otimes w) = \sigma(\Rl^{(1)} \cdot  v \otimes \Rl^{(2)}\cdot w). \] Here $H\Md$ consists of of representations of a bialgbera object $H$ in $k\Md$. The output of the bosonisaiton construction is a new bialgebra $B\#  H$ in $k\Md$ with
its usual symmetric monoidal structure. 

In the Hopf algebra literature elements $b\otimes h$ of $B\# H \simeq B\otimes H$ are denoted $b\# h$.

Multiplication in $B\# H$ is given by the formula 
\[ b_1 \# h_1 \cdot b_2 \# h_2 = b_1 h_1^{(1)} \cdot b_2 \# h_1^{(2)} h_2. \]

The bosonised coproduct is defined by 
\[ \Delta(b\# h) = b_{(1)} \# \Rl^{(2)} h_1 \otimes \Rl^{(1)}  \cdot b_{(2)} \# h_{(2)}. \]

\subsubsection{bosonised coproduct of the shuffle algebra}

We adopt the $\#$ notation for the orthosymplectic shuffle module generalization of bosonisation owing to the added complexity of the coaction in the underlying briaded module category, which adds a layer of complexity over the simplicty of the equation \eqref{eq:extended_coproduct} which gives rise to \eqref{eq:loop_cartan_commutator}.

The extended shuffle algebra $\Al\# \Al^0$ has multiplication induced by \eqref{eq:loop_cartan_commutator} and bosoinzed (localised) coproduct extended mutliplicatively from

\begin{equation}\label{eq:boson_coaction} 
  \Delta(f\# 1) \ = \ 
   \sum_{\db' + \db'' = \db}  \frac{f_{(1)}(x_{\db'}) \# \prod_{i \in Q_0} \varphi_{i}(x_{\db''_i})\otimes  f_{(2)}(x_{\db''}) }{e(Q_0^{\OSpt})_{\db', \db''}/e(Q^{ \OSpt}_1)_{\db', \db''}}
  \end{equation}
where for a set of variables $x_{A} = x_{A_1}, \ldots x_{A_k}$ we have $\varphi_{i}(x_A) = \prod_{j = 1}^k \phi_i(x_{A_j})$ and $x_{\db''_i}$ ranges over the variables at the $i$th vertex of $Q$. We have also adopted a variant of Sweedler's notation where we implicitly sum of monomials for the pullback of $f$ along $\iota: \tk_{\db'}\times \tk_{\db''} \to \tk_{\db}$ where $f(x_{\db})\circ \iota = \sum_{i} f^i_{(1)}(x_{\db'})\otimes f^i_{(2)}(\db'')$.

We are now ready to define the orthosymplectic extended shuffle algebra. 

Consider 
\begin{equation} \label{eq:osp_ext_shuffle} 
  \Al^{ext'} = \Al \# \Al^{0'} := \Al \otimes \Al^{0'}
\end{equation} 
with multiplication defined using the coaction \eqref{eq:loop_cartan_coaction_gens} on $\Al^{0'}$ by the formula
\[(b_1 \# h_1) \cdot (b_2 \# h_2) = b_1 h_1^{(1)} \cdot b_2 \# h_1^{(2)} h_2. \] 
Explicitly this gives rise to the commutation relations for $f \in \Al_{\db}$ given by 
\begin{equation}\label{eq:bosonised_commutation_OSP} \psi_i(w) f \psi_{i}(w)^{-1} = \varphi_{i,\db}(w) \cup \varphi_{\sigma(i), \db}(-w)^{-1} \cup f \end{equation}  
where we have ommitted $\#$ from the notation. 

The bosonised coaction on $\Al^{ext'}$ is defined as $\Delta: \Al^{ext'} \to (\Al^{ext} \widehat{\otimes} \Al^{ext'})_{loc}$ extended mutliplicatively from \eqref{eq:boson_coaction} and \eqref{eq:loop_cartan_coaction_gens} on subalgebras $\Al\# 1$ and $1 \# \Al^{0'}$. We now define a bosonised coaction:

\begin{prop}
  There is a coassociative coaction
  \begin{equation} \label{eq:boson_OSp_coaction_module}
  \Delta^b_L: \Hl^{\OSpt}_{\wb} \to (\Al^{ext'}\widehat{\otimes} \Hl_{\wb}^{\OSpt})_{loc}
  \end{equation} 
\end{prop}
\begin{proof}
  We define it by the formula
  \begin{equation}
  \label{eq:bosonised_shuffle_coaction_w}
  \Delta^b_L(f) = \sum_{(\db', \db'') \in \delta_{\db}}\frac{f_{(1)}(x_{\db'})\# \prod_{i \in Q_0/\sigma} \psi_i(x_{\db''_i}) \otimes f_{(2)}(x_{\db''})}{e(Q_0^{\OSpt})_{\db', \db''}/ e(Q_1^{\OSpt})_{\db', \db''}}. 
  \end{equation} 
\end{proof}

\subsubsection{Shapovalov form}

Consider the form on $\Ht^{\bullet}(M^{\OSpt}(w))$ given by 
\[ (f, g) = \frac{ 1 }{v!}\int_{C} \frac{ f(z) g(z)  \prod_{\alpha\in R(G)} \alpha(z)}{ e(\Rep(Q)^{\OSpt}_{\wb}) } dz \] 

a direct analogue of [Negut 4.31] where $C$ is an appropriate choice of integration contour close to the locus where $|z_i| = 1$ for all $i$. 

We expect $(f,g)$ to serve as a Shapovalov form for the action of a version of the Drinfeld double for the quotient of $\Al\otimes \Al^{0'}$ in its representation on $\Hl^\OSpt_{\wb}$ at generic values of the framing weights $u_j$.

\subsection{Example: the Jordan quiver}  
\subsubsection{} 
Let $Q$ be the Jordan quiver, and $\Ml$ the stack of representations of the tripled quiver. 
The CoHA $\Ht^\bullet(\Ml)$ is isomorphic to 
\[ \bigoplus_{d \ge 0} k[ x_1, \ldots, x_d]^{W_{\GL_d}} \] 
with shuffle produt 
\begin{multline*} 
   f \cdot g = \sum_{\pi \in W_{\GL_{d' + d''}}} f(x_{\pi(1)}, \ldots, x_{\pi(d')})g(x_{\pi(d'+1)},\ldots, x_{\pi(d' + d'')})\times \\ \prod_{\substack{i = 1\\j=1}}^{d', d''}\frac{(x_{\pi(d'+j) } - x_{\pi(i)} + t_1 )(x_{\pi(d'+j) } - x_{\pi(i)} + t_2 )(x_{\pi(d'+j) } - x_{\pi(i)} - t_1 - t_2) }{x_{\pi(d'+j) } - x_{\pi(i)}}. \end{multline*}

It is known that the the preprojective CoHA is generated by elements in degree $d = 1$, and that in this degree the morphism \eqref{eq:preproj_intertwiner} is an isomorphism. 

\subsubsection{Even orthogonal instantons} 
We focus on the framed module for the preprojective CoHA corresponding to even orthogonal instantons where the gauge group of the holomorphic symplectic quotient is of the form $\Spt(2k)$. Using Corollary \ref{cor:preproj_w_intertwiner} we study using the quiver without potential.  The weights of the three loops are $t_1, t_2, -t_1-t_2$

With framing dimension $w = 2r$ and weights $u_1,\ldots u_r, -u_1, \ldots, -u_r$ by Proposition \ref{prop:COHMVectFraming} the action on $\Ht^\bullet(\Ml^{\OSpt}_w)$ for $f \in \Al_d'$ and $g \in \Ht^\bullet(\Ml^{\OSpt}_{w, d''})$ is given by 
\begin{multline} 
f\cdot g = \sum_{\varepsilon \in \mathbf Z/2\mathbf Z^{d' + d''}} \sum_{\pi \in \Sigma_{d' + d''}}  f(\varepsilon_{\pi(1)}x_{\pi(1)}, \ldots, \varepsilon_{\pi(d')}x_{\pi(d')})g(\varepsilon_{\pi(d'+1)}x_{\pi(d'+1)},\ldots, \varepsilon_{\pi(d'+d'')}x_{\pi(d' + d'')})\times \\ 
\prod_{\substack{i = 1\\ j = 1}}^{d', d''} \frac{\prod_{\xi \in \{ t_1, t_2, -t_1 -t_2\}} (-(\varepsilon_{\pi(d'+j)}x_{\pi(d'+j) })^2 + (\varepsilon_{\pi(i)}x_{\pi(i)} - \xi )^2) }{-(\varepsilon_{\pi(d'+j)}x_{\pi(d'+j) })^2 + (\varepsilon_{\pi(i)}x_{\pi(i)})^2} \times \\
\prod_{1 \le i \le k \le d'} \frac{\prod_{\xi \in \{ t_1, t_2, -t_1 -t_2\}} (\varepsilon_{\pi(i)}x_{\pi(i) } + \varepsilon_{\pi(k)}x_{\pi(k)} + \xi ) }{\varepsilon_{\pi(i)}x_{\pi(i) } + \varepsilon_{\pi(k)}x_{\pi(k)}} \times 
\prod_{\substack{i = 1\\ k = 1}}^{d', r} (-u_k - \varepsilon_{\pi(i)}x_{\pi(i)} )(u_k - \varepsilon_{\pi(i)}x_{(\pi(i))}). 
\end{multline}

\subsubsection{Symplectic instantons, odd instanton number} 

With framing dimension $w = 2r$ and framing weights $u_1,\ldots u_r, -u_1, \ldots, -u_r$ the gauge group of the holomorphic symplectic quotient is of the form $O(2k + 1)$ and by Proposition \ref{prop:COHMVectFraming} the action on $\Ht^\bullet(\Ml^{\OSpt}_w)$ for $f \in \Al_d'$ and $g \in \Ht^\bullet(\Ml^{\OSpt}_{w, d''})$ is given by 
\begin{multline} 
f\cdot g = \sum_{\varepsilon \in \mathbf Z/2\mathbf Z^{d' + d''}} \sum_{\pi \in \Sigma_{d' + d''}}  f(\varepsilon_{\pi(1)}x_{\pi(1)}, \ldots, \varepsilon_{\pi(d')}x_{\pi(d')})g(\varepsilon_{\pi(d'+1)}x_{\pi(d'+1)},\ldots, \varepsilon_{\pi(d'+d'')}x_{\pi(d' + d'')})\times \\ 
\prod_{\substack{i = 1\\ j = 1}}^{d', d''} \frac{\prod_{\xi \in \{ t_1, t_2, -t_1 -t_2\}} (-(\varepsilon_{\pi(d'+j)}x_{\pi(d'+j) })^2 + (\varepsilon_{\pi(i)}x_{\pi(i)} - \xi )^2) }{-(\varepsilon_{\pi(d'+j)}x_{\pi(d'+j) })^2 + (\varepsilon_{\pi(i)}x_{\pi(i)})^2} \times \prod_{1 \le i \le d'}\frac{\prod_{\xi \in \{ t_1, t_2, -t_1 -t_2\}} (\varepsilon_{\pi(i)}x_{\pi(i) } + \xi ) }{\varepsilon_{\pi(i)}x_{\pi(i) }} \times  \\
\prod_{1 \le i < k \le d'} \frac{\prod_{\xi \in \{ t_1, t_2, -t_1 -t_2\}} (\varepsilon_{\pi(i)}x_{\pi(i) } + \varepsilon_{\pi(k)}x_{\pi(k)} + \xi ) }{\varepsilon_{\pi(i)}x_{\pi(i) } + \varepsilon_{\pi(k)}x_{\pi(k)}} \times 
\prod_{\substack{i = 1\\ k = 1}}^{d', r} (-u_k - \varepsilon_{\pi(i)}x_{\pi(i)} )(u_k - \varepsilon_{\pi(i)}x_{(\pi(i))}). 
\end{multline}

\subsubsection{Symplectic instantons, even instanton number} 

With framing dimension $w = 2r$ and weights $u_1,\ldots u_r, -u_1, \ldots, -u_r$ by Proposition \ref{prop:COHMVectFraming} the action on $\Ht^\bullet(\Ml^{\OSpt}_w)$ for $f \in \Al_d'$ and $g \in \Ht^\bullet(\Ml^{\OSpt}_{w, d''})$ is given by 
\begin{multline} 
f\cdot g = \sum_{\varepsilon \in \mathbf Z/2\mathbf Z^{d' + d''}} \sum_{\pi \in \Sigma_{d' + d''}}  f(\varepsilon_{\pi(1)}x_{\pi(1)}, \ldots, \varepsilon_{\pi(d')}x_{\pi(d')})g(\varepsilon_{\pi(d'+1)}x_{\pi(d'+1)},\ldots, \varepsilon_{\pi(d'+d'')}x_{\pi(d' + d'')})\times \\ 
\prod_{\substack{i = 1\\ j = 1}}^{d', d''} \frac{\prod_{\xi \in \{ t_1, t_2, -t_1 -t_2\}} (-(\varepsilon_{\pi(d'+j)}x_{\pi(d'+j) })^2 +( \varepsilon_{\pi(i)}x_{\pi(i)} - \xi)^2 ) }{-(\varepsilon_{\pi(d'+j)}x_{\pi(d'+j) })^2 + (\varepsilon_{\pi(i)}x_{\pi(i)})^2} \times  \\
\prod_{1 \le i < k \le d'} \frac{\prod_{\xi \in \{ t_1, t_2, -t_1 -t_2\}} (\varepsilon_{\pi(i)}x_{\pi(i) } + \varepsilon_{\pi(k)}x_{\pi(k)} + \xi ) }{\varepsilon_{\pi(i)}x_{\pi(i) } + \varepsilon_{\pi(k)}x_{\pi(k)}} \times 
\prod_{\substack{i = 1\\ k = 1}}^{d', r} (-u_k - \varepsilon_{\pi(i)}x_{\pi(i)} )(u_k - \varepsilon_{\pi(i)}x_{(\pi(i))}). 
\end{multline}

\subsection{Example: folded linear quivers} 

\subsubsection{} 

Consider a quiver $Q$ with $2n+1$ nodes with the folding $\sigma$ fixing the central node. 

The image of the preprojective algebra under the morphism in \eqref{eq:preproj_intertwiner} is genreated by $x_{i,1}^a, a \ge 0$  in degrees $\delta_i$ 
in 
\[ \Ht^\bullet(\Ml) \simeq \bigoplus_{\db} k[x_{i,1}, \ldots, x_{i, d_i}]_{i =1, \ldots, 2n+1} ^{W_{\GL_\db}}.\] 

We consider even orthogonal framing at the fixed node. 
The module $\Ht^\bullet(\Ml^{\OSpt}_{\wb})$ is isomorphic to 
\[ \bigoplus_{\db\in \mathbf{Z}_{\ge 0}^{n}} k[x_{i, 1}, \ldots, x_{i, d_i}]_{i = 1, \ldots, n}^{W_{\Gt_\db}}. \]

Write elements of $W_{\Gt_{\db}}$ as pairs $(\varepsilon, \pi) \in (\mathbf{Z}/2)^{d_n} \ltimes \Sigma_{\db}$. Given $\varepsilon \in (\mathbf{Z}/ 2)^{d_n}$ let $x_{n, k}^\varepsilon$ denote $(-1)^\varepsilon_k x_{n, k}$. The action of $x_{i, 1}^a$ on $g \in \Ht^\bullet(\Ml^{\OSpt}_\wb)$ in degree $\db - \delta_i$ if $i < n-1$ is given by the usual shuffle algebra action on the left, and if $i> n+1$ is given by the usual shuffle algebra action of $(-1)^ax_{i -n, 1}^a$ on the right. When $i = n$ the action of $x_{n,1}^a$ is given by 
\begin{multline} 
x_{n, 1}^a \cdot g = \sum_{(\varepsilon, \pi) \in W_{\Gt_{\db}}} x_{n, 1}^{\varepsilon a} g((\varepsilon, \pi) x)\prod_{j = 2}^{d_n} \frac{x^\varepsilon_{n, \pi(j)} - x^\varepsilon_{n, \pi(1)} - t_1 - t_2 }{x^\varepsilon_{n, \pi(j)} - x^\varepsilon_{n, \pi(1)}}\times \\ 
\prod_{1 \le i \le k \le d_n} \frac{(x^\varepsilon_{n,\pi(i)} + x^\varepsilon_{n,\pi(k)} -t_1-t_2 ) }{(x^\varepsilon_{n,\pi(i)} + x^\varepsilon_{n,\pi(k)})} \times 
\prod_{\substack{i = 1\\ k = 1}}^{d_n, r} (-u_k - x^\varepsilon_{n,\pi(i)} )(u_k - x^\varepsilon_{n,\pi(i)})\\
\prod_{ 1 \le j \le d_{n-1}} 
   \left( x^\varepsilon_{n,{\pi(1)}} - x^\varepsilon_{n-1, \pi(j)} +t_1\right) 
   \prod_{1 \le j \le d_{n-1} } 
    \left( x^\varepsilon_{n,\pi(1)} + x^\varepsilon_{n-1, \pi(j) } +t_1 \right) \\
    \prod_{ 1 \le j \le d_{n-1}} 
    \left( x^\varepsilon_{n,{\pi(1)}} + x^\varepsilon_{n-1, \pi(j)} +t_2\right) 
    \prod_{1 \le j \le d_{n-1} } 
     \left( x^\varepsilon_{n,\pi(1)} - x^\varepsilon_{n-1, \pi(j) } +t_2 \right).
\end{multline} 

A similar formula holds for the action of $x_{n-1}^a$.

\subsection{Shuffle module twisted Yetter-Drinfeld condition} 
\label{ssec:ShuffleYetDrin}

\subsubsection{} 
In this section we give a direct proof using shuffle algebras of the twisted Yetter-Drinfeld compatibility between the action and left coaction of the CoHA on a framed orthosymplectic module. In order to declutter the notation of the proof we introduce shorthand notation for the localisation contributions of $e(Q_0^\OSpt)_{\db', \db''}(x_{(1)}, x_{(2)}),$ of $e(Q_1^\OSpt)_{\db', \db''}(x_{(1)}, x_{(2)}),$ of  $e(Q_0)_{\db', \db''}(x_{(1)}, x_{(2)}),$ and of $e(Q_0)_{\db', \db''}(x_{(1)}, x_{(2)})$.
We assume that there are no type II orbits of arrows in $Q_1$. 

If $A$ variables are indices for $x_{(1)}$ and $B$ variables are indices for $x_{(2)}$ then 
\begin{alignat*}{1}
(A_1\ldots A_k | B_1\ldots B_\ell)&:= e(Q_1)/e(Q_0)(x_{A_1}\cdots x_{A_k}, x_{B_1} \cdots x_{B_\ell}) \\ 
\langle A | \tau(A) \rangle &:= e(Q_1^{\tau})/e(Q_0^\tau)(x_A, \tau(x_A)) \\ 
[A_1\ldots A_k | B_1, \ldots B_k, B_\ell]&:= e(Q_1^\OSpt)/e(Q_0^\OSpt)(x_{A_1}\cdots x_{A_k}, x_{B_1} \cdots x_{B_\ell}) 
\end{alignat*}
with the convention that the last variables occuring in $[\cdots | \cdots]$ correspond to the self dual object, i.e. which already lie in a torus for a classical type group. 

The ordinary and orthosymplectic hexagon relations (section \ref{sssec:OrthosymplecticModuleHexagonRelations}) therefore become 
\begin{alignat*}{2} 
(A_1 A_2 | B_1) &= (A_1|B_1)(A_2|B_1) & (A_1 | B_1 B_2) &= (A_1 | B_1)(A_1|B_2) \\ 
[A_1A_2|B_1] &= (A_1 |\tau(A_2))[A_1|B_1][A_2|B_1]  & [A|B] &= \langle  A |\tau(A)\rangle (A|B) \\ 
[A_1 | B_1 B_2] &= (A_1 | B_1) (A_1|\tau(B_1)) [A_1| B_2]. & & 
\end{alignat*} 
We also have the relation $ (A|B) = (\tau(B)|\tau(A))$ arising from the fact that $\tau$ is an antiequivalence.

\begin{theorem} \label{thm:braided_YD_shuffle}
The modules in Theorem \ref{thm:ShuffleNoPotential} and Corollary \ref{cor:preproj_w_intertwiner} are $\tau$-twisted Yetter-Drinfeld modules with respect to the coaction in \ref{ssec:coaction_shuffle}.
\end{theorem} 
\begin{proof}  
The $\tau$-twisted compatibility for the action and coaction from Corollary \ref{cor:preproj_w_intertwiner} will follow from that corollary and the compatibility for the no potential case. For simplicity we prove the non-type $D$ case whose prove is very similar.  

Firt we calculate $\Delta_L(f(x_{(1)})\cdot g(x_{(2)}))$ for $f\in \Al_{\db}$ and $g \in \Ht^\bullet(\Ml^{\OSpt})_{\db'}$ let $(\Zb/2)^\ell$ denote the product $\Gamma_s\times \Gamma_v$ of the sign factor of $W_{\Gt_{\db + \db'}}$ with the group  $\Gamma_v$ that sends variables $x$ for vertex $i$ to corresponding variable $\tau(x)$ for vertex $\theta(i)$ supposing $\theta(i) \neq i$. Then let $W'$ the qotient of $W_{\Gt_{\db + \db'}}$ by $\Gamma_s$.

we have that the $\db, \db'$-component of the coaction is
\begin{align*} 
    \Delta_L(f(x_{(1)})\cdot g(x_{(2)}))_{\db, \db'} &= \\ 
    & \Delta_L(\sum_{\varepsilon \in (\Zb/2)^\ell}\sum_{\sigma \in W'}  \sigma(f(\varepsilon x_{(1)}) g(\varepsilon x_{(2)}) \frac{e(Q_{1, \db,\db'}^{\OSpt})}{e(Q_{0,\db,\db'}^\OSpt)})) \\ 
    &= \sum_{A_1, A_2,A_3, B_1, B_2}\sum_{\varepsilon \in (\Zb/2)^\ell}\sum_{\sigma \in W'}  \sigma(f(\varepsilon x_{A_1} \otimes \varepsilon x_{A_2} \otimes x_{A_3} ) g(\varepsilon x_{B_1} \otimes x_{B_2}) \frac{e(Q_{1, \db,\db'}^{\OSpt})}{e(Q_{0,\db,\db'}^\OSpt)} \\ 
    & \frac{e(Q_0^{\OSpt})(\varepsilon x_{A_1}  \varepsilon x_{A_2}  \varepsilon x_{B_1}, x_{A_3} x_{B_2} )}{e(Q_1^{\OSpt})((\varepsilon x_{A_1}  \varepsilon x_{A_2}  \varepsilon x_{B_1}, x_{A_3} x_{B_2} ))})
\end{align*}
where the outer sum on the last line is over partitions of the variables into the terms in the coaction, $A_1$ consists of those variables on $\tk_{\db}$ that lift to themselves under the splitting of the torus in $\tk_{\db + \tau(\db) + \db'} = \tk_{\db + \tau(\db) + \db'}$ while $A_2$ consists of variables which lift to variables in $\tk_{\tau(\db)}$. Thus $\epsilon$ only acts on the $A_1$ variables.  Given $\varepsilon \in \Gamma_s$ let $A_1^+$ and $A_1^-$ respectively denote the variables with index $i$ in $A_1$ where $\varepsilon_i$ is $0$ or $1$ respectively. Then let $\ell_\epsilon$ denote the number of minuses. We define new sets of variables $C = A_1^+, D = \epsilon A_1^- \sqcup A_2, E = A_3, F = B_1, G = B_2$. In these variables the above formula becomes 
\begin{multline*}  \Delta_L(f\cdot g)_{\db, \db'} = \sum_{A_1, A_2, A_3, B_1, B_2} \sum_{\epsilon, \sigma} \sigma(f \cdot g \frac{ [C D E | F G]}{[C \tau(D) F |  E G]})   \\  
  =\sum_{A_1, A_2, A_3, B_1, B_2} \sum_{\epsilon, \sigma} \sigma(f\cdot g \\ 
  \frac{ [CD|G](C|\tau(E))(D|\tau(E))(CD|F)(CD|\tau(F))(E|F)(E|\tau(F))[E|G] }{[C\tau(D)|G](C|\tau(F))(\tau(D)|\tau(F))(C\tau(D)|E)(C\tau(D)|\tau(E))(F|E)(F|\tau(E)) [F|G]}  )\\ 
  =   \sum_{A_1, A_2, A_3, B_1, B_2} \sum_{\epsilon, \sigma} \sigma(f\cdot g \\ 
  \frac{ (C|\tau(D))[D|G] (C|\tau(E))(D|\tau(E))(E|F)(E|\tau(F))(CD|F)(CD|\tau(F))[E|G] }{(C|D)[\tau(D)|G](C|\tau(F))(\tau(D)|\tau(F))(F|E)(F|\tau(E))(C\tau(D)|E)(C\tau(D)|\tau(E)) [F|G]}  )\\ 
  =   \sum_{A_1, A_2, A_3, B_1, B_2} \sum_{\epsilon, \sigma} \sigma(f\cdot g \\ 
  \frac{ (C|\tau(D))[D|G] (C|\tau(E))(D|\tau(E))(E|F)(C|F)(D|F)(D|\tau(F))[E|G] }{(C|D)[\tau(D)|G](\tau(D)|\tau(F))(F|E)(C\tau(D)|E)(C|\tau(E))(\tau(D)|\tau(E)) [F|G]}  )\\
  =  \sum_{A_1, A_2, A_3, B_1, B_2} \sum_{\epsilon, \sigma} \sigma(f\cdot g \\ 
  \frac{ (C|\tau(D))[D|G] (D|\tau(E))(E|F)(C|F)(D|F)(D|\tau(F))[E|G] }{(C|D)[\tau(D)|G](\tau(D)|\tau(F))(F|E)(C|E)(\tau(D)|E)(\tau(D)|\tau(E)) [F|G]}  ) \\ 
  =  \sum_{A_1, A_2, A_3, B_1, B_2} \sum_{\epsilon, \sigma} \sigma(f\cdot g \\ 
  \frac{ (C|\tau(D))[D|G] (E|\tau(D))(E|F)(C|F)(D|F)(F|\tau(D))[E|G] }{(C|D)[\tau(D)|G](F|D)(F|E)(C|E)(\tau(D)|E)(E|D) [F|G]}  )
\end{multline*} 

The coefficient of the right hand side of the $\tau$-twisted Yetter Drinfeld condition without the braiding factors, is the $(\db, \db')$ coefficent of $m_{143}\times m_{25}\circ \tau_3\Delta^2(f)\otimes \Delta_L(g)$ under the obvious identification is given by 
\begin{equation*} 
  \sum_{\epsilon', \sigma'} \sigma' (f\cdot g \frac{ (C|\tau(D))(C|F)(F|\tau(D))[E|G]}{(C|D)(C|E)(E|D)[F|G]})
\end{equation*}

These expressions agree up to the product of braiding factors 
\[ \frac{(E|F)}{(F|E)}\frac{(D|F)}{(F|D)}\frac{(E|\tau(D))}{(\tau(D)|E)} \frac{[D|G]}{[\tau(D)|G]}\] 
which is exactly the braidings required for the $\tau$-twisted Yetter-Drinfeld condition to hold in our braided module category. 
\end{proof}

\appendix

\newpage 
\section{The Cherednik reflection equation and generalisations} \label{sec:Cherednik}

\noindent In this section we show how solutions to the \textit{Cherednik reflection} and \textit{Yang-Baxter equations}
\begin{align*}
  T(w)_{2} S(z+w)_{2\tau(1)}T(z)_1 S(z-w)_{12} &\ = \ S(z-w)_{\tau(2)\tau(1)} T(z)_1 S(z+w)_{1\tau(2)}T(w)_2\\
  S_{12}(z)S_{13}(z-w)S_{23}(w) &\ = \ S_{23}(w)S_{13}(z-w)S_{12}(z)
\end{align*}
arise from braided monoidal \textit{linear-orthosymplectic} categories. For a large set of $\Gt$ we make a definition of (braided monoidal) $\Gt$-categories, generalising usual definition of category for $\Gt=\GL$ and orthosymplectic category for $\Gt=\OSpt$.

Finally, we will develop the theory of \textit{$\Gb$-Borcherds twists}. This allows us to produce many examples of interesting $\Gb$-algebraic structures starting from simpler ones, and is what we will use in section \ref{sec:JL}.

\subsection{Background on braidings} 

\subsubsection{} 
We recall the connection between configuration spaces and braided monoidal categories. Recall the \textit{braid group} is defined as
$$ B \ = \ \sqcup_{n \ge 0} B_n \ = \ \pi_1\left( (\Conf \Ab^1)_\circ \right),$$
where the subscript $\circ$ denotes the open locus of disjoint pairs of points. The $n$th piece of $\Bl = \sqcup_{n \ge 0} \Bl_n $ is generated by elements $\sigma_1,\ldots,\sigma_{n-1}$ subject to the relations
\begin{center}
  \begin{tikzpicture}[scale=0.7]
  \begin{knot}[
    clip width=5,
    flip crossing=1,
    flip crossing=2,
    flip crossing=3,
    flip crossing=4,
    flip crossing=5,
    flip crossing=6,
  ]
  \strand[black,ultra thick] (0,0) .. controls +(0,1.5) and +(0,-1.5) .. (2,3);
  \strand[black,ultra thick] (1,0) .. controls +(0,1) and +(0,-1) .. (0,1.5);
  \strand[black,ultra thick] (0,1.5) .. controls +(0,1) and +(0,-1) .. (1,3);
  \strand[black,ultra thick] (2,0) .. controls +(0,1) and +(0,-1) .. (0,3);
  \node[] at (3.5,1.5) {$=$};
  \node[] at (14,1.5) {$\sigma_i \sigma_{i+1}\sigma_i \ = \ \sigma_{i+1}\sigma_i \sigma_{i+1},$};
  
  \strand[black,ultra thick] (5,0) .. controls +(0,1.5) and +(0,-1.5) .. (7,3);
  \strand[black,ultra thick] (6,0) .. controls +(0,1) and +(0,-1) .. (5+2,1.5);
  \strand[black,ultra thick] (5+2,1.5) .. controls +(0,1) and +(0,-1) .. (6,3);
  \strand[black,ultra thick] (7,0) .. controls +(0,1) and +(0,-1) .. (5,3);
  
  \end{knot}
  \end{tikzpicture}
  \end{center}
and the othe pairs $\sigma_i,\sigma_j$ commute if $i,j$ are not adjacent. A braided monoidal category $\Cl$ is then a monoidal category with a compatible collection of isomorphisms
$$\beta_b \ : \ c_1 \otimes \cdots \otimes c_n \ \stackrel{\sim}{\to} \ c_{b(1)} \otimes \cdots \otimes c_{b(n)}, \hspace{15mm} b \ \in \  \Bl$$
for each tuple of objects and braid. 

We may attempt to make this definition for any generalised configuration space, defining 
$$ BW_\Gb \ = \ \sqcup_{\alpha} BW_{\Gb,\alpha} \ = \ \pi_1\left( (\Conf_{\Gb} \Ab^1)_\circ \right)$$
where $\circ$ means we remove all $\Gb$-hyperplanes, where $\Gb$ is a moduli group stack as in section \ref{sec:ModuliStacksRootData}. Also see \cite{Ha} for more discussion about braid groups attached to root systems; it is an intereting question what $BW_{\Ml}$ is for various moduli stacks $\Ml$.

\subsubsection{Orthosymplectic braids} 
For instance, the braid group of $\Gb=\BSp$ is the fundamental group of 
$$(\Conf_{\Sp}\Ab^1)_\circ \ = \ \sqcup_{n \ge 0} (\Ab^1//\pm)^n_\circ \ = \ \sqcup_{n \ge 0} (\Ab^1//\pm \smallsetminus 0)^n \smallsetminus \Delta$$
where we have removed diagonal and zero hyperplanes. Thus $BW_n$ is generated by $\gamma,\sigma_1,\ldots,\sigma_{n-1}$ subject to the above relations in $B_n$ and the \textit{symplectic hexagon relation}
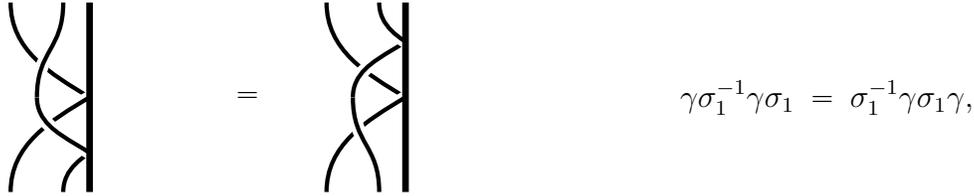
\begin{center}
  \begin{tikzpicture}[scale=0.7,yscale=0.6]
  \draw[white,line width=5pt] (-1.5,6) .. controls +(0,-3) and +(0,3) .. (1.5,0);
  \draw[black,ultra thick] (-1.5,6) .. controls +(0,-3) and +(0,3) .. (1.5,0);
  \draw[white,line width=5pt] (1.5,6) .. controls +(0,-3) and +(0,3) .. (-1.5,0);
  \draw[black,ultra thick] (1.5,6) .. controls +(0,-3) and +(0,3) .. (-1.5,0);

  \draw[white,line width=5pt] (1,3) .. controls +(0,-1.5) and +(0,1.5) .. (-0.5,0);
  \draw[black,ultra thick] (1,3) .. controls +(0,-1.5) and +(0,1.5) .. (-0.5,0);
  \draw[white,line width=5pt] (0.5,6) .. controls +(0,-1.5) and +(0,1.5) .. (1,3);
  \draw[black,ultra thick] (0.5,6) .. controls +(0,-1.5) and +(0,1.5) .. (1,3);

  \draw[white,line width=5pt] (-1,3) .. controls +(0,-1.5) and +(0,1.5) .. (0.5,0);
  \draw[black,ultra thick] (-1,3) .. controls +(0,-1.5) and +(0,1.5) .. (0.5,0);
  \draw[white,line width=5pt] (-0.5,6) .. controls +(0,-1.5) and +(0,1.5) .. (-1,3);
  \draw[black,ultra thick] (-0.5,6) .. controls +(0,-1.5) and +(0,1.5) .. (-1,3);
  
  \node[] at (3,3) {$=$};
  \node[] at (14,3) {$\gamma \sigma^{-1}_1 \gamma \sigma_1 \ = \ \sigma^{-1}_1 \gamma \sigma_1 \gamma,$};
  
  \draw[white,line width=5pt] (6-1.5,6) .. controls +(0,-3) and +(0,3) .. (6+1.5,0);
  \draw[black,ultra thick] (6-1.5,6) .. controls +(0,-3) and +(0,3) .. (6+1.5,0);
  \draw[white,line width=5pt] (6+1.5,6) .. controls +(0,-3) and +(0,3) .. (6-1.5,0);
  \draw[black,ultra thick] (6+1.5,6) .. controls +(0,-3) and +(0,3) .. (6-1.5,0);

  \draw[white,line width=5pt] (6-1,3) .. controls +(0,-1.5) and +(0,1.5) .. (6-0.5,0);
  \draw[black,ultra thick] (6-1,3) .. controls +(0,-1.5) and +(0,1.5) .. (6-0.5,0);
  \draw[white,line width=5pt] (6+0.5,6) .. controls +(0,-1.5) and +(0,1.5) .. (6-1,3);
  \draw[black,ultra thick] (6+0.5,6) .. controls +(0,-1.5) and +(0,1.5) .. (6-1,3);

  \draw[white,line width=5pt] (6+1,3) .. controls +(0,-1.5) and +(0,1.5) .. (6+0.5,0);
  \draw[black,ultra thick] (6+1,3) .. controls +(0,-1.5) and +(0,1.5) .. (6+0.5,0);
  \draw[white,line width=5pt] (6-0.5,6) .. controls +(0,-1.5) and +(0,1.5) .. (6+1,3);
  \draw[black,ultra thick] (6-0.5,6) .. controls +(0,-1.5) and +(0,1.5) .. (6+1,3);

\filldraw[white,opacity=0.9,draw=none] (1.6,-0.5)  rectangle (0,6.5);
\filldraw[white,opacity=0.9,draw=none] (1.6+6,-0.5)  rectangle (6,6.5);

\draw[line width = 2.5pt] (6,-0) -- (6,6);
\draw[line width = 2.5pt] (0,-0) -- (0,6);
  \end{tikzpicture}
\end{center} 
where $\gamma$ is generated by a loop around the origin in the first factor. 
\begin{center}
  \begin{tikzpicture}
    \begin{scope} 
      \filldraw[draw=none, fill=black, fill opacity = 0.2, pattern=north west lines] (-6*0.45,0) -- (-1*0.45,2*0.45) --  (6*0.45,0) -- (1*0.45,-2*0.45) -- cycle;
      \filldraw[draw=none, fill=white, fill opacity = 0.8] (-6*0.45,0) -- (-1*0.45,2*0.45) --  (2.5*0.45,1*0.45) -- (-2.5*0.45,-1*0.45) -- cycle;
      \draw[dashed] (2.5*0.45,1*0.45) -- (-2.5*0.45,-1*0.45);
      
      \filldraw[black, fill=none] (-6*0.45,0) -- (-1*0.45,2*0.45) --  (6*0.45,0) -- (1*0.45,-2*0.45) -- cycle;
  
      \filldraw[black, fill=white] (0,0) circle (1.5pt);

        \node[right] at (7*0.4,0) {$\Cb$};
  
        \begin{scope} 
         [xshift=1.4cm,yshift=-0.25cm] 
         \draw[-,decorate,decoration={snake,amplitude=.4mm,segment length=2mm,post length=1mm},line width = 3pt,white] (-1,0.1) to[out=120, in=180,looseness=20] (-1-0.1,0) ;
         \draw[->,decorate,decoration={snake,amplitude=.4mm,segment length=2mm,post length=1mm}]  (-1,0.1) to[out=120, in=180,looseness=20] (-1-0.1,0) ;
   
         \filldraw (-1,0) circle (1.5pt);
         \end{scope}
  
        \begin{scope} 
         \draw[-,decorate,decoration={snake,amplitude=.4mm,segment length=2mm,post length=1mm},line width = 3pt,white] (1,-0.2+0.1) to[out=45, in=180-45] (1.5,-0.2+0.1) ;
         \draw[->,decorate,decoration={snake,amplitude=.4mm,segment length=2mm,post length=1mm}]  (1,-0.2+0.1) to[out=45, in=180-45] (1.5,-0.2+0.1) ;
         
         \draw[-,decorate,decoration={snake,amplitude=.4mm,segment length=2mm,post length=1mm},line width = 3pt,white] (1.5,-0.2-0.1) to[out=-45, in=-180+45] (1,-0.2-0.1) ;
         \draw[->,decorate,decoration={snake,amplitude=.4mm,segment length=2mm,post length=1mm}]  (1.5,-0.2-0.1) to[out=-45, in=-180+45] (1,-0.2-0.1) ;
      
         \filldraw (1.5,-0.2) circle (1.5pt);
         \filldraw (1,-0.2) circle (1.5pt); 
         \end{scope}
  
     \end{scope}
   \end{tikzpicture}
  \end{center}

\subsubsection{Linear-linear braids} When $\Gb=\BP_{\GL \textup{-}\GL}$ is given by a standard parabolic with Levi $\GL \times \GL$, the $\Gb$-configuration space is the configuration space of two-coloured points on $\Ab^1$.

Thus its braid group $B_{n,m}$ is generated by commuting subgroups $B_n,B_m$ corresponding to braids between points of the same colour, and an element $\delta$ corresponding to a loop of a point of one colour around a point of the other.
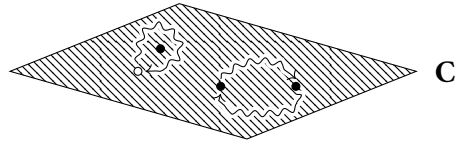
\begin{center}
\begin{tikzpicture}
  \begin{scope} 
    \filldraw[black, fill=black, fill opacity = 0.2, pattern=north west lines] (-6*0.45,0) -- (-1*0.45,2*0.45) --  (6*0.45,0) -- (1*0.45,-2*0.45) -- cycle;

      \node[right] at (7*0.4,0) {$\Cb$};

      \draw[-,decorate,decoration={snake,amplitude=.4mm,segment length=2mm,post length=1mm},line width = 3pt,white] (-1,0.1) to[out=90, in=0,looseness=20] (-1+0.1,0) ;
      \draw[->,decorate,decoration={snake,amplitude=.4mm,segment length=2mm,post length=1mm}]  (-1,0.1) to[out=90, in=0,looseness=20] (-1+0.1,0) ;

      \filldraw[draw=black,fill=white] (-1,0) circle (1.5pt);

      \filldraw (-0.7,0.3) circle (1.5pt);

      \begin{scope} 
       [xshift=-0.4cm]
       \draw[-,decorate,decoration={snake,amplitude=.4mm,segment length=2mm,post length=1mm},line width = 3pt,white] (0.5,-0.2+0.1) to[out=45, in=180-45] (1.5,-0.2+0.1) ;
       \draw[->,decorate,decoration={snake,amplitude=.4mm,segment length=2mm,post length=1mm}]  (0.5,-0.2+0.1) to[out=45, in=180-45] (1.5,-0.2+0.1) ;
       
       \draw[-,decorate,decoration={snake,amplitude=.4mm,segment length=2mm,post length=1mm},line width = 3pt,white] (1.5,-0.2-0.1) to[out=-45, in=-180+45] (0.5,-0.2-0.1) ;
       \draw[->,decorate,decoration={snake,amplitude=.4mm,segment length=2mm,post length=1mm}]  (1.5,-0.2-0.1) to[out=-45, in=-180+45] (0.5,-0.2-0.1) ;

       \filldraw (1.5,-0.2) circle (1.5pt);
       \filldraw (0.5,-0.2) circle (1.5pt); 
       \end{scope}

   \end{scope}
 \end{tikzpicture}
\end{center}
Note that as before, $\sigma_1,\delta$ satisfy the symplectic hexagon relation.

\subsubsection{Linear-orthosymplectic braids} Finally, when $\Gb=\BP_{\GL \textup{-}\Sp}$, the braid group $BW_{n,m}$ is generated by $B_{n}, B_{\Sp,m}$ along with an element $\delta$ corresponding to a loop of a $\GL$ around an $\Sp$ point. Notice that in this group the elements $\gamma,\delta$ are different.

\subsubsection{Remark} Due to  Lurie \cite[5.5.4.10]{Lu} we have an equivalence
$$\Eb_2 \textup{-Cat} \ \simeq\ \FactAlg^{\textup{const}.}(\Ran \Cb,\Cat)$$
between $\Eb_2$-categories (which are essentially braided monoidal categories, see e.g. \cite[7.0.2]{CF}) and factorisable constructible cosheaves of categories on the topological space $\Ran \Cb$. Thus, to define $\Gb$-braided monoidal category for moduli group stack $\Gb$, we just need to define a topological space $\Ran_\Gb\Cb$ with a factorisation structure; this is done in section \ref{sec:OrthosymplecticAlgebras} for the orthosymplectic case. 

We will merely use this as a way to guess the concrete definitions below, and leave the relation between the two definitions a conjecture generalising \cite[7.0.2]{CF}:
 \begin{conj}
 Let $(\El,\otimes)$ be a $\Gt$-monoidal dg category with t-structure for $\Gt$ as in section \ref{ssec:GMonBr}. Then this induces a $\Eb_1$-factorisation cosheaf of categories $\tilde{\El}$ on $\Ran \Cb$, and there is an equivalence between the category
 $$\{\Gt \textup{-braidings of }(\El^\heartsuit,\otimes)\} $$
 and the category of lifts of $\tilde{\El}$ along
 $$\oblv \ : \ \Eb_2\FactAg^{\textup{const}.}(\Ran_{\Gt} \Cb,\Cat) \ \to \ \Eb_1\FactAg^{\textup{const}.}(\Ran_{\Gt} \Cb).$$
 \end{conj}

\subsection{$\Gt$-braided monoidal and braided module categories} \label{ssec:GMonBr}

\subsubsection{} In this section, we define analogues of braided monoidal categories for various groups $\Gt$, specialising to the usual notion when $\Gt=\GL$. These will form symmetric monoidal 2-categories $\textup{BrCat}_\Gt$, satisfying
$$\textup{BrCat}_{\Gt_1} \times \textup{BrCat}_{\Gt_2} \ \simeq \ \textup{BrCat}_{\Gt_1 \times \Gt_2}$$
and with unit denoted $\triv_\Gt$.

\subsubsection{Ordinary braidings} A (\textit{linear}) \textit{monoidal category} $\Al$ is a category with object $1_\Al$ a functor
$$\otimes_\Al \ : \ \Al \times \Al \ \to \ \Al$$
satisfying an associativity and unit condition 
$$\otimes_\Al\cdot (\otimes_\Al\times \id) \ \simeq \ \otimes_\Al\cdot (\id\times \otimes_\Al), \hspace{15mm} 1_\Al\otimes_\Al(-) \ \simeq\ (-)\otimes_\Al 1_\Al\ \simeq \ \id,$$
with these isomorphisms satisfying coherence conditions. 
A \textit{braiding} is a natural isomorphism $\beta_\Al  :  \otimes_\Al  \stackrel{\sim}{\to}  \sigma^*\otimes_\Al$ satisfying the hexagon and unit relations, and is called \textit{symmetric} if $\beta_\Al\cdot \sigma^*\beta_\Al = \id$.

\begin{lem} \label{lem:GLMonoidal}
 If $\Al$ is monoidal, then for any map $f:[m]\to [n]$ we have a functor
  $$\otimes_\Al^f \ : \ \Al^{\times m} \ \to \ \Al^{\times n}$$ 
  which together satisfy an associativity condition 
  \begin{equation}
    \label{eqn:HigherAssociativityCondition} 
    \otimes_\Al^f\cdot (\otimes_{\Al}^{\Gt_1}\times \cdots \times \otimes_{\Al}^{\Gt_m}) \ = \ \otimes_{\Al}^{f\cdot g}
  \end{equation}
  for any surjection $g: [l]\twoheadrightarrow  [m]$ where we write $\Gt_i=g\vert_{g^{-1}(i)}$. 
\end{lem}
\begin{proof}
  We write $f$ as a composition of an order-preserving injection, a permutation, and an order-preserving surjection:
  \begin{equation}
  \label{eqn:FactorisationOfMap} 
  [m] \ \stackrel{i}{\to} \ [m'] \ \stackrel{\mu}{\to} \ [n'] \ \stackrel{s}{\to} \ [n].
  \end{equation}
  We define $\otimes_\Al^i$ as a product of $\id$ and $1_\Al$ factors, define $\otimes_\Al^\mu:\Al^{\times n}\stackrel{\sim}{\to}\Al^{\times n}$ as the associatied permutation, and writing $s$ as a composition of maps $s_i:[l]\twoheadrightarrow [l-1]$ given by contracting the $i,i+1$th elements, we define $\otimes^s_\Al$ as the composition of $\otimes_\Al^{s_i}= \id \times (\otimes_\Al)_{i,i+1}\times \id$. Independence of this definition on the choice of expression and \eqref{eqn:HigherAssociativityCondition} follow from the associativity and unit condition on $\otimes_\Al$.
\end{proof}

We now recall the notion of relative braid. Let $\FinSet_\Dl$ be a category of finite sets with extra structure as in section \ref{ssec:FiniteSets}. Given any morphism
$$f \ : \ ([n],d_n) \ \to \ ([m],d_m)$$
 we define the \textit{relative Weyl group} $W_f$ the relative automorphism group of $f$, i.e. the isomorphisms $\sigma$:
\begin{center}
\begin{tikzcd}[row sep = {30pt,between origins}, column sep = {20pt}]
 ([n],d_n)\ar[r,"f"]\ar[d,"\sigma","\wr"'] & ([m],d_m)\ar[d,equals] \\
 ([n],d_n)\ar[r,"f"] & ([m],d_m) 
\end{tikzcd}
\end{center}
 We may likewise define relative braid groups $BW_f$, either as in the introduction or alternatively like in section \ref{ssec:FiniteSets} define it as relative isomorphisms in a category $\FinSet_\Dl^{\textup{br}}$ of finite sets equipped with an embedding into the two-disk $\Dt^2$ and a compatible $\Dl$-structure.

 In the current case of trivial extra data $\FinSet_\varnothing=\FinSet$, we have 
 $$W_f \ = \   \smprod_{i \in [m]}\Sk_{f^{-1}(i)}, \hspace{15mm} B_f \ = \ \smprod_{i \in [m]}B_{f^{-1}(i)}$$
 and in general for a surjection $W_f$ and $BW_f$ is a product over appropriate Weyl or braid groups of preimages.

\begin{lem} \label{lem:GLBraided}
   If $\Al$ is braided monoidal, then for any map $f:[m]\to [n]$ of finite sets we have a natural isomorphism
\begin{center}
  \begin{tikzcd}[row sep = 20pt, column sep = 30pt]
  \Al^{\times m}\ar[r,"\otimes_{\Al}\cdot \sigma"{xshift=0pt},bend left = 35, ""{name=U,inner sep=0pt,below,xshift=0pt}]\ar[r,"\otimes_{\Al}"{xshift=-0pt},bend right = 35,swap, ""{name=D,inner sep=1pt}]&\Al^{\times n} 
    \arrow[Rightarrow, from=U, to=D, "\beta^f_\sigma", shorten <= 5pt, shorten >= 5pt] 
    \arrow[Rightarrow, from=U, to=D, swap,"\wr", shorten <= 5pt, shorten >= 5pt]
  \end{tikzcd}
  \end{center} 
    for any relative braid $\sigma\in B_f$ (acting on $\Al^{\times m}$ via its projection to $\Sk_f$), such that is compatible with multiplication in the relative braid group
    \begin{equation}
      \label{eqn:BetaSigmaCompatibility} 
      \beta_{\sigma_1}^f\cdot \beta_{\sigma_2}^f \ = \ \beta_{\sigma_1\cdot \sigma_2}^f
    \end{equation}
    and compatible with \eqref{eqn:HigherAssociativityCondition}:
    \begin{equation}
      \label{eqn:BetaFCompatibility} 
      \beta^{f\cdot g}_{\sigma \star \mu} \ = \ \beta^f_\sigma\cdot (\beta^{\Gt_1}_{\mu_1} \times \cdots \times\beta_{\Gt_m}^{\mu_m})
    \end{equation}
    where $f,g$ are composable functions with $\Gt_i=g\vert_{g^{-1}(i)}$, and $\mu_i$ is a permutation of $\Gt^{-1}(i)$. If it is symmetric, $\beta_{(-)}^f$ factors through the surjection $BW_f\to\Sk_f$.
\end{lem}
\begin{proof}
  Again, let us write $f$ as a composition \eqref{eqn:FactorisationOfMap}, we reduce to the case that $f=s$ is an order-preserving surjection; indeed $BW_f$ is trivial for injections and isomorphisms so we define $\beta^f_1$ the trivial natural isomorphism. 

  Since $s$ is a composition of adjacent contractions $s_i$, it suffices to define $\beta^{s_i}_{\sigma_{i,i+1}}=\id \times (\beta)_{i,i+1}\times \id$, thus defining $\beta^f_\sigma$ for every $f$ and $\sigma\in BW_f$. 
\end{proof}

\subsubsection{Generalisations}  \label{sssec:GBrMon} We make Lemmas \ref{lem:GLMonoidal} and \ref{lem:GLBraided} into definitions. See the following sections for explicit examples.

Recall section \ref{ssec:FiniteSets} for details on category of finite $\Gt$-sets for $\Gt$ has classical or standard parabolic type; then a map 
$$f \ : \ ([m_1|\cdots | m_k], d_{m_i}) \ \to \ ([n_1|\cdots| n_k],d_{n_i})$$
 of braided finite $\Gt$-sets is equivalent to an entry path in $\Conf_\Gt\Cb$.

A \textit{$\Gt$-braided monoidal category} $\Al$ consists of a category $\Al_c$ for every simple colour $c$, and 
\begin{itemize}
 \item a functor 
 $$\otimes^{f}_{\Al} \ : \  \Al_{c_1}^{\times m_1}\times \cdots \times \Al^{\times m_k}_{c_k} \ \to \ \Al_{c_1}^{\times n_1}\times \cdots \times \Al_{c_k}^{\times n_k}$$
 for every map $f$ which together satisfy an associativity condition 
 \begin{equation}
   \label{eqn:HigherAssociativityCondition} 
   \otimes_\Al^f\cdot (\otimes_{\Al}^{\Gt_{(1)}}\times \cdots \times \otimes_{\Al}^{\Gt_{(m)}}) \ = \ \otimes_{\Al}^{f\cdot g}
 \end{equation}
 for any composable $\Gt$ where we write $\Gt_{(i)}=g\vert_{g^{-1}(i)}$, and such that if $f=\sigma_{i,i+1}$ is a simple braiding of elements with colour $c$ then
 \begin{equation}
  \otimes^{\sigma_{i,i+1}}_\Al \ = \ (\sigma_{\Al_c})_{i,i+1}
 \end{equation}
  is the induced by the swap map $\Al_c\times \Al_c \stackrel{\sim}{\to} \Al_c\times \Al_c$,
 \item a natural transformation 
  \begin{center}
    \begin{tikzcd}[row sep = 20pt, column sep = 30pt]
    \Al_{c_1}^{\times m_1}\times \cdots \times \Al_{m_k}^{c_k}\ar[r,"\otimes_{\Al}^{f_1}"{xshift=0pt},bend left = 35, ""{name=U,inner sep=0pt,below,xshift=0pt}]\ar[r,"\otimes_{\Al}^{f_2}"{xshift=-0pt},bend right = 35,swap, ""{name=D,inner sep=1pt}]&\Al_{c_1}^{n_1}\times \cdots \times \Al_{c_k}^{n_k} 
      \arrow[Rightarrow, from=U, to=D, "\beta^\eta", shorten <= 5pt, shorten >= 5pt] 
      \arrow[Rightarrow, from=U, to=D, swap,"\wr", shorten <= 5pt, shorten >= 5pt]
    \end{tikzcd}
    \end{center} 
    for every homotopy of braids $\eta:f_1\to f_2$, which is compatible with horizontal and vertical composition of 2-morphisms:
    \begin{equation}
      \beta^{f_1,f_2}_\eta \cdot \beta^{f_1,f_2}_\eta \ = \ \beta^{f_1,f_2}_{\eta\cdot \eta'}, \hspace{15mm}  \beta^{f_1,f_2}_{\eta_f}\cdot \beta^{\Gt_1,g_2}_{\eta_g} \ = \ \beta^{f_1g_1, f_2 g_2}_{\eta_f \star \eta_g} 
    \end{equation}
    for any $\Gt_i$ composable with $f_i$. 
\end{itemize}

\subsubsection{Orthosymplectic braidings} We spell out the above definition in the orthosymplectic case $\Gt=\OSpt$, giving a simpler description analogous to the usual definition of braided monoidal category.

\begin{prop} \label{prop:OrthosymplecticBraided}
  An orthosymplectic braided monoidal category $[\Bl]$ is equivalent to a category $\Bl$ with 
  \begin{equation}
    \label{eqn:BFunctors} 1_\Bl \ \in \ \Bl, \hspace{15mm} \kappa \ : \ \Bl \ \stackrel{\sim}{\to} \ \Bl, \hspace{15mm} \otimes_\Bl \ : \ \Bl \times \Bl \ \to \ \Bl,
  \end{equation}
  satisfying the axioms for a monoidal category as well as $\kappa\cdot \otimes_\Bl \simeq \otimes_\Bl \cdot (\kappa \times \id)\simeq \otimes_{\Bl}\cdot (\id \times \kappa)$ and $\kappa(1_\Bl)\simeq 1_\Bl$, secondly a natural isomorphism
\begin{equation}
  \label{eqn:BNatTrans} 
  \begin{tikzcd}[row sep = 20pt, column sep = 30pt]
  \Bl\times\Bl\ar[r,"\otimes_{\Bl}\cdot \sigma"{xshift=-10pt},bend left = 35, ""{name=U,inner sep=0pt,below,xshift=0pt}]\ar[r,"\otimes_{\Bl}"{xshift=-0pt},bend right = 35,swap, ""{name=D,inner sep=1pt}]&\Bl 
  \arrow[Rightarrow, from=U, to=D, "\beta", shorten <= 5pt, shorten >= 5pt] 
  \arrow[Rightarrow, from=U, to=D, swap,"\wr", shorten <= 5pt, shorten >= 5pt]
  \end{tikzcd}
  \end{equation} 
   which satisfies the hexagon
  $$\beta_{b_1 \otimes_\Bl b_2,b_3} \ = \ \beta_{b_1,b_3} \cdot \beta_{b_2,b_3}, \hspace{15mm} \beta_{b_1,b_2 \otimes_\Bl b_3} \ = \ \beta_{b_1,b_2} \cdot \beta_{b_1,b_2\otimes_\Bl b_3}$$
  and symplectic hexagon relations
  \begin{equation} \label{eqn:SymplecticBraidRelation}
    \kappa_{b_1}\beta_{b_1,\kappa(b_2)}^{-1} \kappa_{b_2}\beta_{b_1,b_2} \ = \ \beta_{\kappa(b_1),\kappa(b_2)}^{-1}\kappa_{b_2}\beta_{\kappa(b_1),b_2}\kappa_{b_1},
  \end{equation}
  and is compatible with $1_\Bl$.  It is symmetric if  $\beta\cdot \sigma^*\beta$ is the identity.
\end{prop}
\begin{proof}
  The functors \eqref{eqn:BFunctors} arise from $\otimes^{f}_{\Bl}$ for the maps $f\in\FinSet_{\OSpt}$
  $$\pm [0] \ \stackrel{i}{\to} \ \pm [1], \hspace{10mm} \pm [1] \ \stackrel[\raisebox{2pt}{$\sim$}]{-1}{\to} \  \pm [1], \hspace{10mm} \pm [2] \ \ \stackrel{s}{\to} \  \ \pm [1],$$
  respectively,  where the first and third are coCartesian lifts of the associated maps of finite sets. The relations follow from the identifications of 1-morphisms in $\FinSet_{\OSpt}$:
  $$(-1) \cdot i \ = \ i, \hspace{15mm} s\cdot (i\times i) \ = \ i, $$
  $$(-1)\cdot s \ = \ s\cdot ((-1)\times \id) \ = \ s\cdot (\id \times (-1)), \hspace{15mm} s \cdot (\id \times s) \ = \ s\cdot (s\times \id).$$
  Conversely, any such data gives rise to a unique functor $\otimes^{f}_{\Bl}$ for every map $f$.  The natural isomorphisms \eqref{eqn:BNatTrans} arise from the generators of the relative braid group
  $$BW_{\pm [2] \to \pm [1]} \ \simeq \ \Zb \sigma.$$
   The hexagon relations, symplectic hexagon relations and compatibility with $1_\Bl$ follow from \eqref{eqn:BetaFCompatibility}.
\end{proof}

\subsubsection{Parabolic braided monoidal categories}  
Likewise we may define \textit{linear-orthosymplectic} braided monoidal categories, and



\begin{prop}\label{prop:GLOSpBrMon}
  A linear-orthosymplectic braided monoidal category $(\Al|\Bl]$ is equivalent to a pair of categories $\Al,\Bl$ with 
  \begin{equation}
    \begin{split} \label{eqn:ABFunctors}
       1_\Al \ \in \ \Al, \hspace{10mm} 1_\Bl \ \in \ \Bl, \hspace{25mm} \\
       \lambda \ : \ \Al \times\Bl \ \stackrel{\sim}{\to} \ \Al \times \Bl,\hspace{10mm}   \kappa \ : \ \Bl \ \stackrel{\sim}{\to} \ \Bl \hspace{15mm}  \\
    \otimes_{\Al} \ : \ \Al \times \Al  \ \to \ \Al \hspace{12mm} \textcolor{white}{(\otimes_{\Al\Bl} \ : \ \Al \times \Bl   \ \to \ \Bl)} \\
    \otimes_{\Al\Bl} \ : \ \Al \times \Bl   \ \to \ \Bl  \hspace{15mm}  \otimes_{\Bl} \ : \ \Bl \times \Bl  \ \to \ \Bl\\
  \end{split}
  \end{equation}
  such that $\otimes_{\Al},(\otimes_{\Bl},\kappa)$ are (orthosymplectic) monoidal products with units $1_\Al,1_\Bl$ and $\otimes_{\Al\Bl}$ is linear over $\otimes_{\Al},\otimes_{\Bl}$, such that 
  $$\kappa\cdot \otimes_{\Al\Bl} \ \simeq \ \otimes_{\Al\Bl}\cdot (\id \times \kappa), \hspace{15mm} 1_\Al \otimes_{\Al \Bl}(-)\ \simeq \ \id_\Bl$$
  and secondly a collection of natural isomorphisms
  \begin{equation}
    \label{eqn:BNatTransGLOSp} 
    \begin{tikzcd}[row sep = 20pt, column sep = 30pt]
     \Al\times\Al\ar[r,"\otimes_{\Al}\cdot \sigma"{xshift=-10pt},bend left = 35, ""{name=UL,inner sep=0pt,below,xshift=0pt}]\ar[r,"\otimes_{\Al}"{xshift=-0pt},bend right = 35,swap, ""{name=DL,inner sep=1pt}]&\Al 
    &&
    \Al\times\Bl\ar[r,"\otimes_{\Al\Bl}\cdot \lambda"{xshift=-10pt},bend left = 35, ""{name=UR,inner sep=0pt,below,xshift=0pt}]\ar[r,"\otimes_{\Al\Bl}"{xshift=-0pt},bend right = 35,swap, ""{name=DR,inner sep=1pt}]&\Bl
    &&
    \Bl\times\Bl\ar[r,"\otimes_{\Bl}\cdot \sigma"{xshift=-10pt},bend left = 35, ""{name=U,inner sep=0pt,below,xshift=0pt}]\ar[r,"\otimes_{\Bl}"{xshift=-0pt},bend right = 35,swap, ""{name=D,inner sep=1pt}]&\Bl 
      \arrow[Rightarrow, from=U, to=D, "\beta_\Bl", shorten <= 5pt, shorten >= 5pt] 
      \arrow[Rightarrow, from=U, to=D, swap,"\wr", shorten <= 5pt, shorten >= 5pt]
      \arrow[Rightarrow, from=UL, to=DL, "\beta_\Al", shorten <= 5pt, shorten >= 5pt] 
      \arrow[Rightarrow, from=UL, to=DL, swap,"\wr", shorten <= 5pt, shorten >= 5pt]
      \arrow[Rightarrow, from=UR, to=DR, "\eta_{\Al\Bl}", shorten <= 5pt, shorten >= 5pt] 
      \arrow[Rightarrow, from=UR, to=DR, swap,"\wr", shorten <= 5pt, shorten >= 5pt]
    \end{tikzcd}
    \end{equation} 
  such that $\beta_\Al$ and $(\beta_\Bl,\kappa)$ satisfy the (orthosymplectic) hexagon relations, and $\beta_\Al,\eta_{\Al\Bl},\beta_\Bl$ satsfies the symplectic hexagon relations for $\otimes_{\Al\Bl}$:\footnote{Here we have abused notation and denoted $\lambda=\tau \times \tau'$.}
  \begin{equation}\label{eqn:EtaBetaSymplecticHexagonRelations}
    \beta_{\tau(a),\tau(a')}^{-1}\cdot \eta_{a,\tau'(b)}\cdot \beta_{\tau(a'),\tau'(b)}\cdot \eta_{a',b} \ = \ \eta_{a',\tau'(b)}\cdot \beta_{a',\tau(a)}^{-1}\cdot \eta_{a,b}\cdot \beta_{a,a'}.
  \end{equation}
\end{prop}
\begin{proof}
  The functors \eqref{eqn:ABFunctors} arise from $\otimes^{f}_{(\Al,\Bl]}$ for the maps $f\in\FinSet_{\GL \textup{-}\OSpt}$ given by
  $$(0|0] \ \to \ (1|0], \hspace{10mm} (0|0] \ \to \ (0|1], \hspace{10mm} (-1) \ : \ (0|1] \ \stackrel{\sim}{\to} \ (0|1],$$
  \begin{align*}
    (2|0]& \ \twoheadrightarrow\ (1|0] & &  \\
    (1|1]& \ \twoheadrightarrow \ (1|0]  & & (0|2] \ \twoheadrightarrow \ (0|1]
  \end{align*}
   and the compatibility conditions from the relations between their iterated compositions. This uniquely determines all the $\otimes_{(\Al|\Bl]}^f$ since any map $f$ is a composition of the above maps with permutations. The braid relations follow by a similar argument as in the proof of Proposition \ref{prop:OrthosymplecticBraided}.
\end{proof}

Likewise, we define a \textit{linear-linear braided monoidal category} $(\Al|\Bl)$ as the same data and conditions as Proposition \ref{prop:GLOSpBrMon} without $\kappa$, similarly for a parabolic $\Pt$ with Levi $\GL^{\times k}\times \Gt$ where $\Gt$ is of classical type.

\subsubsection{Braided module categories} Let $(\Al,\Bl]$ be a linear-orthosymplectic braided monoidal category or $(\Al,\Bl)$ be a linear-linear braided monoidal category.

\begin{cor}
 $\Bl$ is a (lax) braided module category for $\Al$: there is a natural isomorphism 
\begin{center}
  \begin{tikzcd}[row sep = 20pt, column sep = 30pt]
  \Al\times\Bl\ar[r,"\otimes_{\Al\Bl}\cdot \lambda"{xshift=-10pt},bend left = 35, ""{name=U,inner sep=0pt,below,xshift=0pt}]\ar[r,"\otimes_{\Al\Bl}"{xshift=-0pt},bend right = 35,swap, ""{name=D,inner sep=1pt}]&\Bl 
    \arrow[Rightarrow, from=U, to=D, "\eta", shorten <= 5pt, shorten >= 5pt] 
    \arrow[Rightarrow, from=U, to=D, swap,"\wr", shorten <= 5pt, shorten >= 5pt]
  \end{tikzcd}
  \end{center} 
satisfying the hexagon relations with respect to $\otimes_\Al$.
\end{cor}

This recovers the notion \cite{Br} of strict braided module category when $\lambda$ is trivial. In terms of factorisation algebras, one expects the above is induced by pulling back $\Bl$ along $0 \to \Cb/\pm$.

\subsection{Cherednik and Yang-Baxter equations} 

\label{ssec:CherednikSetup}

\subsubsection{} We are now able to give a uniform description of matrices satisfying the Cherednik reflection \cite{Che} or the Yang-Baxter \cite{ES} equations.

\subsubsection{Yang Baxter twists} We now introduce \textit{Yang-Baxter twists} $S$ of braided monoidal categories $\Al$:

\begin{prop} \label{prop:RMatrixBraiding}
Let $(\Al,\otimes_\Al,\beta)$ be a braided monoidal category, and let 
$$S_{a,a'} \ : \ a \otimes_\Al a' \ \stackrel{\sim}{\to} \ a \otimes_\Al a'$$
be maps in $\Al$ natural in the variables $a,a'$. Then $S\cdot \beta$ defines a braiding if and only if $S$ satisfies the hexagon relations
\begin{equation}
  \label{eqn:SHexagonRelations} 
  S_{a \otimes a', a''} \ = \ S_{13}S_{23}, \hspace{15mm}  S_{a,a' \otimes a''} \ = \ S_{12}S_{13}
\end{equation}
and if $S_{1,a}=S_{a,1}=\id$. 
\end{prop}
\begin{proof}
  This follows by a direct computation. Alternatively, viewing an ordinary braiding of a monoidal category $(\Al,\otimes_\Al)$ as a 2-isomorphism 
  \begin{center}
    \begin{tikzcd}[row sep = 20pt, column sep = 30pt]
    \Al\times\Al\ar[r,"\otimes_{\Al}"{xshift=-10pt},bend left = 35, ""{name=UR,inner sep=0pt,below,xshift=0pt}]\ar[r,"\otimes_{\Al}\cdot \sigma_{\Cat}"{xshift=-0pt},bend right = 35,swap, ""{name=DR,inner sep=1pt}]&\Al
      \arrow[Rightarrow, from=UR, to=DR, "\beta", shorten <= 5pt, shorten >= 5pt] 
      \arrow[Rightarrow, from=UR, to=DR, swap,"\wr", shorten <= 5pt, shorten >= 5pt]
    \end{tikzcd}
    \end{center} 
  satisfying the hexagon and unit relations, it is clear that we get a new braiding if we vertically compose with another 2-isomorphism 
  \begin{center}
    \begin{tikzcd}[row sep = 20pt, column sep = 30pt]
    \Al\times\Al\ar[r,"\otimes_{\Al}"{xshift=-10pt},bend left = 35, ""{name=U,inner sep=0pt,below,xshift=0pt}]\ar[r,"\otimes_{\Al}"{xshift=-0pt},bend right = 35,swap, ""{name=D,inner sep=1pt}]&\Al 
      \arrow[Rightarrow, from=U, to=D, "S", shorten <= 5pt, shorten >= 5pt] 
      \arrow[Rightarrow, from=U, to=D, swap,"\wr", shorten <= 5pt, shorten >= 5pt]
    \end{tikzcd}
    \end{center} 
  which also satisfies the hexagon and unit relations; this is precisely the data of $S$.
\end{proof}
Moreover, if $\beta_1,\beta_2$ are any two braidings of $(\Al,\otimes_\Al)$, then $S  =  \beta_2\cdot \beta_1^{-1}$ satisfies the hexagon relations for $\beta_1$. Thus loosely speaking, Yang-Baxter $S$-matrices may be viewed as automorphisms in the category of braidings on a fixed monoidal category. 

\begin{cor}
   $S$ satisfies the Yang-Baxter-equation 
$$S_{12}S_{13}S_{23}  \ = \ S_{23}S_{13}S_{12}.$$ 
\end{cor}

\subsubsection{Remark} To be precise, by equation \eqref{eqn:SHexagonRelations} we mean 
$$S\cdot (\otimes_\Bl \times\id) \ = \ S_{13}S_{23} , \hspace{15mm}  S\cdot (\id \times \otimes_\Bl) \ = \ S_{12}S_{13}$$
where we use the notation $S_{12}=S \otimes\id$, $S_{23}=\id \otimes S$ and $S_{13}= (\id \times \beta)S_{12}(\id \times \beta)^{-1}$.

\subsubsection{Orthosymplectic Yang-Baxter twists} We now write down the orthosymplectic version of the Yang-Baxter equation, using Proposition \ref{prop:OrthosymplecticBraided} as a definition of orthosymplectic braided monoidal category.

\begin{prop} \label{prop:RMatrixBraidingOSp}
  Let $(\Bl,\otimes_\Bl,\kappa,\beta)$ be an orthosymplectic braided monoidal category, and let 
  \begin{center}
    \begin{tikzcd}[row sep = 20pt, column sep = 30pt]
    \Bl\times\Bl\ar[r,"\otimes_{\Bl}"{xshift=-10pt},bend left = 35, ""{name=U,inner sep=0pt,below,xshift=0pt}]\ar[r,"\otimes_{\Bl}"{xshift=-0pt},bend right = 35,swap, ""{name=D,inner sep=1pt}]&\Bl 
      \arrow[Rightarrow, from=U, to=D, "S", shorten <= 5pt, shorten >= 5pt] 
      \arrow[Rightarrow, from=U, to=D, swap,"\wr", shorten <= 5pt, shorten >= 5pt]
    \end{tikzcd}
    \end{center} 
    be a natural transformation. Then $S\cdot \beta$ defines a braiding if and only if $S$ satisfies the hexagon
  \begin{equation}
    \label{eqn:SHexagonRelationsOSp} 
    S_{b \otimes b', b''} \ = \ S_{b,b''}S_{b',b''} , \hspace{15mm}  S_{b,b' \otimes b''} \ = \ S_{b,b'}S_{b,b''}
    \end{equation}
    and symplectic hexagon relations
    \begin{equation} \label{eqn:SSymplecticBraidRelation}
      S_{b_1,\kappa(b_2)}^{-1} S_{b_1,b_2} \ = \ S_{\kappa(b_1),\kappa(b_2)}^{-1}S_{\kappa(b_1),b_2},
    \end{equation}
  and if $S_{1,b}=S_{b,1}=\id$. 
  \end{prop}

The proof is near-identical to that of Proposition \ref{prop:RMatrixBraiding}.

\subsubsection{Cherednik twists}

\begin{prop}\label{prop:LinearOrthosymplecticBraiding} Let $[\Al|\Bl)$ be a linear-orthosymplectic braided monoidal category, and let
  \begin{equation}
    \label{eqn:BNatTransGLOSp} 
    \begin{tikzcd}[row sep = 20pt, column sep = 30pt]
     \Al\times\Al\ar[r,"\otimes_{\Al}\cdot \sigma"{xshift=-10pt},bend left = 35, ""{name=UL,inner sep=0pt,below,xshift=0pt}]\ar[r,"\otimes_{\Al}"{xshift=-0pt},bend right = 35,swap, ""{name=DL,inner sep=1pt}]&\Al 
    &&
    \Al\times\Bl\ar[r,"\otimes_{\Al\Bl}\cdot \lambda"{xshift=-10pt},bend left = 35, ""{name=UR,inner sep=0pt,below,xshift=0pt}]\ar[r,"\otimes_{\Al\Bl}"{xshift=-0pt},bend right = 35,swap, ""{name=DR,inner sep=1pt}]&\Bl
    &&
    \Bl\times\Bl\ar[r,"\otimes_{\Bl}\cdot \sigma"{xshift=-10pt},bend left = 35, ""{name=U,inner sep=0pt,below,xshift=0pt}]\ar[r,"\otimes_{\Bl}"{xshift=-0pt},bend right = 35,swap, ""{name=D,inner sep=1pt}]&\Bl 
      \arrow[Rightarrow, from=U, to=D, "S_\Bl", shorten <= 5pt, shorten >= 5pt] 
      \arrow[Rightarrow, from=U, to=D, swap,"\wr", shorten <= 5pt, shorten >= 5pt]
      \arrow[Rightarrow, from=UL, to=DL, "S", shorten <= 5pt, shorten >= 5pt] 
      \arrow[Rightarrow, from=UL, to=DL, swap,"\wr", shorten <= 5pt, shorten >= 5pt]
      \arrow[Rightarrow, from=UR, to=DR, "T_{\Al\Bl}", shorten <= 5pt, shorten >= 5pt] 
      \arrow[Rightarrow, from=UR, to=DR, swap,"\wr", shorten <= 5pt, shorten >= 5pt]
    \end{tikzcd}
    \end{equation} 
  be natural isomorphisms. Then $S_\Al\cdot \beta_\Al,  T\cdot \eta_{\Al\Bl}, S_{\Bl}\cdot \beta_{\Bl}$ define another orthosymplectic-linear braiding on $[\Al|\Bl)$ if and only if $S$ satisfies the hexagon relations \eqref{eqn:SHexagonRelations}, $S_\Bl,\kappa$ satisfy the symplectic hexagon relations \eqref{eqn:SHexagonRelationsOSp}, \eqref{eqn:SSymplecticBraidRelation}, and $T$ satisfies the Cherednik hexagon relations with respect to $S$ and $S_\Bl$:\footnote{Here we have abused notation and denoted $\lambda=\tau \times \tau'$; the general condition for non-diagonal $\lambda$ is similar but cumbersome to write.} 
    \begin{align}
      T_{a,\tau'(b)}\cdot S_{a,\tau(a')}^{-1}\cdot T_{a',b}\cdot S_{a,a'}& \ = \ S_{\tau(a),\tau(a')}^{-1}\cdot T_{a',\tau'(b)}\cdot S_{\tau(a),a'}\cdot T_{a,b}\\
      T_{\tau(a),b'}\cdot S_{\Bl,\tau'(b),b'}^{-1}\cdot T_{a,b}\cdot S_{\Bl,b,b'} &\ = \  S_{\Bl,\tau'(b),\tau'(b')}^{-1}\cdot T_{\tau(a),b}\cdot S_{\Bl,b,\tau'(b')}\cdot T_{a,b'}.
    \end{align}
    
\end{prop}

\begin{cor} \label{cor:LinearOrthosymplecticBraiding}
  Let $\Al$ be a braided monoidal category with braided module category $\Bl$. Then postcomposing the braiding $\beta$ and module category braiding $\eta$ by the natural transformations 
    $$S_{a,a'} \ : \ a \otimes a' \ \stackrel{\sim}{\to} \ a \otimes a', \hspace{15mm} T_{a,b} \ : \ a \otimes b \ \stackrel{\sim}{\to} \ a \otimes b$$
    gives a new braided module category structure if and only if $S$ satisfied the hexagon relation and if $S,T$ satify the \textbf{finite Cherednik reflection equation} 
    $$T_{a',b} S_{a',a} T_{a,b} S_{a,a'} \ = \ S_{a',a} T_{a,b} S_{a,a'} T_{a',b}.$$
\end{cor}

\subsection{Linear-orthosymplectic algebras and Borcherds twists}

\subsubsection{Yang-Baxter Borcherds twists}\label{ssec:BorcherdsTwists} We first recap the notion of a Borcherds twist of algebraic structures, see \cite{An,La2}. Let $(A,\Delta_{A,0})$ be a bialgebra in braided monoidal category $\Al$. This induces a monoidal forgetful functor 
$$(A \Md , \otimes_\Al) \ \to \ (\Al, \otimes_\Al).$$

\begin{prop}
  If $S\in A \otimes A$ be a Yang-Baxter twist of $\Al$ which is \emph{quasitriangular} for $A$: 
  $$S\cdot \sigma(\Delta_{A,0}(a)) \ : \   \Delta_{A,0}(a)\cdot S,$$
  then $(A,\Delta_A=S\cdot \Delta_{A,0})$ defines a new bialgebra in $\Al$.
\end{prop}

\subsubsection{Remark} We can repeat the construction of Yang-Baxter twists for 2-categories other than $\Cat$. One expects that the Borcherds twist construction arises from taking the 2-category $\Cat_{\obj}$ of pairs $(\Al,a)$ of a category and object.

\subsubsection{Orthosymplectic comodules} Let $\Bl$ be a braided module category for $\Al$. Then we can define a \textit{orthosymplectic module} for bialgebra $A$ as a vector space with involution $B$ with a coaction 
$$\Delta_{AB} \ : \ B \ \to \  A \otimes B$$
colinear over $A$ and with 
$$\Delta_{AB}\cdot \kappa \ = \ (\id \otimes\kappa)\cdot \Delta_{AB}.$$
It is \textit{braided cocommutative} if $\beta\cdot \Delta_A=\Delta_A$ and $(\tau\times \id)\cdot \Delta_{AB}=\Delta_{AB}$.

\subsubsection{Linear-orthosymplectic coalgebras} The pair $(A|B]$ form a \textit{linear-orthosymplectic coalgebra} if in addition there is a coproduct 
$$\Delta_B \ : \ B \ \to \ B \otimes B$$
such that $\Delta_B \cdot \kappa = (\kappa \times \id)\cdot \Delta_B=(\id \times \kappa)\cdot \Delta_B$ and $\Delta_{AB}$ is colinear over $B$. It is \textit{braided cocommutative} if in addition $\beta_\Bl\cdot \Delta_B=\Delta_B$.

One may make the same definitions in arbitrary linear-orthosymplectic braided monoidal categories $[\Al,\Bl)$. Following \cite{La2}, we may likewise define linear-orthosymplectic localised and vertex coalgebras.

\subsubsection{Cherednik Borcherds twists} Let us fix a background linear-orthosymplectic braided monoidal category $(\Al|\Bl]$, and let $(A|B]$ be a linear-orthosymplectic bialgebra such that the induced map 
$$((A\Md|B\Md], \otimes_{(\Al|\Bl]}) \ \to \ ((\Al|\Bl], \otimes_{(\Al|\Bl]})$$
is linear-orthosymplectic monoidal. Now let 
$$S \ \in \ A \otimes A, \hspace{15mm} T \ \in \ A \otimes B, \hspace{15mm} S_B \ \in \ B \otimes B$$
 be linear-orthosymplectic Cherednik twists of $[\Al|\Bl)$ for which the coproducts $\Delta_{A,0},\Delta_{AB,0},\Delta_{B,0}$ are \textit{quasitriangular}: 
$$S\cdot \sigma(\Delta_{A,0}(a)) =  \Delta_{A,0}(a)\cdot S, \hspace{5mm} T\cdot \lambda\cdot(\Delta_{AB,0}(b)) =  \Delta_{AB,0}(b)\cdot T, \hspace{5mm} S_B\cdot \sigma(\Delta_{B,0}(b)) =  \Delta_{B,0}(b)\cdot S_B.$$
Then it follows from the graphical description

\begin{prop}
  \label{prop:OrthosymplecticBorcherdsTwistMain} 
  $(A|B]$ is a linear-orthosymplectic coalgebra in $(\Al|\Bl]$ for the coproducts
  $$\Delta_A \ = \ S\cdot \Delta_{A,0}, \hspace{15mm} \Delta_{AB} \ = \ T\cdot \Delta_{AB,0}, \hspace{15mm} \Delta_B \ = \ S_B\cdot \Delta_{B,0}.$$
  If $B$ is merely an orthosymplectic comodule for bialgebra $A$, then $\Delta_A,\Delta_{AB}$ define a new orthosymplectic comodule structure.
\end{prop}

\subsection{Graphical representation of Cherednik hexagon relations} 

\subsubsection{} 
Let $A,B$ be a linear-orthosymplectic coalgebra. Let us consider endomorphisms 
$$S \ : \ A \otimes A \ \to \  A \otimes  A, \hspace{15mm}  T \ : \  A \otimes B \ \to \  A \otimes B, \hspace{15mm} S_B \ : \ B \otimes B \ \to \ B \otimes B.$$
We will represent the coproducts $\Delta_{A,0},\Delta_{AB,0},\Delta_{B,0}$ as the following string diagrams, and postcompositions by $S$, $T$ and $S_B$ as 
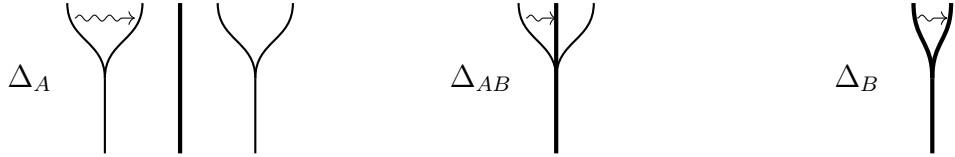
\begin{center}
\begin{tikzpicture}

\draw[black,ultra thick] (0,-1) -- (0,1);
\draw[black,thick] (0,0) .. controls +(0,0.5) and +(0,-0.5) .. (0.5,1);
\draw[black,thick] (0,0) .. controls +(0,0.5) and +(0,-0.5) .. (-0.5,1);

\node[] at (-1,0) {$\Delta_{AB}$};

\draw[->,decorate,decoration={snake,amplitude=.4mm,segment length=2mm,post length=1mm}, shorten >=0pt, shorten <=0pt] (-0.4,0.8) -- (0,0.8);

\begin{scope} 
  [xshift=-5cm]
  
\draw[black,ultra thick] (0,-1) -- (0,1);
\draw[black, thick] (1,-1) -- (1,0);
\draw[black,thick] (1,0) .. controls +(0,0.5) and +(0,-0.5) .. (1.5,1);
\draw[black,thick] (1,0) .. controls +(0,0.5) and +(0,-0.5) .. (0.5,1);
\draw[black, thick] (-1,-1) -- (-1,0);
\draw[black,thick] (-1,0) .. controls +(0,0.5) and +(0,-0.5) .. (-1.5,1);
\draw[black,thick] (-1,0) .. controls +(0,0.5) and +(0,-0.5) .. (-0.5,1);

\draw[->,decorate,decoration={snake,amplitude=.4mm,segment length=2mm,post length=1mm}, shorten >=0pt, shorten <=0pt] (-1.4,0.8) -- (-0.6,0.8);

\node[] at (-2,0) {$\Delta_A$};
 \end{scope}

 \begin{scope} 
  [xshift=5cm]
  
\draw[black,ultra thick] (0,-1) -- (0,0);
\draw[black,ultra thick] (0,0) .. controls +(0,0.5) and +(0,-0.5) .. (0.25,1);
\draw[black,ultra thick] (0,0) .. controls +(0,0.5) and +(0,-0.5) .. (-0.25,1);

\node[] at (-1,0) {$\Delta_{B}$};

\draw[->,decorate,decoration={snake,amplitude=.4mm,segment length=2mm,post length=1mm}, shorten >=0pt, shorten <=0pt] (-0.2,0.8) -- (0.2,0.8);
 \end{scope}
\end{tikzpicture}
\end{center}

\subsubsection{} \label{sssec:OrthosymplecticModuleHexagonRelations} Assuming just that $B$ is an orthosymplectic comodule for $A$ and omitting $S_B$ from the above data, the \textbf{orthosymplectic module} (or \textbf{Cherednik}) \textbf{hexagon relations} for $S$ and $T$ are:
\begin{center}
  \begin{tikzpicture}

    \draw[ thick] (-1,-1) -- (-1,0);
    \draw[black,thick] (-1,0) .. controls +(0,0.5) and +(0,-0.5) .. (-0.5,1);
    \draw[black,thick] (-1,0) .. controls +(0,0.5) and +(0,-0.5) .. (-1.5,1);
    \draw[ thick] (-2,-1) -- (-2,1);

    \draw[->,decorate,decoration={snake,amplitude=.4mm,segment length=2mm,post length=1mm}, shorten >=2pt, shorten <=2pt] (-2,-0.5) -- (-1,-0.5);
    \draw[->,decorate,decoration={snake,amplitude=.4mm,segment length=2mm,post length=1mm}, shorten >=2pt, shorten <=2pt] (-2,0.5) -- (-1.3,0.5);
    \draw[->,decorate,decoration={snake,amplitude=.4mm,segment length=2mm,post length=1mm}, shorten >=2pt, shorten <=2pt] (-2,0.7) -- (-0.6,0.7);

    \node[] at (-1.5,-1.5) {$(\id \times {\Delta_{A,0}})\cdot S \ = \  S_{12}S_{13}\cdot (\id \times {\Delta_{A,0}})$};
    \node[] at (-1.5,0) {$\rotatebox{90}{$=$}$};

\begin{scope} 
[xshift=8cm] 
\draw[ thick] (-1,-1) -- (-1,0);
\draw[black,thick] (-1,0) .. controls +(0,0.5) and +(0,-0.5) .. (-0.5,1);
\draw[black,thick] (-1,0) .. controls +(0,0.5) and +(0,-0.5) .. (-1.5,1);
\draw[ thick] (0,-1) -- (0,1);

\draw[->,decorate,decoration={snake,amplitude=.4mm,segment length=2mm,post length=1mm}, shorten >=2pt, shorten <=2pt] (-1,-0.5) -- (0,-0.5);
\draw[->,decorate,decoration={snake,amplitude=.4mm,segment length=2mm,post length=1mm}, shorten >=2pt, shorten <=2pt] (-1.4,0.7) -- (0,0.7);
\draw[->,decorate,decoration={snake,amplitude=.4mm,segment length=2mm,post length=1mm}, shorten >=2pt, shorten <=2pt] (-0.7,0.5) -- (0,0.5);

\node[] at (-0.5,-1.5) {$( {\Delta_{A,0}}\times \id)\cdot S \ = \  S_{23}S_{13}\cdot (\id \times {\Delta_{A,0}})$};
\node[] at (-0.5,0) {$\rotatebox{90}{$=$}$};
\end{scope}


\begin{scope} 
[yshift=-3.5cm,xshift=-0.5cm] 

\draw[line width=2.5pt] (-1,-1) -- (-1,1);
\draw[black,thick] (-1,0) .. controls +(0,0.5) and +(0,-0.5) .. (-0.5,1);
\draw[black,thick] (-1,0) .. controls +(0,0.5) and +(0,-0.5) .. (-1.5,1);
\draw[ thick] (-2,-1) -- (-2,1);
\draw[ thick] (0,-1) -- (0,1);

\draw[->,decorate,decoration={snake,amplitude=.4mm,segment length=2mm,post length=1mm}, shorten >=2pt, shorten <=2pt] (-2,-0.5) -- (-1,-0.5);
\draw[->,decorate,decoration={snake,amplitude=.4mm,segment length=2mm,post length=1mm}, shorten >=2pt, shorten <=2pt] (-2,0.5) -- (-1.3,0.5);
\draw[->,decorate,decoration={snake,amplitude=.4mm,segment length=2mm,post length=1mm}, shorten >=2pt, shorten <=2pt] (-2,0.7) -- (-1,0.7);
\draw[->,decorate,decoration={snake,amplitude=.4mm,segment length=2mm,post length=1mm}, shorten >=2pt, shorten <=2pt] (-2,0.9) -- (-0.45,0.9);

\node[] at (-0.5,-1.5) {$(\id \times {\Delta_{AB,0}})\cdot T \ = \  S_{12}T_{13}S_{1\tau(2)}\cdot (\id \times {\Delta_{AB,0}})$};
\node[] at (-1.5,0) {$\rotatebox{90}{$=$}$};

\begin{scope} 
 [xshift=9cm]

\draw[line width=2.5pt] (-1,-1) -- (-1,1);
\draw[black,thick] (-2,0) .. controls +(0,0.5) and +(0,-0.5) .. (-1.5,1);
\draw[black,thick] (0,0) .. controls +(0,0.5) and +(0,-0.5) .. (-0.5,1);
\draw[ thick] (-2,-1) -- (-2,1);
\draw[ thick] (0,-1) -- (0,1);

\draw[->,decorate,decoration={snake,amplitude=.4mm,segment length=2mm,post length=1mm}, shorten >=2pt, shorten <=2pt] (-2,-0.5) -- (-1,-0.5);
\draw[->,decorate,decoration={snake,amplitude=.4mm,segment length=2mm,post length=1mm}, shorten >=2pt, shorten <=2pt] (-1.5,0.9) -- (-1,0.9);
\draw[->,decorate,decoration={snake,amplitude=.4mm,segment length=2mm,post length=1mm}, shorten >=2pt, shorten <=2pt] (-2,0.7) -- (-0.4,0.7);
\draw[->,decorate,decoration={snake,amplitude=.4mm,segment length=2mm,post length=1mm}, shorten >=2pt, shorten <=2pt] (-2,0.5) -- (-1,0.5);

\node[] at (-1,-1.5) {$({\Delta_{A,0}} \times \id)\cdot T \ = \  T_{13}S_{1\tau(2)}T_{23}\cdot ({\Delta_{A,0}} \times \id)$};
\node[] at (-1.5,0) {$\rotatebox{90}{$=$}$};

 \end{scope}

\end{scope}

\end{tikzpicture}
\end{center}

\subsubsection{} Finally, if $\tau$ is an involution on $A$, the above $S,T$ are called \textbf{$\tau$-equivariant} if 

\begin{center}
  \begin{tikzpicture}

    \draw[line width=2.5pt] (-1.5,-1) -- (-1.5,1);
    \draw[ thick] (-2.1,-1) -- (-2.1,1);
    \draw[ thick] (-2.5,-1) -- (-2.5,1);
    \draw[ thick] (-0.9,-1) -- (-0.9,1);
    \draw[ thick] (-0.5,-1) -- (-0.5,1);

    \draw[->,decorate,decoration={snake,amplitude=.4mm,segment length=2mm,post length=1mm}, shorten >=2pt, shorten <=2pt] (-2.5,0) -- (-2,0);
    \draw[->,decorate,decoration={snake,amplitude=.4mm,segment length=2mm,post length=1mm}, shorten >=2pt, shorten <=2pt] (-0.9,0) -- (-0.5,0);

    \node[] at (-1.5,-1.5) {$(\tau \otimes \tau) \cdot S \ = \  S^{\textup{op}}\cdot (\tau \otimes \tau)$};
    \node[] at (-1.5,0) {$\rotatebox{0}{$=$}$};

\begin{scope} 
[xshift=9cm]

\draw[line width=2.5pt] (-1.5,-1) -- (-1.5,1);
\draw[ thick] (-2.2,-1) -- (-2.2,1);
\draw[ thick] (-0.8,-1) -- (-0.8,1);

\draw[->,decorate,decoration={snake,amplitude=.4mm,segment length=2mm,post length=1mm}, shorten >=2pt, shorten <=2pt] (-2.2,0) -- (-1.5,0);
\draw[->,decorate,decoration={snake,amplitude=.4mm,segment length=2mm,post length=1mm}, shorten >=2pt, shorten <=2pt] (-0.8,0) -- (-1.5,0);

\node[] at (-1.5,-1.5) {$(\tau \otimes \id) \cdot T \ = \ T \cdot (\tau \otimes \id)$};
\node[] at (-1.5,0.1) {$\rotatebox{0}{$=$}$};
\end{scope}

  \end{tikzpicture}
  \end{center}

\subsubsection{} 
The \textbf{linear-orthosymplectic hexagon relations} for $S,T_B,S_B$ are:

\begin{center}
  \begin{tikzpicture}

    \draw[ thick] (-1,-1) -- (-1,0);
    \draw[black,thick] (-1,0) .. controls +(0,0.5) and +(0,-0.5) .. (-0.5,1);
    \draw[black,thick] (-1,0) .. controls +(0,0.5) and +(0,-0.5) .. (-1.5,1);
    \draw[ thick] (-2,-1) -- (-2,1);

    \draw[->,decorate,decoration={snake,amplitude=.4mm,segment length=2mm,post length=1mm}, shorten >=2pt, shorten <=2pt] (-2,-0.5) -- (-1,-0.5);
    \draw[->,decorate,decoration={snake,amplitude=.4mm,segment length=2mm,post length=1mm}, shorten >=2pt, shorten <=2pt] (-2,0.5) -- (-1.3,0.5);
    \draw[->,decorate,decoration={snake,amplitude=.4mm,segment length=2mm,post length=1mm}, shorten >=2pt, shorten <=2pt] (-2,0.7) -- (-0.6,0.7);

    \node[] at (-1.5,-1.5) {$(\id \times {\Delta_{A,0}})\cdot S \ = \  S_{12}S_{13}\cdot (\id \times {\Delta_{A,0}})$};
    \node[] at (-1.5,0) {$\rotatebox{90}{$=$}$};

\begin{scope} 
[xshift=8cm] 
\draw[ thick] (-1,-1) -- (-1,0);
\draw[black,thick] (-1,0) .. controls +(0,0.5) and +(0,-0.5) .. (-0.5,1);
\draw[black,thick] (-1,0) .. controls +(0,0.5) and +(0,-0.5) .. (-1.5,1);
\draw[ thick] (0,-1) -- (0,1);

\draw[->,decorate,decoration={snake,amplitude=.4mm,segment length=2mm,post length=1mm}, shorten >=2pt, shorten <=2pt] (-1,-0.5) -- (0,-0.5);
\draw[->,decorate,decoration={snake,amplitude=.4mm,segment length=2mm,post length=1mm}, shorten >=2pt, shorten <=2pt] (-1.4,0.7) -- (0,0.7);
\draw[->,decorate,decoration={snake,amplitude=.4mm,segment length=2mm,post length=1mm}, shorten >=2pt, shorten <=2pt] (-0.7,0.5) -- (0,0.5);

\node[] at (-0.5,-1.5) {$( {\Delta_{A,0}}\times \id)\cdot S \ = \  S_{23}S_{13}\cdot (\id \times {\Delta_{A,0}})$};
\node[] at (-0.5,0) {$\rotatebox{90}{$=$}$};
\end{scope}


\begin{scope} 
[yshift=-3.5cm,xshift=-0.5cm] 

\draw[line width=2.5pt] (-1,-1) -- (-1,1);
\draw[black,thick] (-1,0) .. controls +(0,0.5) and +(0,-0.5) .. (-0.5,1);
\draw[black,thick] (-1,0) .. controls +(0,0.5) and +(0,-0.5) .. (-1.5,1);
\draw[ thick] (-2,-1) -- (-2,1);
\draw[ thick] (0,-1) -- (0,1);

\draw[->,decorate,decoration={snake,amplitude=.4mm,segment length=2mm,post length=1mm}, shorten >=2pt, shorten <=2pt] (-2,-0.5) -- (-1,-0.5);
\draw[->,decorate,decoration={snake,amplitude=.4mm,segment length=2mm,post length=1mm}, shorten >=2pt, shorten <=2pt] (-2,0.5) -- (-1.3,0.5);
\draw[->,decorate,decoration={snake,amplitude=.4mm,segment length=2mm,post length=1mm}, shorten >=2pt, shorten <=2pt] (-2,0.7) -- (-1,0.7);
\draw[->,decorate,decoration={snake,amplitude=.4mm,segment length=2mm,post length=1mm}, shorten >=2pt, shorten <=2pt] (-2,0.9) -- (-0.45,0.9);

\node[] at (-0.5,-1.5) {$(\id \times {\Delta_{AB,0}})\cdot T \ = \  S_{12}T_{13}\cdot (\id \times {\Delta_{AB,0}})$};
\node[] at (-1.5,0) {$\rotatebox{90}{$=$}$};

\begin{scope} 
 [xshift=9cm]

\draw[line width=2.5pt] (-1,-1) -- (-1,1);
\draw[black,thick] (-2,0) .. controls +(0,0.5) and +(0,-0.5) .. (-1.5,1);
\draw[black,thick] (0,0) .. controls +(0,0.5) and +(0,-0.5) .. (-0.5,1);
\draw[ thick] (-2,-1) -- (-2,1);
\draw[ thick] (0,-1) -- (0,1);

\draw[->,decorate,decoration={snake,amplitude=.4mm,segment length=2mm,post length=1mm}, shorten >=2pt, shorten <=2pt] (-2,-0.5) -- (-1,-0.5);
\draw[->,decorate,decoration={snake,amplitude=.4mm,segment length=2mm,post length=1mm}, shorten >=2pt, shorten <=2pt] (-1.5,0.9) -- (-1,0.9);
\draw[->,decorate,decoration={snake,amplitude=.4mm,segment length=2mm,post length=1mm}, shorten >=2pt, shorten <=2pt] (-2,0.7) -- (-0.4,0.7);
\draw[->,decorate,decoration={snake,amplitude=.4mm,segment length=2mm,post length=1mm}, shorten >=2pt, shorten <=2pt] (-2,0.5) -- (-1,0.5);

\node[] at (-1,-1.5) {$({\Delta_{A,0}} \times \id)\cdot T \ = \  T_{13}T_{23}\cdot ({\Delta_{A,0}} \times \id)$};
\node[] at (-1.5,0) {$\rotatebox{90}{$=$}$};

 \end{scope}

\end{scope}

\begin{scope} 
[yshift=-7cm,xshift=-0.5cm]

\draw[line width=2.5pt] (-1,-1) -- (-1,0);
\draw[black,line width=2.5pt] (-1,0) .. controls +(0,0.5) and +(0,-0.5) .. (-0.75,1);
\draw[black,line width=2.5pt] (-1,0) .. controls +(0,0.5) and +(0,-0.5) .. (-1.25,1);
\draw[ thick] (-2,-1) -- (-2,1);
\draw[ thick] (0,-1) -- (0,1);

\draw[->,decorate,decoration={snake,amplitude=.4mm,segment length=2mm,post length=1mm}, shorten >=2pt, shorten <=2pt] (-2,-0.5) -- (-1,-0.5);
\draw[->,decorate,decoration={snake,amplitude=.4mm,segment length=2mm,post length=1mm}, shorten >=2pt, shorten <=2pt] (-2,0.7) -- (-1.2,0.7);
\draw[->,decorate,decoration={snake,amplitude=.4mm,segment length=2mm,post length=1mm}, shorten >=2pt, shorten <=2pt] (-2,0.9) -- (-0.7,0.9);

\node[] at (-0.5,-1.5) {$(\id \times {\Delta_{B,0}})\cdot T \ = \  T_{12}T_{13}$};
\node[] at (-1.5,0) {$\rotatebox{90}{$=$}$};
 \end{scope}


\begin{scope} 
[yshift=-10.5cm]

\draw[ line width=2.5pt] (-1,-1) -- (-1,0);
\draw[black,line width=2.5pt] (-1,0) .. controls +(0,0.5) and +(0,-0.5) .. (-0.75,1);
\draw[black,line width=2.5pt] (-1,0) .. controls +(0,0.5) and +(0,-0.5) .. (-1.25,1);
\draw[ line width=2.5pt] (-2,-1) -- (-2,1);

\draw[->,decorate,decoration={snake,amplitude=.4mm,segment length=2mm,post length=1mm}, shorten >=2pt, shorten <=2pt] (-2,-0.5) -- (-1,-0.5);
\draw[->,decorate,decoration={snake,amplitude=.4mm,segment length=2mm,post length=1mm}, shorten >=2pt, shorten <=2pt] (-2,0.5) -- (-1.1,0.5);
\draw[->,decorate,decoration={snake,amplitude=.4mm,segment length=2mm,post length=1mm}, shorten >=2pt, shorten <=2pt] (-2,0.7) -- (-0.8,0.7);

\node[] at (-1.5,-1.5) {$(\id \times {\Delta_{B,0}})\cdot S_B \ = \  S_{B,12}S_{B,13}\cdot (\id \times {\Delta_{B,0}})$};
\node[] at (-1.5,0) {$\rotatebox{90}{$=$}$};

\begin{scope} 
[xshift=8cm] 
\draw[ line width=2.5pt] (-1,-1) -- (-1,0);
\draw[black,line width=2.5pt] (-1,0) .. controls +(0,0.5) and +(0,-0.5) .. (-0.75,1);
\draw[black,line width=2.5pt] (-1,0) .. controls +(0,0.5) and +(0,-0.5) .. (-1.25,1);
\draw[ line width=2.5pt] (0,-1) -- (0,1);

\draw[->,decorate,decoration={snake,amplitude=.4mm,segment length=2mm,post length=1mm}, shorten >=2pt, shorten <=2pt] (-1,-0.5) -- (0,-0.5);
\draw[->,decorate,decoration={snake,amplitude=.4mm,segment length=2mm,post length=1mm}, shorten >=2pt, shorten <=2pt] (-1.2,0.7) -- (0,0.7);
\draw[->,decorate,decoration={snake,amplitude=.4mm,segment length=2mm,post length=1mm}, shorten >=2pt, shorten <=2pt] (-0.8,0.5) -- (0,0.5);

\node[] at (-0.5,-1.5) {$( {\Delta_{B,0}}\times \id)\cdot S_B \ = \  S_{B,23}S_{B,13}\cdot (\id \times {\Delta_{B,0}})$};
\node[] at (-0.5,0) {$\rotatebox{90}{$=$}$};
\end{scope}
 \end{scope}


\begin{scope} 
 [yshift=-14cm]

 \draw[ line width=2.5pt] (-1,-1) -- (-1,0);
\draw[line width=2.5pt] (-1,-1) -- (-1,1);
\draw[black,thick] (-1,0) .. controls +(0,0.5) and +(0,-0.5) .. (-0.5,1);
\draw[black,thick] (-1,0) .. controls +(0,0.5) and +(0,-0.5) .. (-1.5,1);
\draw[line width=2.5pt] (-2,-1) -- (-2,1);

 \draw[->,decorate,decoration={snake,amplitude=.4mm,segment length=2mm,post length=1mm}, shorten >=2pt, shorten <=2pt] (-2,-0.5) -- (-1,-0.5);
 \draw[->,decorate,decoration={snake,amplitude=.4mm,segment length=2mm,post length=1mm}, shorten >=2pt, shorten <=2pt] (-2,0.5) -- (-1.2,0.5);
 \draw[->,decorate,decoration={snake,amplitude=.4mm,segment length=2mm,post length=1mm}, shorten >=2pt, shorten <=2pt] (-2,0.7) -- (-1.1,0.7);
 \draw[->,decorate,decoration={snake,amplitude=.4mm,segment length=2mm,post length=1mm}, shorten >=2pt, shorten <=2pt] (-2,0.9) -- (-0.5,0.9);
 
 \node[] at (-1.5,-1.5) {$(\id \times {\Delta_{AB,0}})\cdot S_B \ = \  S_{B,12}S_{B,13}\cdot (\id \times {\Delta_{AB,0}})$};
 \node[] at (-1.5,0) {$\rotatebox{90}{$=$}$};
 
 \begin{scope} 
 [xshift=8cm] 
 \draw[ line width=2.5pt] (-1,-1) -- (-1,0);
 \draw[black,line width=2.5pt] (-1,0) .. controls +(0,0.5) and +(0,-0.5) .. (-0.75,1);
 \draw[black,line width=2.5pt] (-1,0) .. controls +(0,0.5) and +(0,-0.5) .. (-1.25,1);
 \draw[ line width=2.5pt] (0,-1) -- (0,1);

 \draw[->,decorate,decoration={snake,amplitude=.4mm,segment length=2mm,post length=1mm}, shorten >=2pt, shorten <=2pt] (-1,-0.5) -- (0,-0.5);
 \draw[->,decorate,decoration={snake,amplitude=.4mm,segment length=2mm,post length=1mm}, shorten >=2pt, shorten <=2pt] (-1.2,0.7) -- (0,0.7);
 \draw[->,decorate,decoration={snake,amplitude=.4mm,segment length=2mm,post length=1mm}, shorten >=2pt, shorten <=2pt] (-0.8,0.5) -- (0,0.5);

 \node[] at (-0.5,-1.5) {$( {\Delta_{AB,0}}\times \id)\cdot S_B \ = \  S_{B,23}S_{B,13}\cdot (\id \times {\Delta_{AB,0}})$};
 \node[] at (-0.5,0) {$\rotatebox{90}{$=$}$};
 \end{scope}
 \end{scope}

  \end{tikzpicture}
  \end{center}

\newpage

\section{Partial translation equivariance} \label{ssec:TranslationEquivariance}
 
\subsubsection{Definition} If $\Gt$ is a group acting on a space $\Xl$ and $\Ul \subseteq \Xl$ is an open subspace, then although $\Gt$ does not act on $\Ul$ we have a correspondence 
\begin{center}
\begin{tikzcd}[row sep = {30pt,between origins}, column sep = {45pt, between origins}]
 & (G \times \Ul)_\circ \ar[ld,"j"']  \ar[rd,"a"]  & \\ 
G \times \Ul & & \Ul
\end{tikzcd}
\end{center}
where $(G \times \Ul)_\circ = (G \times \Ul)\times_{\Xl} \Ul$ is the open locus of points $(g,u)$ with $g(u)\in \Ul$. This satisfies a lax associativity condition: that taking the pullbacks of
\begin{center}
\begin{tikzcd}[row sep = {30pt,between origins}, column sep = {45pt, between origins}]
 & G \times (G \times \Ul)_\circ \ar[ld,"\id \times j"'] \ar[rd,"\id \times a"] && (G \times \Ul)_\circ \ar[ld,"j"']  \ar[rd,"a"]  & &
 &[10pt] G \times G \times \Ul \ar[ld,equals] \ar[rd,"m \times \id"] && (G \times \Ul)_\circ \ar[ld,"j"']  \ar[rd,"a"]  & \\ 
G \times G \times \Ul&&G \times \Ul & & \Ul&
G \times G \times \Ul&&G \times \Ul & & \Ul
\end{tikzcd}
\end{center}
there is a map from the left to the right pullback; the left pullback classifies points $(g_1,g_2,u)$ with $\Gt_2(u)\in \Ul$ and $\Gt_1(g_2(u))\in \Ul$ and the right pullback classifies points $(g_1,g_2,u)$ with  with $\Gt_1(g_2(u))\in \Ul$.

\begin{lem}
Viewing $\Gt$ as a group object in $\Stk^{\corr}$, the above correspondence gives a lax group action of $\Gt$ on $\Ul$.
\end{lem}

We can thus define $\Gt$-equivariant coherent sheaves, D-modules, etc. on $\Ul$.

\subsubsection{} Loosely speaking a partially $\Gt$-equivariant sheaf is a sheaf $\El$ of the appropriate type with isomorphisms
$$\varphi_g \ : \ \El_u \ \stackrel{\sim}{\to} \ \El_{g(u)}$$
whenever $g(u)\in \Ul$, satisfying a cocycle condition $\varphi_{\Gt_1g_2}\simeq \varphi_{\Gt_1}\cdot \varphi_{\Gt_2}$ over the locus of points $(g_1,g_2,u)$ with $\Gt_2(u)\in \Ul$ and $\Gt_1(g_2(u))\in \Ul$.

\newpage 
 \section{Decomposition spaces, categories and algebras} \label{sec:Decomposition}

 \subsubsection{Definitions} We recall definitions from \cite{La2} for convenience. Fix $\Xl$ an element in an appropriate category of spaces, for instance we will now take the $\infty$-category $\PreStk$ of prestacks. Form the $(\infty,2)$-category $\PreStk^{\corr}$ whose objects are prestacks, one-morphisms are correspondences, and two-morphisms are maps between correspondences. 

 \subsubsection{}  A \textit{decomposition space} is an element 
 $$\Xl \ \in \ \Alg(\PreStk^{\corr}),$$
 in other words it is a prestack $\Xl$ together with \textit{multiplication} and \textit{unit} correspondences 
 \begin{center}
 \begin{tikzcd}[row sep = {30pt,between origins}, column sep = {45pt, between origins}]
  & \Ml \ar[rd] \ar[ld] & && \El \ar[rd] \ar[ld] &\\
  \Xl \times \Xl & & X  &\pt & & \Xl
  \end{tikzcd}
  \end{center}
  satisfying associativity and unit conditions as maps in $\PreStk^{\corr}$, plus higher coherences.

  \subsubsection{} Likewise, we may form the category $\PreStk_{\Cat}^{\corr}$ whose objects are prestacks $\Xl$ equipped with a sheaf of categories $\Cl \in \ShvCat(\Xl)$. A \textit{decomposition category} is an element
  $$(\Xl,\Cl) \ \in \ \Alg(\PreStk_{\Cat}^{\corr})$$
  and a section $A\in \Gamma(\Xl,\Cl)$ is a \textit{decomposition algebra} on $\Cl$ if it defines an algebra 
  $$(\Xl,\Cl,A) \ \in \ \Alg(\PreStk_{\Cat,\textup{sec}}^{\corr})$$
  in the category $\PreStk_{\Cat,\textup{sec}}^{\corr}$ of prestacks equipped with a sheaf of categories and a section.

  \subsubsection{}   We can likewise consider $\Eb_n\Ag(\PreStk^{\corr}_{(\Cat),(\textup{sec})})$, giving \textit{$\Eb_n$-decomposition spaces}, \textit{categories}, and \textit{algebras}.

\end{document}